\documentclass[reqno,12pt,letterpaper]{amsart}
\usepackage{amsfonts,amsrefs,latexsym,amsmath,amssymb}
\usepackage{mathrsfs,MnSymbol}
\usepackage{url}
\usepackage{upgreek}
\usepackage{fancyhdr}
\usepackage{scalefnt}
\usepackage{url}
\usepackage[usenames,dvipsnames]{color}
\usepackage[colorlinks=true,linkcolor=Red,citecolor=Green]{hyperref}

\newtheorem{proposition}{Proposition}[section]
\newtheorem{lemma}[proposition]{Lemma}

\newtheorem{theorem}{Theorem}[section]

\theoremstyle{definition}
\newtheorem{definition}{Definition}[section]
\newtheorem{remark}{Remark}[section]

\numberwithin{equation}{section}

\newcommand{\mydim}{D}

\newcommand{\TBoot}{T_{(Boot)}}

\newcommand{\Worstexp}{q}
\newcommand{\Blowupexp}{A}
\newcommand{\Room}{\upsigma}
\newcommand{\Pos}{P}

\newcommand{\mycong}{\simeq}
\newcommand{\mycongstar}{\overset{*}{\simeq}}

\newcommand{\leftexp}[2]{{\vphantom{#2}}^{#1}{#2}}

\newcommand{\KasnerMetric}{\widetilde{g}}
\newcommand{\KasnerSecondFund}{\widetilde{k}}

\newcommand{\SecondFund}{k}
\newcommand{\Riemann}{\mbox{\upshape Riem}}
\newcommand{\Ric}{\mbox{\upshape Ric}}
\newcommand{\Sc}{\mbox{\upshape R}}

\newcommand{\gfour}{\mathbf{g}}

\newcommand{\Chfour}{\mathbf{\Gamma}}

\newcommand{\Ricfour}{\mathbf{Ric}}
\newcommand{\Riemfour}{\mathbf{Riem}}
\newcommand{\Rfour}{\mathbf{R}}
\newcommand{\Tfour}{\mathbf{T}}

\newcommand{\Nml}{\hat{\mathbf{N}}}

\newcommand{\Dfour}{\mathbf{D}}

\newcommand{\MetricLownorm}{\mathbb{L}_{(g,\SecondFund)}}
\newcommand{\LapseLownorm}{\mathbb{L}_{(n)}}
\newcommand{\MetricHighnorm}{\mathbb{H}_{(g,\SecondFund)}}
\newcommand{\LapseHighnorm}{\mathbb{H}_{(n)}}

\newcommand{\supMetricLownorm}{\bar{\mathbb{L}}_{(Met)}}

\newcommand{\supMetricHighnorm}{\bar{\mathbb{H}}_{(Met)}}

\newcommand{\Euc}{E}

\topmargin - .23 in

\begin{document}

%\scalefont{1.10}

\title{On the nature of Hawking's incompleteness for the Einstein-vacuum equations:
%Stable Spacelike
%spacelike singularity formation for the 
%Einstein-vacuum singularities evolving from
The regime of moderately spatially anisotropic initial data}
\author{Igor Rodnianski$^{*}$
\and 
Jared Speck$^{**}$}
%\date{\today}
%\email{jspeck@math.mit.edu}

\thanks{$^*$Princeton University, Department of Mathematics, Fine Hall, Washington Road, Princeton, NJ 08544-1000, USA. \texttt{irod@math.princeton.edu}}

\thanks{$^{**}$Massachusetts Institute of Technology, Department of Mathematics, 77 Massachusetts Ave, Room 2-265, Cambridge, MA 02139-4307, USA. \texttt{jspeck@math.mit.edu}}

\thanks{$^{*}$ IR gratefully acknowledges support from NSF grant \# DMS-1001500.}	

\thanks{$^{**}$ JS gratefully acknowledges support from NSF grant \# DMS-1162211 and
from a Solomon Buchsbaum grant administered by the Massachusetts Institute of Technology.
}	

\begin{abstract}
In the mathematical physics literature, 
there are heuristic arguments,
going back three decades,
suggesting that for an open set of initially smooth solutions 
to the Einstein-vacuum equations in high dimensions,
stable, approximately monotonic 
curvature singularities can dynamically form along a spacelike hypersurface.
In this article, we study the Cauchy problem
and give a rigorous proof of this phenomenon in sufficiently high dimensions,
thereby providing the first constructive proof of stable curvature blowup 
(without symmetry assumptions)
along a spacelike hypersurface 
as an effect of pure gravity.
Our proof applies to an open subset of regular initial data 
satisfying the assumptions of 
Hawking's celebrated ``singularity'' theorem,
which shows that the solution is geodesically incomplete 
but does not reveal the nature of the incompleteness.
Specifically, our main result is a proof of the dynamic stability 
of the Kasner curvature singularity
for a subset of Kasner solutions 
whose metrics exhibit only moderately (as opposed to severely) 
spatially anisotropic behavior.
Of independent interest is our method of proof, 
which is more robust than earlier approaches in that
\textbf{i)} it does not rely on approximate monotonicity identities
and \textbf{ii)} it accommodates the possibility that the solution develops
very singular high-order spatial derivatives,
whose blowup rates are allowed to be, 
within the scope of our bootstrap argument,
much worse than those of the base-level
quantities driving the fundamental blowup.
For these reasons, our approach could be used
to obtain similar blowup results for various Einstein-matter
systems in any number of spatial dimensions for solutions corresponding to 
an open set of moderately spatially anisotropic initial data,
thus going beyond the nearly spatially isotropic regime
treated in earlier works.
\bigskip


\bigskip


\end{abstract}

\maketitle

\centerline{\today}

\setcounter{tocdepth}{2}
\tableofcontents

\newpage

\section{Introduction} \label{S:INTRO}
Hawking's celebrated ``singularity'' theorem 
(see, for example, \cite{rW1984}*{Theorem~9.5.1}) 
shows that an interestingly large
set of initial data
for the Einstein-vacuum\footnote{Hawking's theorem also applies to any Einstein-matter system
whose energy-momentum tensor verifies the strong energy condition.} 
equations leads to geodesically incomplete solutions. The chief drawback of this result is that
the proof is by contradiction and therefore does not reveal the nature of the incompleteness;
see Subsect.\,\ref{SS:PREVIOUSWORKS} for further discussion.
In this article, for an open\footnote{By ``open,'' we mean relative to a suitable Sobolev norm topology.} 
subset of regular initial data in high spatial dimensions
that satisfy the assumptions of Hawking's theorem,
we show that the solution's incompleteness is due to the formation of a Big Bang,
that is, the formation of a curvature singularity along a spacelike hypersurface.
For more detailed statements of our main results,
readers can jump ahead to Theorem~\ref{T:MAINTHMROUGHVERSION} 
for a rough summary or to Theorem~\ref{T:MAINTHM} for the precise versions.

Before proceeding, we note that the Einstein-vacuum equations in 
$\mydim$ spatial dimensions are
\begin{align}  \label{E:EINSTEININTRO}
	\Ricfour_{\mu \nu} - \frac{1}{2}\Rfour \gfour_{\mu \nu} & = 0, && (\mu,\nu = 0,1,\cdots,\mydim),
\end{align}
where $\Ricfour$ is the Ricci curvature of the Lorentzian spacetime metric 
$\gfour$ (which has signature $(-,+,\cdots,+)$),
and $\Rfour = \Ricfour_{\ \alpha}^{\alpha}$ is\footnote{We summarize our index and summation conventions in Subsubsect.\,\ref{SSS:INDICES}.} 
its scalar curvature. Throughout, $\mathcal{M}$ denotes the
$1 + \mydim$-dimensional spacetime manifold, on which the system
\eqref{E:EINSTEININTRO} is posed. 
In this article, the spacetimes under study are \emph{cosmological}, 
i.e., they have compact spacelike Cauchy hypersurfaces.
In particular, the solutions that we study are such that
$\mathcal{M} = I \times \mathbb{T}^{\mydim}$, 
where $I$ is an interval of time and 
$\mydim$ is the standard $\mydim$-dimensional torus 
(i.e., $[0,1]^{\mydim}$ with the endpoints identified).

Our work provides the first constructive proof of stable curvature blowup along a spacelike hypersurface 
\underline{as an effect of pure gravity} (i.e., without the presence of matter)
for Einstein's equations in more than one spatial dimension 
without symmetry assumptions. The present work can be viewed as complementary to
our previous works \cites{iRjS2018,iRjS2014b,jS2017e},
in which we proved similar results in the presence of scalar field or stiff fluid\footnote{A stiff fluid 
is such that the speed of sound is equal to the speed of light, i.e., it 
obeys the equation of state $p = \rho$, where $p$ is the pressure and $\rho$ is the density. The stiff fluid is a 
generalization of the scalar field in the sense that it reduces to the scalar field matter model when the fluid's vorticity vanishes.}
matter. As we explain below, out of necessity, we had to devise a new and more robust analytic 
framework to treat the Einstein-vacuum equations, since many of the special structures
that we exploited in \cites{iRjS2018,iRjS2014b,jS2017e} are not available in the vacuum case.

For the singularities that we study here, 
the blowup is rather ``controlled'' in the sense that
the solutions exhibit approximately monotonic behavior 
as the singularity is approached. For example, relative to an appropriately constructed
time coordinate $t$, 
the Kretschmann scalar 
$\Riemfour_{\alpha \beta \gamma \delta} \Riemfour^{\alpha \beta \gamma \delta}$
blows up like $C t^{-4}$ as $t \downarrow 0$
(see Lemma~\ref{L:KRETSCHMANNSCALARESTIMATE} for the precise statement), 
where $\Riemfour$ is the Riemann curvature of $\gfour$.
A main theme of this paper is that the monotonic behavior is not just a curiosity, 
but rather it lies at the heart of our analysis.
Heuristic evidence for the existence of a large family of (non-spatially homogeneous) 
approximately monotonic spacelike-singularity-forming 
Einstein-vacuum solutions for large $\mydim$
goes back more than $30$ years to the work \cite{jDmHpS1985}, which was preceded by
other works of a similar vein that we review in Subsubsect.\,\ref{SSS:HEURISTICANDNUMERICAL}.
In \cite{jDmHpS1985} and in the present work as well, the largeness of $\mydim$
is used only to ensure the existence of background solutions with certain quantitative properties
(in our case Kasner solutions whose exponents verify the condition \eqref{E:KASNEREXPONENTSNOTTOOBIG});
for the Einstein-vacuum equations in low space dimensions,
the only obstacle to the existence of such solutions 
is the Hamiltonian constraint equation \eqref{E:GAUSSINTRO}.
Given such a background solution, 
the rest of our analysis is essentially dimensionally independent.\footnote{Aside from the number of derivatives that we need to close the estimates.}
Although our main theorem applies only when $\mydim \geq 38$,
our approach here is of interest in itself since 
it is more robust compared to prior works on stable blowup for various Einstein-matter systems, 
and since it has further applications, for example in three spatial dimensions; 
see the end of Subsect.\,\ref{SS:ADDITIONALCONTEXT}
for further discussion on this point.

\subsection{The evolution problem and the initial data}
\label{SS:EVOLUTIONPROBLEM}
Before further describing our results, we first provide some background material 
on the evolution problem for Einstein's equations. 
The foundational works \cites{CB1952,cBgR1969}
collectively imply that the Einstein-vacuum equations \eqref{E:EINSTEININTRO}
have an evolution problem formulation in which
all sufficiently regular initial data verifying the constraint equations \eqref{E:GAUSSINTRO}-\eqref{E:CODAZZIINTRO}
launch a unique\footnote{More precisely, 
the maximal globally hyperbolic development is unique up to isometry in the class of globally hyperbolic spacetimes.}
maximal classical solution $(\mathcal{M}_{(Max)},\gfour_{(Max)})$, known as the 
\emph{maximal globally hyperbolic development of the data}. 
Below we will discuss the evolution equations. For now, we recall that 
the initial data are $(\Sigma_1,\mathring{g},\mathring{\SecondFund})$,
where $\Sigma_1$ is as $\mydim$-dimensional manifold,
$\mathring{g}$ is a Riemannian metric on $\Sigma_1$, and $\mathring{\SecondFund}$ is a symmetric type $\binom{0}{2}$
$\Sigma_1$-tangent tensorfield. 
The subscript ``$1$'' on $\Sigma_1$ emphasizes that in our main theorem,
$\Sigma_1$ will be identified with a hypersurface of constant time $1$.
$\mathring{g}$ and $\mathring{\SecondFund}$ represent, 
respectively, 
the first and second fundamental
form of $\Sigma_1$,
viewed as a Riemannian submanifold of the spacetime
%$(\mathcal{M},\gfour)$ 
to be constructed. 
The constraints are the well-known \emph{Gauss} and \emph{Codazzi} equations
(which are often referred to, respectively, as the Hamiltonian and momentum constraint equations):
\begin{subequations}
\begin{align}
	\mathring{\Sc} 
	- 
	\mathring{\SecondFund}_{\ b}^a \mathring{\SecondFund}_{\ a}^b
	+ 
	(\mathring{\SecondFund}_{\ a}^a)^2 
	& = 0, 
	\label{E:GAUSSINTRO} \\
	\nabla_a \mathring{\SecondFund}_{\ j}^a - \nabla_j \mathring{\SecondFund}_{\ a}^a  
		& = 0. 
		\label{E:CODAZZIINTRO}
\end{align}
\end{subequations}
In \eqref{E:GAUSSINTRO}-\eqref{E:CODAZZIINTRO},
$\nabla$ denotes the Levi--Civita connection of $\mathring{g}$.

In analyzing solutions to \eqref{E:EINSTEININTRO}, we will use 
the well-known constant-mean-curvature-transported spatial coordinates
gauge. The corresponding PDE system involves hyperbolic evolution equations for the first
and second fundamental forms of the constant-time hypersurface $\Sigma_t$
coupled to an elliptic PDE for the lapse $n >0$, 
which is defined by $\gfour(\partial_t,\partial_t) = - n^2$.
See Subsubsect.\,\ref{SSS:BASICINGREDIENTS} for a review of this gauge.
Here we only note that 
\emph{the elliptic PDE for $n$ is essential for synchronizing the
singularity across space}. To employ this gauge in the context of our main theorem,
we assume that the initial data verify the constant-mean-curvature (CMC from now on) condition
\begin{align} \label{E:CMCDATA}
	\mathring{\SecondFund}_{\ a}^a = - 1.
\end{align}

\begin{remark}[\textbf{The CMC assumption is not a further restriction}]
\label{R:NONEEDFORCMC}
In the context of our main theorem, 
equation \eqref{E:CMCDATA} should not be viewed 
as a further restriction
on the initial data since for the near-Kasner solutions that we study,
we can always construct a CMC hypersurface
lying near the initial data hypersurface. This can be achieved by
making minor modifications to the arguments given in
\cite{iRjS2014b}, where we constructed a CMC hypersurface 
in a closely related context.
\end{remark}

\subsection{Some additional context and preliminary comments on the new ideas in the proof}
\label{SS:ADDITIONALCONTEXT}
The aforementioned results \cites{CB1952,cBgR1969}, while philosophically
of great importance, reveal very little information about the nature of the
maximal globally hyperbolic development. In our main theorem,
we derive sharp information about the
maximal globally hyperbolic development
of an open set of nearly spatially homogeneous initial data 
(i.e., initial data that vary only slightly in space at each fixed time) 
given along the $\mydim$-dimensional torus 
$\Sigma_1 = \mathbb{T}^{\mydim}$, where $\mydim \geq 38$. 
The largeness of $\mydim$ allows us to consider initial data that are
only moderately (as opposed to severely) spatially anisotropic, 
in the sense that the data are close to a Kasner solution \eqref{E:KASNERSOLUION}
whose Kasner exponents verify \eqref{E:KASNEREXPONENTSNOTTOOBIG}
(see below for further discussion).
Broadly speaking, the main analytic themes of our work can be summarized as follows
(see Subsect.\,\ref{SS:MAINIDEASINPROOF} for a more detailed outline of our proof):
\begin{quote}
	The amount of spatial anisotropy exhibited by the solutions under study
	is tied to the strength of various nonlinear error terms
	that depend on spatial derivatives.
	Below a certain threshold, the error terms become
	subcritical (in the sense of the strength of their singularity)
	with respect to the main terms. This
	allows us to give a perturbative proof of stable blowup through
	a bootstrap argument, where the blowup is driven by
	``ODE-type'' \underline{singular terms that are linear in time derivatives}, 
	and, at least at the low derivative levels,
	the nonlinear spatial derivative terms are strictly weaker than
	the singular linear time derivative terms.
	Put differently, at the low derivative levels, 
	the spatial derivative terms remain negligible throughout the evolution,
	a phenomenon that, in the general relativity literature, 
	is sometimes referred to 
	as \emph{asymptotically velocity term dominated}
	(AVTD) behavior.
	In contrast, at the high derivative levels, the spatial derivatives can be
	very singular; 
	bounding the maximum strength of this singular behavior and 
	showing that it does not
	propagate down into the low derivative levels together 
	constitute the main technical challenge of the proof.
\end{quote}

We stress that in the absence of matter,
spatially homogeneous and isotropic Big Bang solutions \emph{do not exist};
they are precluded by the Hamiltonian constraint equation \eqref{E:GAUSSINTRO}.
Hence, it is out of necessity that our main theorem concerns spatially anisotropic solutions.
In our previous works \cites{iRjS2018,iRjS2014b,jS2017e}, we proved related 
stable blowup results
for the Einstein-scalar field and Einstein-stiff fluid systems in three spatial dimensions.
In those works, the presence of matter allowed for the existence of 
spatially homogeneous, 
spatially isotropic Big-Bang-containing solutions,
specifically, the famous Friedmann--Lema\^{\i}tre--Robertson--Walker (FLRW) solutions,
whose perturbations we studied. The approximate spatial isotropy 
of the perturbed solutions lied at the heart of the analysis of
\cites{iRjS2018,iRjS2014b,jS2017e}, and we therefore had to develop a new approach
to handle the moderately spatially anisotropic solutions under study here.
We now quickly highlight the main new contributions 
of the present work; in Subsect.\,\ref{SS:MAINIDEASINPROOF}, 
we provide a more in depth overview of 
how they fit into our analytical framework.

\begin{itemize}
	\item In \cites{iRjS2018,iRjS2014b,jS2017e}, the \underline{presence} 
		and the precise structure of the matter model was used
		in several key places, leaving open the possibility that the blowup was
		essentially ``caused'' by the matter. In contrast, our present work
		shows that for suitable initial data (which exist for $\mydim \geq 38$),
		stable curvature blowup can occur in general relativity
		as an effect of pure gravity.
	\item Our previous works \cites{iRjS2018,iRjS2014b,jS2017e} relied on
		\emph{approximate monotonicity identities}, which were $L^2$-type integral
		identities in which the special structure of the matter models and the 
		spatially isotropic nature of the FLRW background solutions were exploited
		to exhibit miraculous cancellations. The cancellations were such that
		dangerous singular error integrals with large coefficients were replaced, up to small error terms,
		with \emph{coercive ones}. That is, the identities led to the availability of friction-type
		spacetime integrals that were used to control various error terms up to the singularity.
		The net effect was that we were able to prove that 
		the very high spatial derivatives of the solution are
		not much more singular than the low-order derivatives;
		this was a fundamental ingredient in the analytic
		framework that we used to control error terms.
		In contrast, in the present article, 
		\textbf{we completely avoid relying on approximate monotonicity identities},
		which might not even exist for the kinds of solutions that we study here.
		This leads, at the high spatial derivative levels, to very singular energy estimates near the Big Bang,
		which our proof is able to accommodate; see Subsect.\,\ref{SS:MAINIDEASINPROOF} for 
		an outline of the main ideas behind the estimates.
		For these reasons, 
		our approach here is significantly more robust compared to \cites{iRjS2018,iRjS2014b,jS2017e}
		and could be used, for example, to substantially enlarge
		the set of initial data for the Einstein-scalar field and Einstein-stiff fluid systems
		in three spatial dimensions
		that are guaranteed to lead to curvature blowup.
		More precisely, it could be used to 
		prove stable blowup for perturbations of generalized Kasner solutions
		(i.e., non-vacuum Kasner solutions)
		to these systems
		that exhibit moderate spatial anisotropy.
		We stress that it might not be possible to extend these stable blowup results
		to Einstein-matter systems that feature $\gfour$-timelike characteristics,
		such as the Euler-Einstein system under a general equation of state;
		in \cites{iRjS2018,iRjS2014b,jS2017e}, the analysis relied on the
		fact that for the matter models studied, the characteristics of the system
		are exactly the null hypersurfaces of $\gfour$,
		a feature that is tied to the following crucial structural property: 
		the absence of time-derivative-involving error-terms
		in various equations. 
		Readers can consult \cite{iRjS2018} for further discussion on this point.
\end{itemize}

\subsection{Kasner solutions and a rough summary of the main theorem}
\label{SS:KASNERINITIALDATA}
As we have mentioned, in our main theorem, we consider initial data given on 
\begin{align} \label{E:DATAHYPERSURFACE}
	\Sigma_1 
	& := \mathbb{T}^{\mydim},
\end{align}
where $\mydim$ is the standard $\mydim$-dimensional torus (i.e., $[0,1]^{\mydim}$ with the endpoints identified) 
and, relative to the coordinates that we use in proving our main results,
$\Sigma_1$ will be identified with a spacelike hypersurface of constant time $1$,
i.e., $\Sigma_1 = \lbrace 1 \rbrace \times \mathbb{T}^{\mydim}$.
Our assumption on the topology of $\Sigma_1$ is for convenience and is not of
fundamental importance; we expect that similar results hold for other spatial topologies,
much in the same way that the blowup results
of \cites{iRjS2018,iRjS2014b} in the case of $\mathbb{T}^3$
spatial topology were extended to the case of $\mathbb{S}^3$
spatial topology in \cite{jS2017e}
(see Subsubsect.\,\ref{SSS:CONSTRUCTIVERESULTSNOSYMMETRY} for further discussion of these works).
Our main theorem addresses the evolution of perturbations of the initial data (at time $1$)
of the \emph{Kasner solutions} \cite{eK1921} 
\begin{align} \label{E:KASNERSOLUION}
	\widetilde{\gfour}
	& := - dt \otimes dt 
	+ 
	\KasnerMetric,
	&& (t,x) \in (0,\infty) \times \mathbb{T}^{\mydim},
\end{align}
where 
\begin{align} \label{E:KASNERSPATIALMETRIC}
	\KasnerMetric
	& :=
		\sum_{i=1}^{\mydim} t^{2q_i} dx^i \otimes dx^i,
\end{align}
and the constants $q_i \in [-1,1]$, known as the \emph{Kasner exponents},
are constrained by
\begin{subequations}
\begin{align} \label{E:KASNEREXPONENTSSUMTOONE}
	\sum_{i=1}^{\mydim} q_i 
	& = 1,
		\\
	\sum_{i=1}^{\mydim} q_i^2
	& = 1.
	\label{E:KASNEREXPONENTSSQUARESSUMTOONE}
\end{align} 
In the rest of the paper, we will also use the notation
\begin{align} \label{E:KASNERSECONDFUND}
	\KasnerSecondFund
	& :=
		- t^{-1} \sum_{i=1}^{\mydim} q_i \partial_i \otimes dx^i
\end{align}
to denote the Kasner mixed second fundamental form, i.e.,
$\KasnerSecondFund_{\ j}^i = - \frac{1}{2} (\KasnerMetric^{-1})^{ia} \partial_t \KasnerMetric_{aj}$
relative to the coordinates featured in \eqref{E:KASNERSOLUION}-\eqref{E:KASNERSPATIALMETRIC}.
Aside from exceptional cases in which one of the $q_i$ are equal to $1$ and the rest equal to $0$, 
Kasner solutions have Big Bang singularities 
at $\lbrace t = 0 \rbrace$ where their Kretschmann scalar 
$\Riemfour_{\alpha \beta \gamma \delta} \Riemfour^{\alpha \beta \gamma \delta}$
blows up like $t^{-4}$ (see the proof of Lemma~\ref{L:KRETSCHMANNSCALARESTIMATE} for a proof of this fact).
We now briefly summarize our main results. 
See Theorem~\ref{T:MAINTHM} for the precise statements.

\begin{theorem}[\textbf{Rough summary of main results}]
\label{T:MAINTHMROUGHVERSION}
	Kasner solutions whose exponents satisfy the inequality
	\begin{align} \label{E:KASNEREXPONENTSNOTTOOBIG}
	\max_{i=1,\cdots,\mydim} |q_i|
	& < \frac{1}{6},
	\end{align}
	which is possible when $\mydim \geq 38$
	(see Subsect.\,\ref{SS:EXISTENCEOFKASNER}),
	are nonlinearly stable solutions to the Einstein-vacuum equations \eqref{E:EINSTEININTRO}
	near their Big Bang singularities. More precisely, 
	perturbations 
	(belonging to a suitable high-order Sobolev space)
	of the Kasner initial data along $\Sigma_1 = \lbrace t = 1 \rbrace$ 
	launch a perturbed solution that also has a Big Bang singularity in the past of $\Sigma_1$. 
	In particular, relative to a set of CMC-transported spatial coordinates normalized by
	$\SecondFund_{\ a}^a = - t^{-1}$, where $\SecondFund_{\ j}^i$ is the (mixed) second fundamental form
	of $\Sigma_t$, the perturbed solution's Kretschmann scalar blows up like $C t^{-4}$.
	Hence, the past of $\Sigma_1$ in the maximal (classical) globally hyperbolic development of the data
	is foliated by a family of spacelike CMC hypersurfaces $\Sigma_t$.
	Furthermore, every past-directed causal geodesic 
	that emanates from $\Sigma_1$ crashes into the singular hypersurface 
	$\Sigma_0$ in finite affine parameter time.
	That is, the perturbed solutions are geodesically incomplete to the past,
	and the incompleteness coincides with curvature blowup.
\end{theorem}
\end{subequations}

\begin{remark}[\textbf{Additional information about the solution}]
	\label{R:ADDITONALINFORMATION}
	Theorems \ref{T:MAINTHMROUGHVERSION} (and~\ref{T:MAINTHM})
	reveal only the most fundamental aspects of the singularity formation.
	It is possible to derive substantial additional information about the solution
	using the estimates that we prove in this paper.
	For example, one could show that various time-rescaled
	variables, such as the time-rescaled second fundamental form components 
	$t \SecondFund_{\ j}^i(t,x)$,
	converge to regular functions of $x$ as $t \downarrow 0$,
	which is a manifestation of the AVTD behavior mentioned above.
	For brevity, we have chosen not to state or prove such additional results in this
	article since, thanks to the a priori estimates yielded by
	Prop.\,\ref{P:MAINAPRIORIESTIMATES}, their statements and proofs 
	are similar to the ones given in \cite{iRjS2014b}*{Theorem~2}.
\end{remark}

\begin{remark}[\textbf{On the bound of} $1/6$]
The value of $1/6$ on RHS~\eqref{E:KASNEREXPONENTSNOTTOOBIG}
is possibly not optimal. This bound for $\Worstexp$ 
emerges from considering the strength of various error terms;
see Subsect.\,\ref{SS:MAINIDEASINPROOF} for further discussion on this point.
\end{remark}

\subsection{Previous breakdown results for Einstein's equations}
\label{SS:PREVIOUSWORKS}
The study of the breakdown of solutions in general relativity was ignited by the
classic Hawking--Penrose ``singularity'' theorems 
(see e.g.\ \cites{sH1967,rP1965}, and the discussion in \cite{sHrP1970}),
which show
that for appropriate matter models and spatial topologies,
a compellingly large set 
of initial conditions 
(with non-empty interior relative to a suitable function space topology) 
leads to geodesically incomplete
solutions. 
In particular, Hawking's theorem (see \cite{rW1984}*{Theorem~9.5.1}) 
guarantees that under assumptions 
satisfied by the initial data featured in our main theorem,
all past-directed timelike geodesics are incomplete.
Although these theorems are robust with respect to the kinds of initial data and matter models
to which they apply, they are ``soft'' in that they do not
reveal the nature of the incompleteness, leaving open the possibilities
that \textbf{i)} it is tied the blowup of some invariant quantity, 
such as a spacetime curvature
scalar, or \textbf{ii)} it is due to some other more sinister phenomenon, such as the formation
of a Cauchy horizon (beyond which the solution cannot be classically extended in a unique sense 
due to lack of sufficient information).
In the wake of the Hawking--Penrose theorems, 
many authors studied the nature of the breakdown,
though, with only a few key exceptions,
the works produced were
heuristic, 
numerical,
concerned initial data with symmetry,
or yielded information only about ``special'' solutions
(as opposed to an open set of solutions corresponding
to regular initial data given along a spacelike Cauchy hypersurface).
In the next several subsubsections, we review some of 
these works.

\subsubsection{Heuristic and numerical work}
\label{SSS:HEURISTICANDNUMERICAL}
The famous-but-controversial
work \cite{vBiKeL1971}
of Belinskii--Khalatnikov--Lifshitz
gave heuristic arguments suggesting 
that for the Einstein-vacuum equations in three spatial dimensions,
the ``generic'' (in an unspecified sense) solution
that breaks down should 
\textbf{i)} be local near the breakdown, 
i.e., be well-modeled by a family of 
Bianchi IX\footnote{Readers can consult \cite{pCgGdP2010} for
an overview of the Bianchi IX class and other symmetry classes that we mention later.}  
ODE solutions\footnote{More precisely, the authors argued that they should be well-modeled
by the so-called ``Mixmaster'' \cite{cM1969} solutions, which 
are of the form $\gfour = - dt^2 + \sum_{i=1}^3 \ell_i^2(t) \omega^{(i)} \otimes \omega^{(i)}$,
where $\lbrace \omega^{(i)} \rbrace_{i=1,2,3}$ are one-forms on $\mathbb{S}^3$
verifying $d \omega^{(i)} = \frac{1}{2} [ijk] \omega^{(j)} \wedge \omega^{(k)}$,
and $[ijk]$ is the fully antisymmetric symbol normalized by $[123] = 1$.} 
that are parameterized by space;
\textbf{ii)} become highly oscillatory near the breakdown-points like the Bianchi IX solutions;
and \textbf{iii)} the breakdown points should collectively form a spacelike singularity.
We refer readers to the recent monograph \cite{vBmH2017}
for a detailed, modern account of
\cite{vBiKeL1971} and related works.
Although the work \cite{vBiKeL1971} has stimulated a lot of research activity,
as of the present, it is not clear to what extent 
the heuristic picture it painted holds true. The picture is almost certainly not
true in a generic sense, for the recent breakthrough work
\cite{mDjL2017}
shows, assuming only a widely believed (but not yet proven) 
quantitative version of the stability of the Kerr black hole family of solutions
to the Einstein-vacuum equations,
that the Cauchy horizon inside the black hole is dynamically stable.
Specifically, in \cite{mDjL2017}, the authors showed that the metric is $C^0$-extendible
past the Cauchy horizon, which is a null hypersurface. This in particular
contradicts the vision of spacelike singularities posited in \cite{vBiKeL1971}.
In the opposite direction (i.e., in accordance with \cite{vBiKeL1971}),
in three spatial dimensions, 
Ringstr\"{o}m \cite{hR2001} confirmed
the oscillatory picture\footnote{Ringstr\"{o}m also studied the stiff fluid case $c_s=1$
in Bianchi IX symmetry and proved monotonic-type singularity formation results
similar to the ones we obtained in \cites{iRjS2018,iRjS2014b,jS2017e} and in the present work.} 
of solutions near singularities
for solutions with Bianchi IX symmetry 
to the Einstein-vacuum equations
and
to the Euler--Einstein equations
under the equations of state
$p = c_s^2 \rho$, where $p$ is the pressure, $\rho$ is the density, and
the constant $0 < c_s < 1$ is the speed of sound.
However, outside of the class of spatially homogeneous solutions,
there are currently no known examples of Einstein-vacuum solutions 
that exhibit the oscillatory behavior suggested by \cite{vBiKeL1971}.

There are also heuristic results concerning the existence
non-oscillatory solutions to various Einstein-matter systems
that form a spacelike singularity.
Specifically, in \cite{vBiK1972},
Belinskii and Khalatnikov 
noted that if one considers the Einstein-scalar field system in three spatial dimensions,
then the heuristic arguments given in \cite{vBiKeL1971} concerning oscillations
no longer seem plausible. They argued that instead, the generic incomplete solution
should exhibit monotonic behavior near the breakdown-points,
which should still collectively form a spacelike singularity.
In \cite{jB1978}, Barrow gave a similar heuristic argument
for the Einstein-stiff fluid
system. Most relevant for our work here is the work
\cite{jDmHpS1985} that we mentioned at the beginning, 
in which the authors noted that for $\mydim \geq 10$
(i.e., at least $11$ spacetime dimensions),
there is a \emph{nontrivial regime} for the Einstein-vacuum equations
in which heuristic non-oscillatory arguments along the lines of
\cites{vBiK1972,jB1978} go through.
This suggests that indeed, 
in high spatial dimensions, stable blowup
results of the type that we prove in our main theorem should 
hold. Note that there is a significant gap between the
assumption $\mydim \geq 10$ 
needed for the heuristic arguments in \cite{jDmHpS1985}
and the assumption 
$\mydim \geq 38$ that we make in our 
main theorem. In Subsect.\,\ref{SS:POSSIBLEFUTUREDIRECTIONS},
we will further discuss this gap.

We close this subsubsection by noting that the above 
works and others like them
have stimulated a large number of numerical studies that have been designed
to probe the validity of the heuristic predictions.
We do not attempt to survey the vast literature here, but instead
refer to the works
\cites{bBdGjIvMmW1998,lA2006,dG2007,aR2005c},
which serve as useful starting points for exploring the subject of numerical
analysis in the context of singularities in general relativity.

\subsubsection{Construction (but not the stability of) of singular solutions}
\label{SSS:CONSTRUCTIONOFSINGULAR}
There are a variety of works showing the \emph{existence}, 
but not the stability,
of singularity-containing solutions to various Einstein-matter systems
that exhibit the same kind of AVTD behavior\footnote{That is, 
time-derivative dominated behavior, as we described in Subsect.\,\ref{SS:ADDITIONALCONTEXT}.}
near the singularity exhibited by the solutions in our main theorem.
Typically, the constructions rely on deriving/solving a Fuchsian PDE system.
Roughly speaking, a Fuchsian PDE  
is a system of the following form 
(where for convenience we restrict out attention to one spatial dimension):
\begin{align} \label{E:FUCHSIAN}
	A^0(t,x,u) t \partial_t u + A^1(t,x,u)t \partial_x u + B(t,x,u) u = f(t,x,u),
\end{align}
where $u$ is the array of unknowns, $A^{\alpha}$ and $B$ are matrices
(symmetric when the symmetric hyperbolic framework for energy estimates is invoked), 
and $f$ is an array, all of which must satisfy a collection of technical assumptions.
A typical Fuchsian analysis is based on splitting the solution
as $u = u_0 + w$, where $u_0$ is the leading order
part and $w$ is an error term that one would like to show is small compared to $u_0$ as $t \downarrow 0$.
The leading order part $u_0$ must be guessed/solved for using an appropriate ansatz.
Here we do not describe how this is typically accomplished;
readers can consult \cite{iRjS2018} for a further description of Fuchsian analysis in the context of 
singular solutions in general relativity.
Although the Fuchsian approach can sometimes be used to show the existence of 
singular solutions, it is inadequate for treating
the true stability problem, i.e., the problem of posing initial data
on a regular Cauchy hypersurface 
$\lbrace t = \mbox{\upshape const} \rbrace$ with $\mbox{\upshape const} > 0$
and then solving the equations all the way
down to the singular hypersurface $\lbrace t = 0 \rbrace$.

We now describe some of the Fuchsian-type existence results in more detail.
Notable among these is the work of
Andersson--Rendall \cite{lAaR2001},
in which they constructed a family of spatially analytic 
Big-Bang-containing
solutions to the Einstein-scalar field
and Einstein-stiff fluid systems in three spatial dimensions
that exhibit approximately monotonic behavior near the singularities.
The family of solutions that the authors constructed was large
in the sense that its number of degrees of freedom 
coincides with the number of free functions
in Einstein initial data for the standard Cauchy problem. 
However, 
the results of \cite{lAaR2001} 
do not show the stability of the blowup under Sobolev-class 
perturbations of initial data given along a spacelike hypersurface
near the expected singularity.
Another notable work in the spirit of \cite{lAaR2001} is \cite{tDmHrAmW2002},
in which the authors proved similar results for various Einstein-matter systems
in various spatial dimensions,
including for the Einstein-vacuum equations
in $10$ or more spatial dimension
(i.e., $\mydim \geq 10$ in the notation of the present article).
Note that the Einstein-vacuum result with $\mydim \geq 10$
supports the heuristic work \cite{jDmHpS1985} mentioned
in the previous subsubsection.

There are many additional works that yield the
construction of (but not the stability of)
singularity-containing solutions to select nonlinear Einstein-matter systems.
We do not attempt to exhaustively survey the literature here, but
we mention the following ones, which concern various symmetry reduced equations:
\cites{sKjI1999,gR1996,sKaR1998,kA2000a,yCBjIvM2004,fS2002,fBpL2010c,eAfBjIpL2013b}.

There are also results in which the authors constructed
singular solutions by essentially prescribing
``singular data'' on the singular hypersurface itself and then solving
to the future; see, for example,
\cites{kApT1999a,cCpN1998,rN1993a,rN1993b,pT1990,pT1991,pT2002}. 
One drawback of this approach is that
there are fewer degrees of freedom in the singular data compared to the initial data
corresponding to a standard Cauchy problem on a regular spacelike hypersurface. Thus, 
the solutions constructed in this
fashion are not generic; 
this is explained in more detail in \cite{aR2005b}*{Section~6.1}.

We close this subsubsection by mentioning that 
the Fuchsian techniques behind the above results
have applications outside of general relativity.
There are techniques for constructing solutions
to large classes of hyperbolic Fuchsian PDEs;
see, for example, \cite{eAfBjIpL2013}.
Readers can consult Rendall's work \cite{aR2000b} for a more detailed comparison 
of many of the results described above
as well as application of the Fuchsian framework to prove 
the existence of singular solutions to the Einstein-vacuum equations
with Gowdy symmetry.

\subsubsection{Constructive results under symmetry assumptions}
\label{SSS:CONSTRUCTIVERESULTSSYMMETRY}
There are a variety of works 
on symmetry reduced Einstein-matter systems
in which the authors give 
a constructive proof of stable singularity formation.
In the spatially homogeneous case, in which the equations reduce
to a system of ODEs, there are many results,
including the constructive work of Ringstr\"{o}m \cite{hR2001} mentioned earlier.
We do not attempt to survey the literature here; instead we direct
readers to \cites{aR2005b,jWgE1997} for an overview
of ODE blowup for solutions to Einstein's equations.

There are also constructive proofs of stable singularity formation
for various symmetry reduced Einstein-matter systems
in which the equations reduce to a
$1+1$-dimensional system of PDEs;
see, for example, \cites{pCjIvM1990,jIvM1990,hR2009c,hR2010}.
Chief among these are Christodoulou's remarkable works
\cites{dC1991,dC1999b} on the Einstein-scalar field system 
in three spatial dimensions in spherical symmetry
for $1$- or $2$-ended
asymptotically flat data.
In those works, he showed that the maximal globally hyperbolic
future developments of generic data are future-inextendible
as time-oriented Lorentzian manifolds with a $C^0$ metric;
i.e., the breakdown is severe, at the level of the metric itself.
See also the recent works of \cites{jLsO2017a,jLsO2017b}
on the spherically symmetric Einstein-Maxwell-(real) scalar field system
with asymptotically flat $2$-ended initial data, 
in which the authors proved that
the maximal globally hyperbolic
future developments of generic data are future-inextendible
as time-oriented Lorentzian manifolds with a $C^2$ metric.
This is an especially intriguing result in view of the fact that
Dafermos--Rodnianski \cites{mD2005,mDiR2005a} 
showed that the statement
is \emph{false} for this system if one replaces $C^2$ with $C^0$.

\subsubsection{Constructive results without symmetry assumptions}
\label{SSS:CONSTRUCTIVERESULTSNOSYMMETRY}
The only prior works exhibiting stable spacelike singularity formation 
for solutions to Einstein's equations
without symmetry assumptions are \cites{iRjS2014b,jS2017e},
which are closely related to the present work.
In \cites{iRjS2014b,jS2017e}, 
we showed that curvature blowup develops along a spacelike hypersurface
for open sets of solutions to two Einstein-matter systems: 
the Einstein-scalar field system and the Einstein-stiff fluid system;
see also our related work \cite{iRjS2018} concerning the linear analysis
and Ringstr\"{o}m's recent monograph \cite{hR2017},
which concerns estimates for solutions to a large family of linear
wave equations whose corresponding metrics 
model the behavior that can occur in solutions
to Einstein's equations near cosmological singularities.

Specifically, in \cite{iRjS2014b}, we showed that in three spatial dimensions 
with spatial topology $\mathbb{T}^3$, 
the FLRW solution is nonlinearly stable
in a neighborhood of its Big Bang singularity. In \cite{jS2017e}, 
we proved the same result in the case of $\mathbb{S}^3$ spatial topology, 
a key new feature being that the solutions are not approximately spatially flat
in the $\mathbb{S}^3$ case
(although they are nearly spatially homogeneous and isotropic).
As we stressed earlier, the analytic framework of
\cites{iRjS2018,jS2017e,iRjS2014b}
was quite different from that of the present work,
due to the availability of $L^2$-based approximate monotonicity
identities tied to the special structure of the matter models
and the spatially isotropic nature of the FLRW background solutions.

\subsection{An overview of the proof of the main theorem}
\label{SS:MAINIDEASINPROOF}
In this subsection, we outline the ideas behind the proof of Theorem~\ref{T:MAINTHM}.
As in our prior works \cites{iRjS2018,iRjS2014b,jS2017e},
we analyze solutions relative CMC-transported spatial coordinates, 
in which the spacetime metric is decomposed into the lapse $n$ and 
a Riemannian metric $g$ on the constant time hypersurfaces $\Sigma_t$ 
as follows:
\begin{align} \label{E:INTROSPACETIMEMETRICDECOMPOSEDINTOLAPSEANDFIRSTFUNDAMENTALFORM}
	\gfour 
	& = - n^2 dt \otimes dt 
	+ 
	g_{ab} dx^a \otimes dx^b.
\end{align}
In such coordinates, the Einstein-vacuum evolution problem consists of the Hamiltonian and momentum constraint equations,
hyperbolic evolution equations for the first fundamental form $g$ of $\Sigma_t$
and the second fundamental form $\SecondFund$ of $\Sigma_t$,
and an elliptic PDE for $n$ along $\Sigma_t$,
supplemented by initial data for $g$ and $\SecondFund$ given along $\Sigma_1$ that satisfy the constraints; 
see Prop.\,\ref{P:EINSTEINVACUUMEQUATIONSINCMCTRANSPORTEDSPATIALCOORDINATES} for 
the details. Here we only note that the mixed second fundamental form 
$\SecondFund_{\ j}^i$ 
verifies
$\partial_t g_{ij} = - 2n g_{ia} \SecondFund_{\ j}^a$
and that we normalize $t$ so that 
$\SecondFund_{\ a}^a = - t^{-1}$, i.e.,
$\Sigma_t$ is a hypersurface of constant mean curvature
$- \mydim^{-1} t^{-1}$.

The main part of the proof is showing that
the solution $(g,\SecondFund,n)$ exists classically for $(t,x) \in (0,1] \times \mathbb{T}^{\mydim}$, i.e., 
long enough to form a curvature singularity.
The proofs that the Kretschmann scalar blows up as $t \downarrow 0$
and that the spacetime is geodesically incomplete
follow as straightforward consequences of estimates that we use in 
proving existence on
$(0,1] \times \mathbb{T}^{\mydim}$; we refer readers
to Sect.\,\ref{S:CURVATUREANDTIMELIKECURVELENGTHESTIMATES}
for details on the nature of the breakdown, and we will not discuss them further here.

The main step in proving that the solution exists
classically for $(t,x) \in (0,1] \times \mathbb{T}^{\mydim}$
is to obtain suitable a priori estimates showing that various 
norms along $\Sigma_t$ do not blow up until $t=0$.
For this reason, in our discussion here, we describe only the a priori estimates.
At the heart of the proof lies the following task, whose importance we describe below:
\begin{quote} 
	Showing that for perturbed solutions,
	the $t$-rescaled type $\binom{1}{1}$ spatial Ricci components $t \Ric_{\ j}^i$ 
	are, at each fixed spatial point $x \in \mathbb{T}^{\mydim}$, 
	\underline{integrable in time} over $t \in (0,1]$.
	Here and throughout, $\Ric$ denotes the Ricci curvature of $g$.
\end{quote}
The above task is essentially a quantified version of the following idea,
which has its origins in the heuristic works discussed in Subsubsect.\,\ref{SSS:HEURISTICANDNUMERICAL}:
\begin{quote}
	For near-Kasner initial data, we can prove stable blowup in regimes where we can prove that
	time-derivative terms dominate spatial derivative terms in the equations,
	i.e., when we can prove that the AVTD behavior
	described in Subsubsect.\,\ref{SS:ADDITIONALCONTEXT} occurs.
\end{quote}
In practice, to prove the time-integrability of $t \Ric_{\ j}^i$
and the many other estimates that we need to close the proof, 
we rely on a bootstrap argument involving 
various norms that capture the following behavior:
the high-level derivatives are allowed to be substantially
more singular than the low-level
derivatives. Here we will not describe the logical flow of our bootstrap argument in detail,
but rather only describe how the various estimates fit together consistently.
Moreover, we will not focus on the ``smallness assumptions'' (i.e., near-Kasner assumptions) 
on the initial data that we need in our detailed proof, 
but rather only on the main feature of the analysis: the various powers of $t$ that arise; 
in this subsection,
we will simply denote all quantities that can be estimated in terms of the
near-Kasner initial data by ``$Data$''.
We stress already that our proof relies on commuting the evolution equations
with up to $N$ spatial derivatives, where $N$ has to be chosen sufficiently large
in a manner that we explain below.

We now recall that we are studying perturbations of a Kasner solution \eqref{E:KASNERSOLUION} 
whose exponents satisfy \eqref{E:KASNEREXPONENTSNOTTOOBIG}
(see Subsect.\,\ref{SS:EXISTENCEOFKASNER} for a proof that such Kasner solutions exist when $\mydim \geq 38$).
We also recall that for the Kasner solution
$(\KasnerMetric,\KasnerSecondFund,\widetilde{n})$, 
we have $\widetilde{n} \equiv 1$, while
$\KasnerMetric$ and $\KasnerSecondFund$
are respectively given by
\eqref{E:KASNERSPATIALMETRIC} and \eqref{E:KASNERSECONDFUND}.
Note also that the Kasner solution is spatially flat and thus 
the Ricci tensor of $\KasnerMetric$ vanishes.
For the perturbed solution,
to obtain the desired time integrability of the components $t \Ric_{\ j}^i$ described in the previous paragraph,
it clearly suffices to prove that 
\begin{align} \label{E:OURSPATIALRICCICONDITION}
	|\Ric_{\ j}^i| \lesssim t^{-p} \mbox{ for some constant } p < 2.
\end{align}
We stress that \eqref{E:OURSPATIALRICCICONDITION} 
is essentially the same as the heuristic 
estimates featured in the works
\cites{vBiK1972,vBiK1972,jB1978,jDmHpS1985}, 
and that the estimate was verified by the
solutions that we studied in \cites{iRjS2018,iRjS2014b,jS2017e}.
%We stress already that we designed our assumption \eqref{E:KASNEREXPONENTSNOTTOOBIG} 
%on the Kasner exponents by examining the products in the coordinate expansion
%of $\Ric_{\ j}^i$ and deciding what assumptions would allow us
%to prove that $|\Ric_{\ j}^i| \lesssim t^{-p}$ for some constant $p < 2$
%(see below for more on this).
Before describing how we prove \eqref{E:OURSPATIALRICCICONDITION},
we first outline the two main consequences that it affords:
\begin{enumerate}
	\item $n - 1  \to 0$ as $t \downarrow 0$.
	\item The components $t \SecondFund_{\ j}^i$ remain bounded as $t \downarrow 0$.
\end{enumerate}
The proof of (1) follows from a simple application of the maximum principle to the elliptic PDE
$
t^2 g^{ab} \nabla_a \nabla_b n 
= 
(n - 1) 
+ 
t^2 n \Sc
$,
where $\Sc = \Ric_{\ a}^a$ is the scalar curvature of $g$;
see the proof of \eqref{E:LAPSELOWNORMELLIPTIC} for the details.
The proof of (2)
essentially follows from freezing the spatial point and integrating 
the following evolution
equation (see equation \eqref{E:REWRITTENPARTIALTKCMC})
from time $t$ to time $1$:
$
\partial_t (t \SecondFund_{\ j}^i) = t \Ric_{\ j}^i 
+ \cdots
$,
and from using a few additional estimates that allow one to show that,
like $t \Ric_{\ j}^i$,
the error terms denoted by $\cdots$ are integrable over $t \in (0,1]$;
see the proof of Prop.\,\ref{P:INTEGRALINEQUALITYFORLOWMETRICNORM} for the details.

We now return to the crucial issue of proving that $|\Ric_{\ j}^i| \lesssim t^{-p}$ for some constant $p < 2$, 
a bound that is tied to all aspects of the proof. To achieve this, we first,
in view of \eqref{E:KASNEREXPONENTSNOTTOOBIG}, fix a constant $\Worstexp$ such that
\[
\max_{i=1,\cdots,\mydim} |q_i| < \Worstexp < 1/6.
\]
To control $\Ric_{\ j}^i$, we rely on the following estimates, whose proof we will describe below:
\begin{align} \label{E:INTROMETRICBOUNDS}
	& \max_{i,j=1,\cdots,\mydim} |g_{ij}| \lesssim t^{-2 \Worstexp},
	&&
	\max_{i,j=1,\cdots,\mydim} |g^{ij}| \lesssim t^{-2 \Worstexp}.
\end{align}
Note that \underline{the bounds \eqref{E:INTROMETRICBOUNDS} represent an absolute worst-case scenario}, 
in which all components of $g$ and $g^{-1}$
are allowed to be slightly more singular than the most singular component of the background Kasner
spatial metric and its inverse. As will become clear, the bounds \eqref{E:INTROMETRICBOUNDS}
are the most fundamental ones in our analysis.
One might think of \eqref{E:INTROMETRICBOUNDS} as allowing for
the ``complete mixing'' of the Kasner exponents in the perturbed solutions; 
this is the most glaring spot in the proof that has potential for 
improvements in future studies. 
To control $\Ric_{\ j}^i$, we first express it in terms of the Christoffel symbols of the transported
spatial coordinates (see \eqref{E:RICCICURVATUREEXACT} for the precise expression).
For the solutions under study, the most singular term in the component $\Ric_{\ j}^i$ is not a top-order term,
but rather lower-order terms (i.e., the last two products on RHS~\eqref{E:RICCICURVATUREEXACT})
that have the following schematic form: 
$g^{-3} (\partial g)^2$, where $\partial$ denotes 
the spatial gradient with respect to the transported spatial coordinates.
An interpolation argument,
which heavily relies on \eqref{E:INTROMETRICBOUNDS} and which we explain below,
yields that \emph{for large $N$} 
(where we again stress that $N$ is the maximum number of times that we commute the equations with spatial derivatives),
the low-level spatial derivatives of the components of $g$, including $\partial g$, are
\underline{only slightly more singular than the components of $g$ itself}. Thus, in view of \eqref{E:INTROMETRICBOUNDS},
we see that $g^{-3} (\partial g)^2$ is only slightly more singular than $(t^{- 2 \Worstexp})^5 = t ^{- 10 \Worstexp}$.
Since $\Worstexp < 1/6$, we conclude that the product $g^{-3} (\partial g)^2$ 
is \emph{less singular than $t^{-2}$}, consistent with the desired bound \eqref{E:OURSPATIALRICCICONDITION}.

\begin{remark}
	Since $g^{-3} (\partial g)^2$ is fifth-order in $(g,\partial g)$, the discussion above suggests that the proof
	should close assuming only $\Worstexp < 1/5$. However, in the top-order energy estimates, we encounter some
	below-top-order error terms that seem to prevent the proof from closing unless we assume $\Worstexp < 1/6$;
	we encounter such error terms, for example, 
	in the proof of \eqref{E:TOPORDERDIFFERENTIATEDJUNKMOMENTUMCONSTRAINTINDEXUPERRORTERML2ESTIMATE}.
\end{remark}

Having outlined how to obtain the desired bound for
$\Ric_{\ j}^i$,
we can, as we described above,
show that as $t \downarrow 0$,
$n-1 \to 0$
and 
$t \SecondFund_{\ j}^i$ remains bounded.
Then, given these bounds for $n$ and $\SecondFund$, 
we can integrate the evolution equations
$\partial_t g_{ij} = - 2n g_{ia} \SecondFund_{\ j}^a$
and
$\partial_t g^{ij} = 2n g^{ia} \SecondFund_{\ a}^j$
and use a near-Kasner assumption on the initial data
to obtain, through standard arguments,
the desired estimates \eqref{E:INTROMETRICBOUNDS}
(see Prop. \ref{P:INTEGRALINEQUALITYFORLOWMETRICNORM} and its proof for the precise statements).

We now return to the issue of the interpolation argument mentioned above,
which we used to show, for example, that 
$\partial g$ is only slightly more singular than $g$. To appreciate the role of
interpolation, it is essential to already know that
the best energy estimates we are able to obtain allow for the 
following rather singular behavior:
\begin{quote}
The top-order derivatives of $g$ can be as large (in $L^2$) 
as $Data \times t^{- (\Blowupexp + 1)}$, where
$\Blowupexp > 0$ is a large universal constant, independent of the maximum number of times 
(i.e., $N$) that we
commute the equations with spatial derivatives. 
\end{quote}
That is, under appropriate bootstrap assumptions,
we can prove only that\footnote{The bound \eqref{E:INTROTOPORDERDRIVATIVESOFMETRICBOUND}
is slightly inaccurate in that it does not feature the precise norm
that we use in proving the main theorem;
see Subsect.\,\ref{SS:LEBESGUEANDSOBOLEVNORMS} for the precise definitions of the norms.}
\begin{align} \label{E:INTROTOPORDERDRIVATIVESOFMETRICBOUND}
	t^{\Blowupexp + 1} 
	\| \partial g \|_{\dot{H}^N(\Sigma_t)}
	\lesssim Data.
\end{align}
Then standard Sobolev interpolation
(see \eqref{E:LINFINITYSOBOLEVINTERPOLATION})
yields the following bound for the components of $\partial g$:
\begin{align} \label{E:INTROMODELINTERPOLATION}
\left\| 
	\partial g 
\right\|_{L^{\infty}(\Sigma_t)}
\lesssim
\| g \|_{L^{\infty}(\Sigma_t)}^{1 - (1 + \lfloor \mydim/2 \rfloor)/(N+1)}
\| \partial g \|_{\dot{H}^N(\Sigma_t)}^{(1 + \lfloor \mydim/2 \rfloor)/(N+1)}.
\end{align}
From \eqref{E:INTROMODELINTERPOLATION}, we see that 
even if the top-order homogeneous norm
$\| \partial g \|_{\dot{H}^N(\Sigma_t)}$
is as singular as $Data \times t^{- (\Blowupexp + 1)}$,
as long as we take $N$ to be sufficiently large (relative to $\Blowupexp$ and $\mydim$),
the singular nature of 
$\left\| 
	\partial g 
\right\|_{L^{\infty}(\Sigma_t)} 
$
will not be much worse than that of 
$
\left\| 
	g 
\right\|_{L^{\infty}(\Sigma_t)}
$,
i.e., no worse than $t^{-(2 \Worstexp + \upalpha)}$, where $\upalpha$ is small;
see the beginning of
Subsect.\,\ref{SS:SOBOLEVEMBEDDING} for further discussion of this issue.

It remains for us to discuss the top-order energy estimates
for $g$ and $\SecondFund$.
In reality, these must be complemented with top-order elliptic estimates for $n$,
but we will ignore this (relatively standard) issue in this subsection in order to condense
our summary of the proof.
Let $\vec{I}$ be a top-order spatial derivative multi-index, i.e., $|\vec{I}|=N$.
Using the evolution equations and integration by parts,
and using appropriate bootstrap assumptions to control error terms,
we are able to derive an estimate of the following form, valid for $t \in (0,1]$
(see Prop.\,\ref{P:MAINTOPORDERMETRICENERGYESTIMATE} for the details
and Subsect.\,\ref{SS:LEBESGUEANDSOBOLEVNORMS} for the definition of the norm
$\| \cdot \|_{L_g^2(\Sigma_t)}$):
\begin{align} \label{E:INTROIMPROVEMENTMAINTOPORDERMETRICENERGYESTIMATE}
		& 
		\left\|
			t^{\Blowupexp + 1} \partial_{\vec{I}} \SecondFund 
		\right\|_{L_g^2(\Sigma_t)}^2
		+
		\frac{1}{4}
		\left\|
			t^{\Blowupexp + 1} \partial \partial_{\vec{I}} g 
		\right\|_{L_g^2(\Sigma_t)}^2
			\\
		& \leq 
			Data
			-
			\left\lbrace
				2 \Blowupexp
				- 
				C_*
			\right\rbrace
			\int_{s=t}^1
				s^{-1}
				\left\lbrace
					\left\|
						s^{\Blowupexp + 1} \partial_{\vec{I}} \SecondFund 
					\right\|_{L_g^2(\Sigma_s)}^2
					+
					\frac{1}{4}
					\left\|
						s^{\Blowupexp + 1} \partial_{\vec{I}} \partial g 
					\right\|_{L_g^2(\Sigma_s)}^2
				\right\rbrace
			\, ds
			+
			\cdots,
			\notag
	\end{align}
	where $\cdots$ denotes error terms that, while they require some care to handle, are less singular.
	We stress the following points:
	\begin{itemize}
		\item The constant $C_*$ on RHS~\eqref{E:INTROIMPROVEMENTMAINTOPORDERMETRICENERGYESTIMATE} is \underline{universal},
			i.e., independent of $N$ and $\Blowupexp$. Roughly, $C_*$ is generated by the coefficients of the most singular \emph{linear} terms
			in the evolution equations and the elliptic PDE for $n$; the most singular terms
			are all tied, either directly or indirectly, to the top-order derivatives of $g$ and $\SecondFund$.
			In our prior works \cites{iRjS2018,iRjS2014b,jS2017e}, we were able to show that
			$C_*$ is small or vanishing, thanks to the approximate monotonicity identities that we mentioned earlier in the paper.
			For the solutions under study here, $C_*$ can be large
			(and we do not bother to carefully track the precise value of $C_*$).
		\item We inserted the time weights $t^{\Blowupexp + 1}$ ``by hand'' into the energies in equation
			\eqref{E:INTROIMPROVEMENTMAINTOPORDERMETRICENERGYESTIMATE}. 
			We note that when proving \eqref{E:INTROIMPROVEMENTMAINTOPORDERMETRICENERGYESTIMATE}
			(roughly, by taking the time derivative of the LHS and integrating by parts), 
			one encounters terms
			in which $\partial_t$ falls on the weights. Roughly, this leads 
			to the integrals preceded by the factor of
			$- 2 \Blowupexp$ on RHS~\eqref{E:INTROIMPROVEMENTMAINTOPORDERMETRICENERGYESTIMATE}
			(note that $t \in (0,1]$, which explains why the factors $- 2 \Blowupexp$ are on the RHS). 
	\end{itemize}
	The key point is that if we choose $\Blowupexp$ to be sufficiently large, then
	the factor 
	$
	-
	\left\lbrace
				2 \Blowupexp
				- 
				C_*
			\right\rbrace
$
on RHS~\eqref{E:INTROIMPROVEMENTMAINTOPORDERMETRICENERGYESTIMATE} is negative, 
and the corresponding integral has an overall ``friction'' sign. In particular,
it can be discarded, leaving only the less singular error terms 
$\cdots$ on RHS~\eqref{E:INTROIMPROVEMENTMAINTOPORDERMETRICENERGYESTIMATE}.
A careful analysis of $\cdots$ allows one to conclude, via Gronwall's inequality,
the top-order a priori energy estimate
\[
\left\|
			t^{\Blowupexp + 1} \partial_{\vec{I}} \SecondFund 
		\right\|_{L_g^2(\Sigma_t)}^2
		+
		\frac{1}{4}
		\left\|
			t^{\Blowupexp + 1} \partial \partial_{\vec{I}} g 
		\right\|_{L_g^2(\Sigma_t)}^2
\lesssim
Data
\]
as desired. 

Although the above discussion summarizes the main ideas,
in our detailed proof, we encounter several hurdles
that we overcome using additional ideas.
While conceptually simple, these ideas are somewhat technically involved.
We close this subsection by highlighting some of the features of this analysis.
\begin{itemize}
	\item To control error terms, we rely on a variety of norms.
		In the next point below, we will shed some light on how we use the different kinds of norms;
		see Subsects.\,\ref{SS:POINTWISENORMS} and~\ref{SS:LEBESGUEANDSOBOLEVNORMS} for the precise definitions.
		For example, when bounding $\Sigma_t$-tangent tensors, 
		we rely on pointwise norms $|\cdot|_{Frame}$ that measure the 
		size of components relative to the transported spatial coordinate
		frame as well as the more geometric norm
		$|\cdot|_g$, which measures the size of tensors
		using the dynamic spatial metric $g$. 
		Similarly, our analysis relies on
		Sobolev norms for the frame components,
		such as
		$\| \cdot \|_{H_{Frame}^N(\Sigma_t)}$,
		as well as more geometric Sobolev norms
		$\| \cdot \|_{H_g^N(\Sigma_t)}$.
	\item Although the interpolation estimates described above are sufficient
		for controlling various error terms at the low derivative levels,
		interpolation is not quite sufficient, in itself, for controlling
		some of the error terms in the high-order energy estimates.
		For example, our proof of the error term estimate
		\eqref{E:TOPORDERDIFFERENTIATEDJUNKMOMENTUMCONSTRAINTINDEXUPERRORTERML2ESTIMATE}
		relies not only on interpolation, 
		but also on the structure of Einstein's equations in our gauge.
		We now explain why this is the case.
		First, the top-order energy estimates
		can be obtained only in terms of geometric norms
		such as $\| \cdot \|_{\dot{H}_g^N(\Sigma_t)}$,
		since the basic energy identities involve these kinds of norms.
		In contrast, for $M \leq N$, 
		Sobolev interpolation estimates involving the ``background differential operators'' $\partial_i$
		are most naturally stated and derived in terms of
		the less geometric norms $\| \cdot \|_{\dot{H}_{Frame}^M(\Sigma_t)}$.
		For tensorfields, the discrepancy between the norms
		$\| \cdot \|_{\dot{H}_{Frame}^M(\Sigma_t)}$
		and
		$\| \cdot \|_{\dot{H}_g^M(\Sigma_t)}$
		is factors of $g$ and $g^{-1}$
		(where the number of factors depends on the order of the tensorfield),
		and by \eqref{E:INTROMETRICBOUNDS},
		the two norms can differ in strength by (singular) 
		powers of $t^{- \Worstexp}$; in some cases, 
		these powers of $t^{- \Worstexp}$ are strong enough
		so that in the energy estimates, 
		certain below-top-order error terms seem, at first glance, 
		to be more singular than the main top-order terms. If this were the
		case, then our bootstrap argument would not close.
		Fortunately, this is not the case, but
		to prove this,
		we use an argument that involves deriving energy estimates not only at the very highest level,
		but also at down-to-two derivatives below top. 
		In deriving these below-top-order energy estimates,
		we use arguments that lose one derivative, 
		by treating the evolution equations like
		transport equations along the integral curves of $\partial_t$
		with derivative-losing source terms.
		These transport-type estimates lead to better below-top-order estimates
		compared to the ones that pure interpolation would afford,
		and we stress that
		this approach is viable only because of the structure of Einstein’s equations in our gauge.
		To implement this procedure in a consistent fashion,
		we rely on a hierarchy of energies that features different $t$ weights at
		different orders
		and that involves 
		both
		geometric norms $\| \cdot \|_{\dot{H}_g^M(\Sigma_t)}$
		and coordinate frame norms
		$\| \cdot \|_{\dot{H}_{Frame}^M(\Sigma_t)}$.
		See Defs.\,\ref{D:LOWNORMS} and~\ref{D:HIGHNORMS} 
		for the precise
		definitions of the $t$-weighted norms that use to control the solution;
		we prove that the norms from Defs.\,\ref{D:LOWNORMS} and~\ref{D:HIGHNORMS}
		are uniformly bounded up to the singularity.
	\item In carrying out our bootstrap argument, we find it convenient
		to derive, as a preliminary step, estimates showing that the lapse
		$n$ can be controlled in terms of $g$ and $\SecondFund$.
		This is the content of Sect.\,\ref{S:CONTROLOFLAPSEINTERMSOFMETRICANDSECONDFUND}.
		We have downplayed these estimates in this subsection since they
		are based on deriving standard estimates for elliptic equations.
		One feature worth noting is that for most lapse estimates,
		we rely on the elliptic PDE
		$
t^2 g^{ab} \nabla_a \nabla_b n 
= 
(n - 1) 
+ 
t^2 n \Sc
$.
This is a ``good'' equation in the sense that it involves
source terms that depend only on spatial derivatives of $g$,
which are less singular than time derivatives of $g$ (as represented by $\SecondFund$).
However, to obtain the top-order lapse estimates,
we first use the Hamiltonian constraint
\eqref{E:HAMILTONIAN}
to algebraically replace, in the elliptic lapse PDE, 
$\Sc$ with terms that depend on $\SecondFund$;
see equation \eqref{E:COMMUTEDLAPSEHIGHORDER} and its proof.
This algebraic replacement leads to error terms that can be controlled within the scope of
our bootstrap approach, 
both from the point of view of regularity and from the point of view 
of the structure of the singular error terms that our framework 
can accommodate.
\end{itemize}

\subsection{Further comparisons with two related works}
\label{SS:POSSIBLEFUTUREDIRECTIONS}
In this subsection, we compare and contrast the
reasons behind the assumed minimum values
of $\mydim$ in the present work
and in the aforementioned works \cites{jDmHpS1985,tDmHrAmW2002}.
We start by recalling that,
as we described in Subsubsect.\,\ref{SSS:HEURISTICANDNUMERICAL},
the work \cite{jDmHpS1985} provided heuristic evidence 
for the existence of a large family of monotonic spacelike singularity-forming 
Einstein-vacuum solutions whenever $\mydim \geq 10$,
and that such solution families were constructed in \cite{tDmHrAmW2002}
(though stability in the sense of the present article was not proved there).
There is a substantial gap between the assumption $\mydim \geq 10$
and the assumption $\mydim \geq 38$ of the present article.
In the remainder of this subsection, we provide an overview of how the assumption
$\mydim \geq 10$ is used in \cites{jDmHpS1985,tDmHrAmW2002}, and shed some light
on how one might approach the problem of extending
the results of the present article to apply to a larger range of $\mydim$ values.
Actually, we will focus only on 
the heuristic work \cite{jDmHpS1985} since this allows for a simplified
presentation of the main ideas.

In \cites{jDmHpS1985,tDmHrAmW2002} and the present work, a crucial step is justifying that, in some sense,
the spatial derivative terms in the Einstein-vacuum equations are negligible compared
to the time derivative terms near the singularity.
For example, in the present work
(see in particular Subsect.\,\ref{SS:MAINIDEASINPROOF}),
this step is embodied by the estimate \eqref{E:OURSPATIALRICCICONDITION},
that is,
our proof that the components of the Ricci tensor of $g$ obey
the following bound:
$|\Ric_{\ j}^i| \lesssim t^{-p}$ for some constant $p < 2$.
Similarly, the heuristic arguments given in \cite{jDmHpS1985} were
designed exactly to yield a bound of this type.
From this perspective, a primary analytic 
difference between the present work and
the works \cites{jDmHpS1985,tDmHrAmW2002} is that the latter works
precisely accounted for, in a tensorial fashion, 
\underline{partial cancellations of singular powers of $t$} 
that can occur in the products featured in the coordinate expression 
of the Ricci tensor $\Ric_{\ j}^i$ of $g$
(i.e., the products on RHS~\eqref{E:RICCICURVATUREEXACT}),
at least for metrics that can be interpreted as having
``spatially dependent Kasner exponents''
(see below for more on this point).
In contrast, in the present article, we allow for the possibility that
all components of $g$ and $g^{-1}$
are as singular as $t^{-2 \Worstexp}$
(see, for example, the discussion surrounding \eqref{E:INTROMETRICBOUNDS}), 
and we have therefore ignored the possibility of exploiting such cancellations;
this is apparent from the proof outline that we gave in Subsect.\,\ref{SS:MAINIDEASINPROOF}.
We now further explain the connection between the work \cite{jDmHpS1985}
and the notion of a spatial metric having
``spatially dependent Kasner exponents.''

The starting point of \cite{jDmHpS1985} 
is the hope that there are singular solutions to the Einstein-vacuum equations
that are somehow well-described by a
spacetime metric having ``spatially dependent Kasner exponents'', that is, a metric
of the following form,
defined for\footnote{As in our main results, here we have assumed the spatial topology
$\mathbb{T}^{\mydim}$; in \cite{jDmHpS1985}, the spatial topology was not specified.} 
$(t,x) \in (0,1] \times \mathbb{T}^{\mydim}$:
\begin{align} \label{E:ASSUMEDASYMPTOTICFORMOFKASNERLIKEMETRICS}
	\gfour
	& = - dt^2
		+
		\sum_{i=1}^{\mydim}
		t^{2 q_i(x)}
		\omega^{(i)}(x)
		\otimes
		\omega^{(i)}(x),
\end{align}
where $\lbrace \omega^{(i)}(x) \rbrace_{i=1,\cdots,\mydim}$
are a linearly independent set of time-independent one-forms
on $\mathbb{T}^{\mydim}$ (in particular, $\omega^{(i)}(x) = \omega_a^{(i)}(x) dx^a$ relative to local coordinates), 
and $\lbrace q_i(x) \rbrace_{i=1,\cdots,\mydim}$ are ``$x$-dependent''
Kasner exponents, subject to the following ``spatially dependent Kasner constraints''
(i.e., $x$-dependent analogs of \eqref{E:KASNEREXPONENTSSUMTOONE}-\eqref{E:KASNEREXPONENTSSQUARESSUMTOONE}):
\begin{align} \label{E:XDEPENDENTKASNERCONSTRAINTS}
	\sum_{i=1}^{\mydim} q_i(x)
	& = 1
	=
	\sum_{i=1}^{\mydim} (q_i(x))^2.
\end{align}
Given the above assumptions, the 
authors of \cite{jDmHpS1985} then observe
(through straightforward but tedious computations)
that for metrics of the form \eqref{E:ASSUMEDASYMPTOTICFORMOFKASNERLIKEMETRICS},
the Ricci tensor components of the spatial metric verify 
\begin{align} \label{E:RICCICONDITIONINHEURISTICPAPER}
	\lim_{t \downarrow 0} |t^2 \Ric_{\ j}^i(t,x)| 
	& = 0,
	&&
	(i,j = 1,\cdots,\mydim),
\end{align}
as long as the following system of inequalities holds:
\begin{align} \label{E:KASNEREXPONENTINEQUALITY}
	2 q_i(x)
	+
	\sum_{l \neq i,j,k} q_l(x) 
	&
	> 0,
	\qquad \mbox{whenever } i \neq j, i \neq k,\mbox{and } j \neq k.
\end{align}
Note that \eqref{E:RICCICONDITIONINHEURISTICPAPER}
is essentially equivalent to the estimate 
\eqref{E:OURSPATIALRICCICONDITION} that we
rigorously obtain in proving our main results;
as we described in Subsect.\,\ref{SS:MAINIDEASINPROOF},
\eqref{E:OURSPATIALRICCICONDITION}
is a quantified version of the idea that spatial derivative terms should
be negligible. The main conclusions of \cite{jDmHpS1985} can be summarized as follows.
\begin{quote}
	In view of \eqref{E:XDEPENDENTKASNERCONSTRAINTS},
	it is not possible to simultaneously satisfy all inequalities \eqref{E:KASNEREXPONENTINEQUALITY}
	when $\mydim \leq 9$. However, there \emph{does} exist an open set of 
	$\lbrace q_i \rbrace_{i=1,\cdots,\mydim}$
	satisfying \eqref{E:KASNEREXPONENTINEQUALITY}
	whenever $\mydim \geq 10$.
\end{quote}

We now stress that
the metric $\gfour$ featured in \eqref{E:ASSUMEDASYMPTOTICFORMOFKASNERLIKEMETRICS}
does not generally solve the Einstein-vacuum equations.
However, it does solve a truncated version of the equations in which all spatial derivative terms,
including $\Ric_{\ j}^i$, are discarded; 
in the mathematical general relativity literature,
the truncated system is often referred to as the
Velocity Term Dominated (VTD) system.
For this reason, the condition \eqref{E:RICCICONDITIONINHEURISTICPAPER},
which is supposed to capture the negligibility of the spatial derivative terms,
suggests that one can think of the VTD solution $\gfour$ as
``asymptotically solving'' the Einstein-vacuum equations
of Prop.\,\ref{P:EINSTEINVACUUMEQUATIONSINCMCTRANSPORTEDSPATIALCOORDINATES}
as $t \downarrow 0$. That is, for metrics of the form
\eqref{E:ASSUMEDASYMPTOTICFORMOFKASNERLIKEMETRICS},
the $\SecondFund$-involving products in equation
\eqref{E:PARTIALTKCMC} are of size $t^{-2}$
while, by \eqref{E:RICCICONDITIONINHEURISTICPAPER},
the term $\Ric_{\ j}^i$ is less singular.
Put differently, the metrics
$\gfour$ featured in \eqref{E:ASSUMEDASYMPTOTICFORMOFKASNERLIKEMETRICS}
solve a PDE system obtained from the Einstein equations
by throwing away terms that can be shown to be small in the sense of \eqref{E:RICCICONDITIONINHEURISTICPAPER}.
The work \cite{jDmHpS1985} can therefore be viewed as providing a kind of ``consistency argument''
for metrics \eqref{E:ASSUMEDASYMPTOTICFORMOFKASNERLIKEMETRICS} that satisfy \eqref{E:KASNEREXPONENTINEQUALITY},
i.e., an argument based on ignoring terms in the evolution equations 
that, for the \underline{non-solution} $\gfour$,
are small compared to the main terms.
For this perspective, it is reasonable to speculate that, 
for metrics $\gfour$ of the form \eqref{E:ASSUMEDASYMPTOTICFORMOFKASNERLIKEMETRICS}
that satisfy \eqref{E:RICCICONDITIONINHEURISTICPAPER},
there might be true Einstein-vacuum solutions lying ``close'' to 
$\gfour$.

We now explain some further connections 
between the heuristic picture painted in \cite{jDmHpS1985}
and the results of the present article.
As we described in Subsect.\,\ref{SS:MAINIDEASINPROOF},
the estimate \eqref{E:OURSPATIALRICCICONDITION}
(which is almost the same as the condition \eqref{E:RICCICONDITIONINHEURISTICPAPER} from \cite{jDmHpS1985})
is a main ingredient needed to show that at each fixed $x$,
$\partial_t (t \SecondFund_{\ j}^i(t,x))$
is integrable over the time interval $(0,1]$.
From this integrability condition, one can easily show not only
that $t \SecondFund_{\ j}^i(t,x)$ 
remains bounded as $t \downarrow 0$
(which is what we prove in our main theorem), 
but also that the following stronger result holds:
$t \SecondFund_{\ j}^i(t,x)$ converges uniformly to a function $K_{\ j}^i(x)$ as 
$t \downarrow 0$; see \cites{iRjS2014b,iRjS2018}
for proofs of these kinds of convergent results in
a related context, and see also Remark~\ref{R:ADDITONALINFORMATION}. 
Combining these kinds of convergence estimates
with related ones (in particular for the lapse $n$),
\emph{one could rigorously prove that for the Einstein-vacuum solutions $\gfour$ under study in this article, 
$\gfour$ is asymptotic to a metric that behaves like the family of metrics 
that satisfy \eqref{E:ASSUMEDASYMPTOTICFORMOFKASNERLIKEMETRICS}-\eqref{E:KASNEREXPONENTINEQUALITY}};
again, readers can consult \cites{iRjS2014b,iRjS2018}
for proofs of these kinds of results in
a related context.

The situation can be summarized as follows:
the singular solutions that can be shown to be dynamically stable
under our framework are asymptotic (as the singularity is approached)
to a metric that is well-described by the family of metrics
satisfying \eqref{E:ASSUMEDASYMPTOTICFORMOFKASNERLIKEMETRICS}-\eqref{E:KASNEREXPONENTINEQUALITY}.
The glaring question, then, is whether or not, under the assumptions 
\eqref{E:XDEPENDENTKASNERCONSTRAINTS} and \eqref{E:RICCICONDITIONINHEURISTICPAPER}, 
the ``VTD solutions'' $\gfour$ \underline{defined} by \eqref{E:ASSUMEDASYMPTOTICFORMOFKASNERLIKEMETRICS}
are always the asymptotic end-state
of a true singularity-forming solution of the Einstein-vacuum equations. A natural starting point for 
addressing this question would be to try to extend the stable blowup results
of the present article to apply to perturbations of all Kasner solutions \eqref{E:KASNERSOLUION}
whose Kasner exponents $\lbrace q_i \rbrace_{i=1,\cdots,\mydim}$
satisfy \eqref{E:KASNEREXPONENTINEQUALITY}, which in particular would extend our results
to the cases $\mydim \geq 10$. If such a result is in fact true, then its proof would
almost certainly involve observing the kinds of tensorial cancellations that the authors of \cite{jDmHpS1985}
exploited to derive the bound \eqref{E:RICCICONDITIONINHEURISTICPAPER}, 
a feat that we did not attempt in the present work.
To allow for the observation of these kinds of tensorial cancellations, it is likely
that an important step in the proof would be anticipating, 
in some fashion going beyond our work here,
that the perturbed solution is converging towards a metric that 
is similar to ones of the form \eqref{E:ASSUMEDASYMPTOTICFORMOFKASNERLIKEMETRICS}.
This is a worthy avenue for future investigation, especially since it is intimately
tied to the fundamental question of which terms in Einstein's equations are the ones
driving the breakdown of solutions.

\subsection{Paper outline}
\label{SS:PAPEROUTLINE}
The remainder of the paper is organized as follows.
\begin{itemize}
	\item In Subsect.\,\ref{SS:NOTATION}, we summarize the notation and conventions
		that we use in the rest of the article.
	\item In Sect.\,\ref{S:SETUPOFANALYSIS}, we set up the ensuing analysis 
		by providing the Einstein-vacuum equations in CMC-transported spatial coordinates
		and showing that Kasner solutions verifying our assumptions exist when $\mydim \geq 38$.
	\item In Sect.\,\ref{S:NORMSANDBOOTSTRAPASSUMPTIONS}, we define the norms that we use
		to control solutions and formulate the bootstrap assumptions that we use in our analysis.
	\item In Sect.\,\ref{S:INTERPOLATION}, we provide some preliminary technical estimates,
		which are standard.
	\item In Sect.\,\ref{S:CONTROLOFLAPSEINTERMSOFMETRICANDSECONDFUND}, 
		we derive elliptic estimates showing that the lapse $n$ 
		can be controlled, in various norms, in terms of $g$ and $\SecondFund$.
	\item In Sect.\,\ref{S:ESTIMATESFORLOWDERIVATIVESOFMETRIC}, we derive
		preliminary estimates for $g$ and $\SecondFund$ at the low derivative levels.
	\item In Sect.\,\ref{S:TOPORDERESTIMATESMETRICANDSECONDFUND}, 
		we derive preliminary estimates for $g$ and $\SecondFund$ at the top derivative levels,
		i.e., our main top-order energy estimates.
	\item In Sect.\,\ref{S:ESTIMATESFORJUSTBELOWTOPORDERDERIVATIVESOFMETRICANDSECONDFUNDAMENTALFORM},
		we derive preliminary energy estimates for $g$ and $\SecondFund$ at the near-top-order derivative levels;
		as we explained near the end of Subsect.\,\ref{SS:MAINIDEASINPROOF}, 
		for technical reasons,
		we need these estimates to close our bootstrap argument.
	\item In Sect.\,\ref{S:APRIORIESTIMATES}, we combine the results of
		Sects.\,\ref{S:CONTROLOFLAPSEINTERMSOFMETRICANDSECONDFUND}-\ref{S:ESTIMATESFORJUSTBELOWTOPORDERDERIVATIVESOFMETRICANDSECONDFUNDAMENTALFORM}
		to obtain the main technical result of the article:
		a priori estimates showing that appropriately defined solution norms
		can be uniformly controlled by the initial data, all the way up to the singularity.
	\item In Sect.\,\ref{S:CURVATUREANDTIMELIKECURVELENGTHESTIMATES}, we derive some results
		that describe the breakdown as $t \downarrow 0$, i.e., the curvature blows up 
		and past-directed timelike geodesics terminate with a finite length.
	\item In Sect.\,\ref{S:MAINTHM}, we synthesize the results of the previous sections
		and give a relatively short proof of the main stable blowup theorem.
\end{itemize}

\subsection{Notation and conventions}
\label{SS:NOTATION}
For the reader's convenience, in this subsection,
we provide some notation and conventions that we use
throughout the article. Some of the concepts referred to here
are not defined until later.

\subsubsection{Indices} \label{SSS:INDICES}
Greek ``spacetime'' indices $\alpha, \beta, \cdots$ take on the values $0,1,\cdots,\mydim$, 
while Latin ``spatial'' indices $a,b,\cdots$ 
take on the values $1,2,\cdots,\mydim$. 
Repeated indices are summed over (from $0$ to $\mydim$ if they are Greek, and from $1$ to $\mydim$ if they are Latin). 
We use the same conventions for primed indices such as $a'$ as we do for their non-primed counterparts.
Spatial indices are lowered and raised with the Riemannian metric $g_{ij}$ and its inverse $g^{ij}$.

\subsubsection{Spacetime tensorfields and \texorpdfstring{$\Sigma_t$-}{constant-time hypersurface-}tangent tensorfields}
\label{SSS:TENSORFIELD}
We denote spacetime tensorfields $\Tfour_{\nu_1 \cdots \nu_m}^{\ \ \ \ \ \ \mu_1 \cdots \mu_l}$ in bold font. 
We denote $\Sigma_t$-tangent tensorfields $T_{b_1 \cdots b_m}^{\ \ \ \ \ a_1 \cdots a_l}$ in non-bold font.

\subsubsection{Coordinate systems and differential operators} \label{SSS:COORDINATES}
We often work in a fixed standard local coordinate system $(x^1,\cdots,x^{\mydim})$ on $\mathbb{T}^{\mydim}$. 
The vectorfields 
$\partial_j := \frac{\partial}{\partial x^{j}}$ are globally well-defined even though the coordinates themselves are not. 
Hence, in a slight abuse of notation, we use $\lbrace \partial_1, \cdots, \partial_{\mydim} \rbrace$ to denote the globally defined vectorfield frame. We denote the corresponding dual frame by $\lbrace dx^1, \cdots, dx^{\mydim} \rbrace$. 
In CMC-transported spatial coordinates,
the spatial coordinate functions are transported along the unit normal to $\Sigma_t$, 
thus producing a local coordinate system $(x^0,x^1,x^2,x^3)$ on manifolds-with-boundary of the form 
$(T,1] \times \mathbb{T}^{\mydim}$, and we often write $t$ instead of $x^0$. 
The corresponding vectorfield frame on $(T,1] \times \mathbb{T}^{\mydim}$ 
is $\lbrace \partial_0, \partial_1, \cdots, \partial_{\mydim} \rbrace$, and the corresponding dual frame is
$\lbrace dx^0, dx^1, \cdots, dx^{\mydim} \rbrace$.  
Relative to this frame, Kasner solutions are of the form 
\eqref{E:KASNERSOLUION}. 
The symbol $\partial_{\mu}$ denotes the frame derivative $\frac{\partial}{\partial x^{\mu}}$, and we often write $\partial_t$ instead of $\partial_0$ and $dt$ instead of $dx^0$. Many of our 
equations and estimates are stated relative to the frame 
$\big\lbrace \partial_{\mu} \big\rbrace_{\mu = 0,1,\cdots,\mydim}$ and dual frame 
$\big\lbrace dx^{\mu} \big\rbrace_{\mu = 0,1,\cdots,\mydim}$.

We use the notation $\partial f$ to denote the \emph{spatial coordinate} gradient of the function $f$ 
relative to the transported spatial coordinates.
Similarly, if $\omega$ is a $\Sigma_t$-tangent one-form, 
then $\partial \omega$ denotes the $\Sigma_t$-tangent type $\binom{0}{2}$
tensorfield with components $\partial_i \omega_j$ relative to the frame described above.

If $\vec{I} = (n_1,n_2,\cdots,n_{\mydim})$ is an array comprising 
$\mydim$ non-negative integers, 
then we define the \emph{spatial} multi-indexed 
differential operator $\partial_{\vec{I}}$ by 
$\partial_{\vec{I}} := \partial_1^{n_1} \partial_2^{n_2} \cdots \partial_{\mydim}^{n_{\mydim}}$. The notation 
$|\vec{I}| := n_1 + n_2 + \cdots + n_{\mydim}$ denotes the order of $\vec{I}$. 

Throughout, $\Dfour$ denotes the Levi--Civita connection of $\gfour$.  We write 
\begin{align} \label{E:SPACETIMECOV}
	\Dfour_{\nu} \Tfour_{\nu_1 \cdots \nu_n}^{\ \ \ \ \ \ \mu_1 \cdots \mu_m} = 
		\partial_{\nu} \Tfour_{\nu_1 \cdots \nu_n}^{\ \ \ \ \ \ \mu_1 \cdots \mu_m} + 
		\sum_{r=1}^m \Chfour_{\nu \ \alpha}^{\ \mu_r} \Tfour_{\nu_1 \cdots \nu_n}^{\ \ \ \ \ \ \mu_1 \cdots \mu_{r-1} \alpha \mu_{r+1} 
			\cdots \mu_m} 
		- 
		\sum_{r=1}^n \Chfour_{\nu \ \nu_{r}}^{\ \alpha} 
		\Tfour_{\nu_1 \cdots \nu_{r-1} \alpha \nu_{r+1} \cdots \nu_n}^{\ \ \ \ \ \ \ \ \ \ \ \ \ \ \ \ \ \mu_1 \cdots \mu_m}
\end{align} 
to denote a component of the covariant derivative of a tensorfield $\Tfour$ 
(with components $\Tfour_{\nu_1 \cdots \nu_n}^{\ \ \ \ \ \mu_1 \cdots \mu_m}$) defined on
$(T,1] \times \mathbb{T}^{\mydim}$. 
The Christoffel symbols of $\gfour$, which we denote by $\Chfour_{\mu \ \nu}^{\ \lambda}$, are defined by
\begin{align} \label{E:FOURCHRISTOFFEL}
\Chfour_{\mu \ \nu}^{\ \lambda} & := 
		\frac{1}{2} (\gfour^{-1})^{\lambda \sigma} 
		\left\lbrace
			\partial_{\mu} \gfour_{\sigma \nu} 
			+ \partial_{\nu} \gfour_{\mu \sigma} 
			- \partial_{\sigma} \gfour_{\mu \nu} 
		\right\rbrace.
\end{align}

We use similar notation to denote the covariant derivative of a $\Sigma_t$-tangent tensorfield $T$ 
(with components $T_{b_1 \cdots b_n}^{\ \ \ \ \ a_1 \cdots a_m}$) with respect to the Levi--Civita connection $\nabla$ 
of the Riemannian metric $g$, i.e., the first fundamental form of $\Sigma_t$. 
The Christoffel symbols of $g$, which we denote by $\Gamma_{j \ k}^{\ i}$,
are defined by 
\begin{align} \label{E:SPACECHRISTOFFEL}
	\Gamma_{j \ k}^{\  i} 
	:= \frac{1}{2} g^{ia} 	
			\left\lbrace
				\partial_j g_{ak} 
				+ \partial_k g_{ja} 
				- \partial_a g_{jk}
			\right\rbrace.
\end{align}

\subsubsection{Integrals and basic norms} 
\label{SSS:BASICNORMS}
Throughout this subsection, $f$ denotes a scalar function defined on the hypersurface 
$\Sigma_t = \lbrace (s,x) \in \mathbb{R} \times \mathbb{T}^{\mydim} \ | \ s = t \rbrace$.
We define
\begin{align} \label{E:SIGMATINTEGRALDEF}
	\int_{\Sigma_t}
		f
	\, dx
	:= \int_{\mathbb{T}^{\mydim}} 
				f(t,x^1,\cdots,x^{\mydim}) 
			\, dx.
\end{align}
Above, the notation ``$\int_{\mathbb{T}^{\mydim}} f \, dx$'' 
denotes the integral of $f$ over $\mathbb{T}^{\mydim}$ 
with respect to the measure corresponding to the volume form 
of the \emph{standard Euclidean metric $\Euc$} on $\mathbb{T}^{\mydim}$, which has the components
$\mbox{\upshape diag}(1,1,\cdots,1)$ relative to the coordinate frame described 
in Subsubsect.\,\ref{SSS:COORDINATES}.
Note that $dx$ is \emph{not the canonical integration measure associated to the Riemannian metric $g$}.

All of our Sobolev norms are built out of the (spatial) $L^2$ norms of scalar quantities 
(which may be the components of a tensorfield). 
We define the standard $L^2$ norm $\| \cdot \|_{L^2(\Sigma_t)}$ as follows:
\begin{align} \label{E:SOBOLEVNORMDF}
	\left\| f \right\|_{L^2(\Sigma_t)}
	&:=
	\left(
	\int_{\Sigma_t}
		f^2
	\, dx
	\right)^{1/2}.
\end{align}
For integers $M \geq 0$, we define the standard $H^M$ norm $\| \cdot \|_{H^M(\Sigma_t)}$ as follows:
\begin{align} \label{E:HIGHERSOBOLEVNORMDF}
	\left\| f \right\|_{H^M(\Sigma_t)}
	&:= 
		\left(
			\sum_{|\vec{I}| \leq M}
			\left\| 
				\partial_{\vec{I}} f 
			\right\|_{L^2(\Sigma_t)}^2
		\right)^{1/2}.
\end{align}
We also define the following standard homogeneous analog of \eqref{E:HIGHERSOBOLEVNORMDF}:
\begin{align} \label{E:HOMOGENEOUSHIGHERSOBOLEVNORMDF}
	\left\| f \right\|_{\dot{H}^M(\Sigma_t)}
	&:= 
		\left(
			\sum_{|\vec{I}| = M}
			\left\| 
				\partial_{\vec{I}} f 
			\right\|_{L^2(\Sigma_t)}^2
		\right)^{1/2}.
\end{align}
Finally, we define the Lebesgue norm $\| \cdot \|_{L^{\infty}(\Sigma_t)}$
of scalar functions $f$ in the usual way:
\begin{align} \label{E:LINFINITYNORM}
	\left\| f \right\|_{L^{\infty}(\Sigma_t)}
	&:= 
		\mbox{\upshape ess sup}_{x \in \mathbb{T}^{\mydim}} 
		|f(t,x)|.
\end{align}

In Subsects.\,\ref{SS:POINTWISENORMS} and~\ref{SS:LEBESGUEANDSOBOLEVNORMS},
we will introduce additional norms for tensorfields,
many of which are built out of the basic ones from this subsubsection.

\subsubsection{Parameters}
\label{SSS:PARAMETERS}
\begin{itemize}
	\item $\varepsilon$ is a small ``bootstrap parameter'' that, in our bootstrap argument,
		bounds the size of the solution norms; see \eqref{E:BOOTSTRAPASSUMPTIONS}.
		The smallness of $\varepsilon$ needed to close the estimates is allowed
		to depend on the numbers $N$, $\Blowupexp$, $\Worstexp$, $\Room$, and $\mydim$.
	\item $\Blowupexp \geq 1$ denotes an ``exponent parameter'' that is featured in 
		the high-order solution norms from Def.\,\ref{D:HIGHNORMS}.
		To close our estimates, we will choose $\Blowupexp$ to be large enough to overwhelm
		various universal constants $C_*$ (see Subsubsect.\,\ref{SSS:CONSTANTS}).
	\item $N$ denotes the maximum number of times that we commute the equations with spatial derivatives
		(e.g., $\SecondFund \in H^N(\Sigma_t)$ and $g,n \in  H^{N+1}(\Sigma_t)$).
		To close our estimates, we will choose $N$ to be sufficiently large 
		in a (non-explicit) manner that depends on $\Blowupexp$, $\Worstexp$, $\Room$, and $\mydim$.
	\item $0 < \Worstexp < 1/6$ is a constant, fixed throughout the proof, 
		that bounds the magnitude of the background Kasner exponents.
	\item $\Room > 0$ is a small constant, fixed throughout the argument, 
		that we use to simplify the proofs of various estimates
		that ``have room in them.''
	\item $\Worstexp$ and $\Room$ are constrained by \eqref{E:WORSTEXPANDROOMLOTSOFUSEFULINEQUALITIES}.
	\item $\updelta > 0$ is a small $(N,\mydim)$-dependent parameter that is allowed to vary from line to line
		and that is generated by the estimates of Lemma~\ref{L:SOBOLEVBORROWALITTLEHIGHNORM}.
		In particular, we use the convention that a sum of two $\updelta$'s is another $\updelta$.
		The only important feature of $\updelta$ that we exploit in the proof is the following:
		at fixed $\mydim$, we have $\lim_{N \to \infty} \updelta = 0$.
		In particular, if $\Blowupexp$ is also fixed, then
		$\lim_{N \to \infty} \Blowupexp \updelta = 0$.
\end{itemize}

\subsubsection{Constants}
\label{SSS:CONSTANTS}
\begin{itemize}
	\item $C$ denotes a positive constant that is free to vary from line to line.
		$C$ can depend on $N$, $\Blowupexp$, $\mydim$, $\Worstexp$, and $\Room$,
		but it can be chosen to independent of
		all $\varepsilon > 0$ that are sufficiently small
		in the manner described in Subsubsect.\,\ref{SSS:PARAMETERS}.
	\item $C_*$ denotes a positive constant that is free to vary from line to line
		and that can depend on $\mydim$.
		However, unlike $C$, $C_*$ can be chosen to be \textbf{independent} of
		$N$, $\Blowupexp$, and $\varepsilon$, as long as $\varepsilon$
		is sufficiently small in the manner described in Subsubsect.\,\ref{SSS:PARAMETERS}.
		For example, $1 + C N! \varepsilon = C_*$
		while $N! = C$ and $N!/\Room = C$,
		where $C$ and $C_*$ are as above.
	\item We write $v \lesssim w$ to indicate that $v \leq C w$, with $C$ as above.
	\item We write $v = \mathcal{O}(w)$ to indicate that $|v| \leq C |w|$, with $C$ as above.
\end{itemize}

\subsubsection{Schematic notation}
\label{SSS:SCHEMATIC}
\begin{itemize}
	\item We write
		$v = \prod_{r=1}^R v_r$ to indicate that
		the $\Sigma_t$-tangent tensorfield $v$ is a tensor product,
		possibly involving contractions,
		of the $\Sigma_t$-tangent tensorfields $v_r$.
		We use this notation only when the precise details of the tensor
		product are not important for our analysis.
		We sometimes display the indices of $v$ to indicate its order;
		for example, the expression
		$v_{\ j}^i = \prod_{r=1}^R v_r$
		emphasizes that $v$ is type $\binom{1}{1}$.
	\item We write $v \mycong \sum_{r=1}^R v_r$ to indicate that
	the $\Sigma_t$-tangent tensorfield $v$ is a linear combination
	of the $\Sigma_t$-tangent tensorfields $v_r$,
	where \textbf{the coefficients in the linear combination are constants $\pm C$},
	with $C$ as in Subsubsect.\,\ref{SSS:CONSTANTS}.
	As above, we sometimes display the indices of $v$ to indicate its order.
	For example, $v_{ijk} \mycong v_1 + v_2$ means that 
	$v$ is type $\binom{0}{3}$ and that
	$v = \pm C_1 v_1 \pm C_2 v_2$, where the
	$C_i$ can depend on $N$, $\Blowupexp$, $\mydim$, $\Worstexp$, and $\Room$.
	\item We write $v \mycongstar \sum_{r=1}^R v_r$ to indicate that
	the $\Sigma_t$-tangent tensorfield $v$ is a linear combination
	of the $\Sigma_t$-tangent tensorfields $v_r$,
	where \textbf{the coefficients in the linear combination are constants $\pm C_*$},
	with $C_*$ universal constants enjoying the properties described in Subsubsect.\,\ref{SSS:CONSTANTS}.
	As above, we sometimes display the indices of $v$ to indicate its order.
	For example, $v_{\ j}^i \mycongstar v_1 + v_2$ means that 
	$v$ is type $\binom{1}{1}$ and that
	$v = \pm C_{*,1} v_1 \pm C_{*,1} v_2$, 
	where the $C_{*,i}$ are \textbf{independent of} $N$ and $\Blowupexp$.
\end{itemize}

\section{Setting up the Analysis}
\label{S:SETUPOFANALYSIS}
In this section, we start by providing the Einstein-vacuum equations in the gauge that we use to prove our main theorem.
Next, we provide some standard expressions for the curvature tensors of the first fundamental form of $\Sigma_t$.
Finally, show that when $\mydim \geq 38$, 
there exist Kasner solutions that satisfy the Kasner exponent assumptions in our main theorem.

\subsection{The Einstein-vacuum equations in CMC-transported spatial coordinates}
\label{SS:EQUATIONSINCMCTRANSPORTEDSPATIALCOORDINATES}
In this subsection, we recall the Einstein-vacuum equations in CMC-transported spatial coordinates gauge; 
this is the gauge that we use throughout the article.

\subsubsection{Basic ingredients in the setup}
\label{SSS:BASICINGREDIENTS}
In CMC-transported spatial coordinates, 
the spacetime metric is decomposed into the lapse $n$ and 
a Riemannian metric $g$ on the constant time hypersurfaces $\Sigma_t$ 
(known as the first fundamental form of $\Sigma_t$) as follows:
\begin{align} \label{E:SPACETIMEMETRICDECOMPOSEDINTOLAPSEANDFIRSTFUNDAMENTALFORM}
	\gfour 
	& = - n^2 dt \otimes dt 
	+ 
	g_{ab} dx^a \otimes dx^b.
\end{align}
We use $g^{ab}$ to denote the components of the inverse Riemannian metric $g^{-1}$.
Above and throughout, $t$ is a time function on the spacetime manifold $\mathcal{M}$
that we will describe just below and
$\lbrace x^a \rbrace_{a=1,\cdots,\mydim}$
are standard (locally defined) ``spatial coordinates'' on the constant-time
hypersurfaces $\Sigma_t := \lbrace (s,x) \in \mathcal{M} \ | \ s = t \rbrace$,
which are diffeomorphic to $\mathbb{T}^{\mydim}$.
We sometimes use the alternate notation $x^0 = t$.
We use the notation $\partial_{\mu} = \frac{\partial}{\partial x^{\mu}}$
to denote the partial derivative vectorfields corresponding to the above coordinates,
and we often write $\partial_t$ in place of $\partial_0$.
Note that 
$\lbrace \partial_a \rbrace_{a=1,\cdots,\mydim}$ 
are a globally defined smooth frame on $\Sigma_t$,
even though the coordinate functions $\lbrace \partial_a \rbrace_{a=1,\cdots,\mydim}$ 
are only locally defined. Note also that the dual co-frame
$\lbrace dx^a \rbrace_{a=1,\cdots,\mydim}$ is also globally defined
and smooth. 
In view of \eqref{E:SPACETIMEMETRICDECOMPOSEDINTOLAPSEANDFIRSTFUNDAMENTALFORM},
we see that the future-directed unit normal $\Nml$ to $\Sigma_t$ can be expressed as
\begin{align} \label{E:FUTUREDIRECTEDUNITNORMAL}
	\Nml
	& = - n^{-1} \partial_t.
\end{align}
Note that $\Nml x^a = 0$ for $a=1,\cdots,\mydim$. Thus, 
in the gauge under consideration,
the spatial coordinates are transported along the flow lines of $\Nml$.

The second fundamental form $\SecondFund$ of $\Sigma_t$ is defined by requiring that the following relation 
holds for all vectorfields $X,Y$ tangent to $\Sigma_t$:
\begin{align} \label{E:SECONDFUNDDEF}
	\gfour(\Dfour_X \Nml, Y) = - \SecondFund(X,Y),
\end{align}
where $\Dfour$ is the Levi--Civita connection of $\gfour$.
It is a standard fact that $\SecondFund$ is symmetric:
\begin{align}
	 \SecondFund(X,Y) =  \SecondFund(Y,X).
\end{align}
For such $X,Y$, the action of the spacetime connection $\Dfour$ can be decomposed into
the action of the Levi--Civita connection $\nabla$ of $g$ and $\SecondFund$ as follows:
\begin{align} \label{E:DDECOMP}
	\Dfour_X Y = \nabla_X Y - \SecondFund(X,Y)\Nml.
\end{align}

\begin{remark} \label{R:ALWAYSMIXED}
	When analyzing the components of $\SecondFund$
	(and in particular when differentiating the components of $\SecondFund$),
	\textbf{we will always assume that it is written in mixed form
	(i.e., type $\binom{1}{1}$ form)
	as $\SecondFund_{\ j}^i$ with the first index upstairs and the second one downstairs.} This convention is absolutely essential for 
	some of our analysis; in the problem of interest to us, the evolution and constraint equations verified by the components 
	$\SecondFund_{\ j}^i$ have a more favorable structure than the corresponding equations verified by $\SecondFund_{ij}$.
\end{remark}

Throughout the vast majority of our analysis, we normalize the CMC hypersurfaces $\Sigma_t$ as follows:
\begin{align} \label{E:CMCCONDITION}
	\SecondFund_{\ a}^a & = - \frac{1}{t},&& t \in (0,1].
\end{align}
In order for \eqref{E:CMCCONDITION} to hold, 
the lapse has to verify the elliptic equation \eqref{E:LAPSECMC}.

We adopt the following sign convention for the Riemann curvature $\Riemfour$
of $\gfour$:
\begin{align} \label{E:RIEMANNSIGNCONVENTION}
	\Dfour_{\alpha} \Dfour_{\beta} \mathbf{X}_{\mu} 
		- 
		\Dfour_{\beta} \Dfour_{\alpha} \mathbf{X}_{\mu} 	
	& = \Riemfour_{\alpha \beta \mu \nu} \mathbf{X}^{\nu}.
\end{align}
Similarly, we adopt the following sign convention for the Riemann curvature $\Riemann$
of $g$:
\begin{align} \label{E:THREEMETRICRIEMANNSIGNCONVENTION}
	\nabla_a \nabla_b X_c 
		- 
	\nabla_b \nabla_a X_c	
	& = \Riemann_{abcd} X^d.
\end{align}

\subsubsection{Statement of the equations}
\label{SSS:STATEMENTOFEQUATIONSINCMCTRANSPORTEDSPATIALCOORDINATES}
\begin{proposition}[\textbf{The Einstein-vacuum equations in CMC-transported spatial coordinates}]
\label{P:EINSTEINVACUUMEQUATIONSINCMCTRANSPORTEDSPATIALCOORDINATES}
In CMC-transported spatial coordinates normalized by $\SecondFund_{\ a}^a = - t^{-1}$, 
the Einstein-vacuum equations \eqref{E:EINSTEININTRO} take the following form.

\medskip

$\bullet$ The \textbf{Hamiltonian and momentum constraint equations} verified by $g$ and $\SecondFund$ are
respectively:
\begin{subequations}
\begin{align}
		\Sc
		- 
		\SecondFund_{\ b}^a \SecondFund_{\ a}^b 
		+ 
		t^{-2} 
		& = 0, \label{E:HAMILTONIAN} \\
		\nabla_a \SecondFund_{\ i}^a 
		& = 0, \label{E:MOMENTUM}
\end{align}
\end{subequations}
where $\nabla$ denotes the Levi--Civita connection of $g$,
$\Sc = \Ric_{\ a}^a$ denotes the scalar curvature of $g$,
and $\Ric$ denotes the Ricci curvature of $g$
(a precise expression is given in \eqref{E:RICCICURVATUREEXACT}).

\medskip

$\bullet$ The \textbf{evolution equations} verified by $g$, $g^{-1}$, and $\SecondFund$ are:
\begin{subequations}
\begin{align}
	\partial_t g_{ij} 
	& = - 2 n g_{ia}\SecondFund_{\ j}^a, \label{E:PARTIALTGCMC} \\
	\partial_t g^{ij} 
	& = 2 n g^{ia}\SecondFund_{\ a}^j, \label{E:PARTIALTGINVERSECMC} \\
	\partial_t (\SecondFund_{\ j}^i) 
	& = - g^{ia} \nabla_a \nabla_j n
		+ n 
			\left(
				\Ric_{\ j}^i 
				-
				t^{-1} \SecondFund_{\ j}^i 
			\right).  \label{E:PARTIALTKCMC}
\end{align}
\end{subequations}

\medskip

$\bullet$ The \textbf{elliptic lapse equation} is
\begin{align}  \label{E:LAPSECMC}
	g^{ab} \nabla_a \nabla_b n 
		& = 
				t^{-2}
				(n - 1) 
				+ 
				n \Sc.
	\end{align}

\end{proposition}

\subsection{Standard expressions for curvature tensors of \texorpdfstring{$g$}{the first fundamental form}}
\label{SS:SPATIAL}
For future use, we note the following standard facts:
relative to an arbitrary coordinate system on $\Sigma_t$
(and in particular relative to the transported spatial coordinates that we use in our analysis),
the components of the type $\binom{0}{4}$ Riemann curvature $\Riemann$ of $g$ 
and the type $\binom{1}{1}$ Ricci curvature $\Ric$ of $g$
can be expressed, respectively, as
\begin{subequations}
\begin{align} \label{E:RIEMANNCURVATUREEXACT}
	\Riemann_{ijkl}
		& =
		\frac{1}{2}
		\left\lbrace
			\partial_j \partial_k g_{il}
			+ 
			\partial_i \partial_l g_{jk}
			-
			\partial_i \partial_k g_{jl}
			-
			\partial_j \partial_l g_{ik}
		\right\rbrace
			\\
	& \ \
		+
		g^{ab} \Gamma_{ial} \Gamma_{jbk}
		- 
		g^{ab} \Gamma_{iak} \Gamma_{jbl},
		\notag \\
	\Ric_{\ j}^i
		& =
		\frac{1}{2}
		g^{cd}
		g^{ie}
		\left\lbrace
			\partial_e \partial_c g_{dj}
			+
			\partial_c \partial_j g_{ed}
			-
			\partial_e \partial_j g_{cd}
			-
			\partial_c \partial_d g_{ej}
		\right\rbrace
			\label{E:RICCICURVATUREEXACT} \\
	& \ \
		+
		g^{ab} g^{cd} g^{ie} \Gamma_{eac} \Gamma_{jbd}
		- 
		g^{ab} g^{cd} g^{ie} \Gamma_{eaj} \Gamma_{cbd},
	\notag
\end{align}
\end{subequations}
where 
$\Gamma_{ijk} := g_{ja} \Gamma_{i \ k}^{\ a}$
and $\Gamma_{j \ k}^{\ i}$ 
are the Christoffel symbols of $g$
(see \eqref{E:SPACECHRISTOFFEL}).

\subsection{The existence of Kasner solutions verifying our exponent assumptions 
\texorpdfstring{$\mydim \geq 38$}{in 38 or more spatial dimensions}}
\label{SS:EXISTENCEOFKASNER}
Note that in view of \eqref{E:KASNEREXPONENTSSQUARESSUMTOONE},
if $\mydim \leq 36$, then there do not exist any Kasner solutions that satisfy the 
exponent assumption \eqref{E:KASNEREXPONENTSNOTTOOBIG}.
In this subsection, we show that for $\mydim \geq 38$, such Kasner solutions
\emph{do} exist; recall that this is equivalent to finding real numbers $\lbrace q_i \rbrace_{i=1,\cdots,\mydim}$
that satisfy
\eqref{E:KASNEREXPONENTSSUMTOONE}-\eqref{E:KASNEREXPONENTSSQUARESSUMTOONE}
and \eqref{E:KASNEREXPONENTSNOTTOOBIG}.

To start, we note that in the case $\mydim = 36$,
the following Kasner exponents
satisfy \eqref{E:KASNEREXPONENTSSUMTOONE}-\eqref{E:KASNEREXPONENTSSQUARESSUMTOONE}
but just barely fail to satisfy \eqref{E:KASNEREXPONENTSNOTTOOBIG}:
\begin{align} \label{E:EXPONENTSFORDIMENSION36}
	q_1 
	& 
	= q_2 
	= \cdots 
	= q_{15} 
	:= - \frac{1}{6},
	&&
	q_{16}
	= q_{17}
	= \cdots
	q_{36}
	= \frac{1}{6}.
\end{align}

Considering now the case $\mydim = 38$, 
we let $\epsilon > 0$ be a small parameter, and
we perturb the $36$ Kasner exponents
from \eqref{E:EXPONENTSFORDIMENSION36}
by $\epsilon$ so that they are bounded in magnitude by
$< 1/6$:
\begin{align} \label{E:EXPONENTSFORDIMENSION38}
	q_1 
	& 
	= q_2 
	= \cdots 
	= q_{15} 
	:= - \frac{1}{6} + \epsilon,
	&&
	q_{16}
	= q_{17}
	= \cdots
	q_{36}
	= \frac{1}{6} - \epsilon.
\end{align}
Notice that by \eqref{E:KASNEREXPONENTSSUMTOONE}-\eqref{E:KASNEREXPONENTSSQUARESSUMTOONE},
assuming \eqref{E:EXPONENTSFORDIMENSION38}, any solution 
$(q_{37},q_{38})$ to the following system
yields, when complemented with the exponents \eqref{E:EXPONENTSFORDIMENSION38}, 
a complete set of Kasner exponents:
\begin{subequations}
\begin{align} \label{E:FIRSTEQNFORQ37ANDQ38}
	q_{37} + q_{38}
	& = 6 \epsilon,
		\\
	q_{37}^2 + q_{38}^2
	& = 12 \epsilon
		-
		36 \epsilon^2.
		\label{E:SECONDEQNFORQ37ANDQ38}
\end{align}
\end{subequations}
Using \eqref{E:FIRSTEQNFORQ37ANDQ38} to solve for $q_{38}$ in terms of $q_{37}$ and then
substituting into \eqref{E:SECONDEQNFORQ37ANDQ38},
we obtain the equation
$q_{37}^2 
- 
6 \epsilon q_{37}
- 6 \epsilon
+ 36 \epsilon^2
= 0$,
which has the solutions
\begin{align} \label{E:Q37SOLUTIONS}
	q_{37}
	& = 3 \epsilon
	\pm \sqrt{6 \epsilon - 27 \epsilon^2}.
\end{align}
We now observe that for any $\epsilon > 0$ sufficiently small,
the corresponding solutions $q_{37}$ to \eqref{E:Q37SOLUTIONS}
are real and bounded in magnitude by
$\lesssim \sqrt{\epsilon}$. From \eqref{E:FIRSTEQNFORQ37ANDQ38},
we deduce that the same statement holds for the corresponding
exponent $q_{38}$.
In view of \eqref{E:EXPONENTSFORDIMENSION38},
we see that for $\epsilon > 0$ sufficiently small,
the exponents $\lbrace q_i \rbrace_{i=1,\cdots,38}$ constructed in this fashion
satisfy \eqref{E:KASNEREXPONENTSSUMTOONE}-\eqref{E:KASNEREXPONENTSSQUARESSUMTOONE}
and \eqref{E:KASNEREXPONENTSNOTTOOBIG}. Moreover, in the cases
$\mydim \geq 39$, we can complement these $38$ Kasner exponents with
others as follows: $q_i = 0$ for $39 \leq i \leq \mydim$. In total, we have constructed,
for any $\mydim \geq 38$, 
sets of Kasner exponents $\lbrace q_i \rbrace_{i=1,\cdots,\mydim}$
that satisfy\eqref{E:KASNEREXPONENTSSUMTOONE}-\eqref{E:KASNEREXPONENTSSQUARESSUMTOONE}
and \eqref{E:KASNEREXPONENTSNOTTOOBIG}.
That is, the Kasner solutions whose perturbations
we study in our main theorem exist when $\mydim \geq 38$.

\section{Norms and Bootstrap Assumptions}
\label{S:NORMSANDBOOTSTRAPASSUMPTIONS}
In this section, we define the various norms that we use to control the solution.
We also state bootstrap assumptions for the solution norms; the bootstrap
assumptions are convenient for our analysis in subsequent sections.

\subsection{Constants featured in the norms}
\label{SS:PARAMETERSINNORMS}
The norms that we define in Subsect.\,\ref{SS:NORMSTHATWEUSETOCONTROLSOLUTION} 
involve the positive numbers
$\Worstexp$ and $\Room$ featured in the following definition.

\begin{definition}[\textbf{The constants $\Worstexp$ and $\Room$}]
\label{D:PARAMETERSINNORMS}
Assuming that the Kasner exponents satisfy the condition \eqref{E:KASNEREXPONENTSNOTTOOBIG},
we fix positive numbers $\Worstexp$ and $\Room$ verifying
the following inequalities:
\begin{align} \label{E:WORSTEXPANDROOMLOTSOFUSEFULINEQUALITIES}
	0 
	< 
	\Room 
	< 
	\Room + \max_{i=1,\cdots,\mydim} |q_i| 
	< 
	\Worstexp 
	< 
	\Worstexp + 2 \Room 
	< \frac{1}{6}.
\end{align}

\end{definition}

\begin{remark}
	We consider $\Worstexp$ and $\Room$
	to be fixed for the remainder of the article.
\end{remark}

\subsection{Pointwise norms}
\label{SS:POINTWISENORMS}
To control $\Sigma_t$-tangent tensorfields,
we will rely on two kinds of pointwise norms:
one that refers to the transported spatial coordinate
frame, and the standard geometric norm that is based on the Riemannian metric $g$.

\begin{definition}[\textbf{Pointwise norms}]
	\label{D:POINTWISENORMS}
	For $\Sigma_t$-tangent type $\binom{l}{m}$ tensorfields
	$T$,
	we define
	\begin{subequations}
	\begin{align} \label{E:POINTWISEFRAMENORM}
		|T|_{Frame}
		&  := 
			\left\lbrace
				\sum_{a_1,\cdots,a_l,b_1,\cdots,b_m=1}^{\mydim}
				\left|
					T_{b_1 \cdots b_n}^{\ \ \ \ \ a_1 \cdots a_m}
				\right|^2
			\right\rbrace^{1/2},
				\\
		|T|_g
		&  := 
			\left\lbrace
				g_{a_1 a_1'} 
				\cdots
				g_{a_l a_l'}
				g^{b_1 b_1'}
				\cdots
				g^{b_m b_m'}
				T_{b_1 \cdots b_m}^{\ \ \ \ \ a_1 \cdots a_l}
				T_{b_1' \cdots b_m'}^{\ \ \ \ \ a_1' \cdots a_l'}
			\right\rbrace^{1/2}.
				\label{E:POINTWISEGNORM}
	\end{align}
	\end{subequations}
\end{definition}

\subsection{Lebesgue and Sobolev norms}
\label{SS:LEBESGUEANDSOBOLEVNORMS}
In this subsection, we define the Lebesgue and Sobolev norms
that we will use to control the solution.
We start by defining the $\partial_{\vec{I}}$-derivative
of a tensorfield.

\begin{definition}[\textbf{Derivative of a tensorfield}]
	If $T_{b_1 \cdots b_m}^{\ \ \ \ \ a_1 \cdots a_l}$ is a type $\binom{l}{m}$
	$\Sigma_t$-tangent tensorfield and $\vec{I}$ is a spatial multi-index,
	then we define $\partial_{\vec{I}} T$ to be the tensorfield whose components
	$(\partial_{\vec{I}} T)_{b_1 \cdots b_m}^{\ \ \ \ \ a_1 \cdots a_l}$
	relative to the CMC-transported spatial coordinate frame are the following:
	\begin{align} \label{E:DERIVATIVEOFTENSORFIELD}
		(\partial_{\vec{I}} T)_{b_1 \cdots b_m}^{\ \ \ \ \ a_1 \cdots a_l}
		& := 
		\partial_{\vec{I}} (T_{b_1 \cdots b_m}^{\ \ \ \ \ a_1 \cdots a_l}).
	\end{align}
\end{definition}	

\begin{remark}
	The operator $\partial_{\vec{I}}$, 
	when acting on $\Sigma_t$-tangent tensorfields,
	can be given the following invariant interpretation:
	it can be viewed as repeated Lie differentiation with respect to the 
	(globally defined) coordinate vectorfields 
	$\lbrace \partial_i \rbrace_{i=1,\cdots,\mydim}$.
\end{remark}

In what follows, 
$\| \cdot \|_{L^2(\Sigma_t)}$
and
$\| \cdot \|_{L^{\infty}(\Sigma_t)}$
denote the standard Lebesgue norms for scalar functions on $\Sigma_t$;
see Subsubsect.\,\ref{SSS:BASICNORMS}.
We now define additional Sobolev and Lebesgue norms that we will use in our analysis.

\begin{definition}[\textbf{Sobolev and Lebesgue norms}]
	\label{D:SOBOLEVANDLEBESGUENORMS}
	If $T$ is a type $\binom{l}{m}$
	$\Sigma_t$-tangent tensorfield,
	$p \in \lbrace 2, \infty \rbrace$, 
	and $M \geq 0$ is an integer,
	then we define
	\begin{subequations}
	\begin{align} \label{E:FRAMEL2NORM}
		\left\|
			T
		\right\|_{L_{Frame}^p(\Sigma_t)}
		& :=
		\left\|
			|T|_{Frame}
		\right\|_{L^p(\Sigma_t)},
		&
		\left\|
			T
		\right\|_{L_g^p(\Sigma_t)}
		& :=
		\left\|
			|T|_g
		\right\|_{L^p(\Sigma_t)},
			\\
		\left\|
			T
		\right\|_{W_{Frame}^{M,\infty}(\Sigma_t)}
		& :=
		\sum_{|\vec{I}| \leq M}
		\left\|
			|\partial_{\vec{I}} T|_{Frame}
		\right\|_{L^{\infty}(\Sigma_t)},
		&
		\left\|
			T
		\right\|_{W_g^{M,\infty}(\Sigma_t)}
		& :=
		\sum_{|\vec{I}| \leq M}
		\left\|
			|\partial_{\vec{I}} T|_g
		\right\|_{L^{\infty}(\Sigma_t)},
			\\
		\left\|
			T
		\right\|_{\dot{W}_{Frame}^{M,\infty}(\Sigma_t)}
		& :=
		\sum_{|\vec{I}| = M}
		\left\|
			|\partial_{\vec{I}} T|_{Frame}
		\right\|_{L^{\infty}(\Sigma_t)},
		&
		\left\|
			T
		\right\|_{\dot{W}_g^{M,\infty}(\Sigma_t)}
		& :=
		\sum_{|\vec{I}| = M}
		\left\|
			|\partial_{\vec{I}} T|_g
		\right\|_{L^{\infty}(\Sigma_t)},
			\\
		\left\|
			T
		\right\|_{H_{Frame}^M(\Sigma_t)}
		& :=
		\left\lbrace
			\sum_{|\vec{I}| \leq M}
			\left\|
				|\partial_{\vec{I}} T|_{Frame}
			\right\|_{L^2(\Sigma_t)}^2
		\right\rbrace^{1/2},
		&
		\left\|
			T
		\right\|_{H_g^M(\Sigma_t)}
		& :=
		\left\lbrace
			\sum_{|\vec{I}| \leq M}
			\left\|
				|\partial_{\vec{I}} T|_g
			\right\|_{L^2(\Sigma_t)}^2
		\right\rbrace^{1/2},
			\\
		\left\|
			T
		\right\|_{\dot{H}_{Frame}^M(\Sigma_t)}
		& :=
		\left\lbrace
			\sum_{|\vec{I}| = M}
			\left\|
				|\partial_{\vec{I}} T|_{Frame}
			\right\|_{L^2(\Sigma_t)}^2
		\right\rbrace^{1/2},
		&
		\left\|
			T
		\right\|_{\dot{H}_g^M(\Sigma_t)}
		& :=
		\left\lbrace
			\sum_{|\vec{I}| = M}
			\left\|
				|\partial_{\vec{I}} T|_g
			\right\|_{L^2(\Sigma_t)}^2
		\right\rbrace^{1/2}.
	\end{align}
	\end{subequations}
\end{definition}

\begin{remark}[\textbf{The omission of norm subscripts when appropriate}]
	\label{R:OMISSIONOFNORMSUBSCRIPTS}
	If $T$ is a scalar function, then we typically omit the subscripts ``$Frame$'' and ``$g$''
	in the norms since there is no danger of confusion over how to measure the norm of $T$.
	For example if $f$ is scalar function, then we write 
	$
	\| f \|_{L^{\infty}(\Sigma_t)}
	$
	instead of
	$
	\| f \|_{L_{Frame}^{\infty}(\Sigma_t)}
	$
	or
	$
	\| f \|_{L_g^{\infty}(\Sigma_t)}
	$.
\end{remark}	

\subsection{The specific norms that we use to control the solution}
\label{SS:NORMSTHATWEUSETOCONTROLSOLUTION}
Let $\Blowupexp \gg 1$ be a large parameter, to be chosen later.
We recall that $\Worstexp > 0$ and $\Room > 0$ are the real numbers fixed in
Subsect.\,\ref{SS:PARAMETERSINNORMS}.

To control the solution variables $(g,\SecondFund,n)$, we will rely on a combination of norms for the low-order
derivatives of the solution and norms for its high-order derivatives.
We now define the low-order norms.
\begin{definition}[\textbf{Low norms}]
	\label{D:LOWNORMS}
	Let $\KasnerMetric$ and $\KasnerSecondFund$ denote the background Kasner solution variables
	and let $\Worstexp$ and $\Room$ be the fixed constants 
	(which depend on the background Kasner solution $(\KasnerMetric,\KasnerSecondFund)$)
	satisfying \eqref{E:WORSTEXPANDROOMLOTSOFUSEFULINEQUALITIES}.
	We define
	\begin{subequations}
	\begin{align}
		\MetricLownorm(t)
			:=
			\max
			\Bigg\lbrace
				&
				t^{2 \Worstexp}
				\left\|
					g - \KasnerMetric
				\right\|_{L_{Frame}^{\infty}(\Sigma_t)},
					\,
				t^{2 \Worstexp}
				\left\|
					g^{-1} - \KasnerMetric^{-1}
				\right\|_{L_{Frame}^{\infty}(\Sigma_t)},
						\label{E:METRICLOWNORM} \\
			& 
				t
				\left\|
					\SecondFund - \KasnerSecondFund
				\right\|_{W_{Frame}^{2,\infty}(\Sigma_t)},
					\,
					\left\|
					\left|
						t \SecondFund 
					\right|_g
					-
					1
					\right\|_{L^{\infty}(\Sigma_t)}
				\Bigg\rbrace,
				 \notag	\\		
		\LapseLownorm(t)
		& := 
				 t^{-(2 - 10 \Worstexp - \Room)} 
				\| n-1 \|_{L^{\infty}(\Sigma_t)}.
				\label{E:LAPSELOWNORM}
		\end{align}
		\end{subequations}
\end{definition}

We will use the following norms to control the high-order
derivatives of the solution.
\begin{definition}[\textbf{High norms}]
	\label{D:HIGHNORMS}
	Let $\Worstexp$ and $\Room$ be the fixed constants (which depend on the background Kasner solution $(\KasnerMetric,\KasnerSecondFund)$)
	that satisfy \eqref{E:WORSTEXPANDROOMLOTSOFUSEFULINEQUALITIES}.
	Let $\Blowupexp \gg 1$ be a real parameter (to be chosen later)
	and let $N \gg 1$ be an integer-valued parameter (also to be chosen later).
	We define
	\begin{subequations}
	\begin{align}
		\MetricHighnorm(t)
		:=
		\max
		\Bigg\lbrace
		&
		t^{\Blowupexp + 1} 
		\left\|
			\SecondFund
		\right\|_{\dot{H}_g^N(\Sigma_t)},
			\,
		t^{\Blowupexp + 1}
		\left\|
			\partial g
		\right\|_{\dot{H}_g^N(\Sigma_t)},
		\label{E:METRICHIGHNORM} \\
		&
		t^{\Blowupexp + 3 \Worstexp + \Room} 
		\left\|
			\SecondFund
		\right\|_{\dot{H}_g^{N-1}(\Sigma_t)},
			\,
		t^{\Blowupexp + 3 \Worstexp + \Room} 
		\left\|
			\SecondFund
		\right\|_{\dot{H}_{Frame}^{N-1}(\Sigma_t)},
			\notag \\
	& 
		t^{\Blowupexp + \Worstexp + \Room} 
		\left\|
			g
		\right\|_{\dot{H}_g^N(\Sigma_t)},
			\,
		t^{\Blowupexp + \Worstexp + \Room} 
		\left\|
			g^{-1}
		\right\|_{\dot{H}_g^N(\Sigma_t)},
			\notag 
		 \\
		&
		t^{\Blowupexp + 2 \Worstexp + \Room}
		\left\|
			\partial g
		\right\|_{\dot{H}_g^{N-1}(\Sigma_t)},
		 \notag \\
		&
		t^{\Blowupexp + 2 \Worstexp + \Room}
		\left\|
			g
		\right\|_{\dot{H}_{Frame}^N(\Sigma_t)},
			\,
		t^{\Blowupexp + 2 \Worstexp + \Room}
		\left\|
			g^{-1}
		\right\|_{\dot{H}_{Frame}^N(\Sigma_t)},
			\notag 
		 \\
		&
		t^{\Blowupexp + 5 \Worstexp + 3 \Room - 1}
		\left\|
			g
		\right\|_{\dot{H}_{Frame}^{N-1}(\Sigma_t)},
			\,
		t^{\Blowupexp + 5 \Worstexp + 3 \Room - 1} 
		\left\|
			g^{-1}
		\right\|_{\dot{H}_{Frame}^{N-1}(\Sigma_t)}
			\notag 
	\Bigg\rbrace,
		\notag \\
	\LapseHighnorm(t)
	:=
	\max
	\Bigg\lbrace
	&
		t^{\Blowupexp + 1}
		\left\|
			\partial n
		\right\|_{\dot{H}_g^N(\Sigma_t)},
			\,
		t^{\Blowupexp}
		\left\|
			 n
		\right\|_{\dot{H}^N(\Sigma_t)},
			\,
		t^{\Blowupexp + \Worstexp - 1}
		\left\|
			n
		\right\|_{\dot{H}^{N-1}(\Sigma_t)}
		\Bigg\rbrace.
		\label{E:LAPSEHIGHNORM}
	\end{align}
	\end{subequations}
\end{definition}

\subsection{Bootstrap assumptions}
\label{SS:BOOTSTRAP}
To facilitate our analysis, we find it convenient to rely on bootstrap assumptions.
Let $\TBoot \in (0,1)$ be a ``bootstrap time''.
Until the proof of Theorem~\ref{T:MAINTHM}, we 
assume that the perturbed solution exists classically
for $(t,x) \in (\TBoot,1] \times \mathbb{T}^{\mydim}$
and that the following bootstrap assumptions hold
for the norms from Defs.\,\ref{D:LOWNORMS} and~\ref{D:HIGHNORMS}:
\begin{align} \label{E:BOOTSTRAPASSUMPTIONS}
	\MetricLownorm(t)
		+
	\MetricHighnorm(t)
		+
	\LapseLownorm(t)
		+
	\LapseHighnorm(t)
	& \leq \varepsilon,
	&&
	t \in (\TBoot,1],
\end{align}
where $\varepsilon > 0$ is a small bootstrap parameter.

\begin{remark}[\textbf{The required smallness of $\varepsilon$ depends on various parameters}]
	\label{R:SMALLNESSOFEPSILON}
	We will continually adjust 
	the required smallness of $\varepsilon$
	throughout our analysis.
	The required smallness of $\varepsilon$
	is allowed to depend on $N$, $\Blowupexp$, $\Worstexp$, $\Room$, and $\mydim$,
	but we will often avoid pointing this out.
\end{remark}

\section{Estimates for the Kasner Solution and Preliminary Technical Estimates}
\label{S:INTERPOLATION}
In this section, we derive some simple estimates for the background Kasner solution
and provide other basic estimates that we will use throughout the paper.

\subsection{Basic estimates for the Kasner solution}
\label{SS:BASICESTIMATEFORKASNER}
In controlling various error terms, we will rely on the following
simple estimates for the background Kasner solution.

\begin{lemma}[\textbf{Basic estimates for the Kasner solution}]
\label{L:BASICESTIMATEFORKASNER}
The following estimates hold for $t \in (0,1]$:
\begin{subequations}
\begin{align} \label{E:KASNERMETRICESTIMATES}
	\left\|
		\KasnerMetric
	\right\|_{L_{Frame}^{\infty}(\Sigma_t)}
	& \leq 
		t^{\Room - 2 \Worstexp},
	&&
	\left\|
		\KasnerMetric^{-1}
	\right\|_{L_{Frame}^{\infty}(\Sigma_t)}
	\leq 
		t^{\Room - 2 \Worstexp},
			\\
	\left\|
		\KasnerSecondFund
	\right\|_{L_{Frame}^{\infty}(\Sigma_t)}
	& \leq 
		t^{-1}.
	&&
	\label{E:KASNERSECONDFUNDESTIMATES}
\end{align}
\end{subequations}
\end{lemma}

\begin{proof}
	Recall that the Kasner solution variables
	$\KasnerMetric$
	and
	$\KasnerSecondFund$
	are given, relative to the transported spatial coordinates,
	by the expressions
	\eqref{E:KASNERSPATIALMETRIC}	
	and
	\eqref{E:KASNERSECONDFUND}.
	All estimates stated in the lemma follow as straightforward consequences of
	these expressions and
	\eqref{E:WORSTEXPANDROOMLOTSOFUSEFULINEQUALITIES}.
\end{proof}

\subsection{Norm comparisons}
\label{SS:NORMCOMPARISONS}
In controlling various error terms,
we will compare the pointwise norms
$|\cdot|_{Frame}$
and
$|\cdot|_g$.
Our comparisons will often rely on the following lemma.

\begin{lemma}[\textbf{Norm comparisons}]
	\label{L:NORMCOMPARISONS}
	Let $T$ be a type $\binom{l}{m}$ $\Sigma_t$-tangent tensorfield.
	Under the bootstrap assumptions \eqref{E:BOOTSTRAPASSUMPTIONS},
	there exists a universal constant $C_* > 1$ \underline{independent of $N$ and $\Blowupexp$}
	(but depending on $m$ and $l$)
	such that if $\varepsilon \leq 1$, 
	then the following estimates hold for 
	the pointwise norms of Def.\,\ref{D:POINTWISENORMS}
	for $t \in (\TBoot,1]$:
	\begin{align}	 \label{E:POINTWISENORMCOMPARISON}
			C_*^{-1} t^{(l+m) \Worstexp} |T|_g
			& 
			\leq |T|_{Frame}
			\leq 
			C_* t^{-(l+m) \Worstexp} |T|_g.
	\end{align}
\end{lemma}

\begin{proof}
	Let $\delta$ denote the standard Euclidean metric on $\Sigma_t$,
	i.e., relative to the transported spatial coordinates,
	$\delta$ has components $\delta_{ij}$ equal to $\mbox{\upshape diag}(1,1,\cdots,1)$,
	and likewise for the inverse Euclidean metric $\delta^{-1}$.
	Then in view of the definition of the norm
	$|\cdot|_{Frame}$, we have, schematically,
	$|T|_{Frame} = |(\delta)^l (\delta^{-1})^m T^2|^{1/2}$,
	and by $g$-Cauchy--Schwarz, the RHS of the previous expression is
	\[
	\leq C_* |\delta|_g^{l/2} |\delta^{-1}|_g^{m/2} |T|_g
	= C_* 
		\left\lbrace
			g^{ac} g^{bd} \delta_{ab} \delta_{cd}
		\right\rbrace^{l/4}
		\left\lbrace
			g_{a'c'} g_{b'd'} (\delta^{-1})^{a'b'} (\delta^{-1})^{c'd'}
		\right\rbrace^{m/4}
		|T|_g.
	\]
	From Def.~\ref{D:LOWNORMS},
	the estimate \eqref{E:KASNERMETRICESTIMATES},
	and the bootstrap assumptions,
	we deduce that the RHS of the above expression 
	(which we consider from the perspective of the transported coordinate frame)
	is 
	$
	\leq C_* |g^{-1}|_{Frame}^{l/2} |g|_{Frame}^{m/2} |T|_g
	\leq C_* t^{-(m+l)\Worstexp} |T|_g$,
	which yields the second inequality in \eqref{E:POINTWISENORMCOMPARISON}.
	To obtain the first inequality in
	\eqref{E:POINTWISENORMCOMPARISON},
	we note that
	we have, schematically,
	$|T|_g = |(g)^l (g^{-1})^m T^2|^{1/2}$.
	Writing out the RHS of the previous expression in the
	transported coordinate frame, we see that
	it is $\leq C_* |g|_{Frame}^{l/2} |g^{-1}|_{Frame}^{m/2} |T|_{Frame}$.
	From Def.~\ref{D:LOWNORMS},
	the estimate \eqref{E:KASNERMETRICESTIMATES},
	and the bootstrap assumptions,
	we deduce that the RHS of the previous expression is
	$\leq C_* t^{-(m+l)\Worstexp} |T|_{Frame}$,
	which yields the first inequality in \eqref{E:POINTWISENORMCOMPARISON}.
	\end{proof}

\subsection{Sobolev interpolation and product inequalities}
\label{SS:INTERPOLATION}
In this subsection, we provide some Sobolev interpolation and product inequalities
that we will use to control various error terms. We start with the following
lemma, which provides basic interpolation estimates.

\begin{lemma}[\textbf{Basic interpolation estimates}]
	\label{L:STANDARDSOBOLEVINTERPOLATION}
	If $M_1$ and $M_2$ are non-negative integers
	with $M_1 \leq M_2$ and $v$ is a scalar function,
	then the following inequalities hold:
	\begin{align} \label{E:BASICINTERPOLATION}
		\| v \|_{\dot{H}^{M_1}(\Sigma_t)}
		& \lesssim 
		\| v \|_{L^{\infty}(\Sigma_t)}^{1 - M_1/M_2} \| v \|_{\dot{H}^{M_2}}^{M_1/M_2}
		\lesssim
		\| v \|_{L^{\infty}(\Sigma_t)}
		+
		\| v \|_{\dot{H}^{M_2}}.
	\end{align}
	
	If in addition $M_1 + 1 + \lfloor \mydim/2 \rfloor \leq M_2$,
	then the following inequalities hold:
	\begin{align} \label{E:LINFINITYSOBOLEVINTERPOLATION}
		\left\| 
			v 
		\right\|_{\dot{W}^{M_1,\infty}(\Sigma_t)}
		& \lesssim 
		\| v \|_{H^{M_1 + 1 + \lfloor \mydim/2 \rfloor}(\Sigma_t)}
		\lesssim 
		\| v \|_{L^{\infty}(\Sigma_t)}
		+
		\| v \|_{\dot{H}^{M_2}(\Sigma_t)}.
	\end{align}
	
	\end{lemma}	

	\begin{proof}
		The first inequality in \eqref{E:BASICINTERPOLATION} 
		follows as special case of Nirenberg's famous interpolation results \cite{lN1959},
		except that on the RHS, we have replaced the norm
		$\| \cdot \|_{L^2(\Sigma_t)}$ with $\| \cdot \|_{L^{\infty}(\Sigma_t)}$;
		the replacement is possible because 
		of the estimate
		$\| v \|_{L^2(\Sigma_t)} 
		\lesssim
		\| v \|_{L^{\infty}(\Sigma_t)}
		$
		for scalar functions $v$
		(which holds because $\mathbb{T}^{\mydim}$ is compact).
		Strictly speaking, Nirenberg stated his results for functions
		defined on $\mathbb{R}^{\mydim}$, 
		but the arguments given in his paper can be used to
		derive the same estimates for
		functions defined on $\mathbb{T}^{\mydim}$.
		The second inequality in \eqref{E:BASICINTERPOLATION} 
		follows from the first and Young's inequality.
		
		The first inequality in \eqref{E:LINFINITYSOBOLEVINTERPOLATION}
		is a standard Sobolev embedding result on $\mathbb{T}^{\mydim}$.
		The second inequality in \eqref{E:LINFINITYSOBOLEVINTERPOLATION}
		follows from the first and
		the second inequality in \eqref{E:BASICINTERPOLATION}.
	\end{proof}

	\begin{lemma}[\textbf{Sobolev product estimates for tensorfields}]
	\label{L:L2PRODUCTBOUNDINERMSOFLINFINITYANDHMDOT}
	Let $M$ be a non-negative integer,
	let $\lbrace v_r \rbrace_{r=1}^R$ be a finite collection of
	$\Sigma_t$-tangent tensorfields, 
	and let $v = \prod_{r=1}^R \partial_{\vec{I}_r} v_r$ be a (schematically denoted) tensor product,
	possibly involving contractions.
	Assume that the bootstrap assumptions \eqref{E:BOOTSTRAPASSUMPTIONS} hold
	for some $\varepsilon$ with $\varepsilon \leq 1$.
	Then the following estimate holds for $t \in (\TBoot,1]$:
	\begin{align} \label{E:FRAMENORML2PRODUCTBOUNDINERMSOFLINFINITYANDHMDOT}
		\max_{\sum_{r=1}^R |\vec{I}_r| = M}
		\left\| 
			\prod_{r=1}^R \partial_{\vec{I}_r} v_r 
		\right\|_{L_{Frame}^2(\Sigma_t)}
		& \lesssim 
		\sum_{r=1}^R \| v_r \|_{\dot{H}_{Frame}^M(\Sigma_t)} 
		\prod_{s \neq r} \| v_s \|_{L_{Frame}^{\infty}(\Sigma_t)}.
	\end{align}
	
	Moreover, 
	under the same assumptions as above, 
	and assuming that $v$ is type $\binom{l}{m}$,
	the following estimate holds for $t \in (\TBoot,1]$:
	\begin{align} \label{E:L2PRODUCTBOUNDINERMSOFLINFINITYANDHMDOT}
		\max_{\sum_{r=1}^R |\vec{I}_r| = M}
		\left\| 
			\prod_{r=1}^R \partial_{\vec{I}_r} v_r 
		\right\|_{L_g^2(\Sigma_t)}
		& \lesssim 
		t^{-(l+m)\Worstexp}
		\sum_{r=1}^R \| v_r \|_{\dot{H}_{Frame}^M(\Sigma_t)} 
		\prod_{s \neq r} \| v_s \|_{L_{Frame}^{\infty}(\Sigma_t)}.
	\end{align}
	\end{lemma}
	
	\begin{proof}
		If the $\lbrace v_r \rbrace_{r=1}^R$ are all scalar functions
		(and hence $l=m=0$), then inequality
		\eqref{E:FRAMENORML2PRODUCTBOUNDINERMSOFLINFINITYANDHMDOT}
		is standard; 
		it is proved, for example, 
		as \cite{hR2009b}*{Lemma~6.16}.
		If one or more of the $v_r$ are not scalar functions,
		then the estimate \eqref{E:FRAMENORML2PRODUCTBOUNDINERMSOFLINFINITYANDHMDOT}
		follows a straightforward consequence of the
		estimate \eqref{E:FRAMENORML2PRODUCTBOUNDINERMSOFLINFINITYANDHMDOT}
		for scalar functions, essentially by writing out
		the definition of LHS~\eqref{E:FRAMENORML2PRODUCTBOUNDINERMSOFLINFINITYANDHMDOT}
		relative to the transported spatial coordinate frame and estimating the components of all tensorfields.
		
		To prove \eqref{E:L2PRODUCTBOUNDINERMSOFLINFINITYANDHMDOT},
		we first use \eqref{E:POINTWISENORMCOMPARISON}
		to deduce that
		\[
		\max_{\sum_{r=1}^R |\vec{I}_r| = M}
		\left\| 
			\prod_{r=1}^R \partial_{\vec{I}_r} v_r 
		\right\|_{L_g^2(\Sigma_t)}
		\lesssim
		t^{-(l+m)\Worstexp}
		\max_{\sum_{r=1}^R |\vec{I}_r| = M}
		\left\| 
			\prod_{r=1}^R \partial_{\vec{I}_r} v_r 
		\right\|_{L_{Frame}^2(\Sigma_t)}.
		\]
		We then use \eqref{E:FRAMENORML2PRODUCTBOUNDINERMSOFLINFINITYANDHMDOT}
		to conclude that the RHS of the previous expression is
		$\lesssim \mbox{RHS~\eqref{E:L2PRODUCTBOUNDINERMSOFLINFINITYANDHMDOT}}$ as desired.
	
	\end{proof}

\subsection{Sobolev embedding}
\label{SS:SOBOLEVEMBEDDING}
We will use the following lemma to obtain 
$L^{\infty}(\Sigma_t)$ control over some additional
low-order derivatives that are not directly controlled 
by the low-order norms from Def.\,\ref{D:LOWNORMS}.
To achieve the desired control, we borrow, via interpolation
with sufficiently large $N$,
a small amount of the high norm. 
Because the high norms are 
quite weak near $t=0$ (at least when $\Blowupexp$ is large), 
the interpolation introduces
singular behavior 
into the estimates of the lemma,
represented by the factors of $t^{- \Blowupexp \updelta}$ on the
RHSs of the estimates.
However, $\updelta \to 0$ as $N \to \infty$. 
Thus, at fixed $\Blowupexp$, if $N$ is large,
then the following basic principle, central to our approach,
applies:
\begin{quote}
	\emph{The singular contribution to the behavior of the low-order norms
	coming from the high norms is very small}.
\end{quote}

\begin{lemma}[\textbf{Sobolev embedding, borrowing only a small amount of high norm}]
	\label{L:SOBOLEVBORROWALITTLEHIGHNORM}
	There exists a parameter\footnote{Recall that, as we described in 
	Subsubsect.\,\ref{SSS:PARAMETERS}, we allow $\updelta$ to vary from line to line
	and use the convention that a sum of two $\updelta$'s is another $\updelta$.} 
	$\updelta > 0$, depending on $N$ and $\mydim$, 
	such that
	$\lim_{N \to \infty} \updelta = 0$ 
	and such that if $N \geq 5 + \lfloor \mydim/2 \rfloor$,
	then the following estimates hold for $t \in (\TBoot,1]$:
	\begin{subequations}
	\begin{align} \label{E:UPTOFOURDERIVATIVESOFGLINFINITYSOBOLEV}
		\| g - \KasnerMetric \|_{W_{Frame}^{4,\infty}(\Sigma_t)}
		& \lesssim t^{- 2 \Worstexp - \Blowupexp \updelta} 
			\left\lbrace
				\MetricLownorm(t) + \MetricHighnorm(t)
			\right\rbrace,
			\\
		\| g^{-1} - \KasnerMetric^{-1} \|_{W_{Frame}^{4,\infty}(\Sigma_t)}
		& \lesssim t^{- 2 \Worstexp - \Blowupexp \updelta} 
			\left\lbrace
				\MetricLownorm(t) + \MetricHighnorm(t)
			\right\rbrace,
			\label{E:UPTOFOURDERIVATIVESOFGINVERSELINFINITYSOBOLEV} \\
		\| n - 1 \|_{W^{4,\infty}(\Sigma_t)}
		& \lesssim
			t^{2 - 10 \Worstexp - \Room - \Blowupexp \updelta}
			\left\lbrace
				\LapseLownorm(t) + \LapseHighnorm(t)
			\right\rbrace.
			\label{E:UPTOFOURDERIVATIVESOFLAPSELINFINITYSOBOLEV}
	\end{align}
	\end{subequations}
\end{lemma}

\begin{proof}
	To prove \eqref{E:UPTOFOURDERIVATIVESOFLAPSELINFINITYSOBOLEV}, 
	we first use \eqref{E:LINFINITYSOBOLEVINTERPOLATION} 
	and the second inequality in \eqref{E:BASICINTERPOLATION}
	to deduce
	$
	\| n - 1 \|_{W^{4,\infty}(\Sigma_t)}
	\lesssim 
	\| n - 1 \|_{L^{\infty}(\Sigma_t)}
	+
	\| n - 1 \|_{\dot{H}^{5 + \lfloor \mydim/2 \rfloor}(\Sigma_t)}
	$.
	Then using the first inequality in \eqref{E:BASICINTERPOLATION},
	we deduce that for $N \geq 5 + \lfloor \mydim/2 \rfloor$,
	the RHS of the previous expression is
	$
	\lesssim
	\| n - 1 \|_{L^{\infty}(\Sigma_t)}
	+
	\| n - 1 \|_{L^{\infty}(\Sigma_t)}^{1 - \updelta}
	\| n \|_{\dot{H}^N(\Sigma_t)}^{\updelta}
	$,
	where $\updelta := (5 + \lfloor \mydim/2 \rfloor)/N$.
	Combining these estimates and appealing to definitions
	\eqref{E:LAPSELOWNORM}
	and
	\eqref{E:LAPSEHIGHNORM},
	we find that
	$
	\| n - 1 \|_{W^{4,\infty}(\Sigma_t)}
	\lesssim 
	t^{2 - 10 \Worstexp - \Room} \LapseLownorm(t)
	+
	\left\lbrace
		t^{2 - 10 \Worstexp - \Room} \LapseLownorm(t)
	\right\rbrace^{1 - \updelta}
	\left\lbrace
		t^{- \Blowupexp} \LapseHighnorm(t)
	\right\rbrace^{\updelta}
	$.
	Finally, we use Young's inequality to deduce 
	\begin{align} \label{E:MAINSTEPUPTOFOURDERIVATIVESOFLAPSELINFINITYSOBOLEV}
	&
	t^{2 - 10 \Worstexp - \Room} \LapseLownorm(t)
	+
	\left\lbrace
		t^{2 - 10 \Worstexp - \Room} \LapseLownorm(t)
	\right\rbrace^{1 - \updelta}
	\left\lbrace
		t^{- \Blowupexp} \LapseHighnorm(t)
	\right\rbrace^{\updelta}
		\\
	&
	=
	t^{2 - 10 \Worstexp - \Room} \LapseLownorm(t)
	+
	t^{(2 - 10 \Worstexp - \Room)(1 - \updelta) - \Blowupexp \updelta} 
	\left\lbrace
		\LapseLownorm(t)
	\right\rbrace^{1 - \updelta}
	\left\lbrace
		\LapseHighnorm(t)
	\right\rbrace^{\updelta}
		\notag \\
	&
	\lesssim
	t^{2 - 10 \Worstexp - \Room - \Blowupexp \updelta} 
	\left\lbrace
		\LapseLownorm(t)
			+
		\LapseHighnorm(t)
	\right\rbrace,
	\notag
	\end{align}
	where in obtaining the last inequality in \eqref{E:MAINSTEPUPTOFOURDERIVATIVESOFLAPSELINFINITYSOBOLEV},
	we have used \eqref{E:WORSTEXPANDROOMLOTSOFUSEFULINEQUALITIES},
	the assumption $\Blowupexp \geq 1$, 
	and our running convention 
	that $\updelta$ is free to vary from line to line, subject only to the restriction
	$\lim_{N \to \infty} \updelta = 0$.
	In total, we have derived the desired bound \eqref{E:UPTOFOURDERIVATIVESOFLAPSELINFINITYSOBOLEV}.
	
	The remaining estimates in the lemma can be proved using similar arguments that take into account
	Defs.\,\ref{D:LOWNORMS} and~\ref{D:HIGHNORMS}, and we omit the details.
\end{proof}

\section{Control of \texorpdfstring{$n$}{the Lapse} in Terms of \texorpdfstring{$g$ and $\SecondFund$}{the First and Second Fundamental Forms}}
\label{S:CONTROLOFLAPSEINTERMSOFMETRICANDSECONDFUND}
Our primary goal in this section is to prove the next proposition, 
in which we derive the main estimates for $n$. 
The proof of the proposition is located in
Subsect.\,\ref{SS:PROOFOFPROPOSITIONELLIPTICANDMAXIMUMPRINCIPLEESTIMATESFORTHELAPSE}.
In Subsects.\,\ref{SS:LAPSEEQUATIONS}-\ref{SS:CONTROLOFLAPSEERRORTERMS},
we derive the identities and estimates that we will use when proving the proposition.
The proposition shows
that $n$ is controlled, in various norms, by complementary norms of $g$ and $\SecondFund$.
Achieving such control is possible since $n$ solves an elliptic PDE with source terms that
depend on $\Sc$; here it makes sense to recall that,
as we explained at the end of Subsect.\,\ref{SS:MAINIDEASINPROOF}, 
to obtain suitable estimates for $n$ at the top derivative level,
we use the Hamiltonian constraint \eqref{E:HAMILTONIAN}
to algebraically replace, in the elliptic lapse PDE, 
$\Sc$ with terms that depend on $\SecondFund$.

\begin{proposition}[\textbf{Control of $n$ in terms of $g$ and $\SecondFund$}]
	\label{P:ELLIPTICANDMAXIMUMPRINCIPLEESTIMATESFORTHELAPSE}
	Recall that $\MetricLownorm(t)$, $\LapseLownorm(t)$, $\MetricHighnorm(t)$, and $\LapseHighnorm(t)$
	are the norms from Defs.\,\ref{D:LOWNORMS} and~\ref{D:HIGHNORMS},
	and assume that the bootstrap assumptions \eqref{E:BOOTSTRAPASSUMPTIONS} hold.
	There exists a universal constant $C_* > 0$ \underline{independent of $N$ and $\Blowupexp$}
	such that if $N$ is sufficiently large in a manner that depends on $\Blowupexp$
	and if $\varepsilon$ is sufficiently small,
	then the following estimates hold for $t \in (\TBoot,1]$
	(where, as we described in Subsect.\,\ref{SSS:CONSTANTS}, constants ``$C$'' are allowed to depend on $N$ and other quantities).
	
	\medskip
	
	\noindent \underline{\textbf{Estimates at the lowest order}}.
	The following estimate holds:
	\begin{align} \label{E:LAPSELOWNORMELLIPTIC}
		\LapseLownorm(t)
		& \leq C
				\left\lbrace
					\MetricLownorm(t) + \MetricHighnorm(t)
				\right\rbrace.
	\end{align}
	
	\medskip
	\noindent \underline{\textbf{Estimate for} $\partial_t n$}.
	The following estimate holds:
	\begin{align} 	\label{E:LAPSETIMEDERIVATIVEESTIMATE}	
		\left\| 
			\partial_t n
		\right \|_{L^{\infty}(\Sigma_t)}
		& \leq C
				t^{1 - 10 \Worstexp - \Room}
				\left\lbrace
					\MetricLownorm(t) + \MetricHighnorm(t)
				\right\rbrace.
	\end{align}
	
		\medskip
		\noindent \underline{\textbf{Top-order estimates}}.
		If $|\vec{I}|=N$, then the following estimates hold:
		\begin{align}
		t^{\Blowupexp + 1} 
		\left\|
			\partial \partial_{\vec{I}} n
		\right\|_{L_g^2(\Sigma_t)}
		+
		t^{\Blowupexp}
		\left\|
			\partial_{\vec{I}} n
		\right\|_{L^2(\Sigma_t)}
		& \leq 
			C_* 
			t^{\Blowupexp + 1} 
			\left\|
				\partial_{\vec{I}} \SecondFund
			\right\|_{L_g^2(\Sigma_t)}
				+
				C 
				t^{\Room}
				\left\lbrace
					\MetricLownorm(t) + \MetricHighnorm(t)
				\right\rbrace
					\label{E:LAPSETOPORDERELLIPTIC}
					\\
			& \leq 
				C_* 
				\MetricHighnorm(t)
				+
				C 
				t^{\Room}
				\left\lbrace
					\MetricLownorm(t) + \MetricHighnorm(t)
				\right\rbrace.
				\notag
		\end{align}
		
		\medskip
		\noindent \underline{\textbf{Near-top-order estimates}}.
		The following estimate holds:
		\begin{align}
			t^{\Blowupexp + \Worstexp - 1}
			\left\|
				 n
			\right\|_{\dot{H}^{N-1}(\Sigma_t)}
			&
			\leq
				C 
				\left\lbrace
					\MetricLownorm(t) + \MetricHighnorm(t)
				\right\rbrace.
					\label{E:LAPSEJUSTBELOWTOPORDERELLIPTIC}
		\end{align}
\end{proposition}

\subsection{The equations}
\label{SS:LAPSEEQUATIONS}
In this subsection, we derive the equations 
that we use to control $n$ in terms of $g$ and $\SecondFund$.

\subsubsection{The equation at the lowest order}
\label{SSS:LAPSEEQUATIONATLOWESTORDER}
We start with the following lemma, which
we will use to derive estimates for $\| n-1 \|_{L^{\infty}(\Sigma_t)}$.
   
\begin{lemma}[\textbf{A rewriting of the lapse equation}]
\label{L:EQUATIONFORLAPSEATLOWESTORDER}
The quantity $n-1$ verifies the following elliptic PDE on $\Sigma_t$:
\begin{align} \label{E:REWRITTENLAPSELOWERORDER}
g^{ab} \nabla_a \nabla_b (n-1)
	-
				(n - 1) 
				\left( 
					t^{-2}
					+
					\Sc
				\right)
	& = 
		 \Sc.
\end{align}

\end{lemma}

\begin{proof}
	The lemma is a simple consequence of equation
\eqref{E:LAPSECMC}.
\end{proof}

\subsubsection{The equations verified by the time derivative}
\label{SSS:LAPSEEQUATIONTIMEDERIVATIVE}
The next lemma is an analog of Lemma~\ref{L:EQUATIONFORLAPSEATLOWESTORDER} for $\partial_t n$,
and we will use it to derive estimates for $\| \partial_t n \|_{L^{\infty}(\Sigma_t)}$.
We remark that we need estimates for $\| \partial_t n \|_{L^{\infty}(\Sigma_t)}$ only in
Subsect.\,\ref{SS:LENGTHOFCAUSALGEODESICS}, when we bound the length of 
past-directed causal geodesic segments.

\begin{lemma}[\textbf{The $\partial_t$-commuted lapse equation}]
\label{L:EQUATIONFORLAPSETIMEDERIVATIVE}
The quantity $\partial_t n$ verifies the following elliptic PDE on $\Sigma_t$:
\begin{align} \label{E:REWRITTENLAPSELOWERORDERTIMEDERIVATIVECOMMUTED}
g^{ab} \nabla_a \nabla_b \partial_t n
	-
				(\partial_t n)
				\left( 
					t^{-2}
					+
					\Sc
				\right)
	& = 
		 \mathfrak{N}',
\end{align}
where (see Remark~\ref{R:ALWAYSMIXED})
\begin{align} \label{E:ERRORTERMSTIMEDERIVATIVEOFLAPSE}
	\mathfrak{N}'
	& \mycong	
		\sum_{i_1 + i_2 + i_3 = 2} 
		\sum_{p=0}^1
			(n-1)^p
			(g^{-1})^2 
			(\partial^{i_1} n)
			(\partial^{i_2} g)
			\partial^{i_3} \SecondFund
				\\
	& \ \
		+
		\sum_{i_1 + i_2 + i_3 = 1} 
		\sum_{p=0}^1
			(n-1)^p
			(g^{-1})^3 
			(\partial g)
			(\partial^{i_1} n)
			(\partial^{i_2} g)
			\partial^{i_3} \SecondFund
			+
			t^{-3} (n-1)
			\notag
				\\
		& \ \
			+ 
			n g^{-1} \SecondFund \partial^2 n
			+
			n (g^{-1})^2 \SecondFund (\partial g) \partial n
			+
			\sum_{i_1 + i_2 + i_3 = 1} 
			(g^{-1})^2
			(\partial n)
			(\partial^{i_1} n)
			(\partial^{i_2} g)
			\partial^{i_3} \SecondFund.
			\notag
\end{align}

\end{lemma}

\begin{proof}
	Relative to CMC-transported spatial coordinates, 
	the elliptic PDE \eqref{E:REWRITTENLAPSELOWERORDER} can be expressed as
	$
	g^{ab} \partial_a \partial_b n
	-
	g^{ab} \Gamma_{a \ b}^{\ c} \partial_c n
	-
				(n - 1) 
				\left( 
					t^{-2}
					+
					\Sc
				\right)
	= \Sc
	$,
	where, schematically,
	$
		\Sc
		\mycong 
		(g^{-1})^2 \partial^2 g
		+
		(g^{-1})^3 (\partial g)^2
	$
	(see \eqref{E:RICCICURVATUREEXACT} and recall that $\Sc = \Ric_{\ a}^a$)
	and
	$
	g^{ab}
	\Gamma_{a \ b}^{\ c}
	\mycong
	g^{-2} \partial g
	$.
	Commuting the elliptic PDE with
	$\partial_t$, using
	equations \eqref{E:PARTIALTGCMC}
	and \eqref{E:PARTIALTGINVERSECMC}
	to algebraically substitute for $\partial_t g$ and $\partial_t g^{-1}$,
	and carrying out straightforward computations,
	we conclude the desired equation \eqref{E:REWRITTENLAPSELOWERORDERTIMEDERIVATIVECOMMUTED}.
\end{proof}

\subsubsection{The equations verified by the high-order derivatives}
\label{SSS:LAPSEEQUATIONHIGHORDERS}
We will use the equations in the next lemma to control 
the $L^2(\Sigma_t)$ norms of high-order derivatives of $n$.

\begin{remark}[\textbf{Borderline error terms vs.\ junk error terms}]
	In the rest of the paper, we will denote difficult ``borderline''
	error terms by decorating them with
	``$(Border)$'', e.g., $\leftexp{(Border;\vec{I})}{\mathfrak{N}}$.
	These error terms must treated with care since at the top derivative level, 
	there is no room (in the sense of powers of $t$)
	in our estimates for such terms.
	In contrast, the error terms that we decorate with
	``$(Junk)$'', such as $\leftexp{(Junk;\vec{I})}{\mathfrak{N}}$,
	are easier to treat in the sense that there is some room in our
	estimates.
\end{remark}	

\begin{lemma}[\textbf{The $\vec{I}$-commuted lapse equation}]
\label{L:COMMUTEDLAPSEEQUATIONS}
For each spatial multi-index $\vec{I}$ with $|\vec{I}|=N$, 
$\partial_{\vec{I}} n$ verifies the following equation:
\begin{align}  \label{E:COMMUTEDLAPSEHIGHORDER}
	t^{\Blowupexp + 2} g^{ab} \partial_a \partial_b \partial_{\vec{I}} n
	-
	t^{\Blowupexp} \partial_{\vec{I}} n
	& = 
		\leftexp{(Border;\vec{I})}{\mathfrak{N}}
		+
		\leftexp{(Junk;\vec{I})}{\mathfrak{N}},
	\end{align}
	where (see Subsubsect.\,\ref{SSS:SCHEMATIC} regarding our use of the notation $\mycongstar$ and $\mycong$, and see Remark~\ref{R:ALWAYSMIXED})
	\begin{subequations}
	\begin{align}
		\leftexp{(Border;\vec{I})}{\mathfrak{N}}
		& \mycongstar
			(t^{\Blowupexp} \partial_{\vec{I}} n) 
			\left\lbrace
				(t \SecondFund) \cdot (t \SecondFund)
				-
				1
			\right\rbrace
			+
			(t \SecondFund) \cdot (t^{\Blowupexp + 1} \partial_{\vec{I}} \SecondFund),
			 \label{E:LAPSETOPORDERBORDERLINERRORTERM} \\
		\leftexp{(Junk;\vec{I})}{\mathfrak{N}}
		& \mycong 
			\mathop{\sum_{\vec{I}_1 + \vec{I}_2 + \vec{I}_3 = \vec{I}}}_{|\vec{I}_1| \leq |\vec{I}|-1}
			t^{\Blowupexp + 2} 
			\left\lbrace
				\partial_{\vec{I}_1} (n-1)
			\right\rbrace
			(\partial_{\vec{I}_2} \SecondFund)
			\partial_{\vec{I}_3} \SecondFund
			+
			\mathop{\sum_{\vec{I}_1 + \vec{I}_2 = \vec{I}}}_{|\vec{I}_1|, |\vec{I}_2| \leq |\vec{I}|-1}
			t^{\Blowupexp + 2} 
			(\partial_{\vec{I}_1} \SecondFund)
			\partial_{\vec{I}_2} \SecondFund
				\label{E:LAPSETOPORDERJUNKRRORTERM} \\
		& \ \
		+
		\sum_{\vec{I}_1 + \vec{I}_2 + \vec{I}_3 + \vec{I}_4 = \vec{I}}
			t^{\Blowupexp + 2} 
			(\partial_{\vec{I}_1} g^{-1})
			(\partial_{\vec{I}_2} g^{-1})
			(\partial \partial_{\vec{I}_3} g)
			\partial \partial_{\vec{I}_4} n
			+
			\mathop{\sum_{\vec{I}_1 + \vec{I}_2 = \vec{I}}}_{|\vec{I}_2| \leq |\vec{I}|-1}
			t^{\Blowupexp + 2} 
			(\partial_{\vec{I}_1} g^{-1})
			\partial^2 \partial_{\vec{I}_2} n.
			\notag
	\end{align}
\end{subequations}

	Furthermore, for each spatial multi-index $\vec{I}$ with $|\vec{I}|=N-1$,
	$\partial_{\vec{I}} n$ verifies the following equation:
	\begin{align} \label{E:COMMUTEDLAPSELOWERORDER}
	t^{\Blowupexp + 1 + \Worstexp}
	g^{ab} \partial_a \partial_b \partial_{\vec{I}} n
	-
	t^{\Blowupexp + \Worstexp - 1}
	\left( 
	1
	+
	t^2 
	\Sc
	\right)
	\partial_{\vec{I}} n
	& 
	= 
	\leftexp{\vec{I}}{\widetilde{\mathfrak{N}}},
\end{align}
where
\begin{align} \label{E:LAPSEJUSTBELOWTOPORDERERRORTERM}
		\leftexp{\vec{I}}{\widetilde{\mathfrak{N}}}
		& \mycong 
			\sum_{\vec{I}_1 + \vec{I}_2 + \vec{I}_3 = \vec{I}}
		  t^{\Blowupexp + 1 + \Worstexp}
			(\partial_{\vec{I}_1} g^{-1})
			(\partial_{\vec{I}_2} g^{-1})
			\partial^2 \partial_{\vec{I}_3} g
				 \\
			& \ \
			+
			\sum_{\vec{I}_1 + \vec{I}_2 + \cdots + \vec{I}_5 = \vec{I}}
			t^{\Blowupexp + 1 + \Worstexp}
			(\partial_{\vec{I}_1} g^{-1})
			(\partial_{\vec{I}_2} g^{-1})
			(\partial_{\vec{I}_3} g^{-1})
			(\partial \partial_{\vec{I}_4} g)
			\partial \partial_{\vec{I}_5} g
				\notag \\
			& \ \
			+
			\mathop{\sum_{\vec{I}_1 + \vec{I}_2 + \vec{I}_3 + \vec{I}_4 = \vec{I}}}_{|\vec{I}_1| \leq |\vec{I}|-1}
			t^{\Blowupexp + 1 + \Worstexp}
			\left\lbrace
				\partial_{\vec{I}_1} (n-1)
			\right\rbrace
			(\partial_{\vec{I}_2} g^{-1})
			(\partial_{\vec{I}_3} g^{-1})
			\partial^2 \partial_{\vec{I}_4} g
			\notag
				\\
			& \ \
			+
			\mathop{\sum_{\vec{I}_1 + \vec{I}_2 + \cdots + \vec{I}_6 = \vec{I}}}_{|\vec{I}_1| \leq |\vec{I}|-1}
			t^{\Blowupexp + 1 + \Worstexp}
			\left\lbrace
				\partial_{\vec{I}_1} (n-1)
			\right\rbrace
			(\partial_{\vec{I}_2} g^{-1})
			(\partial_{\vec{I}_3} g^{-1})
			(\partial_{\vec{I}_4} g^{-1})
			(\partial \partial_{\vec{I}_5} g)
			\partial \partial_{\vec{I}_6} g
				\notag \\		
			& \ \
			+
			\sum_{\vec{I}_1 + \vec{I}_2 + \vec{I}_3 + \vec{I}_4 = \vec{I}}
			t^{\Blowupexp + 1 + \Worstexp}
			(\partial_{\vec{I}_1} g^{-1})
			(\partial_{\vec{I}_2} g^{-1})
			(\partial \partial_{\vec{I}_3} g)
			\partial \partial_{\vec{I}_4} n
			\notag
				\\
		& \ \
			+
			\mathop{\sum_{\vec{I}_1 + \vec{I}_2 = \vec{I}}}_{|\vec{I}_2| \leq |\vec{I}|-1}
			t^{\Blowupexp + 1 + \Worstexp}
			(\partial_{\vec{I}_1} g^{-1})
			\partial^2 \partial_{\vec{I}_2} n.
			\notag
	\end{align}

\end{lemma}

\begin{proof}
	First, we use equation \eqref{E:HAMILTONIAN}
	to substitute for $\Sc$ in equation
	\eqref{E:LAPSECMC},
	use that
	$g^{ab} \nabla_a \nabla_b n = g^{ab} \partial_a \partial_b n - g^{ab} \Gamma_{a \ b}^{\ c} \partial_c n$,
	and multiply by $t^{\Blowupexp + 2}$,
	thereby deriving the equation
	\begin{align}  \label{E:LAPSEHIGHORDER}
		t^{\Blowupexp + 2} g^{ab} \partial_a \partial_b n
		-
		t^{\Blowupexp}
		(n-1)
		& = 
				t^{\Blowupexp}
				(n - 1) 
				\left\lbrace
					(t \SecondFund_{\ b}^a)(t \SecondFund_{\ a}^b)
					-
				1
				\right\rbrace
				-
				t^{\Blowupexp}
				+ 
				t^{\Blowupexp + 2} \SecondFund_{\ b}^a \SecondFund_{\ a}^b
				+
				t^{\Blowupexp + 2} g^{ab} \Gamma_{a \ b}^{\ c} \partial_c n.
	\end{align}
	Commuting equation \eqref{E:LAPSEHIGHORDER} with $\partial_{\vec{I}}$
	and using the schematic identity
	$g^{ab} \Gamma_{a \ b}^{\ c} \partial_c n \mycong g^{-2} (\partial g) \partial n$,
	we easily deduce \eqref{E:COMMUTEDLAPSEHIGHORDER}.
	We clarify that the terms on RHS~\eqref{E:LAPSETOPORDERBORDERLINERRORTERM}
	arise, respectively, when all $|\vec{I}|$ derivatives fall on the factor $n-1$ in the first
	product on RHS~\eqref{E:LAPSEHIGHORDER} or on one of the factors of $\SecondFund$
	in the next-to-last product $t^{\Blowupexp + 2} \SecondFund_{\ b}^a \SecondFund_{\ a}^b$
	on RHS~\eqref{E:LAPSEHIGHORDER}. It is for this reason that the coefficients
	in the corresponding products do not depend on $N$, as is indicated by the symbol
	$\mycongstar$ in equation \eqref{E:LAPSETOPORDERBORDERLINERRORTERM}.
	
	Equation \eqref{E:COMMUTEDLAPSELOWERORDER} follows similarly from multiplying equation \eqref{E:REWRITTENLAPSELOWERORDER}
	by $t^{\Blowupexp + 1 + \Worstexp}$
	and using the schematic identity
	$
	\Sc
	\mycong 
		(g^{-1})^2 \partial^2 g
		+
		(g^{-1})^3 (\partial g)^2
	$,
	which follows from
	\eqref{E:RICCICURVATUREEXACT}
	and the fact that 
	$\Sc = \Ric_{\ a}^a$.

\end{proof}

\subsection{Control of error the terms in the lapse estimates}
\label{SS:CONTROLOFLAPSEERRORTERMS}
In this subsection, we derive estimates for the error terms in the equations of
Lemmas \ref{L:EQUATIONFORLAPSETIMEDERIVATIVE} and~\ref{L:COMMUTEDLAPSEEQUATIONS}.

\begin{lemma}[\textbf{Control of the error terms in the elliptic estimates}] 
	Assume that the bootstrap assumptions \eqref{E:BOOTSTRAPASSUMPTIONS} hold.
	There exists a universal constant $C_* > 0$ \underline{independent of $N$ and $\Blowupexp$}
	such that if $N$ is sufficiently large in a manner that depends on $\Blowupexp$
	and if $\varepsilon$ is sufficiently small, 
	then the following estimates hold for $t \in (\TBoot,1]$
	(where, as we described in Subsect.\,\ref{SSS:CONSTANTS}, constants ``$C$'' are allowed to depend on $N$ and other quantities).
	
	\medskip
	\noindent \underline{\textbf{Control of the error term in the equation verified by} $\partial_t n$}.
	The following estimate holds for the error term $\mathfrak{N}'$ 
	defined in \eqref{E:ERRORTERMSTIMEDERIVATIVEOFLAPSE}:
	\begin{align} \label{E:LINFINTYERRORTERMSTIMEDERIVATIVEOFLAPSE}
		\left\|
			\mathfrak{N}'
		\right\|_{L^{\infty}(\Sigma_t)}
		& \leq C 
			t^{-1 - 10 \Worstexp - \Room}
			\left\lbrace
				\MetricLownorm(t)
				+
				\MetricHighnorm(t)
				+
				\LapseLownorm(t)
				+
				\LapseHighnorm(t)
			\right\rbrace.
	\end{align}
	\medskip
	
	\noindent \underline{\textbf{Borderline top-order error term estimates}}.
	If $|\vec{I}|=N$, then the following estimate holds for the error
	term from \eqref{E:LAPSETOPORDERBORDERLINERRORTERM}:
	\begin{align} \label{E:LAPSETOPORDERBORDERERRORTERMBOUNDINTERMSOFMETRIC}
		\left\|
			\leftexp{(Border;\vec{I})}{\mathfrak{N}}
		\right\|_{L^2(\Sigma_t)}
		& \leq 
			C_* \varepsilon
			t^{\Blowupexp}
			\left\|
				\partial_{\vec{I}} n
			\right\|_{L^2(\Sigma_t)}
			+
			C_* 
			t^{\Blowupexp + 1} 
			\left\|
				\partial_{\vec{I}} \SecondFund
			\right\|_{L_g^2(\Sigma_t)}.
	\end{align}
	
	\medskip
	
	\noindent \underline{\textbf{Non-borderline top-order error term estimates}}.
	The following estimate holds for the error
	term from \eqref{E:LAPSETOPORDERJUNKRRORTERM}:
	\begin{align}
		\max_{|\vec{I}|=N}
		\left\|
			\leftexp{(Junk;\vec{I})}{\mathfrak{N}}
		\right\|_{L^2(\Sigma_t)}
		& 
		\leq
		C
		t^{\Room}
		\left\lbrace
			\MetricLownorm(t) + \MetricHighnorm(t)
		\right\rbrace.
		\label{E:LAPSETOPORDERJUNKERRORTERMBOUNDINTERMSOFMETRIC}
	\end{align}

	\noindent \underline{\textbf{Just-below-top-order error term estimates}}.
	The following estimate holds for the error
	term from \eqref{E:LAPSEJUSTBELOWTOPORDERERRORTERM}:
	\begin{align} \label{E:LAPSEJUSTBELOWTOPORDERBORDERERRORTERMBOUNDINTERMSOFMETRIC}
		\max_{|\vec{I}|=N-1}
		\left\|
			\leftexp{\vec{I}}{\widetilde{\mathfrak{N}}}
		\right\|_{L^2(\Sigma_t)}
		& 
		\leq
		C
		\left\lbrace
			\MetricLownorm(t) + \MetricHighnorm(t)
		\right\rbrace.
	\end{align}
\end{lemma}

\begin{proof}
Throughout this proof, we will assume that $\Blowupexp \updelta$ is sufficiently small
(and in particular that $\Blowupexp \updelta < \Room$);
in view of the discussion in Subsect.\,\ref{SS:SOBOLEVEMBEDDING}, we see that
at fixed $\Blowupexp$, this can be achieved by choosing $N$ to be sufficiently large.

\medskip
\noindent \textbf{Proof of \eqref{E:LINFINTYERRORTERMSTIMEDERIVATIVEOFLAPSE}}:
This estimate follows in a straightforward fashion from using
\eqref{E:WORSTEXPANDROOMLOTSOFUSEFULINEQUALITIES},
Def.~\ref{D:LOWNORMS},
Lemma~\ref{L:BASICESTIMATEFORKASNER},
and Lemma~\ref{L:SOBOLEVBORROWALITTLEHIGHNORM}
to control the products on RHS~\eqref{E:ERRORTERMSTIMEDERIVATIVEOFLAPSE} 
in the norm $\| \cdot \|_{L^{\infty}(\Sigma_t)}$,
where we bound each factor in the products in the norm
$\| \cdot \|_{L_{Frame}^{\infty}(\Sigma_t)}$.
We remark that the most singular (in the sense of powers of $t$) 
products on RHS~\eqref{E:ERRORTERMSTIMEDERIVATIVEOFLAPSE} are
$
t^{-3} (n-1)
$
and the terms in the second sum with $p=i_1=0$.
	
\medskip

\noindent \textbf{Proof of \eqref{E:LAPSETOPORDERBORDERERRORTERMBOUNDINTERMSOFMETRIC}}:
From Def.~\ref{D:LOWNORMS}
and the $g$-Cauchy-Schwarz inequality, 
we deduce that the magnitude of RHS~\eqref{E:LAPSETOPORDERBORDERLINERRORTERM} is 
$
\leq
C_* 
\left\lbrace
	1 + \MetricLownorm(t)
\right\rbrace
\left\lbrace
	\MetricLownorm(t)
	|t^{\Blowupexp} \partial_{\vec{I}} n|
	+
	\left|
		t^{\Blowupexp + 1} \partial_{\vec{I}} \SecondFund
	\right|_g
\right\rbrace
$.
From this bound
and the bootstrap assumption bound
$\MetricLownorm(t) \leq \varepsilon$,
we conclude \eqref{E:LAPSETOPORDERBORDERERRORTERMBOUNDINTERMSOFMETRIC}.

\medskip

\noindent \textbf{Proof of \eqref{E:LAPSETOPORDERJUNKERRORTERMBOUNDINTERMSOFMETRIC}}:
Let $\vec{I}$ be a spatial multi-index with $|\vec{I}|=N$.
To bound the first sum on RHS~\eqref{E:LAPSETOPORDERJUNKRRORTERM},
we first consider the top-order case in which $|\vec{I}_2| = N$ or $|\vec{I}_3| = N$.
Then using $g$-Cauchy--Schwarz, we see that
the terms under consideration are bounded in the norm
$\| \cdot \|_{L^2(\Sigma_t)}$
by 
$\lesssim 
t^{\Blowupexp+2}
\| n-1 \|_{L^{\infty}(\Sigma_t)}
\left\|
	\SecondFund
\right\|_{L_g^{\infty}(\Sigma_t)}
\left\|
	\SecondFund
\right\|_{\dot{H}_g^N(\Sigma_t)}
$.
From Defs.\,\ref{D:LOWNORMS} and~\ref{D:HIGHNORMS}
and the bootstrap assumptions,
we see that RHS of the previous expression is
$\lesssim
t^{2 - 10 \Worstexp - \Room} \MetricHighnorm(t)
$.
In view of \eqref{E:WORSTEXPANDROOMLOTSOFUSEFULINEQUALITIES},
we see that RHS of the previous expression is
$\lesssim \mbox{RHS~\eqref{E:LAPSETOPORDERJUNKERRORTERMBOUNDINTERMSOFMETRIC}}$
as desired.
We now consider the remaining cases, 
in which $|\vec{I}_1|, |\vec{I}_2|, |\vec{I}_3| \leq N-1$.
Using 
\eqref{E:BASICINTERPOLATION}
and
\eqref{E:FRAMENORML2PRODUCTBOUNDINERMSOFLINFINITYANDHMDOT},
we bound 
(using that $|\vec{I}|=N$)
the terms under consideration as follows:
\begin{align} \label{E:FIRSTSTEPLAPSETOPORDERJUNKRRORTERMTERM1}
	&
	\mathop{\sum_{\vec{I}_1 + \vec{I}_2 + \vec{I}_3 = \vec{I}}}_{|\vec{I}_1|, |\vec{I}_2|, |\vec{I}_3| \leq N-1}
		t^{\Blowupexp + 2}
		\left\|
			 \left\lbrace
			\partial_{\vec{I}_1} (n-1)
			\right\rbrace
			(\partial_{\vec{I}_2} \SecondFund)
			\partial_{\vec{I}_3} \SecondFund
		\right\|_{L^2(\Sigma_t)}
			\\
	& \lesssim
		t^{\Blowupexp + 2} 
		\left\|
			n-1
		\right\|_{W^{1,\infty}(\Sigma_t)}
		\left\|
			\SecondFund
		\right\|_{W_{Frame}^{1,\infty}(\Sigma_t)}
		\left\|
			\SecondFund
		\right\|_{\dot{H}_{Frame}^{N-1}(\Sigma_t)}
			\notag \\
	& \ \
		+
		t^{\Blowupexp + 2} 
		\left\|
			\SecondFund
		\right\|_{W_{Frame}^{1,\infty}(\Sigma_t)}
		\left\|
			\SecondFund
		\right\|_{\dot{W}_{Frame}^{1,\infty}(\Sigma_t)}
		\left\|
			n
		\right\|_{\dot{H}^{N-1}(\Sigma_t)}
			\notag \\
	& \ \
		+
		t^{\Blowupexp + 2} 
		\left\|
			n-1
		\right\|_{W^{1,\infty}(\Sigma_t)}
		\left\|
			\SecondFund
		\right\|_{W_{Frame}^{1,\infty}(\Sigma_t)}
		\left\|
			\SecondFund
		\right\|_{\dot{W}_{Frame}^{1,\infty}(\Sigma_t)}.
		\notag 
	\end{align}
From Defs.\,\ref{D:LOWNORMS} and~\ref{D:HIGHNORMS},
the estimates \eqref{E:KASNERSECONDFUNDESTIMATES} and \eqref{E:UPTOFOURDERIVATIVESOFLAPSELINFINITYSOBOLEV},
and the bootstrap assumptions, we deduce that
$\mbox{RHS~\eqref{E:FIRSTSTEPLAPSETOPORDERJUNKRRORTERMTERM1}}
\lesssim 
t^{3 - 13 \Worstexp - 2 \Room - \Blowupexp \updelta}
\MetricHighnorm(t)
+
t^{1 - \Worstexp}
\MetricLownorm(t)
$.
In view of \eqref{E:WORSTEXPANDROOMLOTSOFUSEFULINEQUALITIES},
we see that the RHS of the previous expression is
$\lesssim \mbox{RHS~\eqref{E:LAPSETOPORDERJUNKERRORTERMBOUNDINTERMSOFMETRIC}}$
as desired.

To bound the second sum on RHS~\eqref{E:LAPSETOPORDERJUNKRRORTERM},
we first use \eqref{E:FRAMENORML2PRODUCTBOUNDINERMSOFLINFINITYANDHMDOT}
to deduce that the terms under consideration are bounded in the norm
$\| \cdot \|_{L^2(\Sigma_t)}$ by
$\lesssim 
t^{\Blowupexp + 2}
\left\|
	\SecondFund
\right\|_{\dot{W}_{Frame}^{1,\infty}(\Sigma_t)}
\left\|
	\SecondFund
\right\|_{\dot{H}_{Frame}^{N-1}(\Sigma_t)}
$.
From Defs.\,\ref{D:LOWNORMS} and~\ref{D:HIGHNORMS}
and the bootstrap assumptions,
we see that the RHS of the previous expression is 
$\lesssim 
t^{1 - 3 \Worstexp - \Room}
\MetricHighnorm(t)
$,
which, in view of \eqref{E:WORSTEXPANDROOMLOTSOFUSEFULINEQUALITIES},
is $\lesssim \mbox{RHS~\eqref{E:LAPSETOPORDERJUNKERRORTERMBOUNDINTERMSOFMETRIC}}$
as desired.

To bound the third sum on RHS~\eqref{E:LAPSETOPORDERJUNKRRORTERM},
we first consider the top-order case in which $|\vec{I}_3| = N$.
Using $g$-Cauchy--Schwarz and the fact that $|g^{-1}|_g \lesssim 1$, we see that 
the terms under consideration are bounded in the norm
$\| \cdot \|_{L^2(\Sigma_t)}$
by 
$\lesssim 
t^{\Blowupexp + 2}
\left\|
	\partial n
\right\|_{L_g^{\infty}(\Sigma_t)}
\left\|
	\partial g
\right\|_{\dot{H}_g^N(\Sigma_t)}
$.
Applying \eqref{E:POINTWISENORMCOMPARISON}
(with $l=0$ and $m=1$)
to 
$
\left\|
	\partial n
\right\|_{L_g^{\infty}(\Sigma_t)}
$,
we see that the RHS of the previous expression is 
$\lesssim 
t^{\Blowupexp + 2 - \Worstexp}
\| n \|_{\dot{W}^{1,\infty}(\Sigma_t)}
\left\|
	\partial g
\right\|_{\dot{H}_g^N(\Sigma_t)}
$.
From Defs.\,\ref{D:LOWNORMS} and~\ref{D:HIGHNORMS},
the estimate \eqref{E:UPTOFOURDERIVATIVESOFLAPSELINFINITYSOBOLEV},
and the bootstrap assumptions,
we see that the RHS of the previous expression is
$
\lesssim
t^{3 - 11 \Worstexp - \Room - \Blowupexp \updelta}  
\MetricHighnorm(t)
$,
which, in view of \eqref{E:WORSTEXPANDROOMLOTSOFUSEFULINEQUALITIES},
is $\lesssim \mbox{RHS~\eqref{E:LAPSETOPORDERJUNKERRORTERMBOUNDINTERMSOFMETRIC}}$
as desired.
We now consider the top-order case in which $|\vec{I}_4| = N$.
Arguing as above, we deduce that 
the terms under consideration are bounded in the norm
$\| \cdot \|_{L^2(\Sigma_t)}$
by 
$\lesssim 
t^{\Blowupexp + 2}
\left\|
	\partial g
\right\|_{L_g^{\infty}(\Sigma_t)}
\left\|
	\partial n
\right\|_{\dot{H}_g^N(\Sigma_t)}
\lesssim 
t^{\Blowupexp + 2 - 3 \Worstexp}
\| g \|_{\dot{W}_{Frame}^{1,\infty}(\Sigma_t)}
\left\|
	\partial n
\right\|_{\dot{H}_g^N(\Sigma_t)}
$.
From Defs.\,\ref{D:LOWNORMS} and~\ref{D:HIGHNORMS},
the estimate \eqref{E:UPTOFOURDERIVATIVESOFGLINFINITYSOBOLEV},
and the bootstrap assumptions,
we see that the RHS of the previous expression is
$
\lesssim
t^{1 - 5 \Worstexp - \Blowupexp \updelta} 
\MetricLownorm(t)
\LapseHighnorm(t)
\lesssim
t^{1 - 5 \Worstexp - \Blowupexp \updelta} 
\MetricLownorm(t)
$,
which, in view of \eqref{E:WORSTEXPANDROOMLOTSOFUSEFULINEQUALITIES},
is $\lesssim \mbox{RHS~\eqref{E:LAPSETOPORDERJUNKERRORTERMBOUNDINTERMSOFMETRIC}}$
as desired.
It remains for us to handle the cases in which $|\vec{I}_3|, |\vec{I}_4| \leq N-1$.
Using 
\eqref{E:BASICINTERPOLATION}
and
\eqref{E:FRAMENORML2PRODUCTBOUNDINERMSOFLINFINITYANDHMDOT},
we bound 
(using that $|\vec{I}|=N$)
the terms under consideration as follows:
\begin{align} \label{E:FIRSTSTEPLAPSETOPORDERJUNKRRORTERMTERM2}
	&
	\mathop{\sum_{\vec{I}_1 + \vec{I}_2 + \vec{I}_3 + \vec{I}_4 = \vec{I}}}_{|\vec{I}_3|, |\vec{I}_4| \leq N-1}
		t^{\Blowupexp + 2}
		\left\|
			 (\partial_{\vec{I}_1} g^{-1})
			(\partial_{\vec{I}_2} g^{-1})
			(\partial \partial_{\vec{I}_3} g)
			\partial \partial_{\vec{I}_4} n
		\right\|_{L^2(\Sigma_t)}
			\\
	& \lesssim
		t^{\Blowupexp + 2} 
		\left\|
			n-1
		\right\|_{W^{2,\infty}(\Sigma_t)}
		\left\|
			g^{-1}
		\right\|_{W_{Frame}^{1,\infty}(\Sigma_t)}^2
		\left\|
			g
		\right\|_{\dot{H}_{Frame}^N(\Sigma_t)}
			\notag \\
	& \ \
		+
		t^{\Blowupexp + 2} 
		\left\|
			g - \KasnerMetric
		\right\|_{W^{2,\infty}(\Sigma_t)}
		\left\|
			g^{-1}
		\right\|_{W_{Frame}^{1,\infty}(\Sigma_t)}^2
		\left\|
			n
		\right\|_{\dot{H}^N(\Sigma_t)}
		\notag 
			\\
	& \ \
		+
		t^{\Blowupexp + 2} 
		\left\|
			n-1
		\right\|_{W^{2,\infty}(\Sigma_t)}
		\left\|
			g^{-1}
		\right\|_{W_{Frame}^{1,\infty}(\Sigma_t)}
		\left\|
			g - \KasnerMetric
		\right\|_{W_{Frame}^{2,\infty}(\Sigma_t)}
		\left\|
			g^{-1}
		\right\|_{\dot{H}_{Frame}^N(\Sigma_t)}
			\notag \\		
	& \ \
		+
		t^{\Blowupexp + 2} 
		\left\|
			n-1
		\right\|_{W^{2,\infty}(\Sigma_t)}
		\left\|
			g^{-1}
		\right\|_{W_{Frame}^{1,\infty}(\Sigma_t)}^2
		\left\|
			g - \KasnerMetric
		\right\|_{W_{Frame}^{2,\infty}(\Sigma_t)}.
			\notag 
\end{align}
From Defs.\,\ref{D:LOWNORMS} and~\ref{D:HIGHNORMS},
the estimates \eqref{E:KASNERMETRICESTIMATES} and 
\eqref{E:UPTOFOURDERIVATIVESOFGLINFINITYSOBOLEV}-\eqref{E:UPTOFOURDERIVATIVESOFLAPSELINFINITYSOBOLEV},
and the bootstrap assumptions, we deduce that
$\mbox{RHS~\eqref{E:FIRSTSTEPLAPSETOPORDERJUNKRRORTERMTERM2}}
\lesssim 
t^{4 - 16 \Worstexp - 2 \Room - \Blowupexp \updelta}
\left\lbrace
	\MetricLownorm(t) + \MetricHighnorm(t)
\right\rbrace
+
t^{2 - 6 \Worstexp - \Blowupexp \updelta}
\left\lbrace
	\MetricLownorm(t) + \MetricHighnorm(t)
\right\rbrace
$.
In view of \eqref{E:WORSTEXPANDROOMLOTSOFUSEFULINEQUALITIES},
we see that the RHS of the previous expression is
$\lesssim \mbox{RHS~\eqref{E:LAPSETOPORDERJUNKERRORTERMBOUNDINTERMSOFMETRIC}}$
as desired.

To bound the last sum on RHS~\eqref{E:LAPSETOPORDERJUNKRRORTERM},
we first consider the case in which $|\vec{I}_2| = N-1$.
Using \eqref{E:POINTWISENORMCOMPARISON}, we bound
the terms under consideration in the norm
$\| \cdot \|_{L^2(\Sigma_t)}$
by 
$\lesssim 
t^{\Blowupexp + 2}
\left\|
	g^{-1}
\right\|_{\dot{W}_{Frame}^{1,\infty}(\Sigma_t)}
\left\|
	\partial n
\right\|_{\dot{H}_{Frame}^N(\Sigma_t)}
\lesssim
t^{\Blowupexp + 2 - \Worstexp}
\left\|
	g^{-1}
\right\|_{\dot{W}_{Frame}^{1,\infty}(\Sigma_t)}
\left\|
	\partial n
\right\|_{\dot{H}_g^N(\Sigma_t)}
$.
From Defs.\,\ref{D:LOWNORMS} and~\ref{D:HIGHNORMS},
the estimate \eqref{E:UPTOFOURDERIVATIVESOFGINVERSELINFINITYSOBOLEV},
and the bootstrap assumptions,
we see that the RHS of the previous expression is
$
\lesssim
t^{1 - 3 \Worstexp - \Blowupexp \updelta}
\left\lbrace
	\MetricLownorm(t) + \MetricHighnorm(t)
\right\rbrace
\LapseHighnorm(t)
\lesssim
t^{1 - 3 \Worstexp - \Blowupexp \updelta}
\left\lbrace
	\MetricLownorm(t) + \MetricHighnorm(t)
\right\rbrace
$,
which, in view of \eqref{E:WORSTEXPANDROOMLOTSOFUSEFULINEQUALITIES},
is $\lesssim \mbox{RHS~\eqref{E:LAPSETOPORDERJUNKERRORTERMBOUNDINTERMSOFMETRIC}}$
as desired.
We now consider the remaining cases, 
in which $|\vec{I}_2| \leq N-2$.
Using 
\eqref{E:FRAMENORML2PRODUCTBOUNDINERMSOFLINFINITYANDHMDOT},
we bound 
(using that $|\vec{I}|=N$)
the terms under consideration as follows:
\begin{align} \label{E:FIRSTSTEPLAPSETOPORDERJUNKRRORTERMTERMLAST}
	\mathop{\sum_{\vec{I}_1 + \vec{I}_2 = \vec{I}}}_{|\vec{I}_2| \leq N-2}
		t^{\Blowupexp + 2} 
		\left\|
			(\partial_{\vec{I}_1} g^{-1})
			\partial^2 \partial_{\vec{I}_2} n
		\right\|_{L^2(\Sigma_t)}
		& \lesssim
		t^{\Blowupexp + 2} 
		\left\|
			g^{-1}
		\right\|_{\dot{W}_{Frame}^{2,\infty}(\Sigma_t)}
		\left\|
			n
		\right\|_{\dot{H}^N(\Sigma_t)}
			\\
	& \ \
		+
		t^{\Blowupexp + 2} 
		\left\|
			n
		\right\|_{\dot{W}^{2,\infty}(\Sigma_t)}
		\left\|
			g^{-1}
		\right\|_{\dot{H}_{Frame}^N(\Sigma_t)}.
		\notag 
\end{align}
From Def.~\ref{D:HIGHNORMS},
the estimates \eqref{E:UPTOFOURDERIVATIVESOFGINVERSELINFINITYSOBOLEV} 
and \eqref{E:UPTOFOURDERIVATIVESOFLAPSELINFINITYSOBOLEV},
and the bootstrap assumptions, we deduce that
$\mbox{RHS~\eqref{E:FIRSTSTEPLAPSETOPORDERJUNKRRORTERMTERMLAST}}
\lesssim 
t^{2 - 2 \Worstexp - \Blowupexp \updelta} 
\MetricLownorm(t)
+
t^{4 - 12 \Worstexp - 2 \Room - \Blowupexp \updelta}
\MetricHighnorm(t)
$.
In view of \eqref{E:WORSTEXPANDROOMLOTSOFUSEFULINEQUALITIES},
we see that the RHS of the previous expression is
$\lesssim \mbox{RHS~\eqref{E:LAPSETOPORDERJUNKERRORTERMBOUNDINTERMSOFMETRIC}}$
as desired.

\medskip

\noindent \textbf{Proof of \eqref{E:LAPSEJUSTBELOWTOPORDERBORDERERRORTERMBOUNDINTERMSOFMETRIC}}:
Let $\vec{I}$ be a spatial multi-index with $|\vec{I}|=N-1$.
To bound the first sum on RHS~\eqref{E:LAPSEJUSTBELOWTOPORDERERRORTERM},
we first consider the case in which $|\vec{I}_3|=N-1$.
Using $g$-Cauchy--Schwarz and the fact that $|g^{-1}|_g \lesssim 1$, we see that 
the terms under consideration are bounded in the norm
$\| \cdot \|_{L^2(\Sigma_t)}$
by 
$\lesssim 
t^{\Blowupexp + 1 + \Worstexp}
\left\|
	\partial^2 g
\right\|_{\dot{H}_g^{N-1}(\Sigma_t)}
$.
From the definitions of
$\|
	\cdot
\|_{\dot{H}_g^M(\Sigma_t)}
$
and
$\|
	\cdot
\|_{L_{Frame}^{\infty}}
$,
we deduce that the RHS of the previous
estimate is 
$
\lesssim 
t^{\Blowupexp + 1 + \Worstexp}
\|g^{-1}\|_{L_{Frame}^{\infty}}^{1/2}
\left\|
	\partial g
\right\|_{\dot{H}_g^N(\Sigma_t)}
$.
From Defs.\,\ref{D:LOWNORMS} and~\ref{D:HIGHNORMS},
the estimate \eqref{E:KASNERMETRICESTIMATES},
and the bootstrap assumptions,
we see that the RHS of the previous expression is
$
\lesssim
\MetricHighnorm(t)
$,
which is $\lesssim \mbox{RHS~\eqref{E:LAPSEJUSTBELOWTOPORDERBORDERERRORTERMBOUNDINTERMSOFMETRIC}}$
as desired. It remains for us to consider the remaining cases, in which
$|\vec{I}_3| \leq N-2$.
Using 
\eqref{E:BASICINTERPOLATION}
and
\eqref{E:FRAMENORML2PRODUCTBOUNDINERMSOFLINFINITYANDHMDOT},
we bound 
(using that $|\vec{I}|=N-1$)
the terms under consideration as follows:
\begin{align} \label{E:FIRSTSTEPLAPSEJUSTBELOWTOPORDERERRORTERM1}
	&
	\mathop{\sum_{\vec{I}_1 + \vec{I}_2 + \vec{I}_3 = \vec{I}}}_{|\vec{I}_3| \leq N-2}
		t^{\Blowupexp + 1 + \Worstexp}
		\left\|
			(\partial_{\vec{I}_1} g^{-1})
			(\partial_{\vec{I}_2} g^{-1})
			\partial^2 \partial_{\vec{I}_3} g			
		\right\|_{L^2(\Sigma_t)}
			\\
	& \lesssim
	t^{\Blowupexp + 1 + \Worstexp}
	\left\|
		g^{-1}
	\right\|_{W_{Frame}^{1,\infty}(\Sigma_t)}^2
	\left\|
		g
	\right\|_{\dot{H}_{Frame}^N(\Sigma_t)}
		\notag \\
	& \ \
		+
	t^{\Blowupexp + 1 + \Worstexp}
	\left\|
		g^{-1}
	\right\|_{W_{Frame}^{1,\infty}(\Sigma_t)}
	\left\|
		g
	\right\|_{\dot{W}_{Frame}^{2,\infty}(\Sigma_t)}
	\left\|
		g^{-1}
	\right\|_{\dot{H}_{Frame}^N(\Sigma_t)}
	\notag
		\\
	& \ \
		+
	t^{\Blowupexp + 1 + \Worstexp}
	\left\|
		g^{-1}
	\right\|_{W_{Frame}^{1,\infty}(\Sigma_t)}^2
	\left\|
		g
	\right\|_{\dot{W}_{Frame}^{2,\infty}(\Sigma_t)}.
	\notag
\end{align}
From Defs.\,\ref{D:LOWNORMS} and~\ref{D:HIGHNORMS},
the estimates 
\eqref{E:KASNERMETRICESTIMATES},
\eqref{E:UPTOFOURDERIVATIVESOFGLINFINITYSOBOLEV},
and
\eqref{E:UPTOFOURDERIVATIVESOFGINVERSELINFINITYSOBOLEV},
and the bootstrap assumptions,
we see that 
$
\mbox{RHS~\eqref{E:FIRSTSTEPLAPSEJUSTBELOWTOPORDERERRORTERM1}}
\lesssim
t^{1 - 5 \Worstexp - \Room - \Blowupexp \updelta}
\left\lbrace
	\MetricLownorm(t) + \MetricHighnorm(t)
\right\rbrace
$,
which, in view of \eqref{E:WORSTEXPANDROOMLOTSOFUSEFULINEQUALITIES},
is $\lesssim \mbox{RHS~\eqref{E:LAPSEJUSTBELOWTOPORDERBORDERERRORTERMBOUNDINTERMSOFMETRIC}}$
as desired.

To bound the second sum on RHS~\eqref{E:LAPSEJUSTBELOWTOPORDERERRORTERM},
we first consider the cases in which $|\vec{I}_4|=N-1$ or $|\vec{I}_5|=N-1$.
Using $g$-Cauchy--Schwarz and the fact that $|g^{-1}|_g \lesssim 1$, we see that 
the terms under consideration are bounded in the norm
$\| \cdot \|_{L^2(\Sigma_t)}$
by 
$\lesssim 
t^{\Blowupexp + 1 + \Worstexp}
\left\|
	\partial g
\right\|_{L_g^{\infty}(\Sigma_t)}
\left\|
	\partial g
\right\|_{\dot{H}_g^{N-1}(\Sigma_t)}
$.
Applying \eqref{E:POINTWISENORMCOMPARISON}
to 
$
\left\|
	\partial g
\right\|_{L_g^{\infty}(\Sigma_t)}
$
(with $l=0$ and $m=3$),
we deduce that the RHS of the previous expression is 
$
\lesssim 
t^{\Blowupexp + 1 - 2 \Worstexp}
\left\|
	g
\right\|_{\dot{W}_{Frame}^{1,\infty}(\Sigma_t)}
\left\|
	\partial g
\right\|_{\dot{H}_g^{N-1}(\Sigma_t)}
$.
From Defs.\,\ref{D:LOWNORMS} and~\ref{D:HIGHNORMS},
the estimate \eqref{E:UPTOFOURDERIVATIVESOFGLINFINITYSOBOLEV},
and the bootstrap assumptions,
we see that the RHS of the previous expression is
$
\lesssim
t^{1 - 6 \Worstexp - \Room - \Blowupexp \updelta}
\MetricHighnorm(t)
$,
which, in view of \eqref{E:WORSTEXPANDROOMLOTSOFUSEFULINEQUALITIES},
is $\lesssim \mbox{RHS~\eqref{E:LAPSEJUSTBELOWTOPORDERBORDERERRORTERMBOUNDINTERMSOFMETRIC}}$
as desired. It remains for us to consider the remaining cases, in which
$|\vec{I}_4|, |\vec{I}_5| \leq N-2$.
Using 
\eqref{E:BASICINTERPOLATION}
and
\eqref{E:FRAMENORML2PRODUCTBOUNDINERMSOFLINFINITYANDHMDOT},
we bound 
(using that $|\vec{I}|=N-1$)
the terms under consideration as follows:
\begin{align} \label{E:FIRSTSTEPLAPSEJUSTBELOWTOPORDERERRORTERM2}
	&
	\mathop{\sum_{\vec{I}_1 + \vec{I}_2 + \cdots + \vec{I}_5 = \vec{I}}}_{|\vec{I}_4|,|\vec{I}_5| \leq N-2}
		t^{\Blowupexp + 1 + \Worstexp}
		\left\|
			(\partial_{\vec{I}_1} g^{-1})
			(\partial_{\vec{I}_2} g^{-1})
			(\partial_{\vec{I}_3} g^{-1})
			(\partial \partial_{\vec{I}_4} g)
			\partial \partial_{\vec{I}_5} g
		\right\|_{L^2(\Sigma_t)}
			\\
		& \lesssim
		t^{\Blowupexp + 1 + \Worstexp}
		\left\|
			g^{-1}
		\right\|_{W_{Frame}^{1,\infty}(\Sigma_t)}^3
		\left\|
			g - \KasnerMetric
		\right\|_{W_{Frame}^{2,\infty}(\Sigma_t)}
		\left\|
			g
		\right\|_{\dot{H}_{Frame}^{N-1}(\Sigma_t)}
			\notag \\
	& \ \
		+
		t^{\Blowupexp + 1 + \Worstexp}
		\left\|
			g^{-1}
		\right\|_{W_{Frame}^{1,\infty}(\Sigma_t)}^2
		\left\|
			g - \KasnerMetric
		\right\|_{W_{Frame}^{2,\infty}(\Sigma_t)}^2
		\left\|
			g^{-1}
		\right\|_{\dot{H}_{Frame}^{N-1}(\Sigma_t)}
		\notag 
			\\
	& \ \
		+
		t^{\Blowupexp + 1 + \Worstexp}
		\left\|
			g^{-1}
		\right\|_{W_{Frame}^{1,\infty}(\Sigma_t)}^3
		\left\|
			g - \KasnerMetric
		\right\|_{W_{Frame}^{2,\infty}(\Sigma_t)}^2.
		\notag
\end{align}
From Defs.\,\ref{D:LOWNORMS} and~\ref{D:HIGHNORMS},
the estimates 
\eqref{E:KASNERMETRICESTIMATES},
\eqref{E:UPTOFOURDERIVATIVESOFGLINFINITYSOBOLEV},
and
\eqref{E:UPTOFOURDERIVATIVESOFGINVERSELINFINITYSOBOLEV},
and the bootstrap assumptions,
we see that 
$
\mbox{RHS~\eqref{E:FIRSTSTEPLAPSEJUSTBELOWTOPORDERERRORTERM2}}
\lesssim
t^{2 - 12 \Worstexp - 3 \Room - \Blowupexp \updelta}
\left\lbrace
	\MetricLownorm(t) + \MetricHighnorm(t)
\right\rbrace
$,
which, in view of \eqref{E:WORSTEXPANDROOMLOTSOFUSEFULINEQUALITIES},
is $\lesssim \mbox{RHS~\eqref{E:LAPSEJUSTBELOWTOPORDERBORDERERRORTERMBOUNDINTERMSOFMETRIC}}$
as desired.

To bound the third sum on RHS~\eqref{E:LAPSEJUSTBELOWTOPORDERERRORTERM},
we first consider the case $|\vec{I}_4|=N-1$.
Using $g$-Cauchy--Schwarz and the fact that $|g^{-1}|_g \lesssim 1$, 
we see that the terms under consideration are bounded in the norm
$\| \cdot \|_{L^2(\Sigma_t)}$
by 
$
\lesssim 
t^{\Blowupexp + 1 + \Worstexp}
\left\|
	n - 1
\right\|_{L^{\infty}(\Sigma_t)}
\left\|
	\partial^2 g
\right\|_{\dot{H}_g^{N-1}(\Sigma_t)}
$.
Next, considering the definitions 
of
$\left\|
	\cdot
\right\|_{\dot{H}_g^M(\Sigma_t)}
$
and
$\left\|
	\cdot
\right\|_{L_{Frame}^{\infty}(\Sigma_t)}
$,
we deduce that the RHS of the previous expression is 
$
\lesssim 
t^{\Blowupexp + 1 + \Worstexp}
\left\|
	n - 1
\right\|_{L^{\infty}(\Sigma_t)}
\left\|
	g^{-1}
\right\|_{L_{Frame}^{\infty}(\Sigma_t)}^{1/2}
\left\|
	\partial g
\right\|_{\dot{H}_g^N(\Sigma_t)}
$.
From Defs.\,\ref{D:LOWNORMS} and~\ref{D:HIGHNORMS},
the estimate
\eqref{E:KASNERMETRICESTIMATES},
and the bootstrap assumptions,
we see that the RHS of the previous expression is
$
\lesssim
t^{2 - 10 \Worstexp - \Room}
\MetricHighnorm(t)
$,
which, in view of \eqref{E:WORSTEXPANDROOMLOTSOFUSEFULINEQUALITIES},
is $\lesssim \mbox{RHS~\eqref{E:LAPSEJUSTBELOWTOPORDERBORDERERRORTERMBOUNDINTERMSOFMETRIC}}$
as desired. It remains for us to consider the case $|\vec{I}_4| \leq N-2$.
Using 
\eqref{E:BASICINTERPOLATION}
and
\eqref{E:FRAMENORML2PRODUCTBOUNDINERMSOFLINFINITYANDHMDOT},
we bound 
(using that $|\vec{I}|=N-1$)
the terms under consideration as follows:	
\begin{align} \label{E:THIRDSTEPLAPSEJUSTBELOWTOPORDERERRORTERM3}
	&
	\mathop{\sum_{\vec{I}_1 + \vec{I}_2 + \vec{I}_3 + \vec{I}_4 = \vec{I}}}_{|\vec{I}_1|, |\vec{I}_4| \leq N-2}
		t^{\Blowupexp + 1 + \Worstexp}
		\left\|
			\left\lbrace
				\partial_{\vec{I}_1} (n-1)
			\right\rbrace
			(\partial_{\vec{I}_2} g^{-1})
			(\partial_{\vec{I}_3} g^{-1})
			\partial^2 \partial_{\vec{I}_4} g	
		\right\|_{L^2(\Sigma_t)}
			\\
		& \lesssim
		t^{\Blowupexp + 1 + \Worstexp}
		\left\|
			n-1
		\right\|_{W^{1,\infty}(\Sigma_t)}
		\left\|
			g^{-1}
		\right\|_{W_{Frame}^{1,\infty}(\Sigma_t)}^2
		\left\|
			g
		\right\|_{\dot{H}_{Frame}^N(\Sigma_t)}
			\notag \\
	& \ \
		+
		t^{\Blowupexp + 1 + \Worstexp}
		\left\|
			g^{-1}
		\right\|_{W_{Frame}^{1,\infty}(\Sigma_t)}^2
		\left\|
			g - \KasnerMetric
		\right\|_{W_{Frame}^{3,\infty}(\Sigma_t)}
		\left\|
			n
		\right\|_{\dot{H}^{N-1}(\Sigma_t)}
		\notag 
			\\
	& \ \
		+
		t^{\Blowupexp + 1 + \Worstexp}
		\left\|
			n-1
		\right\|_{W^{1,\infty}(\Sigma_t)}
		\left\|
			g^{-1}
		\right\|_{W_{Frame}^{1,\infty}(\Sigma_t)}
		\left\|
			g - \KasnerMetric
		\right\|_{W_{Frame}^{3,\infty}(\Sigma_t)}
		\left\|
			g^{-1}
		\right\|_{\dot{H}_{Frame}^N(\Sigma_t)}
			\notag 
			\\
	& \ \
		+
		t^{\Blowupexp + 1 + \Worstexp}
		\left\|
			n-1
		\right\|_{W^{1,\infty}(\Sigma_t)}
		\left\|
			g^{-1}
		\right\|_{W_{Frame}^{1,\infty}(\Sigma_t)}^2
		\left\|
			g - \KasnerMetric
		\right\|_{W_{Frame}^{3,\infty}(\Sigma_t)}.
		\notag
\end{align}
From Defs.\,\ref{D:LOWNORMS} and~\ref{D:HIGHNORMS},
the estimates 
\eqref{E:KASNERMETRICESTIMATES}
and \eqref{E:UPTOFOURDERIVATIVESOFGLINFINITYSOBOLEV}-\eqref{E:UPTOFOURDERIVATIVESOFLAPSELINFINITYSOBOLEV},
and the bootstrap assumptions,
we see that 
$
\mbox{RHS~\eqref{E:THIRDSTEPLAPSEJUSTBELOWTOPORDERERRORTERM3}}
\lesssim
t^{3 - 15 \Worstexp - 2 \Room - \Blowupexp \updelta}
\left\lbrace
	\MetricLownorm(t) + \MetricHighnorm(t)
\right\rbrace
+
t^{2 - 6 \Worstexp - \Blowupexp \updelta}
\left\lbrace
	\MetricLownorm(t) + \MetricHighnorm(t)
\right\rbrace
$,
which, in view of \eqref{E:WORSTEXPANDROOMLOTSOFUSEFULINEQUALITIES},
is $\lesssim \mbox{RHS~\eqref{E:LAPSEJUSTBELOWTOPORDERBORDERERRORTERMBOUNDINTERMSOFMETRIC}}$
as desired.

To bound the fourth sum on RHS~\eqref{E:LAPSEJUSTBELOWTOPORDERERRORTERM},
we first consider the cases in which $|\vec{I}_5|=N-1$ or 
$|\vec{I}_6| = N-1$. Using $g$-Cauchy--Schwarz,
and the fact that $|g^{-1}|_g \lesssim 1$
we see that the terms under consideration are bounded in the norm
$\| \cdot \|_{L^2(\Sigma_t)}$
by
$
\lesssim
t^{\Blowupexp + 1 + \Worstexp}
\left\|
	n-1
\right\|_{L^{\infty}(\Sigma_t)}
\left\|
	\partial g
\right\|_{L_g^{\infty}(\Sigma_t)}
\left\|
	\partial g
\right\|_{\dot{H}_g^{N-1}(\Sigma_t)}
$.
Applying \eqref{E:POINTWISENORMCOMPARISON}
to 
$
\left\|
	\partial g
\right\|_{L_g^{\infty}(\Sigma_t)}
$
(with $l=0$ and $m=3$),
we deduce that the RHS of the previous expression is
$
\lesssim
t^{\Blowupexp + 1 - 2 \Worstexp}
\left\|
	n-1
\right\|_{L^{\infty}(\Sigma_t)}
\left\|
	g
\right\|_{\dot{W}_{Frame}^{1,\infty}(\Sigma_t)}
\left\|
	\partial g
\right\|_{\dot{H}_g^{N-1}(\Sigma_t)}
$.
From Defs.\,\ref{D:LOWNORMS} and~\ref{D:HIGHNORMS},
the estimates 
\eqref{E:KASNERMETRICESTIMATES} and
\eqref{E:UPTOFOURDERIVATIVESOFGLINFINITYSOBOLEV},
and the bootstrap assumptions,
we deduce that the RHS of the previous expression 
is
$
t^{3 - 16 \Worstexp - 2 \Room - \Blowupexp \updelta}
\left\lbrace
	\MetricLownorm(t) + \MetricHighnorm(t)
\right\rbrace
$,
which, in view of \eqref{E:WORSTEXPANDROOMLOTSOFUSEFULINEQUALITIES},
is $\lesssim \mbox{RHS~\eqref{E:LAPSEJUSTBELOWTOPORDERBORDERERRORTERMBOUNDINTERMSOFMETRIC}}$
as desired.
It remains for us to consider the remaining cases, in which
$|\vec{I}_1|, |\vec{I}_5|, |\vec{I}_6| \leq N-2$.
Using 
\eqref{E:BASICINTERPOLATION}
and
\eqref{E:FRAMENORML2PRODUCTBOUNDINERMSOFLINFINITYANDHMDOT},
we bound 
(using that $|\vec{I}|=N-1$)
the terms under consideration as follows:	
\begin{align} \label{E:FIRSTSTEPLAPSEJUSTBELOWTOPORDERERRORTERM4}
	&
	\mathop{\sum_{\vec{I}_1 + \vec{I}_2 + \cdots + \vec{I}_6 = \vec{I}}}_{|\vec{I}_1|, |\vec{I}_5|, |\vec{I}_6| \leq N-2}
		t^{\Blowupexp + 1 + \Worstexp}
		\left\|
			\left\lbrace
				\partial_{\vec{I}_1} (n-1)
			\right\rbrace
			(\partial_{\vec{I}_2} g^{-1})
			(\partial_{\vec{I}_3} g^{-1})
			(\partial_{\vec{I}_4} g^{-1})
			(\partial \partial_{\vec{I}_5} g)
			\partial \partial_{\vec{I}_6} g	
		\right\|_{L^2(\Sigma_t)}
			\\
		& \lesssim
		t^{\Blowupexp + 1 + \Worstexp}
		\left\|
			n-1
		\right\|_{W^{1,\infty}(\Sigma_t)}
		\left\|
			g^{-1}
		\right\|_{W_{Frame}^{1,\infty}(\Sigma_t)}^3
		\left\|
			g - \KasnerMetric
		\right\|_{W_{Frame}^{2,\infty}(\Sigma_t)}
		\left\|
			g
		\right\|_{\dot{H}_{Frame}^{N-1}(\Sigma_t)}
			\notag \\
	& \ \
		+
		t^{\Blowupexp + 1 + \Worstexp}
		\left\|
			g^{-1}
		\right\|_{W_{Frame}^{1,\infty}(\Sigma_t)}^3
		\left\|
			g - \KasnerMetric
		\right\|_{W_{Frame}^{2,\infty}(\Sigma_t)}^2
		\left\|
			n
		\right\|_{\dot{H}^{N-1}(\Sigma_t)}
		\notag
			\\
	& \ \
		+
		t^{\Blowupexp + 1 + \Worstexp}
		\left\|
			n-1
		\right\|_{W^{1,\infty}(\Sigma_t)}
		\left\|
			g^{-1}
		\right\|_{W_{Frame}^{1,\infty}(\Sigma_t)}^2
		\left\|
			g - \KasnerMetric
		\right\|_{W_{Frame}^{2,\infty}(\Sigma_t)}^2
		\left\|
			g^{-1}
		\right\|_{\dot{H}_{Frame}^{N-1}(\Sigma_t)}
		\notag \\
& \ \
		+
		t^{\Blowupexp + 1 + \Worstexp}
		\left\|
			n-1
		\right\|_{W^{1,\infty}(\Sigma_t)}
		\left\|
			g^{-1}
		\right\|_{W_{Frame}^{1,\infty}(\Sigma_t)}^3
		\left\|
			g - \KasnerMetric
		\right\|_{W_{Frame}^{2,\infty}(\Sigma_t)}^2.
		\notag
\end{align}
From Defs.\,\ref{D:LOWNORMS} and~\ref{D:HIGHNORMS},
the estimates 
\eqref{E:KASNERMETRICESTIMATES}
and \eqref{E:UPTOFOURDERIVATIVESOFGLINFINITYSOBOLEV}-\eqref{E:UPTOFOURDERIVATIVESOFLAPSELINFINITYSOBOLEV},
and the bootstrap assumptions,
we see that 
$
\mbox{RHS~\eqref{E:FIRSTSTEPLAPSEJUSTBELOWTOPORDERERRORTERM4}}
\lesssim
t^{4 - 22 \Worstexp - 4 \Room - \Blowupexp \updelta}
\left\lbrace
	\MetricLownorm(t) + \MetricHighnorm(t)
\right\rbrace
+
t^{2 - 10 \Worstexp - \Blowupexp \updelta}
\left\lbrace
	\MetricLownorm(t) + \MetricHighnorm(t)
\right\rbrace
$,
which, in view of \eqref{E:WORSTEXPANDROOMLOTSOFUSEFULINEQUALITIES},
is $\lesssim \mbox{RHS~\eqref{E:LAPSEJUSTBELOWTOPORDERBORDERERRORTERMBOUNDINTERMSOFMETRIC}}$
as desired.

To bound the fifth sum on RHS~\eqref{E:LAPSEJUSTBELOWTOPORDERERRORTERM},
we first use 
\eqref{E:BASICINTERPOLATION},
\eqref{E:FRAMENORML2PRODUCTBOUNDINERMSOFLINFINITYANDHMDOT},
and the fact that $|\vec{I}|=N-1$
to deduce that the terms under consideration are bounded 
as follows:
\begin{align} \label{E:FIRSTSTEPLAPSEJUSTBELOWTOPORDERERRORTERM5}
	&
	\sum_{\vec{I}_1 + \vec{I}_2 + \vec{I}_3 + \vec{I}_4 = \vec{I}}
		t^{\Blowupexp + 1 + \Worstexp}
		\left\|
			(\partial_{\vec{I}_1} g^{-1})
			(\partial_{\vec{I}_2} g^{-1})
			(\partial \partial_{\vec{I}_3} g)
			(\partial \partial_{\vec{I}_4} n)			
		\right\|_{L^2(\Sigma_t)}
			\\
		& \lesssim
		t^{\Blowupexp + 1 + \Worstexp}
		\left\|
			g^{-1}
		\right\|_{L_{Frame}^{\infty}(\Sigma_t)}^2
		\left\|
			g
		\right\|_{\dot{W}_{Frame}^{1,\infty}(\Sigma_t)}
		\left\|
			n
		\right\|_{\dot{H}^N(\Sigma_t)}
			\notag \\
	& \ \	
		+
		t^{\Blowupexp + 1 + \Worstexp}
		\left\|
			g^{-1}
		\right\|_{L_{Frame}^{\infty}(\Sigma_t)}^2
		\left\|
			n
		\right\|_{\dot{W}^{1,\infty}(\Sigma_t)}
		\left\|
			g
		\right\|_{\dot{H}_{Frame}^N(\Sigma_t)}
			\notag \\
	& \ \
		+
		t^{\Blowupexp + 1 + \Worstexp}
		\left\|
			g^{-1}
		\right\|_{L_{Frame}^{\infty}(\Sigma_t)}
		\left\|
			g
		\right\|_{\dot{W}_{Frame}^{1,\infty}(\Sigma_t)}
		\left\|
			n
		\right\|_{\dot{W}^{1,\infty}(\Sigma_t)}
		\left\|
			g^{-1}
		\right\|_{\dot{H}_{Frame}^N(\Sigma_t)}
		\notag 
			\\
	& \ \
		+
		t^{\Blowupexp + 1 + \Worstexp}
		\left\|
			g^{-1}
		\right\|_{L_{Frame}^{\infty}(\Sigma_t)}^2
		\left\|
			g
		\right\|_{\dot{W}_{Frame}^{1,\infty}(\Sigma_t)}
		\left\|
			n
		\right\|_{\dot{W}^{1,\infty}(\Sigma_t)}.
		\notag
\end{align}
From Defs.\,\ref{D:LOWNORMS} and~\ref{D:HIGHNORMS},
the estimates 
\eqref{E:KASNERMETRICESTIMATES}
and \eqref{E:UPTOFOURDERIVATIVESOFGLINFINITYSOBOLEV}-\eqref{E:UPTOFOURDERIVATIVESOFLAPSELINFINITYSOBOLEV},
and the bootstrap assumptions,
we see that 
$
\mbox{RHS~\eqref{E:FIRSTSTEPLAPSEJUSTBELOWTOPORDERERRORTERM5}}
\lesssim
t^{1 - 5 \Worstexp - \Blowupexp \updelta}
\left\lbrace
	\MetricLownorm(t) + \MetricHighnorm(t)
\right\rbrace
+
t^{3 - 15 \Worstexp - 2 \Room - \Blowupexp \updelta}
\left\lbrace
	\MetricLownorm(t) + \MetricHighnorm(t)
\right\rbrace
$,
which, in view of \eqref{E:WORSTEXPANDROOMLOTSOFUSEFULINEQUALITIES},
is $\lesssim \mbox{RHS~\eqref{E:LAPSEJUSTBELOWTOPORDERBORDERERRORTERMBOUNDINTERMSOFMETRIC}}$
as desired. 

To bound the last sum on RHS~\eqref{E:LAPSEJUSTBELOWTOPORDERERRORTERM},
we first use 
\eqref{E:FRAMENORML2PRODUCTBOUNDINERMSOFLINFINITYANDHMDOT}
and the fact that $|\vec{I}|=N-1$
to deduce that the terms under consideration are bounded 
as follows:
\begin{align} \label{E:FIRSTSTEPLAPSEJUSTBELOWTOPORDERERRORTERMLAST}
	&
	\mathop{\sum_{\vec{I}_1 + \vec{I}_2 = \vec{I}}}_{|\vec{I}_2| \leq N-2}
		\left\|
			t^{\Blowupexp + 1 + \Worstexp}
			(\partial_{\vec{I}_1} g^{-1})
			\partial^2 \partial_{\vec{I}_2} n		
		\right\|_{L^2(\Sigma_t)}
			\\
		& \lesssim
		t^{\Blowupexp + 1 + \Worstexp}
		\left\|
			g^{-1}
		\right\|_{\dot{W}_{Frame}^{1,\infty}(\Sigma_t)}
		\left\|
			n
		\right\|_{\dot{H}^N(\Sigma_t)}
			\notag \\
	& \ \	
		+
		t^{\Blowupexp + 1 + \Worstexp}
		\left\|
			n
		\right\|_{\dot{W}^{2,\infty}(\Sigma_t)}
		\left\|
			g^{-1}
		\right\|_{\dot{H}_{Frame}^{N-1}(\Sigma_t)}.
			\notag
\end{align}
From Defs.\,\ref{D:LOWNORMS} and~\ref{D:HIGHNORMS},
the estimates 
\eqref{E:UPTOFOURDERIVATIVESOFGINVERSELINFINITYSOBOLEV}
and
\eqref{E:UPTOFOURDERIVATIVESOFLAPSELINFINITYSOBOLEV},
and the bootstrap assumptions,
we see that 
$
\mbox{RHS~\eqref{E:FIRSTSTEPLAPSEJUSTBELOWTOPORDERERRORTERMLAST}}
\lesssim
t^{1 - \Worstexp - \Blowupexp \updelta}
\left\lbrace
	\MetricLownorm(t) + \MetricHighnorm(t)
\right\rbrace
+
t^{4 - 14 \Worstexp - 4 \Room - \Blowupexp \updelta}
\left\lbrace
	\MetricLownorm(t) + \MetricHighnorm(t)
\right\rbrace
$,
which, in view of \eqref{E:WORSTEXPANDROOMLOTSOFUSEFULINEQUALITIES},
is $\lesssim \mbox{RHS~\eqref{E:LAPSEJUSTBELOWTOPORDERBORDERERRORTERMBOUNDINTERMSOFMETRIC}}$
as desired. This completes the proof of \eqref{E:LAPSEJUSTBELOWTOPORDERBORDERERRORTERMBOUNDINTERMSOFMETRIC}
and finishes the proof of the lemma.

\end{proof}

\subsection{Proof of Prop.\,\ref{P:ELLIPTICANDMAXIMUMPRINCIPLEESTIMATESFORTHELAPSE}}
\label{SS:PROOFOFPROPOSITIONELLIPTICANDMAXIMUMPRINCIPLEESTIMATESFORTHELAPSE}
In this subsection, we prove Prop.\,\ref{P:ELLIPTICANDMAXIMUMPRINCIPLEESTIMATESFORTHELAPSE}.
Throughout this proof, we will assume that $\Blowupexp \updelta$ is sufficiently small
(and in particular that $\Blowupexp \updelta < \Room$);
in view of the discussion in Subsect.\,\ref{SS:SOBOLEVEMBEDDING}, we see that
at fixed $\Blowupexp$, this can be achieved by choosing $N$ to be sufficiently large.
	
	\medskip
	
	\noindent \underline{\textbf{Proof of \eqref{E:LAPSELOWNORMELLIPTIC}}}:
	From \eqref{E:RICCICURVATUREEXACT},
	we deduce that
	$
	\Sc
	\mycong 
		(g^{-1})^2 \partial^2 g
		+
		(g^{-1})^3 (\partial g)^2
	$.
Hence, bounding these products by bounding each factor in the norm
$\| \cdot \|_{L_{Frame}^{\infty}}$
with the help of
the estimates \eqref{E:KASNERMETRICESTIMATES}
and \eqref{E:UPTOFOURDERIVATIVESOFGLINFINITYSOBOLEV}-\eqref{E:UPTOFOURDERIVATIVESOFGINVERSELINFINITYSOBOLEV}
and the bootstrap assumptions, 
we deduce that
\begin{align} \label{E:SIMPLESCALARCURVATURELINFITYBOUND}
|\Sc| 
\lesssim 
t^{- 10 \Worstexp - \Blowupexp \updelta} 
\left\lbrace
	\MetricLownorm(t) + \MetricHighnorm(t)
\right\rbrace
\lesssim
\varepsilon t^{- 10 \Worstexp - \Blowupexp \updelta}.
\end{align}
From 
\eqref{E:WORSTEXPANDROOMLOTSOFUSEFULINEQUALITIES},
both inequalities in \eqref{E:SIMPLESCALARCURVATURELINFITYBOUND},
and equation \eqref{E:REWRITTENLAPSELOWERORDER} (multiplied by $t^2$), 
we deduce that if $\Blowupexp \updelta < \Room$,
then
\begin{align} \label{E:LAPSELOWNORMELLIPTICALMOSTPROVED}
	\left|
	t^2 g^{ab} \nabla_a \nabla_b n
	-
	(n - 1) 
	\left\lbrace
		1 + \mathcal{O}(\varepsilon)
	\right\rbrace
	\right|
	& \lesssim
	t^{2 - 10 \Worstexp - \Blowupexp \updelta} 
	\left\lbrace
		\MetricLownorm(t) + \MetricHighnorm(t)
	\right\rbrace
		\\
& \lesssim
	t^{2 - 10 \Worstexp - \Room} 
	\left\lbrace
		\MetricLownorm(t) + \MetricHighnorm(t)
	\right\rbrace.
	\notag
\end{align}
At any point $p_{(Max)} \in \Sigma_t$ at which $n-1$ achieves its maximum value, we have that
$g^{ab} \nabla_a \nabla_b n \leq 0$. From this fact and the estimate \eqref{E:LAPSELOWNORMELLIPTICALMOSTPROVED},
it follows that 
$(n-1)|_{p_{(Max})} 
\lesssim 
t^{2 - 10 \Worstexp - \Room} 
	\left\lbrace
		\MetricLownorm(t) + \MetricHighnorm(t)
	\right\rbrace
$. 
Using similar reasoning, we deduce that at any point $p_{(Min)} \in \Sigma_t$ 
at which $n-1$ achieves its minimum value, 
we have the estimate
$(n-1)|_{p_{(Min)}} 
\gtrsim 
	-
	t^{2 - 10 \Worstexp - \Room} 
	\left\lbrace
		\MetricLownorm(t) + \MetricHighnorm(t)
	\right\rbrace
$.
Combining these two estimates, we  
find that 
$
\| n-1 \|_{L^{\infty}(\Sigma_t)}
\lesssim
t^{2 - 10 \Worstexp - \Room} 
	\left\lbrace
		\MetricLownorm(t) + \MetricHighnorm(t)
	\right\rbrace
$,
which, in view of definition \eqref{E:LAPSELOWNORM}, 
yields \eqref{E:LAPSELOWNORMELLIPTIC}.

\medskip

\noindent \textbf{Proof of \eqref{E:LAPSETOPORDERELLIPTIC} and \eqref{E:LAPSEJUSTBELOWTOPORDERELLIPTIC}}:
In view of Def.\,\ref{D:HIGHNORMS},
we see that to obtain \eqref{E:LAPSETOPORDERELLIPTIC} and \eqref{E:LAPSEJUSTBELOWTOPORDERELLIPTIC},
it suffices to prove the following estimates:
\begin{align} \label{E:LAPSETOPBOUND}
	&
	\left[
	\int_{\Sigma_t}
		\left|
			t^{\Blowupexp + 1} \partial \partial_{\vec{I}} n
		\right|_g^2
		+
		\left|
			t^{\Blowupexp} \partial_{\vec{I}} n
		\right|^2
	\, dx
	\right]^{1/2}
	&& \\
	& \leq C_* 
				\left\|
					t^{\Blowupexp + 1} \partial_{\vec{I}} \SecondFund
				\right\|_{L_g^2(\Sigma_t)}
				+
				C 
				t^{\Room}
				\left\lbrace
					\MetricLownorm(t) + \MetricHighnorm(t)
				\right\rbrace,
	&& (|\vec{I}|=N),
		\notag \\
	&
	\left[
	\int_{\Sigma_t}
		\left|
			t^{\Blowupexp + \Worstexp} \partial \partial_{\vec{I}} n
		\right|_g^2
		+
		\left|
			t^{\Blowupexp + \Worstexp - 1} \partial_{\vec{I}} n
		\right|^2
	\, dx
	\right]^{1/2}
		\label{E:LAPSEJUSTBELOWTOPBOUND} \\
	& \leq 
				C 
				\left\lbrace
					\MetricLownorm(t) + \MetricHighnorm(t)
				\right\rbrace,
	&& (|\vec{I}|=N-1).
	\notag
\end{align}

To prove \eqref{E:LAPSETOPBOUND}, we let $\vec{I}$ be any spatial derivative multi-index with $|\vec{I}| = N$.
Multiplying equation \eqref{E:COMMUTEDLAPSEHIGHORDER} by $t^{\Blowupexp} \partial_{\vec{I}} n$
and integrating by parts over $\Sigma_t$,
we deduce
\begin{align} \label{E:HIGHELLIPTICESTIMATEFIRSTSTEP}
	\int_{\Sigma_t}
		\left|
			t^{\Blowupexp + 1} \partial \partial_{\vec{I}} n
		\right|_g^2
		+
		(t^{\Blowupexp} \partial_{\vec{I}} n)^2
	\, dx
	&
	\leq 
	\int_{\Sigma_t}
	\left|
			(t \partial_b g^{ab})
			(t^{\Blowupexp + 1} \partial_a \partial \partial_{\vec{I}} n)
			(t^{\Blowupexp} \partial_{\vec{I}} n)
		\right|
	\, dx
		\\
& \ \
	+
	\int_{\Sigma_t}
		\left|
			\leftexp{(Border;\vec{I})}{\mathfrak{N}}
		\right|
		\left|
			t^{\Blowupexp} \partial_{\vec{I}} n
		\right|
	\, dx
	+
	\int_{\Sigma_t}
		\left|
			\leftexp{(Junk;\vec{I})}{\mathfrak{N}}
		\right|
		\left|
			t^{\Blowupexp} \partial_{\vec{I}} n
		\right|
	\, dx.
	\notag
\end{align}
Next, we use
\eqref{E:WORSTEXPANDROOMLOTSOFUSEFULINEQUALITIES},
\eqref{E:KASNERMETRICESTIMATES},
\eqref{E:UPTOFOURDERIVATIVESOFGINVERSELINFINITYSOBOLEV},
the bootstrap assumptions, 
and $g$-Cauchy--Schwarz
to deduce	that
\begin{align} \label{E:SIMPLECAUCHSCHWARZPOINTWISEESTIMATE}
	\left|
		t \partial_b g^{ab}
	\right|
	\left|
		t^{\Blowupexp + 1} \partial_a \partial_{\vec{I}} n
	\right|
	& 
	\lesssim 
	\left|
		g_{ac} (t \partial_b g^{ab}) (t \partial_d g^{cd})
	\right|^{1/2}
	\left|
		t^{\Blowupexp + 1} \partial \partial_{\vec{I}} n
	\right|_g
		\\
& \lesssim
	t
	\| g \|_{L_{Frame}^{\infty}(\Sigma_t)}^{1/2}
	\| g^{-1} \|_{\dot{W}_{Frame}^{1,\infty}(\Sigma_t)}
	\left|
		t^{\Blowupexp + 1} \partial \partial_{\vec{I}} n
	\right|_g
	\lesssim
	\varepsilon
	\left|
		t^{\Blowupexp + 1} \partial \partial_{\vec{I}} n
	\right|_g.
		\notag
\end{align}
From 
\eqref{E:HIGHELLIPTICESTIMATEFIRSTSTEP},
\eqref{E:SIMPLECAUCHSCHWARZPOINTWISEESTIMATE},
and Young's inequality,
we deduce that if $\varepsilon$ is sufficiently small, 
then
\begin{align} \label{E:HIGHELLIPTICESTIMATESECONDSTEP}
	\int_{\Sigma_t}
		\left|
			t^{\Blowupexp + 1} \partial \partial_{\vec{I}} n
		\right|_g^2
		+
		\left|
			t^{\Blowupexp} \partial_{\vec{I}} n
		\right|^2
	\, dx
	&
	\leq 
	C \varepsilon
	\int_{\Sigma_t}
		\left|
			t^{\Blowupexp + 1} \partial \partial_{\vec{I}} n
		\right|_g^2
	\, dx
	+
	\frac{1}{2}
	\int_{\Sigma_t}
		\left|
			t^{\Blowupexp} \partial \partial_{\vec{I}} n
		\right|^2
	\, dx
		\\
& \ \
	+
	4
	\int_{\Sigma_t}
		\left|
			\leftexp{(Border;\vec{I})}{\mathfrak{N}}
		\right|^2
	\, dx
	+
	4
	\int_{\Sigma_t}
		\left|
			\leftexp{(Junk;\vec{I})}{\mathfrak{N}}
		\right|^2
	\, dx.
	\notag
\end{align}	
Using
\eqref{E:LAPSETOPORDERBORDERERRORTERMBOUNDINTERMSOFMETRIC} and \eqref{E:LAPSETOPORDERJUNKERRORTERMBOUNDINTERMSOFMETRIC}
to bound the last two integrals on
RHS~\eqref{E:HIGHELLIPTICESTIMATESECONDSTEP},
and soaking 
(assuming $\varepsilon$ is sufficiently small)
the first two terms on RHS~\eqref{E:HIGHELLIPTICESTIMATESECONDSTEP} 
and the first term on RHS~\eqref{E:LAPSETOPORDERBORDERERRORTERMBOUNDINTERMSOFMETRIC}
back into LHS~\eqref{E:HIGHELLIPTICESTIMATESECONDSTEP},
we arrive at \eqref{E:LAPSETOPBOUND}.

Similarly, to prove \eqref{E:LAPSEJUSTBELOWTOPBOUND}, 
we let $\vec{I}$ be any spatial derivative multi-index with $|\vec{I}| = N-1$.
We multiply equation \eqref{E:COMMUTEDLAPSELOWERORDER} by 
$t^{\Blowupexp + \Worstexp - 1} \partial_{\vec{I}} n$,
use \eqref{E:SIMPLESCALARCURVATURELINFITYBOUND},
use the bound
$
\left|
	(t \partial_b g^{ab})
	(t^{\Blowupexp + \Worstexp} \partial_a \partial \partial_{\vec{I}} n)
	(t^{\Blowupexp + \Worstexp - 1} \partial_{\vec{I}} n)
\right|
\lesssim 
\varepsilon
|t^{\Blowupexp + \Worstexp} \partial \partial_{\vec{I}} n|_g
|t^{\Blowupexp + \Worstexp - 1} \partial_{\vec{I}} n|
$
(which follows from essentially the same reasoning we used to prove \eqref{E:SIMPLECAUCHSCHWARZPOINTWISEESTIMATE}),
and argue as in the previous paragraph
to deduce the following analog of \eqref{E:HIGHELLIPTICESTIMATESECONDSTEP}:
\begin{align} \label{E:JUSTBELOWTOPELLIPTICSEVERALSTEPS}
	&
	\int_{\Sigma_t}
		\left|
			t^{\Blowupexp + \Worstexp}
			\partial \partial_{\vec{I}} n
		\right|_g^2
		+
		\left\lbrace
			1 + \mathcal{O}(\varepsilon)
		\right\rbrace
		\left|
			t^{\Blowupexp + \Worstexp - 1} \partial_{\vec{I}} n
		\right|^2
	\, dx
		\\
	&
	\leq 
	C \varepsilon
	\int_{\Sigma_t}
		\left|
			t^{\Blowupexp + \Worstexp} \partial \partial_{\vec{I}} n
		\right|_g^2
	\, dx
	+
	\frac{1}{2}
	\int_{\Sigma_t}
		\left|
			t^{\Blowupexp + \Worstexp - 1} \partial \partial_{\vec{I}} n
		\right|^2
	\, dx
	 \notag	\\
& \ \
	+
	4
	\int_{\Sigma_t}
		\left|
			\leftexp{\vec{I}}{\widetilde{\mathfrak{N}}}
		\right|^2
	\, dx.
	\notag
\end{align}
Soaking the first two terms on RHS~\eqref{E:JUSTBELOWTOPELLIPTICSEVERALSTEPS} back into LHS
and using \eqref{E:LAPSEJUSTBELOWTOPORDERBORDERERRORTERMBOUNDINTERMSOFMETRIC} to bound the last integral on
RHS~\eqref{E:JUSTBELOWTOPELLIPTICSEVERALSTEPS},
we arrive at \eqref{E:LAPSEJUSTBELOWTOPBOUND}.

\medskip

\noindent \underline{\textbf{Proof of \eqref{E:LAPSETIMEDERIVATIVEESTIMATE}}}:
We argue as in the proof of \eqref{E:LAPSELOWNORMELLIPTIC},
but using the estimate \eqref{E:LINFINTYERRORTERMSTIMEDERIVATIVEOFLAPSE}
to control RHS~\eqref{E:REWRITTENLAPSELOWERORDERTIMEDERIVATIVECOMMUTED}.
This leads to the bound 
$
\left\| 
	\partial_t n
\right \|_{L^{\infty}(\Sigma_t)}
\leq C
			t^{1 - 10 \Worstexp - \Room}
			\left\lbrace
				\MetricLownorm(t)
				+
				\MetricHighnorm(t)
				+
				\LapseLownorm(t)
				+
				\LapseHighnorm(t)
			\right\rbrace
$.
Using 
Def.\,\ref{D:HIGHNORMS} and 
the estimates \eqref{E:LAPSELOWNORMELLIPTIC},
\eqref{E:LAPSETOPORDERELLIPTIC},
and \eqref{E:LAPSEJUSTBELOWTOPORDERELLIPTIC},
we have that
$
\LapseLownorm(t)
	+
\LapseHighnorm(t)
\leq
C
\left\lbrace
	\MetricLownorm(t)
	+
	\MetricHighnorm(t)
\right\rbrace
$,
which, when combined with the previous estimate,
yields the desired bound \eqref{E:LAPSETIMEDERIVATIVEESTIMATE}.

\hfill $\qed$

\section{Estimates for the Low-Order Derivatives of \texorpdfstring{$g$ and $\SecondFund$}{the First and Second Fundamental Forms}}
\label{S:ESTIMATESFORLOWDERIVATIVESOFMETRIC}
Our main goal in this section is to prove the following proposition,
which provides the integral inequality that we use 
to control the low norm $\MetricLownorm(t)$.
The proof of the proposition
is located in Subsect.\,\ref{SS:INTEGRALINEQUALITYFORTHELOWMETRICNORM}.
In Subsect.\,\ref{SS:LOWORDERDERIVATIVESOFTHEMETRICANDSECONDFUNDAMENTALFORM}, 
we derive the equations that we will use in proving it.

\begin{proposition}[\textbf{Integral inequality for $\MetricLownorm(t)$}]
\label{P:INTEGRALINEQUALITYFORLOWMETRICNORM}
	Recall that $\MetricLownorm(t)$ is the low-order norm from Def.\,\ref{D:LOWNORMS},
	and assume that the bootstrap assumptions \eqref{E:BOOTSTRAPASSUMPTIONS} hold.
	There exists a constant $C > 0$
	such that if $N$ is sufficiently large in a manner that depends on $\Blowupexp$
	and if $\varepsilon$ is sufficiently small, 
	then the following estimate holds for $t \in (\TBoot,1]$:
\begin{align} \label{E:INTEGRALINEQUALITYFORLOWMETRICNORM}
	\MetricLownorm(t)
	&
	\leq
	\MetricLownorm(1)
	+
	C
	\int_{s=t}^1
		s^{\Room - 1} 
		\left\lbrace
			\MetricLownorm(s) + \MetricHighnorm(s)
		\right\rbrace
	\, ds.
\end{align}

\end{proposition}

\subsection{The equations}
\label{SS:LOWORDERDERIVATIVESOFTHEMETRICANDSECONDFUNDAMENTALFORM}
In this subsection, we derive the equations that we use to control
$g$ and $\SecondFund$ at the lowest derivative levels.

\begin{lemma}[\textbf{A rewriting of the equations verified by $g$ and $\SecondFund$}]
\label{L:EQUATIONSFORLOWORDERDERIVATIVESOFMETRICANDSECONDFUNDAMENTALFORM}
Let $\KasnerMetric$ and $\KasnerSecondFund$ denote the background Kasner solution variables
with corresponding Kasner exponents $\lbrace q_i \rbrace_{i=1,\cdots,\mydim}$.
Then the following evolution equations hold, 
where $i \leq j$ in \eqref{E:REWRITTENPARTIALTGCMC}-\eqref{E:REWRITTENPARTIALTGINVERSECMC}
and
there is \underline{no summation over $j$ in} \eqref{E:REWRITTENPARTIALTGCMC}-\eqref{E:REWRITTENPARTIALTGINVERSECMC}
(and note that 
$t^{-2 q_j} \KasnerMetric_{ij}
=
t^{2 q_j} (\KasnerMetric^{-1})^{ij}
=
\mbox{\upshape diag}(1,1,\cdots,1)
$):
\begin{subequations}
\begin{align}
		\partial_t 
		\left\lbrace
			t^{-2 q_j} g_{ij} - t^{-2 q_j} \KasnerMetric_{ij}
		\right\rbrace
		& = 
				- 
				2 t^{- 2 q_j} 
				\left\lbrace
					g_{ia} - \KasnerMetric_{ia}
				\right\rbrace
				\left\lbrace
					\SecondFund_{\ j}^a
					-
					\KasnerSecondFund_{\ j}^a
				\right\rbrace
				\label{E:REWRITTENPARTIALTGCMC} \\
		& \ \
				-
				2 
				t^{2 q_i - 2 q_j}
				\left\lbrace
					\SecondFund_{\ j}^i
					-
					\KasnerSecondFund_{\ j}^i
				\right\rbrace
				- 
				2 t^{- 2 q_j} 
				(n-1)
				g_{ia}
				\SecondFund_{\ j}^a,
				\notag
				\\
		\partial_t 
		\left\lbrace
			t^{2 q_j} g^{ij} - t^{2 q_j} (\KasnerMetric^{-1})^{ij}
		\right\rbrace
		& = 
				2 t^{2 q_j} 
				\left\lbrace
					g^{ia} - (\KasnerMetric^{-1})^{ia}
				\right\rbrace
				\left\lbrace
					\SecondFund_{\ a}^j
					-
					\KasnerSecondFund_{\ a}^j
				\right\rbrace
				\label{E:REWRITTENPARTIALTGINVERSECMC} \\
		& \ \
				+
				2 t^{- 2 q_i + 2 q_j} 
			  \left\lbrace
					\SecondFund_{\ i}^j
					-
					\KasnerSecondFund_{\ i}^j
				\right\rbrace
				+ 
				2 t^{2 q_j} 
				(n-1)
				g^{ia}
				\SecondFund_{\ a}^j.
				\notag
		\notag 
\end{align}
\end{subequations}

Moreover, the following evolution equation holds
(and note that $t \KasnerSecondFund_{\ j}^i = - \mbox{\upshape diag}(q_1,\cdots,q_{\mydim})$):
\begin{align}  \label{E:REWRITTENPARTIALTKCMC} 
	\partial_t 
	\left\lbrace
		t \SecondFund_{\ j}^i - t \KasnerSecondFund_{\ j}^i
	\right\rbrace
	& = 
		(1 - n) \SecondFund_{\ j}^i
		-
		t g^{ia} \partial_a \partial_j n
		+ 
		t g^{ia} \Gamma_{a \ j}^{\ b} \partial_b n
		+
		t n \Ric_{\ j}^i.
\end{align}
\end{lemma}

\begin{proof}
	Throughout this proof, there is no Einstein summation over $j$.
	To derive equation \eqref{E:REWRITTENPARTIALTGCMC}, 
	we first use equation \eqref{E:PARTIALTGCMC}
	and the fact that $t^{-2 q_j} \KasnerMetric_{ij} = \delta_{ij}$
	(where $\delta_{ij}$ is the standard Kronecker delta)
	to deduce
	\begin{align} \label{E:FIRSTREWRITINGPARTIALTGCMC}
		\partial_t 
		\left\lbrace
			t^{-2 q_j} g_{ij} - t^{-2 q_j} \KasnerMetric_{ij}
		\right\rbrace
		& =
		- 2 t^{-2 q_j} g_{ia} \SecondFund_{\ j}^a
		- 2 q_j t^{-2 q_j - 1} g_{ij}
		- 2 t^{-2 q_j} (n-1) g_{ia}\SecondFund_{\ j}^a.
	\end{align}
	Since $\KasnerSecondFund = - t^{-1} \mbox{\upshape diag}(q_1,\cdots,q_{\mydim})$,
	we can express the first two products on RHS~\eqref{E:FIRSTREWRITINGPARTIALTGCMC} as follows:
	$
	- 2 t^{-2 q_j} g_{ia} \SecondFund_{\ j}^a
	- 2 q_j t^{-2 q_j - 1} g_{ij}
	=
	- 2 
	t^{-2 q_j} g_{ia} 
	\left\lbrace
		\SecondFund_{\ j}^a
		-
		\KasnerSecondFund_{\ j}^a
	\right\rbrace
	$.
	Next, using that 
	$\KasnerMetric = \mbox{\upshape diag}(t^{2q_1},\cdots,t^{2q_{\mydim}})$
	and $\KasnerSecondFund = - t^{-1} \mbox{\upshape diag}(q_1,\cdots,q_{\mydim})$,
	we express the RHS of the previous expression as follows:
	$
	- 
	2 
	t^{-2 q_j} g_{ia} 
	\left\lbrace
		\SecondFund_{\ j}^a
		-
		\KasnerSecondFund_{\ j}^a
	\right\rbrace
	=
	- 2 
	t^{-2 q_j} 
	\left\lbrace
		g_{ia} 
		-
	\KasnerMetric_{ia}
	\right\rbrace
	\left\lbrace
		\SecondFund_{\ j}^a
		-
		\KasnerSecondFund_{\ j}^a
	\right\rbrace
	- 
	2 
	t^{2 q_i - 2 q_j}
	\left\lbrace
		\SecondFund_{\ j}^i
		-
		\KasnerSecondFund_{\ j}^i
	\right\rbrace
	$.
	Combining these calculations, we arrive at \eqref{E:REWRITTENPARTIALTGCMC}.
	
	Equation \eqref{E:REWRITTENPARTIALTGINVERSECMC} can be derived by applying similar
	arguments to equation \eqref{E:PARTIALTGINVERSECMC}, and we omit the details.
	
	Equation \eqref{E:REWRITTENPARTIALTKCMC} 
	follows from multiplying both sides of equation \eqref{E:PARTIALTKCMC}
	by $t$ and using that
	$t \KasnerSecondFund = - \mbox{\upshape diag}(q_1,\cdots,q_{\mydim})$.
\end{proof}

\subsection{Proof of Prop.\,\ref{P:INTEGRALINEQUALITYFORLOWMETRICNORM}}
\label{SS:INTEGRALINEQUALITYFORTHELOWMETRICNORM}
In this subsection, we prove Prop.\,\ref{P:INTEGRALINEQUALITYFORLOWMETRICNORM}.
In this proof, we will assume that $\Blowupexp \updelta$ is sufficiently small
(and in particular that $\Blowupexp \updelta < \Room$);
in view of the discussion in Subsect.\,\ref{SS:SOBOLEVEMBEDDING}, we see that
at fixed $\Blowupexp$, this can be achieved by choosing $N$ to be sufficiently large.

	First, we note that to obtain \eqref{E:INTEGRALINEQUALITYFORLOWMETRICNORM}, 
	it suffices to show that the following bounds hold:
	\begin{align}
		\left\|
			t^{2 \Worstexp} g - t^{2 \Worstexp} \KasnerMetric
		\right\|_{L_{Frame}^{\infty}(\Sigma_t)}
		& \leq
		C \MetricLownorm(1) 
	  +
		C
		\int_{s=t}^1
		s^{\Room - 1} 
		\left\lbrace
			\MetricLownorm(s) + \MetricHighnorm(s)
		\right\rbrace
	\, ds,
			\label{E:METRICSUBTRACTEDFROMKASNERMETRICDESIREDINTEGRALINEQUALITY} \\
	\left\|
			t^{2 \Worstexp} g^{-1} - t^{2 \Worstexp} \KasnerMetric^{-1}
		\right\|_{L_{Frame}^{\infty}(\Sigma_t)}
		& \leq
		C \MetricLownorm(1) 
	  +
		C
		\int_{s=t}^1
		s^{\Room - 1} 
		\left\lbrace
			\MetricLownorm(s) + \MetricHighnorm(s)
		\right\rbrace
	\, ds,
			\label{E:METRICINVERSESUBTRACTEDFROMKASNERMETRICINVERSEDESIREDINTEGRALINEQUALITY} \\
	\left\|
		t \SecondFund - t \KasnerSecondFund
	\right\|_{W_{Frame}^{2,\infty}(\Sigma_t)}
	& \leq 
		C \MetricLownorm(1) 
	  +
		C
		\int_{s=t}^1
		s^{\Room - 1} 
		\left\lbrace
			\MetricLownorm(s) + \MetricHighnorm(s)
		\right\rbrace
	\, ds.
	\label{E:SECONDFUNDSUBTRACTEDFROMKASNERSECONFUNDDESIREDINTEGRALINEQUALITY}
	\end{align}
	For we can then use the symmetry property $\SecondFund_{ab} = \SecondFund_{ba}$
	and the fact that $(t \KasnerSecondFund_{\ b}^a) (t \KasnerSecondFund_{\ a}^b) = 1$
	to derive the identity
	$		
	\left|
		t \SecondFund 
	\right|_g
	-
	1
	=
	\sqrt{1 
		+ 
		2 (t \KasnerSecondFund_{\ b}^a)
		\left\lbrace
			t \SecondFund_{\ a}^b - t \KasnerSecondFund_{\ a}^b
		\right\rbrace
		+
		\left\lbrace
			t \SecondFund_{\ b}^a - t \KasnerSecondFund_{\ b}^a
		\right\rbrace
		\left\lbrace
			t \SecondFund_{\ a}^b - t \KasnerSecondFund_{\ a}^b
		\right\rbrace
		}
		-
		1
	$,
	which, in conjunction with
	the bootstrap assumptions \eqref{E:BOOTSTRAPASSUMPTIONS}
	and the estimates 
	\eqref{E:KASNERSECONDFUNDESTIMATES}
	and \eqref{E:SECONDFUNDSUBTRACTEDFROMKASNERSECONFUNDDESIREDINTEGRALINEQUALITY},
	also yields 
	(upon Taylor expanding the square root)
	the estimate
	\begin{align}  \label{E:ORDER0GEOMETRICSECONDFUNDINTEGRALINEQUALITY}
	\left\|
		\left|
			t \SecondFund 
		\right|_g
		-
		1
	\right\|_{L^{\infty}(\Sigma_t)}
	& \leq 
		C \MetricLownorm(1) 
	  +
		\int_{s=t}^1
		s^{\Room - 1} 
		\left\lbrace
			\MetricLownorm(s) + \MetricHighnorm(s)
		\right\rbrace
	\, ds.
\end{align}
In view of definition \eqref{E:METRICLOWNORM},
from \eqref{E:METRICSUBTRACTEDFROMKASNERMETRICDESIREDINTEGRALINEQUALITY}-\eqref{E:ORDER0GEOMETRICSECONDFUNDINTEGRALINEQUALITY},
we conclude the desired estimate \eqref{E:INTEGRALINEQUALITYFORLOWMETRICNORM}.

It remains for us to prove 
\eqref{E:METRICSUBTRACTEDFROMKASNERMETRICDESIREDINTEGRALINEQUALITY}-\eqref{E:SECONDFUNDSUBTRACTEDFROMKASNERSECONFUNDDESIREDINTEGRALINEQUALITY}.
We first prove \eqref{E:METRICSUBTRACTEDFROMKASNERMETRICDESIREDINTEGRALINEQUALITY} by analyzing equation \eqref{E:REWRITTENPARTIALTGCMC}.
From Def.~\ref{D:LOWNORMS},
the estimates \eqref{E:KASNERMETRICESTIMATES},
\eqref{E:KASNERSECONDFUNDESTIMATES}, 
and \eqref{E:LAPSELOWNORMELLIPTIC},
and the bootstrap assumptions, 
we deduce,
by bounding each factor on RHS~\eqref{E:REWRITTENPARTIALTGCMC} in the norm
$\| \cdot \|_{L_{Frame}^{\infty}(\Sigma_t)}$,
that the following estimate holds:
\begin{align} \label{E:POINTWISEBOUNDFORREWRITTENPARTIALTGCMC}
|\mbox{RHS~\eqref{E:REWRITTENPARTIALTGCMC}}| 
&
\leq 
C 
\varepsilon 
t^{-1 - 2 \Worstexp - 2 q_j}
\left\|
	t^{2 \Worstexp} g
	-
  t^{2 \Worstexp} \KasnerMetric
\right\|_{L_{Frame}^{\infty}(\Sigma_t)}
	\\
& \ \
+
C t^{-1 + 2 q_i - 2 q_j}
\left\lbrace
	\MetricLownorm(t) + \MetricHighnorm(t)
\right\rbrace
+
C t^{1 - 12 \Worstexp - \Room - 2 q_j}
\left\lbrace
	\MetricLownorm(t) + \MetricHighnorm(t)
\right\rbrace.
\notag
\end{align}
From \eqref{E:WORSTEXPANDROOMLOTSOFUSEFULINEQUALITIES},
we deduce that the last two products on
RHS~\eqref{E:POINTWISEBOUNDFORREWRITTENPARTIALTGCMC}
are 
$\leq 
C 
t^{\Room -1 - 2 \Worstexp - 2 q_j}
\left\lbrace
	\MetricLownorm(t) + \MetricHighnorm(t)
\right\rbrace
$
for $t \in (\TBoot,1]$.
Hence, with the help of these bounds, 
we can integrate equation
\eqref{E:REWRITTENPARTIALTGCMC}
from time $t$ to time $1$
and use the initial data bound
$
\left\|
	g
	-
	\KasnerMetric
\right\|_{L_{Frame}^{\infty}(\Sigma_1)}
\leq 
C \MetricLownorm(1)$
to deduce the following estimate for components:
\begin{align} \label{E:OTHINTEGRATEDSTEPMETRICSUBTRACTEDFROMKASNERMETRICDESIREDINTEGRALINEQUALITY}
		t^{-2 q_j}
		\left|
			g_{ij}
			-
			\KasnerMetric_{ij}
		\right|
	& \leq 
		C \MetricLownorm(1)
		+
		C \varepsilon
		\int_{s=t}^1
			s^{-1 - 2 \Worstexp - 2 q_j}
			\left\|
				s^{2 \Worstexp} g
				-
				s^{2 \Worstexp} \KasnerMetric
			\right\|_{L_{Frame}^{\infty}(\Sigma_s)}
		\, ds
			\\
	& \ \
		+
		C
		\int_{s=t}^1
			s^{\Room -1 - 2 \Worstexp - 2 q_j}
			\left\lbrace
				\MetricLownorm(s) 
				+
				\MetricHighnorm(s) 
		\right\rbrace
		\, ds.
		\notag
\end{align}
Multiplying both sides of \eqref{E:OTHINTEGRATEDSTEPMETRICSUBTRACTEDFROMKASNERMETRICDESIREDINTEGRALINEQUALITY}
by $t^{2 \Worstexp + 2 q_j}$
and using \eqref{E:WORSTEXPANDROOMLOTSOFUSEFULINEQUALITIES},
we deduce
\begin{align} \label{E:FIRSTINTEGRATEDSTEPMETRICSUBTRACTEDFROMKASNERMETRICDESIREDINTEGRALINEQUALITY}
		\left|
				t^{2 \Worstexp} g_{ij}
				-
				t^{2 \Worstexp} \KasnerMetric_{ij}
		\right|
	& \leq 
		C \MetricLownorm(1)
		+
		C \varepsilon
		t^{2 \Worstexp + 2 q_j}
		\int_{s=t}^1
			s^{-1 - 2 \Worstexp - 2 q_j}
			\left\|
				s^{2 \Worstexp} g
				-
				s^{2 \Worstexp} \KasnerMetric
			\right\|_{L_{Frame}^{\infty}(\Sigma_s)}
		\, ds
			\\
	& \ \
		+
		C
		\int_{s=t}^1
			s^{\Room - 1}
			\left\lbrace
				\MetricLownorm(s) 
				+
				\MetricHighnorm(s) 
		\right\rbrace
		\, ds.
			\notag 
\end{align}
We now define
$G(t) := \sup_{s \in [t,1]} 
		\left\|
			s^{2 \Worstexp} g
			-
			s^{2 \Worstexp} \KasnerMetric
		\right\|_{L_{Frame}^{\infty}(\Sigma_s)}
$.
Next,
with the help of \eqref{E:WORSTEXPANDROOMLOTSOFUSEFULINEQUALITIES}, 
we deduce the following estimate for the first integral
on RHS~\eqref{E:FIRSTINTEGRATEDSTEPMETRICSUBTRACTEDFROMKASNERMETRICDESIREDINTEGRALINEQUALITY}:
\begin{align} \label{E:ESTIMATEFORGRONWALLABLEMETRICPRODUCT}
C \varepsilon
		t^{2 \Worstexp + 2 q_j}
		\int_{s=t}^1
			s^{-1 - 2 \Worstexp - 2 q_j}
			\left\|
				s^{2 \Worstexp} g
				-
				s^{2 \Worstexp} \KasnerMetric
			\right\|_{L_{Frame}^{\infty}(\Sigma_s)}
		\, ds
& \leq
C \varepsilon
		G(t)
		t^{2 \Worstexp + 2 q_j}
		\int_{s=t}^1
			s^{-1 - 2 \Worstexp - 2 q_j}
		\, ds
			\\
& \leq C \varepsilon G(t).
\notag
\end{align}
From \eqref{E:FIRSTINTEGRATEDSTEPMETRICSUBTRACTEDFROMKASNERMETRICDESIREDINTEGRALINEQUALITY}
and \eqref{E:ESTIMATEFORGRONWALLABLEMETRICPRODUCT},
we deduce that for $(t,x) \in (\TBoot,1] \times \mathbb{T}^{\mydim}$, we have
\begin{align} \label{E:SECONDINTEGRATEDSTEPMETRICSUBTRACTEDFROMKASNERMETRICDESIREDINTEGRALINEQUALITY}
		\left|
				t^{2 \Worstexp} g_{ij}
				-
				t^{2 \Worstexp} \KasnerMetric_{ij}
		\right|
	& \leq 
		C \MetricLownorm(1)
		+
		C \varepsilon
		G(t)
		+
		C
		\int_{s=t}^1
			s^{\Room - 1}
			\left\lbrace
				\MetricLownorm(s) 
				+
				\MetricHighnorm(s) 
		\right\rbrace
		\, ds.
\end{align}
From \eqref{E:SECONDINTEGRATEDSTEPMETRICSUBTRACTEDFROMKASNERMETRICDESIREDINTEGRALINEQUALITY},
it follows that
\begin{align} \label{E:EFFECTIVELYBOUNDEDANNOYINGORDERZEROMETRICESTIMATE}
		G(t)
	& \leq 
		C \MetricLownorm(1)
		+
		C \varepsilon
		G(t)
		+
		C
		\int_{s=t}^1
			s^{\Room - 1}
			\left\lbrace
				\MetricLownorm(s) 
				+
				\MetricHighnorm(s) 
		\right\rbrace
		\, ds.
\end{align}
For $\varepsilon$ sufficiently small, 
we can absorb the product 
$
C \varepsilon G(t)
$ 
on RHS~\eqref{E:EFFECTIVELYBOUNDEDANNOYINGORDERZEROMETRICESTIMATE} 
back into the LHS 
(at the minor expense of increasing the constants $C$ on the RHS).
From these arguments, we conclude
the bound
$
G(t)
\leq 
		C \MetricLownorm(1)
		+
		C
		\int_{s=t}^1
			s^{\Room - 1}
			\left\lbrace
				\MetricLownorm(s) 
				+
				\MetricHighnorm(s) 
		\right\rbrace
		\, ds
$,
which in particular implies the desired bound
\eqref{E:METRICSUBTRACTEDFROMKASNERMETRICDESIREDINTEGRALINEQUALITY}.

The estimate \eqref{E:METRICINVERSESUBTRACTEDFROMKASNERMETRICINVERSEDESIREDINTEGRALINEQUALITY} can
be proved using a similar argument based on the evolution equation \eqref{E:REWRITTENPARTIALTGINVERSECMC}, 
and we omit these details.

To prove \eqref{E:SECONDFUNDSUBTRACTEDFROMKASNERSECONFUNDDESIREDINTEGRALINEQUALITY},
we note (with the help of \eqref{E:RICCICURVATUREEXACT}) 
that RHS~\eqref{E:REWRITTENPARTIALTKCMC} is of the schematic form
$
(1-n) \SecondFund 
+
t g^{-1} \partial^2 n
+ 
t g^{-2} (\partial g) \partial n 
+
t n g^{-2} \partial^2 g
+
t n g^{-3} (\partial g) \partial g
$.
Hence, commuting \eqref{E:REWRITTENPARTIALTKCMC} with up to two spatial derivatives
and bounding all of these products by bounding each factor in the norm
$\| \cdot \|_{L_{Frame}^{\infty}}$ with the help of
Def.~\ref{D:LOWNORMS},
Lemma~\ref{L:BASICESTIMATEFORKASNER},
Lemma~\ref{L:SOBOLEVBORROWALITTLEHIGHNORM}, 
\eqref{E:LAPSELOWNORMELLIPTIC},
\eqref{E:LAPSETOPORDERELLIPTIC},
\eqref{E:LAPSEJUSTBELOWTOPORDERELLIPTIC},
and the bootstrap assumptions,
and also using Young's inequality,
we find,
in view of \eqref{E:WORSTEXPANDROOMLOTSOFUSEFULINEQUALITIES},
that if $\Blowupexp \updelta$ is sufficiently small, then
\begin{align} \label{E:SECONDFUNDORDER0TIMEDERIVATIVEESTIMATE}
\left\|
	\partial_t 
	\left\lbrace
		t \SecondFund - t \KasnerSecondFund
	\right\rbrace
\right\|_{W_{Frame}^{2,\infty}(\Sigma_t)}
&
\leq
C
t^{1 - 10 \Worstexp - \Room - \Blowupexp \updelta}
\left\lbrace
	\MetricLownorm(t) + \MetricHighnorm(t)
\right\rbrace
+
C
t^{3 - 16 \Worstexp - \Room - \Blowupexp \updelta}
\left\lbrace
	\MetricLownorm(t) + \MetricHighnorm(t)
\right\rbrace
		\\
&
\leq
C
t^{\Room - 1}
\left\lbrace
	\MetricLownorm(t) + \MetricHighnorm(t)
\right\rbrace.
\notag
\end{align}
Integrating \eqref{E:SECONDFUNDORDER0TIMEDERIVATIVEESTIMATE} in time
and using the initial data bound
$
\left\|
	\SecondFund - \KasnerSecondFund
\right\|_{W_{Frame}^{2,\infty}(\Sigma_1)}
\leq
\MetricLownorm(1)
$,
we conclude \eqref{E:SECONDFUNDSUBTRACTEDFROMKASNERSECONFUNDDESIREDINTEGRALINEQUALITY}.
We have therefore proved the proposition.

\hfill $\qed$

\section{Estimates for the Top-Order Derivatives of \texorpdfstring{$g$ and $\SecondFund$}{the First and Second Fundamental Forms}}
\label{S:TOPORDERESTIMATESMETRICANDSECONDFUND}
Our main goal in this section is to prove the following proposition,
which provides the main integral inequality for the top-order derivatives of 
$g$ and $\SecondFund$.
The proof is located in Subsect.\,\ref{SS:PROOFOFMAINTOPORDERMETRICENERGYESTIMATE}.
In Subsects.\,\ref{SS:EQUATIONSFORTOPDERIVATIVESOFMETRIC}-\ref{SS:TOPORDERMETRICENERGYESTIMATESCONTROLOFERRORTERMS},
we derive the identities and estimates that we will use when proving the proposition.

\begin{proposition}[\textbf{Integral inequality for the top-order derivatives of $g$ and $\SecondFund$}]
	\label{P:MAINTOPORDERMETRICENERGYESTIMATE}
	Let $\vec{I}$ be top-order spatial multi-index, that is, a multi-index with $|\vec{I}|=N$.
	Assume that the bootstrap assumptions \eqref{E:BOOTSTRAPASSUMPTIONS} hold.
	There exists a universal constant $C_* > 0$ \underline{independent of $N$ and $\Blowupexp$}
	such that if $N$ is sufficiently large in a manner that depends on $\Blowupexp$
	and if $\varepsilon$ is sufficiently small, 
	then the following integral inequality holds for $t \in (\TBoot,1]$
	(where, as we described in Subsect.\,\ref{SSS:CONSTANTS}, constants ``$C$'' are allowed to depend on $N$ and other quantities):
	\begin{align} \label{E:MAINTOPORDERMETRICENERGYESTIMATE}
		\left\|
			t^{\Blowupexp + 1} \partial_{\vec{I}} \SecondFund 
		\right\|_{L_g^2(\Sigma_t)}^2
		+
		\frac{1}{4}
		\left\|
			t^{\Blowupexp + 1} \partial \partial_{\vec{I}} g 
		\right\|_{L_g^2(\Sigma_t)}^2
		& \leq 
			C \MetricHighnorm^2(1)
				 \\
		& \ \
			-
			\left\lbrace
				2 \Blowupexp
				- 
				C_*
			\right\rbrace
			\int_{s=t}^1
				s^{-1}
				\left\lbrace
					\left\|
						s^{\Blowupexp + 1} \partial_{\vec{I}} \SecondFund 
					\right\|_{L_g^2(\Sigma_s)}^2
					+
					\frac{1}{4}
					\left\|
						s^{\Blowupexp + 1} \partial \partial_{\vec{I}} g 
					\right\|_{L_g^2(\Sigma_s)}^2
				\right\rbrace
			\, ds
			\notag	\\
		& \ \
				+
				C
				\int_{s=t}^1
					s^{\Room - 1}
					\left\lbrace
						\MetricLownorm^2(s) + \MetricHighnorm^2(s)
					\right\rbrace
			\, ds.
			\notag
	\end{align}
\end{proposition}

\subsection{The equations}
\label{SS:EQUATIONSFORTOPDERIVATIVESOFMETRIC}
In this subsection, we derive the equations that we will use when deriving estimates for
$g$ and $\SecondFund$ at the high derivative levels.

\begin{lemma}[\textbf{The equations verified by the high-order derivatives of the metric and second fundamental form}]
\label{L:EQUATIONSFORTOOPORDERDERIVATIVESOFMETRICANDSECONDFUND}
Let $\vec{I}$ be top-order spatial multi-index, that is, a multi-index with $|\vec{I}|=N$.
Then the following \textbf{commuted momentum constraint equations} hold:
\begin{subequations}
\begin{align}
		t^{\Blowupexp + 1} \partial_a \partial_{\vec{I}} \SecondFund_{\ i}^a
		& 	= 
				\leftexp{(Border;\vec{I})}{\mathfrak{M}}_i
				+
				\leftexp{(Junk;\vec{I})}{\mathfrak{M}}_i,
					\label{E:TOPORDERDIFFERENTIATEDBORDERMOMENTUMCONSTRAINTINDEXDOWN} \\
		t^{\Blowupexp + 1} g^{ab} \partial_a \partial_{\vec{I}} \SecondFund_{\ b}^i
		& 	= 
				\leftexp{(Border;\vec{I})}{\widetilde{\mathfrak{M}}}^i
				+
				\leftexp{(Junk;\vec{I})}{\widetilde{\mathfrak{M}}}^i,
				\label{E:TOPORDERDIFFERENTIATEDBORDERMOMENTUMCONSTRAINTINDEXUP}
\end{align}
\end{subequations}
where (see Subsubsect.\,\ref{SSS:SCHEMATIC} regarding our use of the notation $\mycongstar$ and $\mycong$, and see Remark~\ref{R:ALWAYSMIXED})
\begin{subequations}
\begin{align}
	\leftexp{(Border;\vec{I})}{\mathfrak{M}}_i
	& \mycongstar t^{\Blowupexp + 1} g^{-1} (\partial \partial_{\vec{I}} g) \SecondFund,
		\label{E:TOPORDERDIFFERENTIATEDBORDERMOMENTUMCONSTRAINTINDEXDOWNERRORTERM} \\
	\leftexp{(Junk;\vec{I})}{\mathfrak{M}}_i
	& \mycong 
	\mathop{\sum_{\vec{I}_1 + \vec{I}_2 + \vec{I}_3 = \vec{I}}}_{|\vec{I}_2| \leq N-1}
	t^{\Blowupexp + 1}	
	(\partial_{\vec{I}_1} g^{-1}) (\partial \partial_{\vec{I}_2} g) \partial_{\vec{I}_3} \SecondFund,
		\label{E:TOPORDERDIFFERENTIATEDJUNKMOMENTUMCONSTRAINTINDEXDOWNERRORTERM} \\
\leftexp{(Border;\vec{I})}{\widetilde{\mathfrak{M}}}^i
& \mycongstar t^{\Blowupexp + 1} (g^{-1})^2 (\partial \partial_{\vec{I}} g) \SecondFund,
	\label{E:TOPORDERDIFFERENTIATEDBORDERMOMENTUMCONSTRAINTINDEXUPERRORTERM} \\
\leftexp{(Junk;\vec{I})}{\widetilde{\mathfrak{M}}}^i
& \mycong 
	\mathop{\sum_{\vec{I}_1 + \vec{I}_2 + \vec{I}_3 + \vec{I}_4 = \vec{I}}}_{|\vec{I}_3| \leq N-1}
	t^{\Blowupexp + 1}
	(\partial_{\vec{I}_1} g^{-1}) 
	(\partial_{\vec{I}_2} g^{-1})
	(\partial \partial_{\vec{I}_3} g) 
	\partial_{\vec{I}_4} \SecondFund.
	\label{E:TOPORDERDIFFERENTIATEDJUNKMOMENTUMCONSTRAINTINDEXUPERRORTERM}
\end{align}	
\end{subequations}

Moreover, for any constant $\Pos \geq 0$
and for any spatial multi-index $\vec{I}$,
the following \textbf{commuted evolution equations} hold:
\begin{subequations}
\begin{align}
\partial_t (t^{\Blowupexp + \Pos} \partial_e \partial_{\vec{I}} g_{ij})
	& = 
	\frac{1}{t}
	\left\lbrace
		(\Blowupexp + \Pos)
		\delta_{\ j}^a
		- 
		2 t \SecondFund_{\ j}^a
	\right\rbrace
	(t^{\Blowupexp + \Pos} \partial_e \partial_{\vec{I}} g_{ia})
	- 
	2 t^{\Blowupexp + \Pos} n g_{ia} \partial_e \partial_{\vec{I}} \SecondFund_{\ j}^a
		\label{E:COMMUTEDPARTIALTONEDERIVATIVEOFMETRICCMC} \\
& \ \ 
	+ 
	\leftexp{(Border;\Pos;\vec{I})}{\mathfrak{H}}_{eij}
	+ 
	\leftexp{(Junk;\Pos;\vec{I})}{\mathfrak{H}}_{eij},
		\notag \\
	\partial_t (t^{\Blowupexp + \Pos} \partial_{\vec{I}} \SecondFund_{\ j}^i) 
	& = 
			(\Blowupexp + \Pos - 1)
			t^{\Blowupexp + \Pos - 1} 
			\partial_{\vec{I}} \SecondFund_{\ j}^i
			- 
			t^{\Blowupexp + \Pos} 
			g^{ia} 
			\partial_a \partial_j \partial_{\vec{I}} n
			\label{E:COMMUTEDPARTIALTKCMC} \\
	& \ \
		+ 
		\frac{1}{2}
		t^{\Blowupexp + \Pos} 
		n 
		g^{ic} 
		g^{ab}
		\left\lbrace
			 \partial_a \partial_c \partial_{\vec{I}} g_{bj} 
				+ 
				\partial_a \partial_j \partial_{\vec{I}} g_{bc} 
				- 
				\partial_a \partial_b \partial_{\vec{I}} g_{cj}
				-
				\partial_c \partial_j \partial_{\vec{I}} g_{ab} 
		\right\rbrace
			\notag \\
	& \ \
		+
		\leftexp{(Border;\Pos;\vec{I})}{\mathfrak{K}}_{\ j}^i
		+
		\leftexp{(Junk;\Pos;\vec{I})}{\mathfrak{K}}_{\ j}^i,
		\notag
\end{align}
\end{subequations}
where $\delta_{\ j}^i$ is the standard Kronecker delta,
\begin{subequations}
	\begin{align}
	\leftexp{(Border;\Pos;\vec{I})}{\mathfrak{H}}_{eij}
	& \mycongstar
		t^{\Blowupexp + \Pos} (\partial \partial_{\vec{I}} n) g \SecondFund,
		\label{E:TOPORDERDIFFERENTIATEDBORDERONCEDIFFERENTIATEDMETRICERRORTERM} \\
\leftexp{(Junk;\Pos;\vec{I})}{\mathfrak{H}}_{eij}
& \mycong
		t^{\Blowupexp + \Pos} (n-1) (\partial \partial_{\vec{I}} g) \SecondFund
		+
		\mathop{\sum_{\vec{I}_1 + \vec{I}_2 + \vec{I}_3 = \vec{I}}}_{|\vec{I}_1| \leq |\vec{I}|-1}
		t^{\Blowupexp + \Pos}
		(\partial \partial_{\vec{I}_1} n) 
		(\partial_{\vec{I}_2} g) 
		\partial_{\vec{I}_3} \SecondFund
		\label{E:TOPORDERDIFFERENTIATEDJUNKONCEDIFFERENTIATEDMETRICERRORTERM} \\
&  \ \
		+
		\mathop{\sum_{\vec{I}_1 + \vec{I}_2 + \vec{I}_3 = \vec{I}}}_{|\vec{I}_2| \leq |\vec{I}|-1}
		t^{\Blowupexp + \Pos} 
		(\partial_{\vec{I}_1} n) 
		(\partial \partial_{\vec{I}_2} g) 
		\partial_{\vec{I}_3} \SecondFund
		+
		\mathop{\sum_{\vec{I}_1 + \vec{I}_2 + \vec{I}_3 = \vec{I}}}_{|\vec{I}_3| \leq |\vec{I}|-1}
		t^{\Blowupexp + \Pos}
		(\partial_{\vec{I}_1} n) 
		(\partial_{\vec{I}_2} g) 
		\partial \partial_{\vec{I}_3} \SecondFund,
		\notag
\end{align}
\end{subequations}

\begin{subequations}
\begin{align}
\leftexp{(Border;\Pos;\vec{I})}{\mathfrak{K}}_{\ j}^i		
	& \mycongstar
		t^{\Blowupexp + \Pos - 1} (\partial_{\vec{I}} n) \SecondFund,
		\label{E:SECONDFUNDCOMMUTEDBORDERLINETERM} \\
\leftexp{(Junk;\Pos;\vec{I})}{\mathfrak{K}}_{\ j}^i		
& \mycong 	
	\mathop{\sum_{\vec{I}_1 + \vec{I}_2 = \vec{I}}}_{|\vec{I}_1| \leq |\vec{I}|-1}
	t^{\Blowupexp + \Pos - 1}
	\left\lbrace
		\partial_{\vec{I}_1} (n-1) 
	\right\rbrace	
	\partial_{\vec{I}_2} \SecondFund
		\label{E:SECONDFUNDCOMMUTEDJUNKTERM} \\
& \ \
	+
	\sum_{\vec{I}_1 + \vec{I}_2 + \cdots + \vec{I}_6 = \vec{I}}
	t^{\Blowupexp + \Pos} 
	(\partial_{\vec{I}_1} n)
	(\partial_{\vec{I}_2} g^{-1})
	(\partial_{\vec{I}_3} g^{-1})
	(\partial_{\vec{I}_4} g^{-1})
	(\partial \partial_{\vec{I}_5} g)
	\partial \partial_{\vec{I}_6} g
		\notag \\
	& \ \
	+
	\mathop{\sum_{\vec{I}_1 + \vec{I}_2 + \vec{I}_3 + \vec{I}_4 = \vec{I}}}_{|\vec{I}_4| \leq |\vec{I}|-1}
		t^{\Blowupexp + \Pos}
		(\partial_{\vec{I}_1} n)
		(\partial_{\vec{I}_2} g^{-1})
		(\partial_{\vec{I}_3} g^{-1})
		\partial^2 \partial_{\vec{I}_4} g
			\notag \\
	& \ \
	+
	\sum_{\vec{I}_1 + \vec{I}_2 + \vec{I}_3 + \vec{I}_4 = \vec{I}}
		t^{\Blowupexp + \Pos} 
		(\partial_{\vec{I}_1} g^{-1})
		(\partial_{\vec{I}_2} g^{-1})
		(\partial \partial_{\vec{I}_3} g)
		\partial \partial_{\vec{I}_4} n.
	\notag
\end{align}
\end{subequations}

\end{lemma}

\begin{proof}
 To prove \eqref{E:TOPORDERDIFFERENTIATEDBORDERMOMENTUMCONSTRAINTINDEXDOWN},
 we first write equation \eqref{E:MOMENTUM} relative to the transported spatial
coordinates in the schematic form
$\partial_a \SecondFund_{\ i}^a = g^{-1} (\partial g) \SecondFund$.
Commuting this equation with
$t^{\Blowupexp + 1} \partial_{\vec{I}}$,
we arrive at \eqref{E:TOPORDERDIFFERENTIATEDBORDERMOMENTUMCONSTRAINTINDEXDOWN}.
The proof of \eqref{E:TOPORDERDIFFERENTIATEDBORDERMOMENTUMCONSTRAINTINDEXUP} 
is similar, but we
start by raising the index $i$ in equation \eqref{E:MOMENTUM}
to deduce
$g^{ab} \nabla_a \SecondFund_{\ b}^i
=
0
$
and then (schematically) writing this equation in coordinates as
$g^{ab} \partial_a \SecondFund_{\ b}^i = g^{-2} (\partial g) \SecondFund$.
	
	To prove \eqref{E:COMMUTEDPARTIALTONEDERIVATIVEOFMETRICCMC},
	we commute equation \eqref{E:PARTIALTGCMC}
	with $\partial_e \partial_{\vec{I}}$
	and then with 
	$t^{\Blowupexp + \Pos}$
	and carry out straightforward computations.
	
	To prove \eqref{E:COMMUTEDPARTIALTKCMC},
	we first use \eqref{E:RICCICURVATUREEXACT} 
to decompose
\begin{align} \label{E:RICCIDECOMPOSITIONINTOPRINCIPLEANDERROR}
		\Ric_{\ j}^i
		& =	
		\frac{1}{2}
		g^{ic} 
		g^{ab}
		\left\lbrace
			 \partial_a \partial_c g_{bj} 
				+ 
				\partial_a \partial_j g_{bc} 
				- 
				\partial_a \partial_b g_{cj}
				-
				\partial_c \partial_j g_{ab} 
		\right\rbrace
		+
		\mbox{\upshape Error}_{\ j}^i,
\end{align}
where (schematically)
$\mbox{\upshape Error} = g^{-3} (\partial g)^2$.
We now use \eqref{E:RICCIDECOMPOSITIONINTOPRINCIPLEANDERROR}
to decompose the term
$
\Ric_{\ j}^i
$	
on RHS~\eqref{E:PARTIALTKCMC},
commute the evolution equation \eqref{E:PARTIALTKCMC}
with $\partial_{\vec{I}}$
and then with $t^{\Blowupexp + \Pos}$,
and carry out straightforward computations,
thereby arriving at \eqref{E:COMMUTEDPARTIALTKCMC} .
\end{proof}

\subsection{Energy currents}
\label{SS:ENERGYCURRENTS}
When deriving top-order energy estimates (see the proof of Prop.\,\ref{P:MAINTOPORDERMETRICENERGYESTIMATE}), 
we will integrate by parts by applying the divergence theorem
on spacetime regions of the form 
$[t,1] \times \mathbb{T}^{\mydim}$
with the help of the vectorfields ${\leftexp{(\vec{I})}{\mathbf{J}}}$ 
featured in the following definition. Put differently,
the vectorfields ${\leftexp{(\vec{I})}{\mathbf{J}}}$ are convenient
for bookkeeping during integration by parts.

\begin{definition}[\textbf{Energy current vectorfields}]
	\label{D:ENERGYCURRENT}
To each top-order spatial multi-index $\vec{I}$ (i.e., $|\vec{I}| = N$), 
we associate the 
energy current ${\leftexp{(\vec{I})}{\mathbf{J}}}$,
which we define to be the vectorfield with 
the following components relative to the
CMC-transported spatial coordinates:
\begin{subequations}
	\begin{align}
		{\leftexp{(\vec{I})}{\mathbf{J}^0}}
		& :=
			\left|
				t^{\Blowupexp + 1} \partial_{\vec{I}} \SecondFund 
			\right|_g^2
			 + 
				\frac{1}{4} 
				\left|
					t^{\Blowupexp + 1} \partial \partial_{\vec{I}} g 
				\right|_g^2, 
				\label{E:METRICCURRENT0} \\
		{\leftexp{(\vec{I})}{\mathbf{J}^j}}
		& := 
		2 g^{ab} 
			(t^{\Blowupexp + 1} \partial_{\vec{I}} \SecondFund_{\ a}^j)
			(t^{\Blowupexp + 1} \partial_b \partial_{\vec{I}} n)
			\label{E:METRICCURRENTSPACE} \\
	& \ \
			+ 
			n
			g^{ij}
			g^{bc}
			(t^{\Blowupexp + 1} \partial_{\vec{I}} \SecondFund_{\ b}^a)
			(t^{\Blowupexp + 1} \partial_i \partial_{\vec{I}} g_{ac})
			+ 
			n
			g^{ab} g^{cd}
			(t^{\Blowupexp + 1} \partial_{\vec{I}} \SecondFund_{\ c}^j)
			(t^{\Blowupexp + 1} \partial_d \partial_{\vec{I}} g_{ab})
			\notag \\
	& \ \ 
		- n
			g^{ab} 
			g^{cd} 
			(t^{\Blowupexp + 1} \partial_{\vec{I}} \SecondFund_{\ c}^j)
			(t^{\Blowupexp + 1} \partial_a \partial_{\vec{I}} g_{bd}) 
			- 
			n
			g^{ab} 
			g^{jc} 
			(t^{\Blowupexp + 1} \partial_{\vec{I}} \SecondFund_{\ c}^d)
			(t^{\Blowupexp + 1} \partial_a \partial_{\vec{I}} g_{bd}).
			\notag
\end{align}	
\end{subequations}
\end{definition}

\begin{remark}
	Note that the components $\lbrace {\leftexp{(\vec{I})}{\mathbf{J}^{\alpha}}} \rbrace_{\alpha = 0,\cdots,\mydim}$
	are quadratic forms in $(\partial_{\vec{I}} \SecondFund, \partial \partial_{\vec{I}} g, \partial \partial_{\vec{I}} n)$
	with coefficients that depend on $g$, $g^{-1}$, and $n$.
\end{remark}

\begin{lemma}[\textbf{Divergence identity verified by} ${\leftexp{(\vec{I})}{\mathbf{J}}}$]
	\label{L:DIVIDFORENERGYCURRENT}
	Let $\vec{I}$ be a top-order spatial multi-index, that is, a multi-index with $|\vec{I}|=N$. Then
	for solutions to the commuted equations of Lemma~\ref{L:EQUATIONSFORTOOPORDERDERIVATIVESOFMETRICANDSECONDFUND}
	with $\Pos = 1$ in equations \eqref{E:COMMUTEDPARTIALTONEDERIVATIVEOFMETRICCMC} and \eqref{E:COMMUTEDPARTIALTKCMC},
	the spacetime vectorfield ${\leftexp{(\vec{I})}{\mathbf{J}}}$ from
	Def.~\ref{D:ENERGYCURRENT} verifies the following divergence
	identity relative to the CMC-transported spatial coordinates:
	\begin{align} \label{E:DIVIDFORENERGYCURRENT}
		\partial_{\alpha} {\leftexp{(\vec{I})}{\mathbf{J}^{\alpha}}}
		& = 
				\frac{2 \Blowupexp}{t} 
				\left|
					t^{\Blowupexp + 1} \partial_{\vec{I}} \SecondFund 
				\right|_g^2
				+ 
				\frac{\Blowupexp + 1}{2 t} 
				\left|
					t^{\Blowupexp + 1} \partial \partial_{\vec{I}} g 
				\right|_g^2
				+
				\leftexp{(Border;\vec{I})}{\mathfrak{J}}
				+
				\leftexp{(Junk;\vec{I})}{\mathfrak{J}},
	\end{align}
	where
	\begin{subequations}
	\begin{align}
		\leftexp{(Border;\vec{I})}{\mathfrak{J}}
		& := 2 t^{\Blowupexp + 1}
				g_{ac}
				g^{bd}
				(\partial_{\vec{I}} \SecondFund_{\ b}^a)
				\leftexp{(Border;1;\vec{I})}{\mathfrak{K}}_{\ d}^c	
				+
				\frac{1}{2}
				t^{\Blowupexp + 1}
		 		g^{ad} g^{be} g^{cf}
				(\partial_a \partial_{\vec{I}} g_{bc})
				\leftexp{(Border;1;\vec{I})}{\mathfrak{H}}_{def}
				\label{E:METRICCURRENTBORDERTERMS}	\\
		& \ \
			+ 
			2 t^{\Blowupexp + 1} 
			g^{ab} 
			(\partial_a \partial_{\vec{I}} n)
			\leftexp{(Border;\vec{I})}{\mathfrak{M}}_b
			\notag
				\\
		& \ \
			+
			t^{\Blowupexp + 1} 
			n
			g^{ab} g^{cd}
			(\partial_c \partial_{\vec{I}} g_{ab})
			\leftexp{(Border;\vec{I})}{\mathfrak{M}}_d
			\notag \\
	& \ \ 
		- t^{\Blowupexp + 1} 
			n
			g^{ab} 
			g^{cd} 
			(\partial_a \partial_{\vec{I}} g_{bc}) 
			\leftexp{(Border;\vec{I})}{\mathfrak{M}}_d
			- 
			t^{\Blowupexp + 1}
			n
			g^{ab} 
			(\partial_a \partial_{\vec{I}} g_{bc})
			\leftexp{(Border;\vec{I})}{\widetilde{\mathfrak{M}}}^c
			\notag
				\\
		& \ \
				+
				2 t^{2 \Blowupexp + 2} 
				n
				g_{ac}
				g^{id}
				\SecondFund_{\ i}^b
				(\partial_{\vec{I}} \SecondFund_{\ b}^a)
				(\partial_{\vec{I}} \SecondFund_{\ d}^c)
				-
				2 t^{2 \Blowupexp + 2} 
				n
				g_{ai}
				g^{bd}
				\SecondFund_{\ c}^i
				(\partial_{\vec{I}} \SecondFund_{\ b}^a)
				(\partial_{\vec{I}} \SecondFund_{\ d}^c)
					\notag \\
			& \ \
				+ 
				\frac{t^{2 \Blowupexp + 2}}{2}
				n
				g^{be} g^{cf} 
				g^{id}
				\SecondFund_{\ i}^a
				(\partial_a \partial_{\vec{I}} g_{bc}) 
				(\partial_d \partial_{\vec{I}} g_{ef})
				+ 
				t^{2 \Blowupexp + 2} 
				n
				g^{ad} g^{ie} g^{cf}
				\SecondFund_{\ e}^b
				(\partial_a \partial_{\vec{I}} g_{bc}) 
				(\partial_d \partial_{\vec{I}} g_{if}),
				\notag
	\end{align}
	
	\begin{align}
		\leftexp{(Junk;\vec{I})}{\mathfrak{J}}
		& :=
			2 t^{\Blowupexp + 1}
				g_{ac}
				g^{bd}
				(\partial_{\vec{I}} \SecondFund_{\ b}^a)
				\leftexp{(Junk;1;\vec{I})}{\mathfrak{K}}_{\ d}^c	
				+
				\frac{1}{2}
				t^{\Blowupexp + 1}
				g^{ad} g^{be} g^{cf}
				(\partial_a \partial_{\vec{I}} g_{bc})
				\leftexp{(Junk;1;\vec{I})}{\mathfrak{H}}_{def}
				\label{E:METRICCURRENTJUNKTERMS}	\\
		& \ \
			+ 
			2 t^{\Blowupexp + 1} 
			g^{ab} 
			(\partial_a \partial_{\vec{I}} n)
			\leftexp{(Junk;\vec{I})}{\mathfrak{M}}_b
			\notag
				\\
		& \ \
			+
			t^{\Blowupexp + 1} 
			n
			g^{ab} g^{cd}
			(\partial_c \partial_{\vec{I}} g_{ab})
			\leftexp{(Junk;\vec{I})}{\mathfrak{M}}_d
			\notag \\
	& \ \ 
		- t^{\Blowupexp + 1} 
			n
			g^{ab} 
			g^{cd} 
			(\partial_a \partial_{\vec{I}} g_{bc}) 
			\leftexp{(Junk;\vec{I})}{\mathfrak{M}}_d
			- 
			t^{\Blowupexp + 1}
			n
			g^{ab} 
			(\partial_a \partial_{\vec{I}} g_{bc})
			\leftexp{(Junk;\vec{I})}{\widetilde{\mathfrak{M}}}^c
			\notag
				\\
	& \ \
		+ 2 t^{2 \Blowupexp + 2}
			(\partial_j g^{ab} )
			(\partial_{\vec{I}} \SecondFund_{\ a}^j)
			(\partial_b \partial_{\vec{I}} n)
			\notag \\
	& \ \
			+ 
			t^{2 \Blowupexp + 2} 
			\left\lbrace
				\partial_j (n g^{ij} g^{bc})
			\right\rbrace
			(\partial_{\vec{I}} \SecondFund_{\ b}^a)
			(\partial_{\vec{I}} \partial_i g_{ac})
			+ 
			t^{2 \Blowupexp + 2}
			\left\lbrace
				\partial_j (n g^{ab} g^{cd})
			\right\rbrace
			(\partial_{\vec{I}} \SecondFund_{\ c}^j)
			(\partial_d \partial_{\vec{I}} g_{ab})
			\notag \\
	& \ \ 
		- t^{2 \Blowupexp + 2}
			\left\lbrace
				\partial_j (n g^{ab} g^{cd})
			\right\rbrace
			(\partial_{\vec{I}} \SecondFund_{\ c}^j)
			(\partial_a \partial_{\vec{I}} g_{bd}) 
			- 
			t^{2 \Blowupexp + 2}
			\left\lbrace
				\partial_j (n g^{ab} g^{jc})
			\right\rbrace
			(\partial_{\vec{I}} \SecondFund_{\ c}^d)
			(\partial_a \partial_{\vec{I}} g_{bd}).
			\notag
	\end{align}
	\end{subequations}
	
\end{lemma}

\begin{proof}
	We view RHSs~\eqref{E:METRICCURRENT0}-\eqref{E:METRICCURRENTSPACE}
	as quadratic forms in $(\partial_{\vec{I}} \SecondFund, \partial \partial_{\vec{I}} g, \partial \partial_{\vec{I}} n)$
	with coefficients that depend on $g$, $g^{-1}$, and $n$.
	We now consider the expression
	$\partial_{\alpha} {\leftexp{(\vec{I})}{\mathbf{J}^{\alpha}}}$.
	On RHS~\eqref{E:METRICCURRENTJUNKTERMS},
	we place all terms in which spatial derivatives $\partial_j$ fall on the coefficients. 
	In contrast, when $\partial_t$ falls on
	$g$ or $g^{-1}$, we use \eqref{E:PARTIALTGCMC}-\eqref{E:PARTIALTGINVERSECMC}
	to substitute for $\partial_t g$ and $\partial_t g^{-1}$ and
	then place the resulting terms on RHS~\eqref{E:METRICCURRENTBORDERTERMS}
	(as the last four products).
	Next, we consider all of the terms 
	in the expression $\partial_{\alpha} {\leftexp{(\vec{I})}{\mathbf{J}^{\alpha}}}$
	in which a derivative falls on one of
	$\partial_{\vec{I}} \SecondFund$, 
	$\partial \partial_{\vec{I}} g$,
	or
	$\partial \partial_{\vec{I}} n$.
	For these terms, we use equations
	\eqref{E:TOPORDERDIFFERENTIATEDBORDERMOMENTUMCONSTRAINTINDEXDOWN}-\eqref{E:TOPORDERDIFFERENTIATEDBORDERMOMENTUMCONSTRAINTINDEXUP}
	and
	\eqref{E:COMMUTEDPARTIALTONEDERIVATIVEOFMETRICCMC}-\eqref{E:COMMUTEDPARTIALTKCMC} 
	for algebraic substitution, where
	$\Pos = 1$ (by assumption) in \eqref{E:COMMUTEDPARTIALTONEDERIVATIVEOFMETRICCMC}-\eqref{E:COMMUTEDPARTIALTKCMC}.
	More precisely, 
	in the expression $\partial_{\alpha} {\leftexp{(\vec{I})}{\mathbf{J}^{\alpha}}}$,
	we use equation
	\eqref{E:COMMUTEDPARTIALTKCMC}
	to substitute for the factor
	$\partial_t (t^{\Blowupexp + 1} \partial_{\vec{I}} \SecondFund_{\ d}^c)$
	in
	$
	2 g_{ac} g^{bd}
	(t^{\Blowupexp + 1} \partial_{\vec{I}} \SecondFund_{\ b}^a)
	\partial_t (t^{\Blowupexp + 1} \partial_{\vec{I}} \SecondFund_{\ d}^c)
	$,
	equation \eqref{E:COMMUTEDPARTIALTONEDERIVATIVEOFMETRICCMC}
	to substitute for the factor $\partial_t (t^{\Blowupexp + 1} \partial_{\vec{I}} \partial_d \partial g_{ef})$ in
	$ 
	\frac{1}{2} 
	g^{ad} g^{be} g^{cf}
	(t^{\Blowupexp + 1} \partial_{\vec{I}} \partial_a \partial g_{bc})
	\partial_t (t^{\Blowupexp + 1} \partial_{\vec{I}} \partial_d \partial g_{ef})
	$,
	equation \eqref{E:TOPORDERDIFFERENTIATEDBORDERMOMENTUMCONSTRAINTINDEXDOWN}
	to substitute for the factor 
	$(t^{\Blowupexp + 1} \partial_j \partial_{\vec{I}} \SecondFund_{\ a}^j)$ in
	$
	2 g^{ab} 
			(t^{\Blowupexp + 1} \partial_j \partial_{\vec{I}} \SecondFund_{\ a}^j)
			(t^{\Blowupexp + 1} \partial_b \partial_{\vec{I}} n)
	$,
	the factor $(t^{\Blowupexp + 1} \partial_j \partial_{\vec{I}} \SecondFund_{\ c}^j)$ in
	$
	n
			g^{ab} g^{cd}
			(t^{\Blowupexp + 1} \partial_j \partial_{\vec{I}} \SecondFund_{\ c}^j)
			(t^{\Blowupexp + 1} \partial_d \partial_{\vec{I}} g_{ab})
	$,
	and the factor $(t^{\Blowupexp + 1} \partial_j \partial_{\vec{I}} \SecondFund_{\ c}^j)$ in
	$
	- n
			g^{ab} 
			g^{cd} 
			(t^{\Blowupexp + 1} \partial_j \partial_{\vec{I}} \SecondFund_{\ c}^j)
			(t^{\Blowupexp + 1} \partial_a \partial_{\vec{I}} g_{bd}) 
	$,
	and equation \eqref{E:TOPORDERDIFFERENTIATEDBORDERMOMENTUMCONSTRAINTINDEXUP}
	to substitute for the factor
	$g^{jc} (t^{\Blowupexp + 1} \partial_j \partial_{\vec{I}} \SecondFund_{\ c}^d)$
	in
	$
	- 
			n
			g^{ab} 
			g^{jc} 
			(t^{\Blowupexp + 1} \partial_j \partial_{\vec{I}} \SecondFund_{\ c}^d)
			(t^{\Blowupexp + 1} \partial_a \partial_{\vec{I}} g_{bd})
	$.
	These algebraic substitutions lead to the elimination of all products that depend
	on a derivative of 
	$(\partial_{\vec{I}} \SecondFund, \partial \partial_{\vec{I}} g, \partial \partial_{\vec{I}} n)$,
	at the expense of introducing products that depend on the inhomogeneous terms, 
	for example
	$
	2 g_{ac} g^{bd}
	(t^{\Blowupexp + 1} \partial_{\vec{I}} \SecondFund_{\ b}^a)
	\leftexp{(Border;1;\vec{I})}{\mathfrak{K}}_{\ d}^c
	$
	and
	$
	2 g_{ac} g^{bd}
	(t^{\Blowupexp + 1} \partial_{\vec{I}} \SecondFund_{\ b}^a)
	\leftexp{(Junk;1;\vec{I})}{\mathfrak{K}}_{\ d}^c
	$.
	We place the borderline error term products such as
	$
	2 g_{ac} g^{bd}
	(t^{\Blowupexp + 1} \partial_{\vec{I}} \SecondFund_{\ b}^a)
	\leftexp{(Border;1;\vec{I})}{\mathfrak{K}}_{\ d}^c
	$
	on RHS~\eqref{E:METRICCURRENTBORDERTERMS}
	and the junk error term products such as
	$
	2 g_{ac} g^{bd}
	(t^{\Blowupexp + 1} \partial_{\vec{I}} \SecondFund_{\ b}^a)
	\leftexp{(Junk;1;\vec{I})}{\mathfrak{K}}_{\ d}^c
	$
	on RHS~\eqref{E:METRICCURRENTJUNKTERMS}.
	Moreover,
	as the first product on RHS~\eqref{E:DIVIDFORENERGYCURRENT},
	we place 
	$
	\frac{2 \Blowupexp}{t} 
			\left|
				t^{\Blowupexp + 1} \partial_{\vec{I}} \SecondFund 
			\right|_g^2
	$,
	which is generated by the first term on RHS~\eqref{E:COMMUTEDPARTIALTKCMC}
	when we use equation \eqref{E:COMMUTEDPARTIALTKCMC} (with $\Pos = 1$)
	to substitute for the factor $\partial_t (t^{\Blowupexp + 1} \partial_{\vec{I}} \SecondFund_{\ d}^c)$
	in the expression
	$
	2 g_{ac} g^{bd}
	(t^{\Blowupexp + 1} \partial_{\vec{I}} \SecondFund_{\ b}^a)
	\partial_t (t^{\Blowupexp + 1} \partial_{\vec{I}} \SecondFund_{\ d}^c)
	$.
	Finally,
	as the second product on RHS~\eqref{E:DIVIDFORENERGYCURRENT},
	we place 
	$
		\frac{\Blowupexp + 1}{2 t} 
				\left|
					t^{\Blowupexp + 1} \partial \partial_{\vec{I}} g 
				\right|_g^2
	$,
	which is generated by the first term on RHS~\eqref{E:COMMUTEDPARTIALTONEDERIVATIVEOFMETRICCMC}
	when we use equation \eqref{E:COMMUTEDPARTIALTONEDERIVATIVEOFMETRICCMC} 
	to substitute for the factor $\partial_t (t^{\Blowupexp + 1} \partial_{\vec{I}} \partial_d \partial g_{ef})$
	in the expression
	$ 
	\frac{1}{2} 
	g^{ad} g^{be} g^{cf}
	(t^{\Blowupexp + 1} \partial_{\vec{I}} \partial_a \partial g_{bc})
	\partial_t (t^{\Blowupexp + 1} \partial_{\vec{I}} \partial_d \partial g_{ef})
	$.
	In total, these steps yield the lemma.
	
\end{proof}

\subsection{Control of the error terms}
\label{SS:TOPORDERMETRICENERGYESTIMATESCONTROLOFERRORTERMS}
In this subsection, we derive estimates for
the error terms that will arise when we derive energy estimates
for solutions to the equations of
Lemma~\ref{SS:EQUATIONSFORTOPDERIVATIVESOFMETRIC}.

\subsubsection{Pointwise estimates for the error terms in the divergence of the energy current}
\label{SSS:POINTWISECONTROLOFERRORTERMSINDIVERGENCEOFTHEENERGYCURRENTS}
In this subsubsection, we derive pointwise estimates for 
the error terms 
$\leftexp{(Border;\vec{I})}{\mathfrak{J}}$
and
$\leftexp{(Junk;\vec{I})}{\mathfrak{J}}$
that appear in the expression \eqref{E:DIVIDFORENERGYCURRENT} for
$\partial_{\alpha} {\leftexp{(\vec{I})}{\mathbf{J}^{\alpha}}}$.

\begin{lemma}[\textbf{Pointwise estimates for the error terms in the divergence of the energy current}]
	Let $\vec{I}$ be a top-order spatial multi-index, that is, a multi-index with $|\vec{I}|=N$.
	Assume that the bootstrap assumptions \eqref{E:BOOTSTRAPASSUMPTIONS} hold.
	There exists a universal constant $C_* > 0$ \underline{independent of $N$ and $\Blowupexp$}
	such that if $N$ is sufficiently large in a manner that depends on $\Blowupexp$
	and if $\varepsilon$ is sufficiently small,
	then the error terms $\leftexp{(Border;\vec{I})}{\mathfrak{J}}$
	and $\leftexp{(Junk;\vec{I})}{\mathfrak{J}}$
	from \eqref{E:METRICCURRENTBORDERTERMS} and \eqref{E:METRICCURRENTJUNKTERMS}
	verify the following pointwise estimates for $t \in (\TBoot,1]$
	(where, as we described in Subsect.\,\ref{SSS:CONSTANTS}, constants ``$C$'' are allowed to depend on $N$ and other quantities):
\begin{subequations}
\begin{align} \label{E:POINTWISEBOUNDMETRICCURRENTBORDERTERMS}
	\left|
		\leftexp{(Border;\vec{I})}{\mathfrak{J}}
	\right|
	& \leq 
				C_*
				t^{2 \Blowupexp+1}
				\left|
					\partial_{\vec{I}} \SecondFund
				\right|_g^2
				+
				C_*
				t^{2 \Blowupexp+1}
				\left|
					\partial \partial_{\vec{I}} g
				\right|_g^2
				+
				C_*
				t^{2 \Blowupexp+1}
				\left|
					\partial \partial_{\vec{I}} n
				\right|_g^2
					\\
	& \ \
				+
				C_*
				t
				\left|
					\leftexp{(Border;1;\vec{I})}{\mathfrak{K}}
				\right|_g^2
				+
				C_*
				t
				\left|
					\leftexp{(Border;1;\vec{I})}{\mathfrak{H}}
				\right|_g^2
				\notag	\\
	& \ \
				+
				C_*
				t
				\left|
					\leftexp{(Border;\vec{I})}{\mathfrak{M}}
				\right|_g^2
				+
				C_*
				t
				\left|
					\leftexp{(Border;\vec{I})}{\widetilde{\mathfrak{M}}}
				\right|_g^2,
				\notag
\end{align}

\begin{align} \label{E:POINTWISEBOUNDMETRICCURRENTJUNKTERMS}
	\left|
		\leftexp{(Junk;\vec{I})}{\mathfrak{J}}
	\right|
	& \leq
				C
				t^{2 \Blowupexp + 1 + \Room}
				\left|
					\partial_{\vec{I}} \SecondFund
				\right|_g^2
				+
				C
				t^{2 \Blowupexp + 1 + \Room}
				\left|
					\partial \partial_{\vec{I}} g
				\right|_g^2
				+
				C
				t^{2 \Blowupexp + 1 + \Room}
				\left|
					\partial \partial_{\vec{I}} n
				\right|_g^2
					\\
	& \ \
				+
				C
				t^{1 - \Room}
				\left|
					\leftexp{(Junk;1;\vec{I})}{\mathfrak{K}}
				\right|_g^2
				+
				C
				t^{1 - \Room}
				\left|
					\leftexp{(Junk;1;\vec{I})}{\mathfrak{H}}
				\right|_g^2
				\notag	\\
	& \ \
				+
				C
				t^{1 - \Room}
				\left|
					\leftexp{(Junk;\vec{I})}{\mathfrak{M}}
				\right|_g^2
				+
				C
				t^{1 - \Room}
				\left|
					\leftexp{(Junk;\vec{I})}{\mathfrak{M}}
				\right|_g^2
				+
				C
				t^{1 - \Room}
				\left|
					\leftexp{(Junk;\vec{I})}{\widetilde{\mathfrak{M}}}
				\right|_g^2.
				\notag
\end{align}
\end{subequations}

\end{lemma}

\begin{proof}
Throughout this proof, we will assume that $\Blowupexp \updelta$ is sufficiently small
(and in particular that $\Blowupexp \updelta < \Room$);
in view of the discussion in Subsect.\,\ref{SS:SOBOLEVEMBEDDING}, we see that
at fixed $\Blowupexp$, this can be achieved by choosing $N$ to be sufficiently large.
	
	The estimate \eqref{E:POINTWISEBOUNDMETRICCURRENTBORDERTERMS}
	follows in a straightforward fashion from applying the $g$-Cauchy--Schwarz inequality
	and Young's inequality (in the form $|ab| \leq \frac{1}{2}(a^2 + b^2)$)
	to the products on RHS~\eqref{E:METRICCURRENTBORDERTERMS}
	and using the estimate $|g|_g = |g^{-1}|_g \leq C_*$
	as well as the pointwise estimates
	$|n| \leq C_*$
	and
	$|\SecondFund|_g \leq C_* t^{-1}$,
	which follow from
	\eqref{E:WORSTEXPANDROOMLOTSOFUSEFULINEQUALITIES},
	Def.\,\ref{D:LOWNORMS},
	and the bootstrap assumptions.
	
	To prove \eqref{E:POINTWISEBOUNDMETRICCURRENTJUNKTERMS},
	we first note the estimate
	$
	\| n-1 \|_{L^{\infty}(\Sigma_t)}
	\leq C t^{\Room}
	$,
	which follows from
	Def.~\ref{D:LOWNORMS},
	\eqref{E:WORSTEXPANDROOMLOTSOFUSEFULINEQUALITIES},
	and the bootstrap assumptions.
	Next, we use \eqref{E:POINTWISENORMCOMPARISON} with $l=2$ and $m=1$ to deduce
	$
	\| \partial g^{-1} \|_{L_g^{\infty}(\Sigma_t)}
	\leq
	t^{-3 \Worstexp}
	\| g^{-1} \|_{\dot{W}_{Frame}^{1,\infty}(\Sigma_t)}
	$.
	Also using
	\eqref{E:WORSTEXPANDROOMLOTSOFUSEFULINEQUALITIES}
	and
	\eqref{E:UPTOFOURDERIVATIVESOFGINVERSELINFINITYSOBOLEV},
	we deduce that
	$	
	\| \partial g^{-1} \|_{L_g^{\infty}(\Sigma_t)}
	\leq 
	C t^{-5 \Worstexp - \Blowupexp \updelta}
	\leq 
	C
	t^{\Room - 1}
	$.
	Using similar reasoning and the bound \eqref{E:UPTOFOURDERIVATIVESOFLAPSELINFINITYSOBOLEV},
	we deduce that
	$	
	\| \partial n \|_{L_g^{\infty}(\Sigma_t)}
	\leq 
	C t^{2 - 11 \Worstexp - \Room - \Blowupexp \updelta}
	\leq 
	C $.
	Using these three bounds
	to help control the $|\cdot|_g$ norms of the products
	on RHS~\eqref{E:METRICCURRENTJUNKTERMS},
	we can derive inequality
	\eqref{E:POINTWISEBOUNDMETRICCURRENTJUNKTERMS}
	using arguments similar to the ones
	we used to deduce \eqref{E:POINTWISEBOUNDMETRICCURRENTBORDERTERMS}.
	
	\end{proof}

\subsubsection{$L^2$ estimates for the error terms in the divergence of the energy current}
\label{SSS:L2CONTROLOFERRORTERMSINDIVERGENCEOFTHEENERGYCURRENTS}
In this subsubsection,
we bound the error terms from Lemma~\ref{L:EQUATIONSFORTOOPORDERDERIVATIVESOFMETRICANDSECONDFUND}
in the norm $\| \cdot \|_{L_g^2(\Sigma_t)}$.

\begin{lemma}[$L^2$ \textbf{control of the error terms in the top-order energy estimates for $g$ and $\SecondFund$}] 
	\label{L:L2CONTROLOFERRORTERMSINTOPORDERMETRICENERGYESTIMATES}
	Assume that the bootstrap assumptions \eqref{E:BOOTSTRAPASSUMPTIONS} hold.
	There exists a universal constant $C_* > 0$ \underline{independent of $N$ and $\Blowupexp$}
	such that if $N$ is sufficiently large 
	and if $\varepsilon$ is sufficiently small,
	then the following estimates hold for $t \in (\TBoot,1]$
	(where, as we described in Subsect.\,\ref{SSS:CONSTANTS}, constants ``$C$'' are allowed to depend on $N$ and other quantities).
	
	\medskip
	\noindent \underline{\textbf{Borderline top-order error term estimates}}.
	For each top-order spatial multi-index $\vec{I}$ (i.e., $|\vec{I}| = N$),
	the following estimates hold for the
	error terms from
	\eqref{E:TOPORDERDIFFERENTIATEDBORDERMOMENTUMCONSTRAINTINDEXDOWNERRORTERM},
	\eqref{E:TOPORDERDIFFERENTIATEDBORDERMOMENTUMCONSTRAINTINDEXUPERRORTERM},
	\eqref{E:TOPORDERDIFFERENTIATEDBORDERONCEDIFFERENTIATEDMETRICERRORTERM},
	and 
	\eqref{E:SECONDFUNDCOMMUTEDBORDERLINETERM},
	where $\Pos = 1$ 
	on RHSs~\eqref{E:TOPORDERDIFFERENTIATEDBORDERONCEDIFFERENTIATEDMETRICERRORTERM},
	and 
	\eqref{E:SECONDFUNDCOMMUTEDBORDERLINETERM}:
	\begin{subequations}
	\begin{align}
		\left\|
			\leftexp{(Border;\vec{I})}{\mathfrak{M}}
		\right\|_{L_g^2(\Sigma_t)}
		& \leq 
			C_* 
			\left\|
				t^{\Blowupexp} \partial \partial_{\vec{I}} g
			\right\|_{L_g^2(\Sigma_t)},
				\label{E:TOPORDERDIFFERENTIATEDBORDERMOMENTUMCONSTRAINTINDEXDOWNERRORTERML2ESTIMATE}	\\
		\left\|
			\leftexp{(Border;\vec{I})}{\widetilde{\mathfrak{M}}}
		\right\|_{L_g^2(\Sigma_t)}
		& \leq 
			C_* 
			\left\|
				t^{\Blowupexp} \partial \partial_{\vec{I}} g
			\right\|_{L_g^2(\Sigma_t)},
				\label{E:TOPORDERDIFFERENTIATEDBORDERMOMENTUMCONSTRAINTINDEXUPERRORTERML2ESTIMATE} 
				\\
		\left\|
				\leftexp{(Border;1,\vec{I})}{\mathfrak{H}}
		\right\|_{L_g^2(\Sigma_t)}
		& \leq 
			C_* 
			\left\|
				t^{\Blowupexp} \partial_{\vec{I}} \SecondFund
			\right\|_{L_g^2(\Sigma_t)}
			+
			C
			t^{\Room - 1}
			\left\lbrace
				\MetricLownorm(t) + \MetricHighnorm(t)
			\right\rbrace,
			\label{E:TOPORDERDIFFERENTIATEDBORDERMETRICERRORTERML2ESTIMATE}	
				\\
		\left\|
			\leftexp{(Border;1;\vec{I})}{\mathfrak{K}}
		\right\|_{L_g^2(\Sigma_t)}
		& \leq 
			C_* 
			\left\|
				t^{\Blowupexp} \partial_{\vec{I}} \SecondFund
			\right\|_{L_g^2(\Sigma_t)}
			+
			C
			t^{\Room - 1}
			\left\lbrace
				\MetricLownorm(t) + \MetricHighnorm(t)
			\right\rbrace.
			\label{E:SECONDFUNDCOMMUTEDBORDERLINETERML2ESTIMATE}
	\end{align}
	\end{subequations}
	
	\medskip
	
	\noindent \underline{\textbf{Non-borderline top-order error term estimates}}.
	The following estimates hold for the
	error terms from
	\eqref{E:TOPORDERDIFFERENTIATEDJUNKMOMENTUMCONSTRAINTINDEXDOWNERRORTERM},
	\eqref{E:TOPORDERDIFFERENTIATEDJUNKMOMENTUMCONSTRAINTINDEXUPERRORTERM},
	\eqref{E:TOPORDERDIFFERENTIATEDJUNKONCEDIFFERENTIATEDMETRICERRORTERM},
	and 
	\eqref{E:SECONDFUNDCOMMUTEDJUNKTERM},
	where $\Pos = 1$ on 
	RHSs~\eqref{E:TOPORDERDIFFERENTIATEDJUNKONCEDIFFERENTIATEDMETRICERRORTERM}
	and
	\eqref{E:SECONDFUNDCOMMUTEDJUNKTERM}:
	\begin{subequations}
	\begin{align}
	\max_{|\vec{I}|=N}
	\left\|
		\leftexp{(Junk;\vec{I})}{\mathfrak{M}}
	\right\|_{L_g^2(\Sigma_t)}
		& \leq 
			C
			t^{\Room - 1}
			\left\lbrace
				\MetricLownorm(t) + \MetricHighnorm(t)
			\right\rbrace,
				\label{E:TOPORDERDIFFERENTIATEDJUNKMOMENTUMCONSTRAINTINDEXDOWNERRORTERML2ESTIMATE}	\\
		\max_{|\vec{I}|=N}
		\left\|
			\leftexp{(Junk;\vec{I})}{\widetilde{\mathfrak{M}}}
		\right\|_{L_g^2(\Sigma_t)}
		& \leq 
			C
			t^{\Room - 1}
			\left\lbrace
				\MetricLownorm(t) + \MetricHighnorm(t)
			\right\rbrace,
				\label{E:TOPORDERDIFFERENTIATEDJUNKMOMENTUMCONSTRAINTINDEXUPERRORTERML2ESTIMATE} 
				\\
	\max_{|\vec{I}|=N}
	\left\|
		\leftexp{(Junk;1;\vec{I})}{\mathfrak{H}}
	\right\|_{L_g^2(\Sigma_t)}
	& \leq C 
		t^{\Room - 1}
		\left\lbrace
			\MetricLownorm(t) + \MetricHighnorm(t)
		\right\rbrace,
			\label{E:TOPORDERDIFFERENTIATEDJUNKMETRICERRORTERML2ESTIMATE} \\
	\max_{|\vec{I}|=N}
	\left\|
		\leftexp{(Junk;1;\vec{I})}{\mathfrak{K}}
	\right\|_{L_g^2(\Sigma_t)}
	& \leq C 
		t^{\Room - 1}
		\left\lbrace
			\MetricLownorm(t) + \MetricHighnorm(t)
		\right\rbrace.
		\label{E:TOPORDERSECONDFUNDJUNKERRORTERML2BOUND}
	\end{align}
	\end{subequations}
	
\end{lemma}

\begin{proof}
	Throughout this proof, we will assume that $\Blowupexp \updelta$ is sufficiently small
	(and in particular that $\Blowupexp \updelta < \Room$);
	in view of the discussion in Subsect.\,\ref{SS:SOBOLEVEMBEDDING}, we see that
	at fixed $\Blowupexp$, this can be achieved by choosing $N$ to be sufficiently large.
	
	\medskip
	
	\noindent \underline{\textbf{Proof of \eqref{E:TOPORDERDIFFERENTIATEDBORDERMOMENTUMCONSTRAINTINDEXDOWNERRORTERML2ESTIMATE}-\eqref{E:SECONDFUNDCOMMUTEDBORDERLINETERML2ESTIMATE}}}:
	Let $\vec{I}$ be a spatial multi-index with $|\vec{I}|=N$.
	We first use 
	Def.~\ref{D:LOWNORMS},
	the fact that $|g^{-1}|_g \leq C_*$,
	the bootstrap assumptions,
	and $g$-Cauchy--Schwarz
	to bound the product on RHS~\eqref{E:TOPORDERDIFFERENTIATEDBORDERMOMENTUMCONSTRAINTINDEXDOWNERRORTERM}
	as follows:
	$\left|
		\leftexp{(Border;\vec{I})}{\mathfrak{M}}
	 \right|_g
		\leq C_* t^{\Blowupexp + 1} 
		\left\|
			\SecondFund
		\right\|_{L_g^{\infty}(\Sigma_t)}
		|\partial \partial_{\vec{I}} g|_g
		\leq C_* t^{\Blowupexp} |\partial \partial_{\vec{I}} g|_g
	$.
	Taking the norm $\| \cdot \|_{L^2(\Sigma_t)}$ of this inequality,
	we conclude that
	$
		\left\|
			\leftexp{(Border;\vec{I})}{\mathfrak{M}}
		\right\|_{L_g^2(\Sigma_t)}
	\leq
		C_* 
		\left\|
			t^{\Blowupexp} \partial \partial_{\vec{I}} g
		\right\|_{L_g^2(\Sigma_t)}
	$
	as desired.
	The estimate \eqref{E:TOPORDERDIFFERENTIATEDBORDERMOMENTUMCONSTRAINTINDEXUPERRORTERML2ESTIMATE}
	follows similarly based on equation \eqref{E:TOPORDERDIFFERENTIATEDBORDERMOMENTUMCONSTRAINTINDEXUPERRORTERM},
	and we omit the details.
	The estimates \eqref{E:TOPORDERDIFFERENTIATEDBORDERMETRICERRORTERML2ESTIMATE} and \eqref{E:SECONDFUNDCOMMUTEDBORDERLINETERML2ESTIMATE}
	follow similarly based on equations
	\eqref{E:TOPORDERDIFFERENTIATEDBORDERONCEDIFFERENTIATEDMETRICERRORTERM}
	and 
	\eqref{E:SECONDFUNDCOMMUTEDBORDERLINETERM}
	(with $\Pos = 1$ by assumption)
	and the elliptic estimate \eqref{E:LAPSETOPORDERELLIPTIC},
	and we omit the details.
	
	\medskip

\noindent \underline{\textbf{Proof of \eqref{E:TOPORDERDIFFERENTIATEDJUNKMOMENTUMCONSTRAINTINDEXDOWNERRORTERML2ESTIMATE}-\eqref{E:TOPORDERDIFFERENTIATEDJUNKMOMENTUMCONSTRAINTINDEXUPERRORTERML2ESTIMATE}}}:	
We prove only \eqref{E:TOPORDERDIFFERENTIATEDJUNKMOMENTUMCONSTRAINTINDEXUPERRORTERML2ESTIMATE};
the estimate 
\eqref{E:TOPORDERDIFFERENTIATEDJUNKMOMENTUMCONSTRAINTINDEXDOWNERRORTERML2ESTIMATE}
can be proved by applying a similar argument to the products on 
RHS~\eqref{E:TOPORDERDIFFERENTIATEDJUNKMOMENTUMCONSTRAINTINDEXDOWNERRORTERM}
and we omit those details.
Let $\vec{I}$ be a spatial multi-index with $|\vec{I}|=N$.
To obtain the desired bound for the sum on RHS~\eqref{E:TOPORDERDIFFERENTIATEDJUNKMOMENTUMCONSTRAINTINDEXUPERRORTERM},
we first consider the cases in which $|\vec{I}_1| = N$ or $|\vec{I}_2| = N$.
Using that $|g^{-1}|_g \lesssim 1$
and $g$-Cauchy--Schwarz,
we first bound the products under consideration in the norm $\| \cdot \|_{L_g^2(\Sigma_t)}$
by
$
\lesssim
t^{\Blowupexp + 1}
\| \partial g \|_{L_g^{\infty}} 
\| \SecondFund \|_{L_g^{\infty}} 
\| g^{-1} \|_{\dot{H}_g^N}
$.
Applying \eqref{E:POINTWISENORMCOMPARISON}
to 
$\| \partial g \|_{L_g^{\infty}}$ 
(with $l=0$ and $m=3$), 
we deduce that the RHS of the previous expression is
$
\lesssim
t^{\Blowupexp + 1 - 3 \Worstexp}
\| g \|_{\dot{W}_{Frame}^{1,\infty}} 
\| \SecondFund \|_{L_g^{\infty}} 
\| g^{-1} \|_{\dot{H}_g^N}
$.
From Defs.\,\ref{D:LOWNORMS} and~\ref{D:HIGHNORMS},
the estimate
\eqref{E:UPTOFOURDERIVATIVESOFGLINFINITYSOBOLEV},
and the bootstrap assumptions, 
we deduce that the RHS of the previous expression is
$
\lesssim
t^{- 6 \Worstexp - \Room - \Blowupexp \updelta}
\MetricHighnorm(t)
$,
which, 
in view of \eqref{E:WORSTEXPANDROOMLOTSOFUSEFULINEQUALITIES}, 
is $\lesssim \mbox{RHS~\eqref{E:TOPORDERDIFFERENTIATEDJUNKMOMENTUMCONSTRAINTINDEXUPERRORTERML2ESTIMATE}}$ as desired. 
We now consider the case in which $|\vec{I}_3| = N-1$ 
on RHS~\eqref{E:TOPORDERDIFFERENTIATEDJUNKMOMENTUMCONSTRAINTINDEXUPERRORTERM}.
Using that $|g^{-1}|_g \lesssim 1$
and $g$-Cauchy--Schwarz,
we first bound the products under consideration in the norm $\| \cdot \|_{L_g^2(\Sigma_t)}$
by
$
\lesssim
t^{\Blowupexp + 1}
\| g^{-1} \|_{\dot{W}_g^{1,\infty}} \| \SecondFund \|_{L_g^{\infty}} \| \partial g \|_{\dot{H}_g^{N-1}}
+
t^{\Blowupexp + 1}
\| \SecondFund \|_{\dot{W}_g^{1,\infty}} \| \partial g \|_{\dot{H}_g^{N-1}}
$.
Applying \eqref{E:POINTWISENORMCOMPARISON}
to $\| g^{-1} \|_{\dot{W}_g^{1,\infty}}$
(with $l=2$ and $m=0$)
and to
$\| \SecondFund \|_{\dot{W}_g^{1,\infty}}$
(with $l=m=1$),
we deduce that the RHS of the previous expression is
$
\lesssim
t^{\Blowupexp + 1 - 2 \Worstexp}
\| g^{-1} \|_{\dot{W}_{Frame}^{1,\infty}} 
\| \SecondFund \|_{L_g^{\infty}} 
\| \partial g \|_{\dot{H}_g^{N-1}}
+
t^{\Blowupexp + 1 - 2 \Worstexp}
\| \SecondFund \|_{\dot{W}_{Frame}^{1,\infty}} 
\| \partial g \|_{\dot{H}_g^{N-1}}
$.
From Defs.\,\ref{D:LOWNORMS} and~\ref{D:HIGHNORMS},
the estimate
\eqref{E:UPTOFOURDERIVATIVESOFGINVERSELINFINITYSOBOLEV},
and the bootstrap assumptions, 
we deduce that the RHS of the previous expression is
$
\lesssim
t^{- 6 \Worstexp - \Room - \Blowupexp \updelta}
\MetricHighnorm(t)
$,
which, 
in view of \eqref{E:WORSTEXPANDROOMLOTSOFUSEFULINEQUALITIES}, 
is $\lesssim \mbox{RHS~\eqref{E:TOPORDERDIFFERENTIATEDJUNKMOMENTUMCONSTRAINTINDEXUPERRORTERML2ESTIMATE}}$
as desired.
We now consider the case in which $|\vec{I}_4| = N$
on RHS~\eqref{E:TOPORDERDIFFERENTIATEDJUNKMOMENTUMCONSTRAINTINDEXUPERRORTERM}.
Using that $|g^{-1}|_g \lesssim 1$
and $g$-Cauchy--Schwarz,
we first bound the products under consideration in the norm $\| \cdot \|_{L_g^2(\Sigma_t)}$
by
$
\lesssim
t^{\Blowupexp + 1}
\| \partial g \|_{L_g^{\infty}} 
\| \SecondFund \|_{\dot{H}_g^N}
$.
Applying \eqref{E:POINTWISENORMCOMPARISON}
to $\| \partial g \|_{L_g^{\infty}}$
(with $l=0$ and $m=3$), 
we deduce that the RHS of the previous expression is
$
\lesssim
t^{\Blowupexp + 1 - 3 \Worstexp}
\| g \|_{\dot{W}_{Frame}^{1,\infty}} 
\| \SecondFund \|_{\dot{H}_g^N}
$.
From Def.\,\ref{D:HIGHNORMS},
the estimate
\eqref{E:UPTOFOURDERIVATIVESOFGLINFINITYSOBOLEV},
and the bootstrap assumptions, 
we deduce that the RHS of the previous expression is
$
\lesssim
t^{- 5 \Worstexp - \Blowupexp \updelta}
\MetricHighnorm(t)
$,
which, in view of \eqref{E:WORSTEXPANDROOMLOTSOFUSEFULINEQUALITIES}, 
is $\lesssim \mbox{RHS~\eqref{E:TOPORDERDIFFERENTIATEDJUNKMOMENTUMCONSTRAINTINDEXUPERRORTERML2ESTIMATE}}$
as desired.
It remains for us to consider the remaining terms on RHS~\eqref{E:TOPORDERDIFFERENTIATEDJUNKMOMENTUMCONSTRAINTINDEXUPERRORTERM}, 
in which 
$|\vec{I}_1|,|\vec{I}_2|,|\vec{I}_4| \leq N-1$
and
$|\vec{I}_3| \leq N-2$.
We first use 
\eqref{E:POINTWISENORMCOMPARISON} with $l=1$ and $m=0$
(since $\leftexp{(Junk;\vec{I})}{\widetilde{\mathfrak{M}}}^i$ is a type $\binom{1}{0}$ tensorfield),
\eqref{E:BASICINTERPOLATION}, 
and \eqref{E:FRAMENORML2PRODUCTBOUNDINERMSOFLINFINITYANDHMDOT}
to bound 
(using that $|\vec{I}|=N$)
the terms under consideration as follows:
\begin{align} \label{E:FIRSTSTEPTOPORDERDIFFERENTIATEDJUNKMOMENTUMCONSTRAINTINDEXUPERRORTERML2ESTIMATEL2BOUND}
&
\mathop{\mathop{\sum_{\vec{I}_1 + \vec{I}_2 + \vec{I}_3 + \vec{I}_4 = \vec{I}}}_{|\vec{I}_1|,|\vec{I}_2|,|\vec{I}_4| \leq N-1}}_{|\vec{I}_3| \leq N-2}
	t^{\Blowupexp + 1}
	\left\|
		(\partial_{\vec{I}_1} g^{-1}) 
		(\partial_{\vec{I}_2} g^{-1})
		(\partial \partial_{\vec{I}_3} g) 
		\partial_{\vec{I}_4} \SecondFund
	\right\|_{L_g^2(\Sigma_t)}
		\\
& \ \	
	\lesssim
	t^{\Blowupexp + 1 - \Worstexp}
	\left\|
		g^{-1}
	\right\|_{W_{Frame}^{2,\infty}(\Sigma_t)}^2
	\left\|
		g - \KasnerMetric
	\right\|_{W_{Frame}^{2,\infty}(\Sigma_t)}
	\left\|
		\SecondFund
	\right\|_{\dot{H}_{Frame}^{N-1}(\Sigma_t)}
		\notag \\
& \ \
	+
	t^{\Blowupexp + 1 - \Worstexp}
	\left\|
		g^{-1}
	\right\|_{W_{Frame}^{2,\infty}(\Sigma_t)}^2
	\left\|
		\SecondFund
	\right\|_{W_{Frame}^{2,\infty}(\Sigma_t)}
	\left\|
		g
	\right\|_{\dot{H}_{Frame}^{N-1}(\Sigma_t)}
		\notag \\
& \ \
		+
		t^{\Blowupexp + 1 - \Worstexp}
	\left\|
		g^{-1}
	\right\|_{W_{Frame}^{2,\infty}(\Sigma_t)}
	\left\|
		g - \KasnerMetric
	\right\|_{W_{Frame}^{2,\infty}(\Sigma_t)}
	\left\|
		\SecondFund
	\right\|_{W_{Frame}^{2,\infty}(\Sigma_t)}
	\left\|
		g^{-1}
	\right\|_{\dot{H}_{Frame}^{N-1}(\Sigma_t)}
		\notag \\	
& \ \
	+
	t^{\Blowupexp + 1 - \Worstexp}
	\left\|
		g^{-1}
	\right\|_{W_{Frame}^{2,\infty}(\Sigma_t)}^2
	\left\|
		g - \KasnerMetric
	\right\|_{W_{Frame}^{2,\infty}(\Sigma_t)}
	\left\|
		\SecondFund
	\right\|_{W_{Frame}^{2,\infty}(\Sigma_t)}.
	\notag
\end{align}
From Defs.\,\ref{D:LOWNORMS} and~\ref{D:HIGHNORMS},
the estimates 
\eqref{E:KASNERMETRICESTIMATES},
\eqref{E:KASNERSECONDFUNDESTIMATES},
\eqref{E:UPTOFOURDERIVATIVESOFGLINFINITYSOBOLEV},
and
\eqref{E:UPTOFOURDERIVATIVESOFGINVERSELINFINITYSOBOLEV},
and the bootstrap assumptions, we deduce that
$\mbox{RHS~\eqref{E:FIRSTSTEPTOPORDERDIFFERENTIATEDJUNKMOMENTUMCONSTRAINTINDEXUPERRORTERML2ESTIMATEL2BOUND}}
\lesssim
t^{1 - 10 \Worstexp - 3 \Room - \Blowupexp \updelta}
\left\lbrace
	\MetricLownorm(t)
	+
	\MetricHighnorm(t) 
\right\rbrace
$.
In view of \eqref{E:WORSTEXPANDROOMLOTSOFUSEFULINEQUALITIES}, 
we see that the RHS of the previous expression is
$\lesssim \mbox{RHS~\eqref{E:TOPORDERDIFFERENTIATEDJUNKMOMENTUMCONSTRAINTINDEXUPERRORTERML2ESTIMATE}}$
as desired.

\medskip

\noindent \underline{\textbf{Proof of \eqref{E:TOPORDERDIFFERENTIATEDJUNKMETRICERRORTERML2ESTIMATE}}}:	
We stress that for this estimate, on RHS~\eqref{E:TOPORDERDIFFERENTIATEDJUNKONCEDIFFERENTIATEDMETRICERRORTERM},
we have $\Pos = 1$ and $|\vec{I}|=N$.

To bound the first product on RHS~\eqref{E:TOPORDERDIFFERENTIATEDJUNKONCEDIFFERENTIATEDMETRICERRORTERM},
we first use $g$-Cauchy--Schwarz to deduce that it is bounded 
 in the norm
$\| \cdot \|_{L_g^2(\Sigma_t)}$
by
$
\leq
t^{\Blowupexp + 1} 
\| n-1 \|_{L^{\infty}(\Sigma_t)}
\left\|
	\SecondFund
\right\|_{L_g^{\infty}(\Sigma_t)}
\left\|
	\partial g
\right\|_{\dot{H}_g^N(\Sigma_t)}
$.
From Defs.\,\ref{D:LOWNORMS} and~\ref{D:HIGHNORMS},
\eqref{E:WORSTEXPANDROOMLOTSOFUSEFULINEQUALITIES},
and the bootstrap assumptions,
we see that the RHS of the previous expression is
$
\lesssim
t^{1 - 10 \Worstexp - \Room} \MetricHighnorm(t)
\lesssim
t^{\Room - 1} \MetricHighnorm(t)
$
as desired.

To complete the proof of \eqref{E:TOPORDERDIFFERENTIATEDJUNKMETRICERRORTERML2ESTIMATE},
we must bound the three sums on RHS~\eqref{E:TOPORDERDIFFERENTIATEDJUNKONCEDIFFERENTIATEDMETRICERRORTERM}
in the norm $\| \cdot \|_{L_g^2(\Sigma_t)}$.
To bound the first sum on RHS~\eqref{E:TOPORDERDIFFERENTIATEDJUNKONCEDIFFERENTIATEDMETRICERRORTERM},
we first consider the products with $|\vec{I}_3|=N$.
Using that $|g|_g \lesssim 1$ and $g$-Cauchy--Schwarz,
we first bound the products under consideration in the norm $\| \cdot \|_{L_g^2(\Sigma_t)}$
by
$
\lesssim
t^{\Blowupexp + 1}
\| \partial n \|_{L_g^{\infty}(\Sigma_t)} 
 \| \SecondFund \|_{\dot{H}_g^N(\Sigma_t)}
$.
Applying \eqref{E:POINTWISENORMCOMPARISON}
to $\| \partial n \|_{L_g^{\infty}(\Sigma_t)}$
(with $l=0$ and $m=1$), 
we deduce that the RHS of the previous expression is
$
\lesssim
t^{\Blowupexp + 1 - \Worstexp}
\| n \|_{\dot{W}^{1,\infty}(\Sigma_t)} 
\| \SecondFund \|_{\dot{H}_g^N(\Sigma_t)}
$.
From Defs.\,\ref{D:HIGHNORMS},
the estimate \eqref{E:UPTOFOURDERIVATIVESOFLAPSELINFINITYSOBOLEV},
and the bootstrap assumptions, 
we deduce that the RHS of the previous expression is
$
\lesssim
t^{2 - 11 \Worstexp - \Room - \Blowupexp \updelta}
\MetricHighnorm(t)
$,
which, 
in view of \eqref{E:WORSTEXPANDROOMLOTSOFUSEFULINEQUALITIES}, 
is $\lesssim \mbox{RHS~\eqref{E:TOPORDERDIFFERENTIATEDJUNKMETRICERRORTERML2ESTIMATE}}$ as desired. 
It remains for us to bound the products
in the first sum on RHS~\eqref{E:TOPORDERDIFFERENTIATEDJUNKONCEDIFFERENTIATEDMETRICERRORTERM}
with $|\vec{I}_1|, |\vec{I}_3| \leq N-1$.
We first use
\eqref{E:L2PRODUCTBOUNDINERMSOFLINFINITYANDHMDOT} with $l=0$ and $m=3$ 
(since $\leftexp{(Junk;\Pos;\vec{I})}{\mathfrak{H}}$ is a type $\binom{0}{3}$ tensorfield)
and \eqref{E:BASICINTERPOLATION}
to deduce 
(using that $|\vec{I}|=N$)
that the products under consideration are bounded as follows:
\begin{align} \label{E:FIRSTSTEPTOPORDERJUNKMETRICERRRORTERM1L2BOUND2}
&
\mathop{\sum_{\vec{I}_1 + \vec{I}_2 + \vec{I}_3 = \vec{I}}}_{|\vec{I}_1|, |\vec{I}_3| \leq N-1}
	t^{\Blowupexp + 1}
	\left\|
		(\partial \partial_{\vec{I}_1} n) 
		(\partial_{\vec{I}_2} g) 
		\partial_{\vec{I}_3} \SecondFund
	\right\|_{L_g^2(\Sigma_t)}
		\\
& \ \	
	\lesssim
	t^{\Blowupexp + 1 - 3 \Worstexp}
	\left\|
		n - 1
	\right\|_{W^{2,\infty}(\Sigma_t)}
	\left\|
		g 
	\right\|_{W_{Frame}^{1,\infty}(\Sigma_t)}
	\left\|
		\SecondFund
	\right\|_{\dot{H}_{Frame}^{N-1}(\Sigma_t)}
		\notag \\
& \ \
	+
	t^{\Blowupexp + 1 - 3 \Worstexp}
	\left\|
		n - 1
	\right\|_{W^{2,\infty}(\Sigma_t)}
	\left\|
		\SecondFund 
	\right\|_{W_{Frame}^{1,\infty}(\Sigma_t)}
	\left\|
		g
	\right\|_{\dot{H}_{Frame}^N(\Sigma_t)}
		\notag \\
& \ \
	+
	t^{\Blowupexp + 1 - 3 \Worstexp}
	\left\|
		g 
	\right\|_{W_{Frame}^{1,\infty}(\Sigma_t)}
	\left\|
		 \SecondFund
	\right\|_{W_{Frame}^{1,\infty}(\Sigma_t)}
	\left\|
		n
	\right\|_{\dot{H}^N(\Sigma_t)}
	\notag
		\\	
	& \ \
	+
	t^{\Blowupexp + 1 - 3 \Worstexp}
	\left\|
		n - 1
	\right\|_{W^{2,\infty}(\Sigma_t)}
	\left\|
		\SecondFund
	\right\|_{W_{Frame}^{1,\infty}(\Sigma_t)}
	\left\|
		g 
	\right\|_{W_{Frame}^{1,\infty}(\Sigma_t)}.
	\notag
\end{align}
From Defs.\,\ref{D:LOWNORMS} and~\ref{D:HIGHNORMS},
the estimates 

\eqref{E:KASNERSECONDFUNDESTIMATES},
\eqref{E:UPTOFOURDERIVATIVESOFGLINFINITYSOBOLEV},
and
\eqref{E:UPTOFOURDERIVATIVESOFLAPSELINFINITYSOBOLEV},
the elliptic estimates 
\eqref{E:LAPSELOWNORMELLIPTIC} 
and
\eqref{E:LAPSETOPORDERELLIPTIC},
and the bootstrap assumptions, 
we deduce that
\begin{align} \label{E:BOUNDFORLOWERORDERTERMSFIRSTSTEPTOPORDERJUNKMETRICERRRORTERM1L2BOUND2}
\mbox{RHS~\eqref{E:FIRSTSTEPTOPORDERJUNKMETRICERRRORTERM1L2BOUND2}}
&
\lesssim
t^{3 - 18 \Worstexp - 2 \Room - \Blowupexp \updelta}
\left\lbrace
	\MetricLownorm(t)
	+
	\MetricHighnorm(t) 
\right\rbrace
+
t^{2 - 15 \Worstexp - 2 \Room - \Blowupexp \updelta}
\left\lbrace
	\MetricLownorm(t)
	+
	\MetricHighnorm(t) 
\right\rbrace
	\\
& \ \
	+
	t^{- 5 \Worstexp - \Blowupexp \updelta}
	\left\lbrace
		\MetricLownorm(t)
		+
		\MetricHighnorm(t) 
	\right\rbrace.
	\notag
\end{align}
In view of \eqref{E:WORSTEXPANDROOMLOTSOFUSEFULINEQUALITIES}, 
we see that 
$\mbox{RHS~\eqref{E:BOUNDFORLOWERORDERTERMSFIRSTSTEPTOPORDERJUNKMETRICERRRORTERM1L2BOUND2}} 
\lesssim \mbox{RHS~\eqref{E:TOPORDERDIFFERENTIATEDJUNKMETRICERRORTERML2ESTIMATE}}$
as desired.

To bound the second sum on RHS~\eqref{E:TOPORDERDIFFERENTIATEDJUNKONCEDIFFERENTIATEDMETRICERRORTERM},
we first consider the products with $|\vec{I}_3|=N$.
Using $g$-Cauchy--Schwarz,
we first bound the products under consideration in the norm $\| \cdot \|_{L_g^2(\Sigma_t)}$
by
$
\lesssim
t^{\Blowupexp + 1}
\| n \|_{L^{\infty}(\Sigma_t)} 
\| \partial g \|_{L_g^{\infty}(\Sigma_t)} 
\| \SecondFund \|_{\dot{H}_g^N(\Sigma_t)}
$.
Applying \eqref{E:POINTWISENORMCOMPARISON}
to $\| \partial g \|_{L_g^{\infty}(\Sigma_t)} $
(with $l=0$ and $m=3$), 
we deduce that the RHS of the previous expression is
$
\lesssim
t^{\Blowupexp + 1 - 3 \Worstexp}
\| n \|_{L^{\infty}(\Sigma_t)} 
\| g \|_{\dot{W}_{Frame}^{1,\infty}(\Sigma_t)} 
\| \SecondFund \|_{\dot{H}_g^N(\Sigma_t)}
$.
From 
\eqref{E:WORSTEXPANDROOMLOTSOFUSEFULINEQUALITIES},
Defs.\,\ref{D:LOWNORMS} and~\ref{D:HIGHNORMS},
the estimate
\eqref{E:UPTOFOURDERIVATIVESOFGLINFINITYSOBOLEV},
and the bootstrap assumptions, 
we deduce that the RHS of the previous expression is
$
\lesssim
t^{- 5 \Worstexp - \Blowupexp \updelta}
\MetricHighnorm(t)
$,
which, 
in view of \eqref{E:WORSTEXPANDROOMLOTSOFUSEFULINEQUALITIES}, 
is $\lesssim \mbox{RHS~\eqref{E:TOPORDERDIFFERENTIATEDJUNKMETRICERRORTERML2ESTIMATE}}$ as desired. 
It remains for us to consider the remaining terms 
in the second sum on RHS~\eqref{E:TOPORDERDIFFERENTIATEDJUNKONCEDIFFERENTIATEDMETRICERRORTERM},
in which $|\vec{I}_2|, |\vec{I}_3| \leq N-1$.
We first use
\eqref{E:L2PRODUCTBOUNDINERMSOFLINFINITYANDHMDOT} with $l=0$ and $m=3$ 
(since $\leftexp{(Junk;\Pos;\vec{I})}{\mathfrak{H}}$ is a type $\binom{0}{3}$ tensorfield)
and \eqref{E:BASICINTERPOLATION}
to deduce 
(using that $|\vec{I}|=N$)
that the terms under consideration are bounded as follows:
\begin{align} \label{E:FIRSTSTEPTOPORDERJUNKMETRICERRRORTERM2L2BOUND2}
&
\mathop{\sum_{\vec{I}_1 + \vec{I}_2 + \vec{I}_3 = \vec{I}}}_{|\vec{I}_2|, |\vec{I}_3| \leq N-1}
	t^{\Blowupexp + 1}
	\left\|
		(\partial_{\vec{I}_1} n) 
		(\partial \partial_{\vec{I}_2} g) 
		\partial_{\vec{I}_3} \SecondFund
	\right\|_{L_g^2(\Sigma_t)}
		\\
& \ \	
	\lesssim
	t^{\Blowupexp + 1 - 3 \Worstexp}
	\left\|
		n 
	\right\|_{W^{1,\infty}(\Sigma_t)}
	\left\|
		g - \KasnerMetric
	\right\|_{W_{Frame}^{2,\infty}(\Sigma_t)}
	\left\|
		\SecondFund
	\right\|_{\dot{H}_{Frame}^{N-1}(\Sigma_t)}
		\notag \\
& \ \
	+
	t^{\Blowupexp + 1 - 3 \Worstexp}
	\left\|
		n 
	\right\|_{W^{1,\infty}(\Sigma_t)}
	\left\|
		\SecondFund 
	\right\|_{W_{Frame}^{1,\infty}(\Sigma_t)}
	\left\|
		g
	\right\|_{\dot{H}_{Frame}^N(\Sigma_t)}
		\notag \\
& \ \
	+
	t^{\Blowupexp + 1 - 3 \Worstexp}
	\left\|
		g - \KasnerMetric
	\right\|_{W_{Frame}^{2,\infty}(\Sigma_t)}
	\left\|
		 \SecondFund
	\right\|_{W_{Frame}^{1,\infty}(\Sigma_t)}
	\left\|
		n
	\right\|_{\dot{H}^N(\Sigma_t)}
	\notag
		\\	
	& \ \
	+
	t^{\Blowupexp + 1 - 3 \Worstexp}
	\left\|
		n 
	\right\|_{W^{1,\infty}(\Sigma_t)}
	\left\|
		\SecondFund
	\right\|_{L_{Frame}^{\infty}(\Sigma_t)}
	\left\|
		g - \KasnerMetric
	\right\|_{W_{Frame}^{2,\infty}(\Sigma_t)}.
	\notag
\end{align}
From 
\eqref{E:WORSTEXPANDROOMLOTSOFUSEFULINEQUALITIES},
Defs.\,\ref{D:LOWNORMS} and~\ref{D:HIGHNORMS},
the estimates 
\eqref{E:KASNERSECONDFUNDESTIMATES},
\eqref{E:UPTOFOURDERIVATIVESOFGLINFINITYSOBOLEV},
and
\eqref{E:UPTOFOURDERIVATIVESOFLAPSELINFINITYSOBOLEV},
and the bootstrap assumptions, 
we deduce that
$\mbox{RHS~\eqref{E:FIRSTSTEPTOPORDERJUNKMETRICERRRORTERM2L2BOUND2}}
\lesssim
t^{1 - 8 \Worstexp - \Room - \Blowupexp \updelta}
\left\lbrace
	\MetricLownorm(t)
	+
	\MetricHighnorm(t) 
\right\rbrace
+
t^{- 5 \Worstexp - \Room - \Blowupexp \updelta}
\left\lbrace
	\MetricLownorm(t)
	+
	\MetricHighnorm(t) 
\right\rbrace
$.
In view of \eqref{E:WORSTEXPANDROOMLOTSOFUSEFULINEQUALITIES}, 
we see that the RHS of the previous expression is
$\lesssim \mbox{RHS~\eqref{E:TOPORDERDIFFERENTIATEDJUNKMETRICERRORTERML2ESTIMATE}}$
as desired.

To bound the last sum on RHS~\eqref{E:TOPORDERDIFFERENTIATEDJUNKONCEDIFFERENTIATEDMETRICERRORTERM},
we first consider the products 
with $|\vec{I}_3|=N-1$.
Using that $|g|_g \lesssim 1$
and $g$-Cauchy--Schwarz,
we first bound the products under consideration in the norm $\| \cdot \|_{L_g^2(\Sigma_t)}$
by
$
\lesssim
t^{\Blowupexp + 1}
\| n \|_{\dot{W}^{1,\infty}(\Sigma_t)} 
\| \partial \SecondFund \|_{\dot{H}_g^{N-1}(\Sigma_t)}
+
t^{\Blowupexp + 1}
\| n \|_{L^{\infty}(\Sigma_t)} 
\| g \|_{\dot{W}_g^{1,\infty}(\Sigma_t)} 
\| \partial \SecondFund \|_{\dot{H}_g^{N-1}(\Sigma_t)}
$.
Using \eqref{E:POINTWISENORMCOMPARISON}
(with $l=0$ and $m=2$)
to estimate $\| g \|_{\dot{W}_g^{1,\infty}(\Sigma_t)}$,
and using the definition of the norms
$\| \cdot \|_{\dot{H}_g^M}$
and
$\| \cdot \|_{L_{Frame}^2(\Sigma_t)}$,
we deduce that the RHS of the previous expression is
$
\lesssim
t^{\Blowupexp + 1}
\| n \|_{\dot{W}^{1,\infty}(\Sigma_t)} 
\| g^{-1} \|_{L_{Frame}^{\infty}(\Sigma_t)}^{1/2}
\| \SecondFund \|_{\dot{H}_g^N(\Sigma_t)}
+
t^{\Blowupexp + 1 - 2 \Worstexp}
\| n \|_{L^{\infty}(\Sigma_t)} 
\| g \|_{\dot{W}_{Frame}^{1,\infty}(\Sigma_t)} 
\| g^{-1} \|_{L_{Frame}^{\infty}(\Sigma_t)}^{1/2}
\| \SecondFund \|_{\dot{H}_g^N(\Sigma_t)}
$.
From 
\eqref{E:WORSTEXPANDROOMLOTSOFUSEFULINEQUALITIES},
Defs.\,\ref{D:LOWNORMS} and~\ref{D:HIGHNORMS},
the estimates \eqref{E:KASNERMETRICESTIMATES},
\eqref{E:UPTOFOURDERIVATIVESOFGLINFINITYSOBOLEV}
and \eqref{E:UPTOFOURDERIVATIVESOFLAPSELINFINITYSOBOLEV},
and the bootstrap assumptions, 
we deduce that the RHS of the previous expression is
$
\lesssim
t^{2 - 11 \Worstexp - \Room - \Blowupexp \updelta}
\MetricHighnorm(t)
+
t^{- 5 \Worstexp - \Blowupexp \updelta}
\MetricHighnorm(t)
$,
which, 
in view of \eqref{E:WORSTEXPANDROOMLOTSOFUSEFULINEQUALITIES}, 
is $\lesssim \mbox{RHS~\eqref{E:TOPORDERDIFFERENTIATEDJUNKMETRICERRORTERML2ESTIMATE}}$ as desired. 
It remains for us to consider the remaining terms 
in the last sum on RHS~\eqref{E:TOPORDERDIFFERENTIATEDJUNKONCEDIFFERENTIATEDMETRICERRORTERM},
in which $|\vec{I}_3| \leq N-2$.
We first use
\eqref{E:L2PRODUCTBOUNDINERMSOFLINFINITYANDHMDOT} with $l=0$ and $m=3$ 
(since $\leftexp{(Junk;\Pos;\vec{I})}{\mathfrak{H}}$ is a type $\binom{0}{3}$ tensorfield)
and \eqref{E:BASICINTERPOLATION}
to deduce 
(using that $|\vec{I}|=N$)
that the terms under consideration are bounded as follows:
\begin{align} \label{E:FIRSTSTEPTOPORDERJUNKMETRICERRRORTERM3L2BOUND2}
&
\mathop{\sum_{\vec{I}_1 + \vec{I}_2 + \vec{I}_3 = \vec{I}}}_{|\vec{I}_3| \leq N-2}
	t^{\Blowupexp + 1}
	\left\|
		(\partial_{\vec{I}_1} n) 
		(\partial_{\vec{I}_2} g) 
		\partial \partial_{\vec{I}_3} \SecondFund
	\right\|_{L_g^2(\Sigma_t)}
		\\
& \ \	
	\lesssim
	t^{\Blowupexp + 1 - 3 \Worstexp}
	\left\|
		n 
	\right\|_{W^{2,\infty}(\Sigma_t)}
	\left\|
		g 
	\right\|_{W_{Frame}^{2,\infty}(\Sigma_t)}
	\left\|
		\SecondFund
	\right\|_{\dot{H}_{Frame}^{N-1}(\Sigma_t)}
		\notag \\
& \ \
	+
	t^{\Blowupexp + 1 - 3 \Worstexp}
	\left\|
		n 
	\right\|_{W^{2,\infty}(\Sigma_t)}
	\left\|
		\SecondFund 
	\right\|_{\dot{W}_{Frame}^{1,\infty}(\Sigma_t)}
	\left\|
		g
	\right\|_{\dot{H}_{Frame}^N(\Sigma_t)}
		\notag \\
& \ \
	+
	t^{\Blowupexp + 1 - 3 \Worstexp}
	\left\|
		g 
	\right\|_{W_{Frame}^{2,\infty}(\Sigma_t)}
	\left\|
		\SecondFund 
	\right\|_{\dot{W}_{Frame}^{1,\infty}(\Sigma_t)}
	\left\|
		n
	\right\|_{\dot{H}^N(\Sigma_t)}
	\notag
		\\	
	& \ \
	+
	t^{\Blowupexp + 1 - 3 \Worstexp}
	\left\|
		n 
	\right\|_{W^{2,\infty}(\Sigma_t)}
	\left\|
		\SecondFund 
	\right\|_{\dot{W}_{Frame}^{1,\infty}(\Sigma_t)}
	\left\|
		g 
	\right\|_{W_{Frame}^{2,\infty}(\Sigma_t)}.
	\notag
\end{align}
From 
\eqref{E:WORSTEXPANDROOMLOTSOFUSEFULINEQUALITIES},
Defs.\,\ref{D:LOWNORMS} and~\ref{D:HIGHNORMS},
the estimates 
\eqref{E:KASNERMETRICESTIMATES},
\eqref{E:UPTOFOURDERIVATIVESOFGLINFINITYSOBOLEV},
and
\eqref{E:UPTOFOURDERIVATIVESOFLAPSELINFINITYSOBOLEV},
and the bootstrap assumptions, 
we deduce that
$\mbox{RHS~\eqref{E:FIRSTSTEPTOPORDERJUNKMETRICERRRORTERM3L2BOUND2}}
\lesssim
t^{1 - 8 \Worstexp - \Room - \Blowupexp \updelta}
\left\lbrace
	\MetricLownorm(t)
	+
	\MetricHighnorm(t) 
\right\rbrace
+
t^{- 5 \Worstexp - \Room - \Blowupexp \updelta}
\left\lbrace
	\MetricLownorm(t)
	+
	\MetricHighnorm(t) 
\right\rbrace
$.
In view of \eqref{E:WORSTEXPANDROOMLOTSOFUSEFULINEQUALITIES}, 
we see that the RHS of the previous expression is
$\lesssim \mbox{RHS~\eqref{E:TOPORDERDIFFERENTIATEDJUNKMETRICERRORTERML2ESTIMATE}}$
as desired.

\medskip

\noindent \underline{\textbf{Proof of \eqref{E:TOPORDERSECONDFUNDJUNKERRORTERML2BOUND}}}:
We stress that for this estimate, on RHS~\eqref{E:SECONDFUNDCOMMUTEDJUNKTERM},
we have $\Pos = 1$ and $|\vec{I}|=N$.

To bound the first sum on RHS~\eqref{E:SECONDFUNDCOMMUTEDJUNKTERM},
we first consider the case in which $|\vec{I}_2| = N$.
Then the terms under consideration are bounded in the norm
$\| \cdot \|_{L_g^2(\Sigma_t)}$
by 
$\lesssim 
t^{\Blowupexp}
\| n-1 \|_{L^{\infty}(\Sigma_t)}
\left\|
	\SecondFund
\right\|_{\dot{H}_g^N(\Sigma_t)}
$.
From Defs.\,\ref{D:LOWNORMS} and~\ref{D:HIGHNORMS}
and \eqref{E:WORSTEXPANDROOMLOTSOFUSEFULINEQUALITIES},
we see that the RHS of the previous expression is
$
\lesssim
t^{1 - 10 \Worstexp - \Room} \MetricHighnorm(t)
\lesssim
t^{\Room - 1} \MetricHighnorm(t)
$
as desired.
We now consider the remaining cases, 
in which $|\vec{I}_1| \leq N-1$ and $|\vec{I}_2| \leq N-1$.
First, using \eqref{E:L2PRODUCTBOUNDINERMSOFLINFINITYANDHMDOT}
with $l=m=1$
(since $\leftexp{(Junk;1;\vec{I})}{\mathfrak{K}}$ is a type $\binom{1}{1}$ tensorfield)
we bound 
(using that $|\vec{I}|=N$)
the products under consideration as follows:
\begin{align} \label{E:FIRSTSTEPTOPORDERSECONDFUNDNONBORDERLINEERRRORTERM1}
&
\mathop{\sum_{\vec{I}_1 + \vec{I}_2 = \vec{I}}}_{|\vec{I}_1|, |\vec{I}_2| \leq N-1}
	\left\|
		t^{\Blowupexp}
		\left\lbrace
			\partial_{\vec{I}_1} (n-1) 
		\right\rbrace	
		\partial_{\vec{I}_2} \SecondFund
	\right\|_{L_g^2(\Sigma_t)}
		\\
& \ \	
	\lesssim
	t^{\Blowupexp - 2 \Worstexp}
	\left\|
		n - 1
	\right\|_{\dot{W}^{1,\infty}(\Sigma_t)}
	\left\|
		\SecondFund
	\right\|_{\dot{H}_{Frame}^{N-1}(\Sigma_t)}
	+
	t^{\Blowupexp - 2 \Worstexp}
	\left\|
		\SecondFund
	\right\|_{\dot{W}_{Frame}^{1,\infty}(\Sigma_t)}
	\left\|
		n
	\right\|_{\dot{H}^{N-1}(\Sigma_t)}.
	\notag
\end{align}
From Defs.\,\ref{D:LOWNORMS} and~\ref{D:HIGHNORMS},
the estimate \eqref{E:UPTOFOURDERIVATIVESOFLAPSELINFINITYSOBOLEV},
and the bootstrap assumptions, we deduce that
$\mbox{RHS~\eqref{E:FIRSTSTEPTOPORDERSECONDFUNDNONBORDERLINEERRRORTERM1}}
\lesssim
t^{2 - 15 \Worstexp - 2 \Room - \Blowupexp \updelta}
\MetricHighnorm(t) 
+  
t^{-3 \Worstexp} 
\left\lbrace
	\MetricLownorm(t) + \MetricHighnorm(t)
\right\rbrace
$.
In view of \eqref{E:WORSTEXPANDROOMLOTSOFUSEFULINEQUALITIES}, 
we see that the RHS of the previous expression is
$\lesssim \mbox{RHS~\eqref{E:TOPORDERSECONDFUNDJUNKERRORTERML2BOUND}}$
as desired.

To bound the second sum on RHS~\eqref{E:SECONDFUNDCOMMUTEDJUNKTERM},
we first consider the cases in which $|\vec{I}_5| = N$
or $|\vec{I}_6| = N$. Using that $|g^{-1}|_g \lesssim 1$,
we deduce that 
the terms under consideration are bounded in the norm
$\| \cdot \|_{L_g^2(\Sigma_t)}$
by 
$\lesssim 
t^{\Blowupexp + 1}
\| n \|_{L^{\infty}(\Sigma_t)}
\| \partial g \|_{L_g^{\infty}(\Sigma_t)}
\left\|
	\partial g
\right\|_{\dot{H}_g^N(\Sigma_t)}
$.
Using \eqref{E:POINTWISENORMCOMPARISON}
(with $l=0$ and $m=3$)
to estimate 
$\| \partial g \|_{L_g^{\infty}(\Sigma_t)}$,
we see that the RHS of the previous expression is 
$\lesssim 
t^{\Blowupexp + 1 - 3 \Worstexp}
\| n \|_{L^{\infty}(\Sigma_t)}
\| g \|_{\dot{W}_{Frame}^{1,\infty}(\Sigma_t)}
\left\|
	\partial g
\right\|_{\dot{H}_g^N(\Sigma_t)}
$.
From 
\eqref{E:WORSTEXPANDROOMLOTSOFUSEFULINEQUALITIES},
Defs.\,\ref{D:LOWNORMS} and~\ref{D:HIGHNORMS},
the estimate \eqref{E:UPTOFOURDERIVATIVESOFGLINFINITYSOBOLEV},
and the bootstrap assumptions,
we see that the RHS of the previous expression is
$\lesssim
t^{- 5 \Worstexp - \Blowupexp \updelta}  \MetricHighnorm(t)
$.
Using \eqref{E:WORSTEXPANDROOMLOTSOFUSEFULINEQUALITIES}, 
we see that the RHS of the previous expression is
$\lesssim \mbox{RHS~\eqref{E:TOPORDERSECONDFUNDJUNKERRORTERML2BOUND}}$
as desired.
We now consider the remaining cases, 
in which $|\vec{I}_5| \leq N-1$ and $|\vec{I}_6| \leq N-1$.
Using \eqref{E:L2PRODUCTBOUNDINERMSOFLINFINITYANDHMDOT}
with $l=m=1$
(since $\leftexp{(Junk;1;\vec{I})}{\mathfrak{K}}$ is a type $\binom{1}{1}$ tensorfield)
and \eqref{E:BASICINTERPOLATION},
we deduce 
(using that $|\vec{I}|=N$) that
\begin{align} \label{E:FIRSTSTEPTOPORDERSECONDFUNDNONBORDERLINEERRRORTERM2}
&
\mathop{\sum_{\vec{I}_1 + \vec{I}_2 + \cdots + \vec{I}_6 = \vec{I}}}_{|\vec{I}_5|, |\vec{I}_6| \leq N-1}
	t^{\Blowupexp + 1} 
	\left\|
		(\partial_{\vec{I}_1} n)
		(\partial_{\vec{I}_2} g^{-1})
		(\partial_{\vec{I}_3} g^{-1})
		(\partial_{\vec{I}_4} g^{-1})
		(\partial \partial_{\vec{I}_5} g)
		\partial \partial_{\vec{I}_6} g
	\right\|_{L_g^2(\Sigma_t)}
		\\
& \ \	
	\lesssim
	t^{\Blowupexp + 1 - 2 \Worstexp} 
	\left\|
		n
	\right\|_{W^{1,\infty}(\Sigma_t)}
	\left\|
		g^{-1}
	\right\|_{W_{Frame}^{1,\infty}(\Sigma_t)}^3
	\left\|
		g
	\right\|_{W_{Frame}^{1,\infty}(\Sigma_t)}
	\left\|
		g
	\right\|_{\dot{H}_{Frame}^N(\Sigma_t)}
		\notag \\
& \ \
	+
	t^{\Blowupexp + 1 - 2 \Worstexp} 
	\left\|
		n
	\right\|_{W^{1,\infty}(\Sigma_t)}
	\left\|
		g^{-1}
	\right\|_{W_{Frame}^{1,\infty}(\Sigma_t)}^2
	\left\|
		g
	\right\|_{W_{Frame}^{1,\infty}(\Sigma_t)}^2
	\left\|
		g^{-1}
	\right\|_{\dot{H}_{Frame}^N(\Sigma_t)}
		\notag \\
& \ \
	+
	t^{\Blowupexp + 1 - 2 \Worstexp} 
	\left\|
		g
	\right\|_{W_{Frame}^{1,\infty}(\Sigma_t)}^2
	\left\|
		g^{-1}
	\right\|_{W_{Frame}^{1,\infty}(\Sigma_t)}^3
	\left\|
		n
	\right\|_{\dot{H}^N(\Sigma_t)}
		\notag
			\\
& \ \
	+
	t^{\Blowupexp + 1 - 2 \Worstexp} 
	\left\|
		n
	\right\|_{W^{1,\infty}(\Sigma_t)}
	\left\|
		g^{-1}
	\right\|_{W_{Frame}^{1,\infty}(\Sigma_t)}^3
	\left\|
		g
	\right\|_{W_{Frame}^{1,\infty}(\Sigma_t)}^2.
	\notag
\end{align}
From 
\eqref{E:WORSTEXPANDROOMLOTSOFUSEFULINEQUALITIES},
Defs.\,\ref{D:LOWNORMS} and~\ref{D:HIGHNORMS}, 
the estimates 
\eqref{E:KASNERMETRICESTIMATES},
\eqref{E:UPTOFOURDERIVATIVESOFGLINFINITYSOBOLEV},
\eqref{E:UPTOFOURDERIVATIVESOFGINVERSELINFINITYSOBOLEV},
and
\eqref{E:UPTOFOURDERIVATIVESOFLAPSELINFINITYSOBOLEV},
the elliptic estimate \eqref{E:LAPSETOPORDERELLIPTIC},
and the bootstrap assumptions, we deduce that
$\mbox{RHS~\eqref{E:FIRSTSTEPTOPORDERSECONDFUNDNONBORDERLINEERRRORTERM2}}
\lesssim
t^{1 - 12 \Worstexp - \Room - \Blowupexp \updelta} 
\left\lbrace
	\MetricLownorm(t) 
	+
	\MetricHighnorm(t) 
\right\rbrace
$.
In view of \eqref{E:WORSTEXPANDROOMLOTSOFUSEFULINEQUALITIES}, 
we see that the RHS of the previous expression is
$\lesssim \mbox{RHS~\eqref{E:TOPORDERSECONDFUNDJUNKERRORTERML2BOUND}}$
as desired.

To bound the third sum on RHS~\eqref{E:SECONDFUNDCOMMUTEDJUNKTERM},
we first consider the case in which
$|\vec{I}_4| = N-1$.
Using that $|g^{-1}|_g \lesssim 1$
and $g$-Cauchy--Schwarz,
we deduce that 
the terms under consideration are bounded in the norm
$\| \cdot \|_{L_g^2(\Sigma_t)}$
by 
$
\lesssim 
t^{\Blowupexp + 1}
\| n \|_{\dot{W}^{1,\infty}(\Sigma_t)}
\left\|
	\partial^2 g
\right\|_{\dot{H}_g^{N-1}(\Sigma_t)}
+
t^{\Blowupexp + 1}
\| n \|_{L^{\infty}(\Sigma_t)}
\| \partial g \|_{L_g^{\infty}(\Sigma_t)}
\left\|
	\partial^2 g
\right\|_{\dot{H}_g^{N-1}(\Sigma_t)}
$.
Next, 
using \eqref{E:POINTWISENORMCOMPARISON} 
(with $l=0$ and $m=3$)
to estimate $\| \partial g \|_{L_g^{\infty}(\Sigma_t)}$,
and using the definition of the norms
$\| \cdot \|_{L_{Frame}^{\infty}(\Sigma_t)}$,
$\| \cdot \|_{L_g^{\infty}(\Sigma_t)}$,
and
$\| \cdot \|_{\dot{H}_g^M(\Sigma_t)}$,
we see that the RHS of the previous expression is 
$\lesssim 
t^{\Blowupexp + 1}
\| n \|_{\dot{W}^{1,\infty}(\Sigma_t)}
\| g^{-1} \|_{L_{Frame}^{\infty}(\Sigma_t)}^{1/2}
\left\|
	\partial g
\right\|_{\dot{H}_g^N(\Sigma_t)}
+
t^{\Blowupexp + 1 - 3 \Worstexp}
\| n \|_{L^{\infty}(\Sigma_t)}
\| g^{-1} \|_{L_{Frame}^{\infty}(\Sigma_t)}^{1/2}
\| g \|_{\dot{W}_{Frame}^{1,\infty}(\Sigma_t)}
\left\|
	\partial g
\right\|_{\dot{H}_g^N(\Sigma_t)}
$.
From 
\eqref{E:WORSTEXPANDROOMLOTSOFUSEFULINEQUALITIES},
Defs.\,\ref{D:LOWNORMS} and~\ref{D:HIGHNORMS},
the estimates 
\eqref{E:KASNERMETRICESTIMATES},
\eqref{E:UPTOFOURDERIVATIVESOFGLINFINITYSOBOLEV}
and \eqref{E:UPTOFOURDERIVATIVESOFLAPSELINFINITYSOBOLEV},
and the bootstrap assumptions,
we see that the RHS of the previous expression is
$\lesssim
t^{2 - 11 \Worstexp - \Room - \Blowupexp \updelta} \MetricHighnorm(t)
+
t^{-6 \Worstexp - \Blowupexp \updelta} \MetricHighnorm(t)
$.
In view of \eqref{E:WORSTEXPANDROOMLOTSOFUSEFULINEQUALITIES}, 
we see that the RHS of the previous expression is
$\lesssim \mbox{RHS~\eqref{E:TOPORDERSECONDFUNDJUNKERRORTERML2BOUND}}$
as desired.
We now consider the remaining cases, 
in which $|\vec{I}_4| \leq N-2$.
Using \eqref{E:L2PRODUCTBOUNDINERMSOFLINFINITYANDHMDOT}
with $l=m=1$ 
(since $\leftexp{(Junk;1;\vec{I})}{\mathfrak{K}}$ is a type $\binom{1}{1}$ tensorfield)
and \eqref{E:BASICINTERPOLATION},
we deduce 
(using that $|\vec{I}|=N$) 
that
\begin{align} \label{E:FIRSTSTEPTOPORDERSECONDFUNDNONBORDERLINEERRRORTERM3}
&
\mathop{\sum_{\vec{I}_1 + \vec{I}_2 + \vec{I}_3 + \vec{I}_4 = \vec{I}}}_{|\vec{I}_4| \leq N-2}
	t^{\Blowupexp + 1} 
	\left\|
		(\partial_{\vec{I}_1} n)
		(\partial_{\vec{I}_2} g^{-1})
		(\partial_{\vec{I}_3} g^{-1})
		\partial^2 \partial_{\vec{I}_4} g
	\right\|_{L_g^2(\Sigma_t)}
		\\
& \ \	
	\lesssim
	t^{\Blowupexp + 1 - 2 \Worstexp} 
	\left\|
		n
	\right\|_{W^{2,\infty}(\Sigma_t)}
	\left\|
		g^{-1}
	\right\|_{W_{Frame}^{2,\infty}(\Sigma_t)}^2
	\left\|
		g
	\right\|_{\dot{H}_{Frame}^N(\Sigma_t)}
		\notag \\
& \ \
	+
	t^{\Blowupexp + 1 - 2 \Worstexp} 
	\left\|
		n
	\right\|_{W^{2,\infty}(\Sigma_t)}
	\left\|
		g^{-1}
	\right\|_{W_{Frame}^{2,\infty}(\Sigma_t)}
	\left\|
		g
	\right\|_{W_{Frame}^{2,\infty}(\Sigma_t)}
	\left\|
		g^{-1}
	\right\|_{\dot{H}_{Frame}^N(\Sigma_t)}
		\notag \\
& \ \
	+
	t^{\Blowupexp + 1 - 2 \Worstexp} 
	\left\|
		g^{-1}
	\right\|_{W_{Frame}^{2,\infty}(\Sigma_t)}^2
	\left\|
		g
	\right\|_{W_{Frame}^{2,\infty}(\Sigma_t)}
	\left\|
		n
	\right\|_{\dot{H}^N(\Sigma_t)}.
	\notag
\end{align}
From 
\eqref{E:WORSTEXPANDROOMLOTSOFUSEFULINEQUALITIES},
Defs.\,\ref{D:LOWNORMS} and~\ref{D:HIGHNORMS}, 
\eqref{E:WORSTEXPANDROOMLOTSOFUSEFULINEQUALITIES},
the estimates 
\eqref{E:KASNERMETRICESTIMATES},
\eqref{E:UPTOFOURDERIVATIVESOFGLINFINITYSOBOLEV},
\eqref{E:UPTOFOURDERIVATIVESOFGINVERSELINFINITYSOBOLEV},
and
\eqref{E:UPTOFOURDERIVATIVESOFLAPSELINFINITYSOBOLEV},
the elliptic estimate \eqref{E:LAPSETOPORDERELLIPTIC},
and the bootstrap assumptions, we deduce that
$\mbox{RHS~\eqref{E:FIRSTSTEPTOPORDERSECONDFUNDNONBORDERLINEERRRORTERM3}}
\lesssim
t^{1 - 8 \Worstexp - \Room - \Blowupexp \updelta} 
\left\lbrace
	\MetricLownorm(t)
	+
	\MetricHighnorm(t) 
\right\rbrace
$.
In view of \eqref{E:WORSTEXPANDROOMLOTSOFUSEFULINEQUALITIES}, 
we see that the RHS of the previous expression is
$\lesssim \mbox{RHS~\eqref{E:TOPORDERSECONDFUNDJUNKERRORTERML2BOUND}}$
as desired.

To bound the last sum on RHS~\eqref{E:SECONDFUNDCOMMUTEDJUNKTERM},
we first consider the case in which
$|\vec{I}_3| = N$.
Using that $|g^{-1}|_g \lesssim 1$
and $g$-Cauchy--Schwarz,
we deduce that 
the terms under consideration are bounded in the norm
$\| \cdot \|_{L_g^2(\Sigma_t)}$
by 
$
\lesssim 
t^{\Blowupexp + 1}
\| \partial n \|_{L_g^{\infty}(\Sigma_t)}
\left\|
	\partial g
\right\|_{\dot{H}_g^N(\Sigma_t)}
$.
Next, using \eqref{E:POINTWISENORMCOMPARISON} 
(with $l=0$ and $m=1$)
to estimate $\| \partial n \|_{L_g^{\infty}(\Sigma_t)}$,
we see that the RHS of the previous expression is 
$\lesssim 
t^{\Blowupexp + 1 - \Worstexp}
\| n \|_{\dot{W}^{1,\infty}(\Sigma_t)}
\left\|
	\partial g
\right\|_{\dot{H}_g^N(\Sigma_t)}
$.
From Defs.\,\ref{D:LOWNORMS} and~\ref{D:HIGHNORMS},
the estimate \eqref{E:UPTOFOURDERIVATIVESOFLAPSELINFINITYSOBOLEV},
and the bootstrap assumptions,
we see that the RHS of the previous expression is
$
\lesssim
t^{2 - 11 \Worstexp - \Room - \Blowupexp \updelta}  \MetricHighnorm(t)
$,
which, in view of \eqref{E:WORSTEXPANDROOMLOTSOFUSEFULINEQUALITIES},
is $\lesssim \mbox{RHS~\eqref{E:TOPORDERSECONDFUNDJUNKERRORTERML2BOUND}}$
as desired.
We now consider the case in which
$|\vec{I}_4| = N$.
Using that $|g^{-1}|_g \lesssim 1$
and $g$-Cauchy--Schwarz,
we deduce that 
the terms under consideration are bounded in the norm
$\| \cdot \|_{L_g^2(\Sigma_t)}$
by 
$
\lesssim 
t^{\Blowupexp + 1}
\| \partial g \|_{L_g^{\infty}(\Sigma_t)}
\left\|
	\partial n
\right\|_{\dot{H}_g^N(\Sigma_t)}
$.
Next, using \eqref{E:POINTWISENORMCOMPARISON} 
(with $l=0$ and $m=3$)
to estimate $\| \partial g \|_{L_g^{\infty}(\Sigma_t)}$,
we see that the RHS of the previous expression is 
$\lesssim 
t^{\Blowupexp + 1 - 3 \Worstexp}
\| g \|_{\dot{W}_{Frame}^{1,\infty}(\Sigma_t)}
\left\|
	\partial n
\right\|_{\dot{H}_g^N(\Sigma_t)}
$.
From Defs.\,\ref{D:LOWNORMS} and~\ref{D:HIGHNORMS}
and the estimate \eqref{E:UPTOFOURDERIVATIVESOFGLINFINITYSOBOLEV},
we see that the RHS of the previous expression is
$\lesssim
t^{- 5 \Worstexp - \Blowupexp \updelta} 
\MetricLownorm(t)
\LapseHighnorm(t)
\lesssim
t^{- 5 \Worstexp - \Blowupexp \updelta} 
\MetricLownorm(t)
$.
In view of \eqref{E:WORSTEXPANDROOMLOTSOFUSEFULINEQUALITIES}, 
we see that the RHS of the previous expression is
$\lesssim \mbox{RHS~\eqref{E:TOPORDERSECONDFUNDJUNKERRORTERML2BOUND}}$
as desired.
We now consider the remaining cases, 
in which $|\vec{I}_3| \leq N-1$
and 
$|\vec{I}_4| \leq N-1$.
Using \eqref{E:L2PRODUCTBOUNDINERMSOFLINFINITYANDHMDOT}
with $l=m=1$ 
(since $\leftexp{(Junk;1;\vec{I})}{\mathfrak{K}}$ is a type $\binom{1}{1}$ tensorfield)
and \eqref{E:BASICINTERPOLATION},
we deduce 
(using that $|\vec{I}|=N$)
that
\begin{align} \label{E:FIRSTSTEPTOPORDERSECONDFUNDNONBORDERLINEERRRORTERM4}
&
\mathop{\sum_{\vec{I}_1 + \vec{I}_2 + \vec{I}_3 + \vec{I}_4 = \vec{I}}}_{|\vec{I}_3|, |\vec{I}_4| \leq N-1}
	t^{\Blowupexp + 1} 
	\left\|
		(\partial_{\vec{I}_1} g^{-1})
		(\partial_{\vec{I}_2} g^{-1})
		(\partial \partial_{\vec{I}_3} g)
		\partial \partial_{\vec{I}_4} n
	\right\|_{L_g^2(\Sigma_t)}
		\\
& \ \	
	\lesssim
	t^{\Blowupexp + 1 - 2 \Worstexp} 
	\left\|
		g^{-1}
	\right\|_{W_{Frame}^{1,\infty}(\Sigma_t)}^2
	\left\|
		g - \KasnerMetric
	\right\|_{W_{Frame}^{2,\infty}(\Sigma_t)}
	\left\|
		n
	\right\|_{\dot{H}^N(\Sigma_t)}
		\notag \\
& \ \
	+
	t^{\Blowupexp + 1 - 2 \Worstexp} 
	\left\|
		n
	\right\|_{W_{Frame}^{2,\infty}(\Sigma_t)}
	\left\|
		g^{-1}
	\right\|_{W_{Frame}^{1,\infty}(\Sigma_t)}^2
	\left\|
		g
	\right\|_{\dot{H}_{Frame}^N(\Sigma_t)}
		\notag \\
& \ \
	+
	t^{\Blowupexp + 1 - 2 \Worstexp} 
	\left\|
		n
	\right\|_{W_{Frame}^{2,\infty}(\Sigma_t)}
	\left\|
		g - \KasnerMetric
	\right\|_{W_{Frame}^{2,\infty}(\Sigma_t)}
	\left\|
		g^{-1}
	\right\|_{W_{Frame}^{1,\infty}(\Sigma_t)}
	\left\|
		g^{-1}
	\right\|_{\dot{H}_{Frame}^N(\Sigma_t)}
	\notag
		\\
	& \ \
	+
	t^{\Blowupexp + 1 - 2 \Worstexp} 
	\left\|
		n
	\right\|_{W_{Frame}^{2,\infty}(\Sigma_t)}
	\left\|
		g^{-1}
	\right\|_{W_{Frame}^{1,\infty}(\Sigma_t)}^2
	\left\|
		g - \KasnerMetric
	\right\|_{W_{Frame}^{2,\infty}(\Sigma_t)}.
	\notag
\end{align}
From 
\eqref{E:WORSTEXPANDROOMLOTSOFUSEFULINEQUALITIES},
Defs.\,\ref{D:LOWNORMS} and~\ref{D:HIGHNORMS}, 
the estimates \eqref{E:KASNERMETRICESTIMATES},
\eqref{E:UPTOFOURDERIVATIVESOFGLINFINITYSOBOLEV},
\eqref{E:UPTOFOURDERIVATIVESOFGINVERSELINFINITYSOBOLEV},
and \eqref{E:UPTOFOURDERIVATIVESOFLAPSELINFINITYSOBOLEV},
and the bootstrap assumptions, we deduce that
$\mbox{RHS~\eqref{E:FIRSTSTEPTOPORDERSECONDFUNDNONBORDERLINEERRRORTERM4}}
\lesssim
t^{1 - 8 \Worstexp - \Room - \Blowupexp \updelta} 
\left\lbrace
	\MetricLownorm(t)
	+
	\MetricHighnorm(t) 
\right\rbrace
$.
In view of \eqref{E:WORSTEXPANDROOMLOTSOFUSEFULINEQUALITIES}, 
we see that the RHS of the previous expression is
$\lesssim \mbox{RHS~\eqref{E:TOPORDERSECONDFUNDJUNKERRORTERML2BOUND}}$
as desired. We have therefore proved \eqref{E:TOPORDERSECONDFUNDJUNKERRORTERML2BOUND},
which completes the proof of the lemma.

\end{proof}

\subsection{Proof of Prop.\,\ref{P:MAINTOPORDERMETRICENERGYESTIMATE}}
\label{SS:PROOFOFMAINTOPORDERMETRICENERGYESTIMATE}
In this subsection, we prove Prop.\,\ref{P:MAINTOPORDERMETRICENERGYESTIMATE}.
Throughout this proof, we will assume that $\Blowupexp \updelta$ is sufficiently small
(and in particular that $\Blowupexp \updelta < \Room$);
in view of the discussion in Subsect.\,\ref{SS:SOBOLEVEMBEDDING}, we see that
at fixed $\Blowupexp$, this can be achieved by choosing $N$ to be sufficiently large.

Let $\vec{I}$ be a spatial multi-index with $|\vec{I}|=N$
and let ${\leftexp{(\vec{I})}{\mathbf{J}}}$ be the energy current from Def.~\ref{D:ENERGYCURRENT}.
Applying the divergence theorem on the region $[t,1] \times \mathbb{T}^{\mydim}$,
and considering equation \eqref{E:METRICCURRENT0},
we deduce that for $t \in (\TBoot,1]$, we have
	\begin{align} \label{E:DIVTHMAPPLIEDTOCURRENT}
		&
		\int_{\Sigma_t}
			\left\lbrace
			\left|
				t^{\Blowupexp + 1} \partial_{\vec{I}} \SecondFund 
			\right|_g^2
			 + 
			\frac{1}{4} 
			\left|
				t^{\Blowupexp + 1} \partial \partial_{\vec{I}} g 
			\right|_g^2
			\right\rbrace
		\, dx
			\\
		& =
		\int_{\Sigma_1}
			\left\lbrace
			\left|
				\partial_{\vec{I}} \SecondFund 
			\right|_g^2
			 + 
			\frac{1}{4} 
			\left|
				\partial \partial_{\vec{I}} g 
			\right|_g^2
			\right\rbrace
		\, dx
		-
		\int_{s=t}^1
			\int_{\Sigma_s}
				\partial_{\alpha} {\leftexp{(\vec{I})}{\mathbf{J}^{\alpha}}}
			\, dx
		\, ds.
		\notag
	\end{align}
	We now use equation \eqref{E:DIVIDFORENERGYCURRENT}
	to substitute for the last integral on
	RHS~\eqref{E:DIVTHMAPPLIEDTOCURRENT},
	thereby obtaining
	\begin{align} \label{E:POSTSUBSTITUTIONDIVTHMAPPLIEDTOCURRENT}
		&
		\int_{\Sigma_t}
			\left\lbrace
			\left|
				t^{\Blowupexp + 1} \partial_{\vec{I}} \SecondFund 
			\right|_g^2
			 + 
			\frac{1}{4} 
			\left|
				t^{\Blowupexp + 1} \partial \partial_{\vec{I}} g 
			\right|_g^2
			\right\rbrace
		\, dx
			\\
		& =
		\int_{\Sigma_1}
			\left\lbrace
			\left|
				\partial_{\vec{I}} \SecondFund 
			\right|_g^2
			 + 
			\frac{1}{4} 
			\left|
				\partial \partial_{\vec{I}} g 
			\right|_g^2
			\right\rbrace
		\, dx
			\notag \\
	& \ \
		-
		\int_{s=t}^1
			s^{-1}
				\int_{\Sigma_s}
					\left\lbrace
						2 \Blowupexp
						\left|
							s^{\Blowupexp + 1} \partial_{\vec{I}} \SecondFund 
						\right|_g^2
						+ 
						\frac{\Blowupexp + 1}{2} 
						\left|
							s^{\Blowupexp + 1} \partial \partial_{\vec{I}} g 
						\right|_g^2
				\right\rbrace
			\, dx
		\, ds
		\notag
			\\
		& \ \
			-
			\int_{s=t}^1
			\int_{\Sigma_s}
				\leftexp{(Border;\vec{I})}{\mathfrak{J}}
			\, dx
		\, ds
		-
			\int_{s=t}^1
			\int_{\Sigma_s}
				\leftexp{(Junk;\vec{I})}{\mathfrak{J}}
			\, dx
		\, ds.
		\notag
	\end{align}
	Next, we use  
	the pointwise estimates 
	\eqref{E:POINTWISEBOUNDMETRICCURRENTBORDERTERMS}
	and
	\eqref{E:POINTWISEBOUNDMETRICCURRENTJUNKTERMS}
	and the elliptic estimate
	\eqref{E:LAPSETOPORDERELLIPTIC}
	to deduce that the
	last two integrals on RHS~\eqref{E:POSTSUBSTITUTIONDIVTHMAPPLIEDTOCURRENT}
	are bounded in magnitude by
	\begin{align} \label{E:BOUNDFORERRORTERMSINPOSTSUBSTITUTIONDIVTHMAPPLIEDTOCURRENT}
		& \leq
		C_*
		\int_{s=t}^1
				s^{-1}
				\left\lbrace
					\left\|
						s^{\Blowupexp + 1} \partial_{\vec{I}} \SecondFund 
					\right\|_{L_g^2(\Sigma_s)}^2
					+
					\frac{1}{4}
					\left\|
						s^{\Blowupexp + 1} \partial \partial_{\vec{I}} g 
					\right\|_{L_g^2(\Sigma_s)}^2
				\right\rbrace
			\, ds
				\\
		& \ \
			+
			C_*
			\int_{s=t}^1
				s
				\left\|
					\leftexp{(Border;\vec{I})}{\mathfrak{K}}
				\right\|_{L_g^2(\Sigma_s)}^2
			\, ds
			+
			C_*
			\int_{s=t}^1
				s
				\left\|
					\leftexp{(Border;\vec{I})}{\mathfrak{H}}
				\right\|_{L_g^2(\Sigma_s)}^2
			\, ds
				\notag \\
	& \ \
			+
			C_*
			\int_{s=t}^1
				s
				\left\|
					\leftexp{(Border;\vec{I})}{\mathfrak{M}}
				\right\|_{L_g^2(\Sigma_s)}^2
			\, ds
			+
			C_*
			\int_{s=t}^1
				s
				\left\|
					\leftexp{(Border;\vec{I})}{\widetilde{\mathfrak{M}}}
				\right\|_{L_g^2(\Sigma_s)}^2
			\, ds
				\notag 
					\\
		& \ \
			+
			C
			\int_{s=t}^1
				s^{1 - \Room}
				\left\|
					\leftexp{(Junk;\vec{I})}{\mathfrak{K}}
				\right\|_{L_g^2(\Sigma_s)}^2
			\, ds
			+
			C
			\int_{s=t}^1
				s^{1 - \Room}
				\left\|
					\leftexp{(Junk;\vec{I})}{\mathfrak{H}}
				\right\|_{L_g^2(\Sigma_s)}^2
			\, ds
				\notag \\
	& \ \
			+
			C
			\int_{s=t}^1
				s^{1 - \Room}
				\left\|
					\leftexp{(Junk;\vec{I})}{\mathfrak{M}}
				\right\|_{L_g^2(\Sigma_s)}^2
			\, ds
			+
			C
			\int_{s=t}^1
				s^{1 - \Room}
				\left\|
					\leftexp{(Junk;\vec{I})}{\widetilde{\mathfrak{M}}}
				\right\|_{L_g^2(\Sigma_s)}^2
			\, ds
				\notag 
					\\
		& \ \
			+
			C
			\int_{s=t}^1
				s^{\Room - 1}
				\left\lbrace
					\MetricLownorm^2(s) 
					+ 
					\MetricHighnorm^2(s)
				\right\rbrace
			\, ds.
			\notag
	\end{align}
	Next, we use the estimates 
	\eqref{E:TOPORDERDIFFERENTIATEDBORDERMOMENTUMCONSTRAINTINDEXDOWNERRORTERML2ESTIMATE}-\eqref{E:TOPORDERSECONDFUNDJUNKERRORTERML2BOUND}
	to bound the terms
	on lines two to five of
	RHS~\eqref{E:BOUNDFORERRORTERMSINPOSTSUBSTITUTIONDIVTHMAPPLIEDTOCURRENT}.
	Also noting that the term
	$
	\int_{\Sigma_1}
			\left\lbrace
			\left|
				\partial_{\vec{I}} \SecondFund 
			\right|_g^2
			 + 
			\frac{1}{4} 
			\left|
				\partial \partial_{\vec{I}} g 
			\right|_g^2
			\right\rbrace
		\, dx
	$
	on RHS~\eqref{E:POSTSUBSTITUTIONDIVTHMAPPLIEDTOCURRENT}
	is $\leq C \MetricHighnorm^2(1)$,
	we arrive at the desired bound \eqref{E:MAINTOPORDERMETRICENERGYESTIMATE}.
	
	\hfill $\qed$

\section{Estimates for the Near-Top-Order Derivatives of \texorpdfstring{$g$ and $\SecondFund$}{the First and Second Fundamental Forms}}
\label{S:ESTIMATESFORJUSTBELOWTOPORDERDERIVATIVESOFMETRICANDSECONDFUNDAMENTALFORM}
Our primary goal in this section is to prove the following proposition,
which provides the main integral inequalities that we will use
to control the near-top-order derivatives of $g$ and $\SecondFund$;
recall that, as we explained near the end of Subsect.\,\ref{SS:MAINIDEASINPROOF}, 
for technical reasons,
we need these estimates to close our bootstrap argument.
The proof of the proposition is located in Subsect.\,\ref{SS:INTEGRALINEQUALITIESFORJUSTBELOWTOPDERIVATIVESOFMETRICANDSECONDFUNDAMENTALFORM}.
In Subsects.\,\ref{SS:EQUATIONSVERIFIEDBYMETRICOMMMUTED}-\ref{SS:CONTROLOFERRORTERMSINNEARTOPORDERMETRICESTIMATES},
we provide the identities and estimates that we will use when proving the proposition.

\begin{proposition}[\textbf{Integral inequalities for the near-top-order derivatives of $g$ and $\SecondFund$}]
	\label{P:INTEGRALINEQUALITIESFORJUSTBELOWTOPDERIVATIVESOFMETRICANDSECONDFUND}
		Under the bootstrap assumptions \eqref{E:BOOTSTRAPASSUMPTIONS},
		there exists a universal constant $C_* > 0$ \underline{independent of $N$ and $\Blowupexp$}
	such that if $N$ is sufficiently large in a manner that depends on $\Blowupexp$
	and if $\varepsilon$ is sufficiently small, 
	then the following integral inequalities hold for $t \in (\TBoot,1]$
	(where, as we described in Subsect.\,\ref{SSS:CONSTANTS}, constants ``$C$'' are allowed to depend on $N$ and other quantities):
	\begin{subequations}
	\begin{align} \label{E:JUSTBELOWTOPORDERMETRICENERGYINTEGRALINEQUALITY}
		\left\|	
			t^{\Blowupexp + 2 \Worstexp + \Room} g
		\right\|_{\dot{H}_{Frame}^N(\Sigma_t)}^2
		& \leq
			C \MetricHighnorm^2(1)
			-
			2 \Blowupexp
			\int_{s=t}^1
				s^{-1}
				\left\|	
					s^{\Blowupexp + 2 \Worstexp + \Room}  g
				\right\|_{\dot{H}_{Frame}^N(\Sigma_s)}^2
			\, ds
				\\
		& \ \
		+
		C
		\int_{s=t}^1
			s^{\Room - 1}
			\left\lbrace
				\MetricLownorm^2(s) + \MetricHighnorm^2(s)
			\right\rbrace
		\, ds,
		\notag
			\\
		\left\|	
			t^{\Blowupexp + 2 \Worstexp + \Room} g^{-1}
		\right\|_{\dot{H}_{Frame}^N(\Sigma_t)}^2
		& \leq
			C \MetricHighnorm^2(1)
			-
			2 \Blowupexp
			\int_{s=t}^1
				s^{-1}
				\left\|	
					s^{\Blowupexp + 2 \Worstexp + \Room} g^{-1}
				\right\|_{\dot{H}_{Frame}^N(\Sigma_s)}^2
			\, ds
				\label{E:JUSTBELOWTOPORDERINVERSEMETRICENERGYINTEGRALINEQUALITY} \\
		& \ \
		+
		C
		\int_{s=t}^1
			s^{\Room - 1}
			\left\lbrace
				\MetricLownorm^2(s) + \MetricHighnorm^2(s)
			\right\rbrace
		\, ds,
		\notag
		\end{align}
		\end{subequations}
		
	\begin{subequations}
	\begin{align} \label{E:GNORMJUSTBELOWTOPORDERMETRICENERGYINTEGRALINEQUALITY}
		\left\|	
			t^{\Blowupexp + \Worstexp + \Room} g
		\right\|_{\dot{H}_g^N(\Sigma_t)}^2
		& \leq
			C \MetricHighnorm^2(1)
			-
			2 \Blowupexp
			\int_{s=t}^1
				s^{-1}
				\left\|	
					s^{\Blowupexp + \Worstexp + \Room}  g
				\right\|_{\dot{H}_g^N(\Sigma_s)}^2
			\, ds
				\\
		& \ \
		+
		C
		\int_{s=t}^1
			s^{\Room - 1}
			\left\lbrace
				\MetricLownorm^2(s) + \MetricHighnorm^2(s)
			\right\rbrace
		\, ds,
		\notag
			\\
		\left\|	
			t^{\Blowupexp + \Worstexp + \Room} g^{-1}
		\right\|_{\dot{H}_g^N(\Sigma_t)}^2
		& \leq
			C \MetricHighnorm^2(1)
			-
			2 \Blowupexp
			\int_{s=t}^1
				s^{-1}
				\left\|	
					s^{\Blowupexp + \Worstexp + \Room} g^{-1}
				\right\|_{\dot{H}_g^N(\Sigma_s)}^2
			\, ds
				\label{E:GNORMJUSTBELOWTOPORDERINVERSEMETRICENERGYINTEGRALINEQUALITY} \\
		& \ \
		+
		C
		\int_{s=t}^1
			s^{\Room - 1}
			\left\lbrace
				\MetricLownorm^2(s) + \MetricHighnorm^2(s)
			\right\rbrace
		\, ds,
		\notag
		\end{align}
		\end{subequations}
	
		\begin{align} \label{E:GNORMJUSTBELOWTOPORDERWITHGRADIENTMETRICENERGYINTEGRALINEQUALITY}
		\left\|	
			t^{\Blowupexp + 2 \Worstexp + \Room} \partial g
		\right\|_{\dot{H}_g^{N-1}(\Sigma_t)}^2
		& \leq
			C \MetricHighnorm^2(1)
			-
			\left\lbrace
				2 \Blowupexp
				- 
				C_*
			\right\rbrace
			\int_{s=t}^1
				s^{-1}
				\left\|	
					s^{\Blowupexp + 2 \Worstexp + \Room} \partial g
				\right\|_{\dot{H}_g^{N-1}(\Sigma_s)}^2
			\, ds
				\\
		& \ \
		+
		C
		\int_{s=t}^1
			s^{\Room - 1}
			\left\lbrace
				\MetricLownorm^2(s) + \MetricHighnorm^2(s)
			\right\rbrace
		\, ds,
		\notag
\end{align}

	\begin{subequations}
	\begin{align} \label{E:TWOBELOWTOPORDERMETRICENERGYINTEGRALINEQUALITY}
		\left\|	
			t^{\Blowupexp + 5 \Worstexp + 3 \Room - 1} g
		\right\|_{\dot{H}_{Frame}^{N-1}(\Sigma_t)}^2
		& \leq
			C \MetricHighnorm^2(1)
			-
			\left\lbrace
				2 \Blowupexp
				-
				C_*
			\right\rbrace
			\int_{s=t}^1
				s^{-1}
				\left\|	
					s^{\Blowupexp + 5 \Worstexp + 3 \Room - 1} g
			\right\|_{\dot{H}_{Frame}^{N-1}(\Sigma_s)}^2
			\, ds
				\\
		& \ \
		+
		C
		\int_{s=t}^1
			s^{\Room - 1}
			\left\lbrace
				\MetricLownorm^2(s) + \MetricHighnorm^2(s)
			\right\rbrace
		\, ds,
		\notag
			\\
		\left\|	
			t^{\Blowupexp + 5 \Worstexp + 3 \Room - 1} g^{-1}
		\right\|_{\dot{H}_{Frame}^{N-1}(\Sigma_t)}^2
		& \leq
			C \MetricHighnorm^2(1)
			-
			\left\lbrace
				2 \Blowupexp
				-
				C_*
			\right\rbrace
			\int_{s=t}^1
				s^{-1}
				\left\|	
					s^{\Blowupexp + 5 \Worstexp + 3 \Room - 1} g^{-1}
				\right\|_{\dot{H}_{Frame}^{N-1}(\Sigma_s)}^2
			\, ds
				\label{E:TWOBELOWTOPORDERINVERSEMETRICENERGYINTEGRALINEQUALITY} \\
		& \ \
		+
		C
		\int_{s=t}^1
			s^{\Room - 1}
			\left\lbrace
				\MetricLownorm^2(s) + \MetricHighnorm^2(s)
			\right\rbrace
		\, ds,
		\notag
		\end{align}
		\end{subequations}
		
		\begin{subequations}
		\begin{align}
		\left\|	
			t^{\Blowupexp + 3 \Worstexp + \Room} \SecondFund
		\right\|_{\dot{H}_{Frame}^{N-1}(\Sigma_t)}^2
		& \leq
			C \MetricHighnorm^2(1)
			-
			\left\lbrace
				2 \Blowupexp 
				- 
				C_*
			\right\rbrace
			\int_{s=t}^1
				s^{-1}
				\left\|	
					s^{\Blowupexp + 3 \Worstexp + \Room} \SecondFund
				\right\|_{\dot{H}_{Frame}^{N-1}(\Sigma_s)}^2
			\, ds
				\label{E:JUSTBELOWTOPORDERSECONDFUNDINTEGRALINEQUALITY} \\
		& \ \
			+
			\int_{s=t}^1
			s^{\Room - 1}
			\left\lbrace
				\MetricLownorm^2(s) + \MetricHighnorm^2(s)
			\right\rbrace
		\, ds,
			\notag
	\end{align}
	
	\begin{align}
		\left\|	
			t^{\Blowupexp + 3 \Worstexp + \Room} \SecondFund
		\right\|_{\dot{H}_g^{N-1}(\Sigma_t)}^2
		& \leq
			C \MetricHighnorm^2(1)
			-
			\left\lbrace
				2 \Blowupexp 
				- 
				C_*
			\right\rbrace
			\int_{s=t}^1
				s^{-1}
				\left\|	
					s^{\Blowupexp + 3 \Worstexp + \Room} \SecondFund
				\right\|_{\dot{H}_g^{N-1}(\Sigma_s)}^2
			\, ds
				\label{E:JUSTBELOWTOPORDERGNORMSECONDFUNDINTEGRALINEQUALITY} \\
		& \ \
			+
			\int_{s=t}^1
			s^{\Room - 1}
			\left\lbrace
				\MetricLownorm^2(s) + \MetricHighnorm^2(s)
			\right\rbrace
		\, ds.
			\notag
	\end{align}
	\end{subequations}
	
\end{proposition}

\subsection{The equations}
\label{SS:EQUATIONSVERIFIEDBYMETRICOMMMUTED}
In this subsection, we derive the evolution equations verified by the
near-top-order derivatives of $g$.

\begin{lemma}[\textbf{The equations}]
	\label{L:EQUATIONSFORMETRICITSELFCOMMUTED}
	Let $\vec{I}$ be a spatial multi-index and let $\Pos \geq 0$ be a constant.
	Then the following \textbf{commuted metric evolution equations} hold:
	\begin{subequations}
	\begin{align}
	\partial_t (t^{\Blowupexp + \Pos} \partial_{\vec{I}} g_{ij})
		& = 
			\frac{1}{t}
			\left\lbrace
				(\Blowupexp + \Pos)
				\delta_{\ j}^a
				- 2 n t \SecondFund_{\ j}^a
				\right\rbrace
				(t^{\Blowupexp + \Pos} \partial_{\vec{I}} g_{ia}) 
			+ 
		\leftexp{(\Pos;\vec{I})}{\mathfrak{G}}_{ij}, 
		\label{E:COMMUTEDPARTIALTGCMC} \\
	\partial_t (t^{\Blowupexp + \Pos} \partial_{\vec{I}} g^{ij})
	& = 
	\frac{1}{t}
	\left\lbrace
		(\Blowupexp + \Pos)
		\delta_{\ a}^j
		+ 
		2 n t \SecondFund_{\ a}^j
	\right\rbrace
	(t^{\Blowupexp + \Pos} \partial_{\vec{I}} g^{ia}) 
	+
	\leftexp{(\Pos;\vec{I})}{\widetilde{\mathfrak{G}}}^{ij}, 
		\label{E:COMMUTEDPARTIALTGINVERSECMC} 
\end{align}
\end{subequations}
where
\begin{subequations}
\begin{align}
		\leftexp{(\Pos;\vec{I})}{\mathfrak{G}}_{ij}
		& \mycong
		\mathop{\sum_{\vec{I}_1 + \vec{I}_2 + \vec{I}_3 = \vec{I}}}_{|\vec{I}_2| \leq |\vec{I}|-1}
		t^{\Blowupexp + \Pos}
		(\partial_{\vec{I}_1} n) 
		(\partial_{\vec{I}_2} g) 
		\partial_{\vec{I}_3} \SecondFund,
			\label{E:COMMUTEDPARTIALTGERRORTERM} \\
	\leftexp{(\Pos;\vec{I})}{\widetilde{\mathfrak{G}}}^{ij}
	& \mycong
		\mathop{\sum_{\vec{I}_1 + \vec{I}_2 + \vec{I}_3 = \vec{I}}}_{|\vec{I}_2| \leq |\vec{I}|-1}
		t^{\Blowupexp + \Pos}
		(\partial_{\vec{I}_1} n) 
		(\partial_{\vec{I}_2} g^{-1}) 
		\partial_{\vec{I}_3} \SecondFund.
		\label{E:COMMUTEDPARTIALTGINVERSEERRORTERM} 
	\end{align}
\end{subequations}

\end{lemma}

\begin{proof}
	Equation \eqref{E:COMMUTEDPARTIALTGCMC} follows in a straightforward fashion
	from commuting equation \eqref{E:PARTIALTGCMC}
	first with $\partial_{\vec{I}}$ and then with $t^{\Blowupexp + \Pos}$.
	Equation \eqref{E:COMMUTEDPARTIALTGINVERSECMC} follows from
	applying the same procedure to equation \eqref{E:PARTIALTGINVERSECMC}.
\end{proof}

\subsection{Control of the error terms in the near-top-order energy estimates}
\label{SS:CONTROLOFERRORTERMSINNEARTOPORDERMETRICESTIMATES}
We now bound various $L^2$ norms of the error terms from Lemma~\ref{L:EQUATIONSFORMETRICITSELFCOMMUTED}
in terms of the solution norms.

\begin{lemma}[$L^2$ \textbf{control of the error terms in the near-below-top-order 
		energy estimates for the $g$ and $\SecondFund$}] 
	\label{L:L2CONTROLOFERRORTERMSINJUSTBELOWTOPORDERMETRICENERGYESTIMATES}
	Assume that the bootstrap assumptions \eqref{E:BOOTSTRAPASSUMPTIONS} hold.
	There exists a constant $C > 0$
	such that if $N$ is sufficiently large in a manner that depends on $\Blowupexp$
	and if $\varepsilon$ is sufficiently small,  
	then the following estimates hold for $t \in (\TBoot,1]$.
	
	\medskip
	
	\noindent \underline{\textbf{Error term estimates for the near-top-order derivatives of} $g$}.
	The following estimates hold for the
	error terms from
	\eqref{E:COMMUTEDPARTIALTGERRORTERM}-\eqref{E:COMMUTEDPARTIALTGINVERSEERRORTERM}:
	\begin{subequations}
	\begin{align} \label{E:METRICJUSTBELOWTOPORDERERRORL2ESTIMATE}
		\max_{|\vec{I}| = N}
		\left\|
			\leftexp{(2 \Worstexp + \Room;\vec{I})}{\mathfrak{G}}
		\right\|_{L_{Frame}^2(\Sigma_t)}
		& 
		\leq
		C
		t^{\Room - 1}
		\left\lbrace
			\MetricLownorm(t) + \MetricHighnorm(t)
		\right\rbrace,
				\\
		\max_{|\vec{I}| = N}
		\left\|
			\leftexp{(2 \Worstexp + \Room;\vec{I})}{\widetilde{\mathfrak{G}}}
		\right\|_{L_{Frame}^2(\Sigma_t)}
		& 
		\leq
		C
		t^{\Room - 1}
		\left\lbrace
			\MetricLownorm(t) + \MetricHighnorm(t)
		\right\rbrace,
		\label{E:INVERSEMETRICJUSTBELOWTOPORDERERRORL2ESTIMATE}
	\end{align}
	\end{subequations}
	
	\begin{subequations}
	\begin{align} \label{E:METRICJUSTBELOWTOPORDERERRORGNORML2ESTIMATE}
		\max_{|\vec{I}| = N}
		\left\|
			\leftexp{(\Worstexp + \Room;\vec{I})}{\mathfrak{G}}
		\right\|_{L_g^2(\Sigma_t)}
		& 
		\leq
		C t^{\Room-1}
		\left\lbrace
			\MetricLownorm(t) + \MetricHighnorm(t)
		\right\rbrace,
			\\
		\max_{|\vec{I}| = N}
		\left\|
			\leftexp{(\Worstexp + \Room;\vec{I})}{\widetilde{\mathfrak{G}}}
		\right\|_{L_g^2(\Sigma_t)}
		& 
		\leq
		C t^{\Room-1}
		\left\lbrace
			\MetricLownorm(t) + \MetricHighnorm(t)
		\right\rbrace.
		\label{E:INVERSEMETRICJUSTBELOWTOPORDERERRORGNORML2ESTIMATE}
	\end{align}
	\end{subequations}
	
	Furthermore,
	the following estimates hold for the error terms from
	\eqref{E:COMMUTEDPARTIALTGERRORTERM}-\eqref{E:COMMUTEDPARTIALTGINVERSEERRORTERM}:
	\begin{subequations}
	\begin{align} \label{E:METRICTWOBELOWTOPORDERERRORL2ESTIMATE}
		\max_{|\vec{I}| = N-1}
		\left\|
			\leftexp{(\Blowupexp + 5 \Worstexp + 3 \Room - 1;\vec{I})}{\mathfrak{G}}
		\right\|_{L_{Frame}^2(\Sigma_t)}
		& 
		\leq
		C 
		\varepsilon
		t^{\Blowupexp + 5 \Worstexp + 3 \Room - 2}
		\left\|
			g 
		\right\|_{\dot{H}_{Frame}^{N-1}(\Sigma_t)}
		+
		C 
		t^{\Room - 1}
		\left\lbrace
			\MetricLownorm(t) + \MetricHighnorm(t)
		\right\rbrace,
			\\
		\max_{|\vec{I}| = N-1}
		\left\|
			\leftexp{(\Blowupexp + 5 \Worstexp + 3 \Room - 1;\vec{I})}{\widetilde{\mathfrak{G}}}
		\right\|_{L_{Frame}^2(\Sigma_t)}
		& 
		\leq
		C 
		\varepsilon
		t^{\Blowupexp + 5 \Worstexp + 3 \Room - 2}
		\left\|
			g^{-1} 
		\right\|_{\dot{H}_{Frame}^{N-1}(\Sigma_t)}
		+
		C
		t^{\Room - 1}
		\left\lbrace
			\MetricLownorm(t) + \MetricHighnorm(t)
		\right\rbrace.
		\label{E:INVERSEMETRICTWOBELOWTOPORDERERRORL2ESTIMATE}
	\end{align}
	\end{subequations}
	
	In addition, 
	the following estimates hold for the
	error terms from
	\eqref{E:TOPORDERDIFFERENTIATEDBORDERONCEDIFFERENTIATEDMETRICERRORTERM}-\eqref{E:TOPORDERDIFFERENTIATEDJUNKONCEDIFFERENTIATEDMETRICERRORTERM}:
	\begin{subequations}
	\begin{align}
		\max_{|\vec{I}| = N-1}
		\left\|
			\leftexp{(Border;2 \Worstexp + \Room;\vec{I})}{\mathfrak{H}}
		\right\|_{L_g^2(\Sigma_t)}
		& \leq
		C 
		t^{\Room - 1} 
		\left\lbrace
			\MetricLownorm(t)
			+
			\MetricHighnorm(t) 
		\right\rbrace,
			\label{E:ONEBELOWTOPORDERDIFFERENTIATEDBORDERONCEDIFFERENTIATEDMETRICERRORTERML2ESTIMATE} \\
		\max_{|\vec{I}| = N-1}
		\left\|
			\leftexp{(Junk;2 \Worstexp + \Room;\vec{I})}{\mathfrak{H}}
		\right\|_{L_g^2(\Sigma_t)}
		& \leq
		C 
		t^{\Room - 1} 
		\left\lbrace
			\MetricLownorm(t)
			+
			\MetricHighnorm(t) 
		\right\rbrace.
		\label{E:ONEBELOWTOPORDERDIFFERENTIATEDJUNKONCEDIFFERENTIATEDMETRICERRORTERML2ESTIMATE}
	\end{align}	
	\end{subequations}
	
	Let $\vec{I}$ be a spatial multi-index with $|\vec{I}|=N-1$ and let
	$T$ be the type $\binom{0}{3}$ $\Sigma_t$-tangent tensorfield
	with the following components relative to the transported
	spatial coordinates:
	$
	\displaystyle
	T_{eij}
	=:
	- 2 n g_{ia} \partial_e \partial_{\vec{I}} \SecondFund_{\ j}^a
	$.
	Then the following estimate holds:
	\begin{align} \label{E:L2GNORMESTIMATEFORTOPORDERSECONDFUNDTERMSINJUSTBELOWTOPORDERSTIMATESFORMETRIC}
		t^{\Blowupexp + 2 \Worstexp + \Room} 
		\left\|
			T
		\right\|_{L_g^2(\Sigma_t)}
		& \leq
		C
		t^{\Room - 1}
		\MetricHighnorm(t).
	\end{align}

	\medskip
	\noindent \underline{\textbf{Error term estimates for the just-below-top-order derivatives of} $\SecondFund$}.
	The following estimates hold for the
	error terms from
	\eqref{E:SECONDFUNDCOMMUTEDBORDERLINETERM}
	and
	\eqref{E:SECONDFUNDCOMMUTEDJUNKTERM}:
	\begin{subequations}
	\begin{align} \label{E:JUSTBELOWTOPORDERSECONDFUNDBORDERLINEERRORL2}
		\max_{|\vec{I}| = N-1}
		\left\|
			\leftexp{(Border;3 \Worstexp + \Room;\vec{I}}{\mathfrak{K}}
		\right\|_{L_{Frame}^2(\Sigma_t)}
		& 
		\leq
	C
	t^{\Room - 1}
	\left\lbrace
		\MetricLownorm(t)
		+
		\MetricHighnorm(t) 
	\right\rbrace,
			\\
		\max_{|\vec{I}| = N-1}
		\left\|
			\leftexp{(Junk;3 \Worstexp + \Room;\vec{I}}{\mathfrak{K}}
		\right\|_{L_{Frame}^2(\Sigma_t)}
		& \leq 
		C
		t^{\Room - 1}
		\left\lbrace
			\MetricLownorm(t) + \MetricHighnorm(t)
		\right\rbrace,
			\label{E:JUSTBELOWTOPORDERSECONDFUNDJUNKLINEERRORL2}
	\end{align}
	
	\begin{align} \label{E:JUSTBELOWTOPORDERSECONDFUNDBORDERLINEERRORL2GNORM}
		\max_{|\vec{I}| = N-1}
		\left\|
			\leftexp{(Border;3 \Worstexp + \Room;\vec{I}}{\mathfrak{K}}
		\right\|_{L_g^2(\Sigma_t)}
		& 
		\leq
	C
	t^{\Room - 1}
	\left\lbrace
		\MetricLownorm(t)
		+
		\MetricHighnorm(t) 
	\right\rbrace,
			\\
		\max_{|\vec{I}| = N-1}
		\left\|
			\leftexp{(Junk;3 \Worstexp + \Room;\vec{I}}{\mathfrak{K}}
		\right\|_{L_g^2(\Sigma_t)}
		& \leq 
		C
		t^{\Room - 1}
		\left\lbrace
			\MetricLownorm(t) + \MetricHighnorm(t)
		\right\rbrace.
			\label{E:JUSTBELOWTOPORDERSECONDFUNDJUNKLINEERRORL2GNORM}
	\end{align}
	\end{subequations}
	
	Let $\vec{I}$ be any spatial multi-index with $|\vec{I}| = N-1$
	and let $T$ be the type $\binom{1}{1}$ $\Sigma_t$-tangent tensorfield
	with the following components relative to the transported
	spatial coordinates:
	$
	\displaystyle
	T_{\ j}^i
	=:
	-
	g^{ia} \partial_a \partial_j \partial_{\vec{I}} n
	+
	\frac{1}{2}
	n 
			g^{ic} 
			g^{ab}
			\left\lbrace
			 \partial_a \partial_c \partial_{\vec{I}} g_{bj} 
				+ 
				\partial_a \partial_j \partial_{\vec{I}} g_{bc} 
				- 
				\partial_a \partial_b \partial_{\vec{I}} g_{cj}
				-
				\partial_c \partial_j \partial_{\vec{I}} g_{ab} 
			\right\rbrace
	$.
	Then the following estimates hold:
	\begin{subequations}
	\begin{align} \label{E:L2ESTIMATEFORTOPORDERMETRICANDLAPSEERRORTERMSINJUSTBELOWTOPORDERSTIMATESFORSECONDFUND}
		t^{\Blowupexp + 3 \Worstexp + \Room} 
		\left\|
			 T
		\right\|_{L_{Frame}^2(\Sigma_t)}
		& \leq
		C
		t^{\Room - 1}
		\MetricHighnorm(t),
	\end{align}
	
	\begin{align} \label{E:L2GNORMESTIMATEFORTOPORDERMETRICANDLAPSEERRORTERMSINJUSTBELOWTOPORDERSTIMATESFORSECONDFUND}
		t^{\Blowupexp + 3 \Worstexp + \Room} 
		\left\|
			 T
		\right\|_{L_g^2(\Sigma_t)}
		& \leq
		C
		t^{\Room - 1}
		\MetricHighnorm(t).
	\end{align}
	\end{subequations}
\end{lemma}

\begin{proof}
Throughout this proof, we will assume that $\Blowupexp \updelta$ is sufficiently small
(and in particular that $\Blowupexp \updelta < \Room$);
in view of the discussion in Subsect.\,\ref{SS:SOBOLEVEMBEDDING}, we see that
at fixed $\Blowupexp$, this can be achieved by choosing $N$ to be sufficiently large.
	
	\medskip
	\noindent \underline{\textbf{Proof of 
	\eqref{E:METRICJUSTBELOWTOPORDERERRORL2ESTIMATE}-\eqref{E:INVERSEMETRICJUSTBELOWTOPORDERERRORL2ESTIMATE}}}:
	We first prove \eqref{E:METRICJUSTBELOWTOPORDERERRORL2ESTIMATE}.
	We stress that for this estimate, 
	on RHS~\eqref{E:COMMUTEDPARTIALTGERRORTERM},
	we have $\Pos = 2 \Worstexp + \Room$ 
	and $|\vec{I}|=N$.
	
	We first consider the case in which $|\vec{I}_1| = N$ on RHS~\eqref{E:COMMUTEDPARTIALTGERRORTERM}.
	Using \eqref{E:POINTWISENORMCOMPARISON} with $l = 0$ and $m=2$
	(since $\leftexp{(2 \Worstexp + \Room;\vec{I})}{\mathfrak{G}}$ is type $\binom{0}{2}$),
	the fact that $|g|_g \lesssim 1$,
	and $g$-Cauchy--Schwarz,
	we deduce that the products under consideration are bounded in the norm
	$\| \cdot \|_{L_{Frame}^2(\Sigma_t)}$
	by
	$\lesssim
	t^{\Blowupexp + \Room}
	\left\| 
		\SecondFund 
	\right\|_{L_g^{\infty}(\Sigma_t)}
	\left\|
		n
	\right\|_{\dot{H}^N(\Sigma_t)}
	$.
	From Defs.\,\ref{D:LOWNORMS} and~\ref{D:HIGHNORMS},
	the elliptic estimate \eqref{E:LAPSETOPORDERELLIPTIC},
	and the bootstrap assumptions,
	we deduce that the RHS of the previous expression
	is 
	$
	\lesssim
	t^{\Room - 1}
	\left\lbrace
		\MetricLownorm(t)
		+
		\MetricHighnorm(t) 
	\right\rbrace
	$
	as desired.
	We next consider the case in which $|\vec{I}_3| = N$ on RHS~\eqref{E:COMMUTEDPARTIALTGERRORTERM}.
  Using \eqref{E:POINTWISENORMCOMPARISON} with $l = 0$ and $m=2$
	(since $\leftexp{(2 \Worstexp + \Room;\vec{I})}{\mathfrak{G}}$ is type $\binom{0}{2}$),
	and $g$-Cauchy--Schwarz,
	we deduce that the products under consideration are bounded in the norm
	$\| \cdot \|_{L_{Frame}^2(\Sigma_t)}$
	by
	$\lesssim
	t^{\Blowupexp + \Room}
	\| n \|_{L^{\infty}(\Sigma_t)}
	\left\|
		\SecondFund
	\right\|_{\dot{H}_g^N(\Sigma_t)}
	$.
	From 
	\eqref{E:WORSTEXPANDROOMLOTSOFUSEFULINEQUALITIES},
	Defs.\,\ref{D:LOWNORMS} and~\ref{D:HIGHNORMS},
	and the bootstrap assumptions,
	we deduce that the RHS of the previous expression
	is 
	$
	\lesssim
	t^{\Room - 1}
	\MetricHighnorm(t) 
	$
	as desired.
	It remains for us to consider the cases in which $|\vec{I}_1|,|\vec{I}_2|,|\vec{I}_3| \leq N-1$
	on RHS~\eqref{E:COMMUTEDPARTIALTGERRORTERM}.
	We first use 
	\eqref{E:BASICINTERPOLATION}
	and
	\eqref{E:FRAMENORML2PRODUCTBOUNDINERMSOFLINFINITYANDHMDOT}
	to bound 
	(using that $|\vec{I}|=N$)
	the products under consideration as follows:
	\begin{align} \label{E:FIRSTSTEPMETRICJUSTBELOWTOPORDERJUNKERRORLASTTERM}
		&
		\mathop{\sum_{\vec{I}_1 + \vec{I}_2 + \vec{I}_3 = \vec{I}}}_{|\vec{I}_1|,|\vec{I}_2|,|\vec{I}_3| \leq N-1}
		t^{\Blowupexp + 2 \Worstexp + \Room}
		\left\|
			(\partial_{\vec{I}_1} n) 
			(\partial_{\vec{I}_2} g) 
			\partial_{\vec{I}_3} \SecondFund
		\right\|_{L_{Frame}^2(\Sigma_t)}
			\\
		& \lesssim
		t^{\Blowupexp + 2 \Worstexp + \Room}
		\left\|
			n
		\right\|_{W_{Frame}^{1,\infty}(\Sigma_t)}
		\left\|
			\SecondFund
		\right\|_{W_{Frame}^{1,\infty}(\Sigma_t)}
		\left\|
			g
		\right\|_{\dot{H}_{Frame}^{N-1}(\Sigma_t)}
		\notag
			\\
		& 
			\ \
		+ 
		t^{\Blowupexp + 2 \Worstexp + \Room}
		\left\|
			n
		\right\|_{W^{1,\infty}(\Sigma_t)}
		\left\|
			g
		\right\|_{W_{Frame}^{1,\infty}(\Sigma_t)}
		\left\|
			\SecondFund
		\right\|_{\dot{H}_{Frame}^{N-1}(\Sigma_t)}
		\notag
			\\
		& 
			\ \
		+ 
		t^{\Blowupexp + 2 \Worstexp + \Room}
		\left\|
			g
		\right\|_{W_{Frame}^{1,\infty}(\Sigma_t)}
		\left\|
			\SecondFund
		\right\|_{W_{Frame}^{1,\infty}(\Sigma_t)}
		\left\|
			n
		\right\|_{\dot{H}^{N-1}(\Sigma_t)}
		\notag
			\\
			\ \
	& \ \
		+ 
		t^{\Blowupexp + 2 \Worstexp + \Room}
		\left\|
			n
		\right\|_{W^{1,\infty}(\Sigma_t)}
		\left\|
			g
		\right\|_{W_{Frame}^{1,\infty}(\Sigma_t)}
		\left\|
			\SecondFund
		\right\|_{W_{Frame}^{1,\infty}(\Sigma_t)}.
		\notag
	\end{align}
From 
\eqref{E:WORSTEXPANDROOMLOTSOFUSEFULINEQUALITIES},
Defs.\,\ref{D:LOWNORMS} and~\ref{D:HIGHNORMS},
the estimates \eqref{E:KASNERMETRICESTIMATES},
\eqref{E:KASNERSECONDFUNDESTIMATES},
\eqref{E:UPTOFOURDERIVATIVESOFGLINFINITYSOBOLEV},
and
\eqref{E:UPTOFOURDERIVATIVESOFLAPSELINFINITYSOBOLEV},
the elliptic estimate \eqref{E:LAPSEJUSTBELOWTOPORDERELLIPTIC},
and the bootstrap assumptions, 
we deduce that
$
	\mbox{RHS~\eqref{E:FIRSTSTEPMETRICJUSTBELOWTOPORDERJUNKERRORLASTTERM}}
	\lesssim
	t^{- 3 \Worstexp - 2 \Room - \Blowupexp \updelta}
	\left\lbrace
		\MetricLownorm(t)
		+
		\MetricHighnorm(t) 
	\right\rbrace
	$,
	which, in view of \eqref{E:WORSTEXPANDROOMLOTSOFUSEFULINEQUALITIES},
	is 
	$
	\lesssim
	\mbox{RHS~\eqref{E:METRICJUSTBELOWTOPORDERERRORL2ESTIMATE}}
	$
	as desired. We have therefore proved \eqref{E:METRICJUSTBELOWTOPORDERERRORL2ESTIMATE}.
	
	The estimate \eqref{E:INVERSEMETRICJUSTBELOWTOPORDERERRORL2ESTIMATE} can be proved
	by applying nearly identical arguments to
	RHS~\eqref{E:COMMUTEDPARTIALTGINVERSEERRORTERM},
	and we omit the details.
	
	\medskip
	\noindent \underline{\textbf{Proof of 
	\eqref{E:METRICJUSTBELOWTOPORDERERRORGNORML2ESTIMATE}-\eqref{E:INVERSEMETRICJUSTBELOWTOPORDERERRORGNORML2ESTIMATE}}}:
	We first prove \eqref{E:METRICJUSTBELOWTOPORDERERRORGNORML2ESTIMATE}.
	We stress that for this estimate, 
	on RHS~\eqref{E:COMMUTEDPARTIALTGERRORTERM},
	we have $\Pos = \Worstexp + \Room$ 
	and $|\vec{I}|=N$.
	
	To bound RHS~\eqref{E:COMMUTEDPARTIALTGERRORTERM},
	we first consider the case in which $|\vec{I}_3|= N$.
	Using that $|g|_g \lesssim 1$ and $g$-Cauchy--Schwarz,
	we deduce that the products under consideration are
	bounded in the norm $\| \cdot \|_{L_g^2(\Sigma_t)}$
	by 
	$
	\lesssim 
	t^{\Blowupexp + \Worstexp + \Room} 
	\| n \|_{L^{\infty}(\Sigma_t)}
	\left\|
		\SecondFund
	\right\|_{\dot{H}_g^N(\Sigma_t)}
	$.
	Next, from 
	\eqref{E:WORSTEXPANDROOMLOTSOFUSEFULINEQUALITIES},
	Defs.\,\ref{D:LOWNORMS} and~\ref{D:HIGHNORMS},
	and the bootstrap assumptions, 
	we deduce that the
	RHS of the previous expression is
	$
	\lesssim
	t^{\Worstexp + \Room - 1}
	\MetricHighnorm(t)
	$,
	which is $\lesssim \mbox{RHS~\eqref{E:METRICJUSTBELOWTOPORDERERRORGNORML2ESTIMATE}}$ 
	as desired.
	We now consider the case in which $|\vec{I}_1|= N$ on RHS~\eqref{E:COMMUTEDPARTIALTGERRORTERM}.
	Using that $|g|_g \lesssim 1$ and $g$-Cauchy--Schwarz,
	we deduce that the products under consideration are
	bounded in the norm $\| \cdot \|_{L_g^2(\Sigma_t)}$
	by 
	$
	\lesssim 
	t^{\Blowupexp + \Worstexp + \Room} 
	\| \SecondFund \|_{L_g^{\infty}(\Sigma_t)}
	\left\|
		n
	\right\|_{\dot{H}_g^N(\Sigma_t)}
	$.
	Using Def.~\ref{D:LOWNORMS},
	the estimate \eqref{E:KASNERSECONDFUNDESTIMATES},
	and the elliptic estimate
	\eqref{E:LAPSETOPORDERELLIPTIC},
	we deduce that the
	RHS of the previous expression is
	$
	\lesssim
	t^{\Worstexp + \Room - 1} 
	\left\lbrace
		\MetricLownorm(t) 
		+ 
		\MetricHighnorm(t)
	\right\rbrace
	$,
	which is $\lesssim \mbox{RHS~\eqref{E:METRICJUSTBELOWTOPORDERERRORGNORML2ESTIMATE}}$ 
	as desired.
	It remains for us to consider the remaining cases, in which
$|\vec{I}_1|, |\vec{I}_2|, |\vec{I}_3| \leq N-1$
on RHS~\eqref{E:COMMUTEDPARTIALTGERRORTERM}.
We first use
\eqref{E:L2PRODUCTBOUNDINERMSOFLINFINITYANDHMDOT} with $l=0$ and $m=2$
(since $\leftexp{(\Worstexp + \Room;\vec{I})}{\mathfrak{G}}$ is type $\binom{0}{2}$) 
and \eqref{E:BASICINTERPOLATION}
to deduce 
(using that $|\vec{I}|=N$)
that the products under consideration are bounded as follows:	
\begin{align} \label{E:FIRSTSTEPMETRICJUSTBELOWTOPORDERERRORGNORML2ESTIMATELOWERORDERERRORL2TERM2}
	\mathop{\sum_{\vec{I}_1 + \vec{I}_2 + \vec{I}_3 = \vec{I}}}_{|\vec{I}_1|, |\vec{I}_2|, |\vec{I}_3| \leq N-1}
		t^{\Blowupexp + \Worstexp + \Room}
		\left\|
			(\partial_{\vec{I}_1} n) 
			(\partial_{\vec{I}_2} g) 
			\partial_{\vec{I}_3} \SecondFund
		\right\|_{L_g^2(\Sigma_t)}
		\lesssim
		\mathop{\sum_{\vec{I}_1 + \vec{I}_2 + \vec{I}_3 = \vec{I}}}_{|\vec{I}_1|, |\vec{I}_2|, |\vec{I}_3| \leq N-1}
		t^{\Blowupexp - \Worstexp + \Room}
		\left\|
			(\partial_{\vec{I}_1} n) 
			(\partial_{\vec{I}_2} g) 
			\partial_{\vec{I}_3} \SecondFund
		\right\|_{L_{Frame}^2(\Sigma_t)}.
	\end{align}
	Since RHS~\eqref{E:FIRSTSTEPMETRICJUSTBELOWTOPORDERERRORGNORML2ESTIMATELOWERORDERERRORL2TERM2} 
	is equal to $t^{-3 \Worstexp}$ times
	LHS~\eqref{E:FIRSTSTEPMETRICJUSTBELOWTOPORDERJUNKERRORLASTTERM}, the arguments
	surrounding \eqref{E:FIRSTSTEPMETRICJUSTBELOWTOPORDERJUNKERRORLASTTERM}
	imply that 
	$
\mbox{RHS~\eqref{E:FIRSTSTEPMETRICJUSTBELOWTOPORDERERRORGNORML2ESTIMATELOWERORDERERRORL2TERM2}}
\lesssim
t^{- 6 \Worstexp - 2 \Room - \Blowupexp \updelta}
\left\lbrace
	\MetricLownorm(t)
	+
	\MetricHighnorm(t)
\right\rbrace
$,
which, in view of \eqref{E:WORSTEXPANDROOMLOTSOFUSEFULINEQUALITIES},
is $\lesssim \mbox{RHS~\eqref{E:METRICJUSTBELOWTOPORDERERRORGNORML2ESTIMATE}}$ as desired.
We have therefore proved \eqref{E:METRICJUSTBELOWTOPORDERERRORGNORML2ESTIMATE}.	
	
The estimate \eqref{E:INVERSEMETRICJUSTBELOWTOPORDERERRORGNORML2ESTIMATE}
can be proved by applying nearly identical arguments to the products on
RHS~\eqref{E:COMMUTEDPARTIALTGINVERSEERRORTERM}	
(with $\Pos = \Worstexp + \Room$ and $|\vec{I}|=N$)
and we omit the details.
	
	\medskip
	\noindent \underline{\textbf{Proof of 
	\eqref{E:METRICTWOBELOWTOPORDERERRORL2ESTIMATE}-\eqref{E:INVERSEMETRICTWOBELOWTOPORDERERRORL2ESTIMATE}}}:
	We first prove \eqref{E:METRICTWOBELOWTOPORDERERRORL2ESTIMATE}.
	We stress that for this estimate, 
	on RHS~\eqref{E:COMMUTEDPARTIALTGERRORTERM},
	we have $\Pos = \Blowupexp + 5 \Worstexp + 3 \Room - 1$
	and $|\vec{I}| = N-1$.
	
	We first use 
	\eqref{E:BASICINTERPOLATION}
	and
	\eqref{E:FRAMENORML2PRODUCTBOUNDINERMSOFLINFINITYANDHMDOT}
	to bound the terms on RHS~\eqref{E:COMMUTEDPARTIALTGERRORTERM} as follows:
	\begin{align} \label{E:TWOBELOWTOPORDERJUNKERRORL2BOUND}
		&
		\mathop{\sum_{\vec{I}_1 + \vec{I}_2 + \vec{I}_3 = \vec{I}}}_{|\vec{I}_2| \leq N-2}
		t^{\Blowupexp + 5 \Worstexp + 3 \Room - 1}
		\left\|
			(\partial_{\vec{I}_1} n) 
			(\partial_{\vec{I}_2} g) 
			\partial_{\vec{I}_3} \SecondFund
		\right\|_{L_{Frame}^2(\Sigma_t)}
			\\
		& \lesssim
		t^{\Blowupexp + 5 \Worstexp + 3 \Room - 1}
		\left\|
			n
		\right\|_{W^{1,\infty}(\Sigma_t)}
		\left\|
			g
		\right\|_{L_{Frame}^{\infty}(\Sigma_t)}
		\left\|
			\SecondFund
		\right\|_{\dot{H}_{Frame}^{N-1}(\Sigma_t)}
		\notag
			\\
		& 
			\ \
		+ 
		t^{\Blowupexp + 5 \Worstexp + 3 \Room - 1}
		\left\|
			g
		\right\|_{L_{Frame}^{\infty}(\Sigma_t)}
		\left\|
			\SecondFund
		\right\|_{W_{Frame}^{1,\infty}(\Sigma_t)}
		\left\|
			n
		\right\|_{\dot{H}^{N-1}(\Sigma_t)}
		\notag
		\\
		& 
			\ \
		+ 
		t^{\Blowupexp + 5 \Worstexp + 3 \Room - 1}
		\sum_{|\vec{I}_1| + |\vec{I}_2| = 1}
		\left\|
			\partial_{\vec{I}_1} n
		\right\|_{L{\infty}(\Sigma_t)}
		\left\|
			\partial_{\vec{I}_2} \SecondFund
		\right\|_{L_{Frame}^{\infty}(\Sigma_t)}
		\left\|
			g 
		\right\|_{\dot{H}_{Frame}^{N-2}(\Sigma_t)}
		\notag
			\\
		& \ \
		+
		t^{\Blowupexp + 5 \Worstexp + 3 \Room - 1}
		\sum_{|\vec{I}_1| + |\vec{I}_2| = 1}
		\left\|
			\partial_{\vec{I}_1} n
		\right\|_{L{\infty}(\Sigma_t)}
		\left\|
			\partial_{\vec{I}_2} \SecondFund
		\right\|_{L_{Frame}^{\infty}(\Sigma_t)}
		\left\|
			g 
		\right\|_{L_{Frame}^{\infty}(\Sigma_t)}.
			\notag
	\end{align}
From \eqref{E:WORSTEXPANDROOMLOTSOFUSEFULINEQUALITIES},
Defs.\,\ref{D:LOWNORMS} and~\ref{D:HIGHNORMS},
the estimates
\eqref{E:KASNERMETRICESTIMATES},
\eqref{E:KASNERSECONDFUNDESTIMATES},
and \eqref{E:UPTOFOURDERIVATIVESOFLAPSELINFINITYSOBOLEV},
the elliptic estimates \eqref{E:LAPSELOWNORMELLIPTIC}
and
\eqref{E:LAPSEJUSTBELOWTOPORDERELLIPTIC},
and the bootstrap assumptions, 
we deduce that the products on RHS~\eqref{E:TWOBELOWTOPORDERJUNKERRORL2BOUND},
except for the sum on the next-to-last line,
are bounded by
$
	\lesssim
	t^{2 \Room - 1 - \Blowupexp \updelta}
	\left\lbrace
		\MetricLownorm(t)
		+
		\MetricHighnorm(t) 
	\right\rbrace
	$,
	which is 
	$
	\lesssim
	\mbox{RHS~\eqref{E:METRICTWOBELOWTOPORDERERRORL2ESTIMATE}}
	$
	as desired. 
	To handle the remaining sum on the next-to-last line of RHS~\eqref{E:TWOBELOWTOPORDERJUNKERRORL2BOUND},
	we first note the bound
	$
	\sum_{|\vec{I}_1| + |\vec{I}_2| = 1}
		\left\|
			\partial_{\vec{I}_1} n
		\right\|_{L{\infty}(\Sigma_t)}
		\left\|
			\partial_{\vec{I}_1} \SecondFund
		\right\|_{L_{Frame}^{\infty}(\Sigma_t)}
	\lesssim \varepsilon t^{-1}
	$,
	which follows from
	Defs.\,\ref{D:LOWNORMS} and~\ref{D:HIGHNORMS},
	\eqref{E:WORSTEXPANDROOMLOTSOFUSEFULINEQUALITIES},
	the estimates 
	\eqref{E:KASNERSECONDFUNDESTIMATES}
	and
	\eqref{E:UPTOFOURDERIVATIVESOFLAPSELINFINITYSOBOLEV},
	and the bootstrap assumptions. From this bound
	and the interpolation estimate \eqref{E:BASICINTERPOLATION},
	it follows that the sum on the next-to-last line of RHS~\eqref{E:TWOBELOWTOPORDERJUNKERRORL2BOUND} is
	$
	\lesssim
	\varepsilon
	t^{\Blowupexp + 5 \Worstexp + 3 \Room - 2}
	\left\|
		g 
	\right\|_{\dot{H}_{Frame}^{N-2}(\Sigma_t)}
	\lesssim
	\varepsilon
	t^{\Blowupexp + 5 \Worstexp + 3 \Room - 2}
	\left\|
		g - \KasnerMetric
	\right\|_{L_{Frame}^{\infty}(\Sigma_t)}
	+
	\varepsilon
	t^{\Blowupexp + 5 \Worstexp + 3 \Room - 2}
	\left\|
		g 
	\right\|_{\dot{H}_{Frame}^{N-1}(\Sigma_t)}
	$.
	From Defs.\,\ref{D:LOWNORMS} and~\ref{D:HIGHNORMS},
	it follows that if $\Blowupexp \geq 1$,
	then the RHS of the previous estimate is
	$
	\lesssim
	\varepsilon
	t^{\Room - 1} \MetricLownorm(t)
	+
	\varepsilon
	t^{\Blowupexp + 5 \Worstexp + 3 \Room - 2}
	\left\|
		g 
	\right\|_{\dot{H}_{Frame}^{N-1}(\Sigma_t)}
	$,
	which is $\lesssim \mbox{\upshape RHS}~\eqref{E:METRICTWOBELOWTOPORDERERRORL2ESTIMATE}$
	as desired. We have therefore proved \eqref{E:METRICTWOBELOWTOPORDERERRORL2ESTIMATE}.
	
	The estimate \eqref{E:INVERSEMETRICTWOBELOWTOPORDERERRORL2ESTIMATE} can be proved
	by applying nearly identical arguments to
	RHS~\eqref{E:COMMUTEDPARTIALTGINVERSEERRORTERM}
	(with $\Pos = \Blowupexp + 5 \Worstexp + 3 \Room - 1$ and $|\vec{I}| = N-1$),
	and we omit the details.

\medskip
\noindent \underline{\textbf{Proof of \eqref{E:ONEBELOWTOPORDERDIFFERENTIATEDBORDERONCEDIFFERENTIATEDMETRICERRORTERML2ESTIMATE}}}:
We stress that for this estimate,
on RHS~\eqref{E:TOPORDERDIFFERENTIATEDBORDERONCEDIFFERENTIATEDMETRICERRORTERM},
we have 
$\Pos = 2 \Worstexp + \Room$
and
$|\vec{I}|=N-1$.

Using that $|g|_g \lesssim 1$ and $g$-Cauchy--Schwarz,
we deduce that the product on RHS~\eqref{E:TOPORDERDIFFERENTIATEDBORDERONCEDIFFERENTIATEDMETRICERRORTERM}
is bounded in the norm $\| \cdot \|_{L_g^2(\Sigma_t)}$
by 
$
\lesssim 
	t^{\Blowupexp + 2 \Worstexp + \Room} 
	\| \SecondFund \|_{L_g^{\infty}(\Sigma_t)}
	\left\|
		\partial n
	\right\|_{\dot{H}_g^{N-1}(\Sigma_t)}
$.
Using \eqref{E:POINTWISENORMCOMPARISON} 
(with $l = 0$ and $m=1$)
to estimate
$\left\|
		\partial n
\right\|_{\dot{H}_g^{N-1}(\Sigma_t)}
$,
we deduce that the RHS of the previous expression is
$
\lesssim 
	t^{\Blowupexp + \Worstexp + \Room} 
	\| \SecondFund \|_{L_g^{\infty}(\Sigma_t)}
	\left\|
		n
	\right\|_{\dot{H}^N(\Sigma_t)}
$.
From Def.~\ref{D:LOWNORMS},
the elliptic estimate
\eqref{E:LAPSETOPORDERELLIPTIC},
and the bootstrap assumptions, 
we deduce that the
RHS of the previous expression is
	$
	\lesssim 
	t^{\Blowupexp + \Worstexp + \Room - 1} 
	\left\|
		n
	\right\|_{\dot{H}^N(\Sigma_t)}
	\lesssim
	t^{\Worstexp + \Room - 1}
	\left\lbrace
		\MetricLownorm(t) + \MetricHighnorm(t)
	\right\rbrace
	$,
	which is $\lesssim \mbox{RHS~\eqref{E:ONEBELOWTOPORDERDIFFERENTIATEDBORDERONCEDIFFERENTIATEDMETRICERRORTERML2ESTIMATE}}$ 
	as desired.

\medskip
\noindent \underline{\textbf{Proof of \eqref{E:ONEBELOWTOPORDERDIFFERENTIATEDJUNKONCEDIFFERENTIATEDMETRICERRORTERML2ESTIMATE}}}:
We stress that for this estimate,
on RHS~\eqref{E:TOPORDERDIFFERENTIATEDJUNKONCEDIFFERENTIATEDMETRICERRORTERM},
we have 
$\Pos = 2 \Worstexp + \Room$
and 
$|\vec{I}|=N-1$.

To bound the first product on RHS~\eqref{E:TOPORDERDIFFERENTIATEDJUNKONCEDIFFERENTIATEDMETRICERRORTERM},
we first use $g$-Cauchy--Schwarz to deduce that it is bounded 
in the norm
$\| \cdot \|_{L_g^2(\Sigma_t)}$
by
$
\leq
t^{\Blowupexp + 2 \Worstexp + \Room} 
\| n-1 \|_{L^{\infty}(\Sigma_t)}
\left\|
	\SecondFund
\right\|_{L_g^{\infty}(\Sigma_t)}
\left\|
	\partial g
\right\|_{\dot{H}_g^{N-1}(\Sigma_t)}
$.
From Defs.\,\ref{D:LOWNORMS} and~\ref{D:HIGHNORMS}
and the bootstrap assumptions,
we see that the RHS of the previous expression is
$
\lesssim
t^{1 - 10 \Worstexp - \Room}
\MetricHighnorm(t)
$, 
which, in view of \eqref{E:WORSTEXPANDROOMLOTSOFUSEFULINEQUALITIES},
is 
$\lesssim \mbox{RHS~\eqref{E:ONEBELOWTOPORDERDIFFERENTIATEDJUNKONCEDIFFERENTIATEDMETRICERRORTERML2ESTIMATE}}$
as desired.

To bound the first sum on RHS~\eqref{E:TOPORDERDIFFERENTIATEDJUNKONCEDIFFERENTIATEDMETRICERRORTERM},
we first use
\eqref{E:L2PRODUCTBOUNDINERMSOFLINFINITYANDHMDOT} with $l=0$ and $m=3$ 
(since $\leftexp{(Junk;2 \Worstexp + \Room;\vec{I})}{\mathfrak{H}}$ is type $\binom{0}{3}$)
and \eqref{E:BASICINTERPOLATION}
to deduce 
(using that $|\vec{I}|=N-1$)
that the products under consideration are bounded as follows:	
	\begin{align} \label{E:TOPORDERDIFFERENTIATEDJUNKONCEDIFFERENTIATEDMETRICERRORTERML2TERM2}
	&
	\mathop{\sum_{\vec{I}_1 + \vec{I}_2 + \vec{I}_3 = \vec{I}}}_{|\vec{I}_1| \leq N-2}
		t^{\Blowupexp + 2 \Worstexp + \Room}
		\left\|
			(\partial \partial_{\vec{I}_1} n) 
			(\partial_{\vec{I}_2} g) 
			\partial_{\vec{I}_3} \SecondFund
		\right\|_{L_g^2(\Sigma_t)}
			\\
		& \lesssim
		t^{\Blowupexp - \Worstexp + \Room}
		\left\|
			n 
		\right\|_{\dot{W}^{1,\infty}(\Sigma_t)}
		\left\|
			g
		\right\|_{W_{Frame}^{1,\infty}(\Sigma_t)}
		\left\|
			\SecondFund
		\right\|_{\dot{H}_{Frame}^{N-1}(\Sigma_t)}
			\notag \\
		& \ \
		+
		t^{\Blowupexp - \Worstexp + \Room}
		\left\|
			n 
		\right\|_{\dot{W}^{1,\infty}(\Sigma_t)}
		\left\|
			\SecondFund
		\right\|_{W_{Frame}^{1,\infty}(\Sigma_t)}
		\left\|
			g
		\right\|_{\dot{H}_{Frame}^{N-1}(\Sigma_t)}
		\notag
			\\
		& \ \
		+
		t^{\Blowupexp - \Worstexp + \Room}
		\left\|
			g
		\right\|_{W_{Frame}^{1,\infty}(\Sigma_t)}
		\left\|
			\SecondFund
		\right\|_{W_{Frame}^{1,\infty}(\Sigma_t)}
		\left\| 
			n 
		\right\|_{\dot{H}^{N-1}(\Sigma_t)}
		\notag
			\\
		& \ \
			+
		t^{\Blowupexp - \Worstexp + \Room}
		\left\| 
			n 
		\right\|_{\dot{W}^{1,\infty}(\Sigma_t)}
		\left\|
			g
		\right\|_{W_{Frame}^{1,\infty}(\Sigma_t)}
		\left\|
			\SecondFund
		\right\|_{W_{Frame}^{1,\infty}(\Sigma_t)}.
		\notag
	\end{align}
From Defs.\,\ref{D:LOWNORMS} and~\ref{D:HIGHNORMS},
the estimates \eqref{E:KASNERMETRICESTIMATES},
\eqref{E:KASNERSECONDFUNDESTIMATES},
\eqref{E:UPTOFOURDERIVATIVESOFGLINFINITYSOBOLEV},
and \eqref{E:UPTOFOURDERIVATIVESOFLAPSELINFINITYSOBOLEV},
the elliptic estimates
\eqref{E:LAPSELOWNORMELLIPTIC} 
and
\eqref{E:LAPSEJUSTBELOWTOPORDERELLIPTIC},
and the bootstrap assumptions,	
we see that 
$
\mbox{RHS~\eqref{E:TOPORDERDIFFERENTIATEDJUNKONCEDIFFERENTIATEDMETRICERRORTERML2TERM2}}
\lesssim
t^{2 - 16 \Worstexp - 3 \Room - \Blowupexp \updelta}
\left\lbrace
	\MetricLownorm(t)
	+
	\MetricHighnorm(t)
\right\rbrace
+
t^{- 4 \Worstexp + \Room - \Blowupexp \updelta}
\left\lbrace
	\MetricLownorm(t)
	+
	\MetricHighnorm(t)
\right\rbrace
$,
which, in view of \eqref{E:WORSTEXPANDROOMLOTSOFUSEFULINEQUALITIES},
is $\lesssim \mbox{RHS~\eqref{E:ONEBELOWTOPORDERDIFFERENTIATEDJUNKONCEDIFFERENTIATEDMETRICERRORTERML2ESTIMATE}}$ as desired.

Using essentially the same reasoning,
we find that the second and third sums on RHS~\eqref{E:TOPORDERDIFFERENTIATEDJUNKONCEDIFFERENTIATEDMETRICERRORTERM}
are bounded in the norm $\| \cdot \|_{L_g^2(\Sigma_t)}$
by
$\lesssim 
\eqref{E:TOPORDERDIFFERENTIATEDJUNKONCEDIFFERENTIATEDMETRICERRORTERML2TERM2}
$
and thus are 
 $\lesssim \mbox{RHS~\eqref{E:ONEBELOWTOPORDERDIFFERENTIATEDJUNKONCEDIFFERENTIATEDMETRICERRORTERML2ESTIMATE}}$
as well. We have therefore proved \eqref{E:ONEBELOWTOPORDERDIFFERENTIATEDJUNKONCEDIFFERENTIATEDMETRICERRORTERML2ESTIMATE}.

\medskip
\noindent \underline{\textbf{Proof of \eqref{E:L2GNORMESTIMATEFORTOPORDERSECONDFUNDTERMSINJUSTBELOWTOPORDERSTIMATESFORMETRIC}}}:
First, using
that $|g|_g \lesssim 1$ and $g$-Cauchy--Schwarz,
we deduce that 
$
\| T \|_{L_g^2(\Sigma_t)}
\lesssim 
	t^{\Blowupexp + 2 \Worstexp + \Room} 
	\| n \|_{L^{\infty}(\Sigma_t)}
	\left\|
		\partial \SecondFund
	\right\|_{\dot{H}_g^{N-1}(\Sigma_t)}
$.
Using the definition of the norms
$\| \cdot \|_{\dot{H}_g^M}$
and
$\| \cdot \|_{L_{Frame}^2(\Sigma_t)}$,
we deduce that the RHS of the previous expression is
$\lesssim 
	t^{\Blowupexp + 2 \Worstexp + \Room} 
	\| n \|_{L^{\infty}(\Sigma_t)}
	\left\|
		g^{-1}
	\right\|_{L_{Frame}^{\infty}}^{1/2}
	\left\|
		\SecondFund
	\right\|_{\dot{H}_g^N(\Sigma_t)}
$.
From 
\eqref{E:WORSTEXPANDROOMLOTSOFUSEFULINEQUALITIES},
Defs.\,\ref{D:LOWNORMS} and~\ref{D:HIGHNORMS},
the estimate
\eqref{E:KASNERMETRICESTIMATES},
and the bootstrap assumptions, 
we deduce that the RHS of the previous expression is
$
\lesssim
t^{\Worstexp + \Room - 1} 
\MetricHighnorm(t)
$,
which is $\lesssim \mbox{RHS~\eqref{E:L2GNORMESTIMATEFORTOPORDERSECONDFUNDTERMSINJUSTBELOWTOPORDERSTIMATESFORMETRIC}}$
as desired.

\medskip
	\noindent \underline{\textbf{Proof of \eqref{E:JUSTBELOWTOPORDERSECONDFUNDBORDERLINEERRORL2} and 
	\eqref{E:JUSTBELOWTOPORDERSECONDFUNDBORDERLINEERRORL2GNORM}}}:
	We stress that for these estimates,
	on RHS~\eqref{E:SECONDFUNDCOMMUTEDBORDERLINETERM}, 
	we have $\Pos = 3 \Worstexp + \Room$ and $|\vec{I}| = N-1$. 
	
	We first prove \eqref{E:JUSTBELOWTOPORDERSECONDFUNDBORDERLINEERRORL2}. 
	We start by noting that
	RHS~\eqref{E:SECONDFUNDCOMMUTEDBORDERLINETERM} 
	is bounded in the norm $\| \cdot \|_{L_{Frame}^2(\Sigma_t)}$
	by
	$
	\leq
	C
	t^{\Blowupexp + 3 \Worstexp + \Room - 1} 
	\left\|
		\SecondFund
	\right\|_{L_{Frame}^{\infty}(\Sigma_t)}
	\left\|
		n
	\right\|_{\dot{H}^{N-1}(\Sigma_t)}
	$.
	From Defs.\,\ref{D:LOWNORMS} and~\ref{D:HIGHNORMS},
	the estimate
	\eqref{E:KASNERSECONDFUNDESTIMATES},
	the elliptic estimate 
	\eqref{E:LAPSEJUSTBELOWTOPORDERELLIPTIC},
	and the bootstrap assumptions,
	we find that the RHS of the previous expression is
	$
	\lesssim
	t^{2 \Worstexp + \Room - 1}
	\left\lbrace
		\MetricLownorm(t)
		+
		\MetricHighnorm(t) 
	\right\rbrace$,
	which is $\lesssim \mbox{RHS~\eqref{E:JUSTBELOWTOPORDERSECONDFUNDBORDERLINEERRORL2}}$
	as desired. 
	
	The estimate \eqref{E:JUSTBELOWTOPORDERSECONDFUNDBORDERLINEERRORL2GNORM}
	can be proved using a nearly identical argument,
	the key point being that like
	$
	\left\|
		\SecondFund
	\right\|_{L_{Frame}^{\infty}(\Sigma_t)}
	$,
	the term
	$
	\left\|
		\SecondFund
	\right\|_{L_g^{\infty}(\Sigma_t)}
	$
	is bounded by $\lesssim t^{-1}$.
	
	\medskip
	\noindent \underline{\textbf{Proof of \eqref{E:JUSTBELOWTOPORDERSECONDFUNDJUNKLINEERRORL2}}}:
	We stress that for this estimate,
	on RHS~\eqref{E:SECONDFUNDCOMMUTEDJUNKTERM}, 
	we have $\Pos = 3 \Worstexp + \Room$  and 
	$|\vec{I}| = N-1$. 
	
	To bound the first sum on RHS~\eqref{E:SECONDFUNDCOMMUTEDJUNKTERM},
	we first use
	\eqref{E:FRAMENORML2PRODUCTBOUNDINERMSOFLINFINITYANDHMDOT}
	and 
	\eqref{E:BASICINTERPOLATION}
	to bound 
	(using that $|\vec{I}| = N-1$)
	the terms under consideration as follows:
	\begin{align} \label{E:FIRSTSTEPJUSTBELOWTOPORDERSECONDFUNDJUNKLINEERRORL2TERM1}
	&
	\mathop{\sum_{\vec{I}_1 + \vec{I}_2 = \vec{I}}}_{|\vec{I}_1| \leq N-2}
		t^{\Blowupexp + 3 \Worstexp + \Room - 1}
		\left\|
			\left\lbrace
				\partial_{\vec{I}_1} (n-1) 
			\right\rbrace	
			\partial_{\vec{I}_2} \SecondFund
		\right\|_{L_{Frame}^2(\Sigma_t)}
			\\
		& \lesssim
		t^{\Blowupexp + 3 \Worstexp + \Room - 1}
		\left\|
			n - 1
		\right\|_{L^{\infty}(\Sigma_t)}
		\left\|
			\SecondFund
		\right\|_{\dot{H}_{Frame}^{N-1}(\Sigma_t)}
			\notag \\
	& \ \
		+
		t^{\Blowupexp + 3 \Worstexp + \Room - 1}
		\left\|
			\SecondFund
		\right\|_{\dot{W}_{Frame}^{1,\infty}(\Sigma_t)}
		\left\|
			n
		\right\|_{\dot{H}^{N-1}(\Sigma_t)}
			\notag
			\\
	& \ \
		+
		t^{\Blowupexp + 3 \Worstexp + \Room - 1}
		\left\|
			n -1
		\right\|_{L^{\infty}(\Sigma_t)}
		\left\|
			\SecondFund
		\right\|_{\dot{W}_{Frame}^{1,\infty}(\Sigma_t)}.
		\notag
\end{align}
From Defs.\,\ref{D:LOWNORMS} and~\ref{D:HIGHNORMS}
and the bootstrap assumptions,
we see that 
$
\mbox{RHS~\eqref{E:FIRSTSTEPJUSTBELOWTOPORDERSECONDFUNDJUNKLINEERRORL2TERM1}}
\lesssim
t^{1 - 10 \Worstexp - \Room}
\MetricHighnorm(t)
+
t^{2 \Worstexp + \Room - 1}
\MetricLownorm(t)
$,
which, in view of \eqref{E:WORSTEXPANDROOMLOTSOFUSEFULINEQUALITIES},
is $\lesssim \mbox{RHS~\eqref{E:JUSTBELOWTOPORDERSECONDFUNDJUNKLINEERRORL2}}$ as desired.
	
	To bound the second sum on RHS~\eqref{E:SECONDFUNDCOMMUTEDJUNKTERM},
	we first consider the cases in which
	$|\vec{I}_5| = N-1$
	or
	$|\vec{I}_6| = N-1$.
	Using that $|g^{-1}|_g \lesssim 1$, 
	using the estimate \eqref{E:POINTWISENORMCOMPARISON} with $l=m=1$
	(since $\leftexp{(Junk;3 \Worstexp + \Room;\vec{I}}{\mathfrak{K}}$ is type $\binom{1}{1}$),
	and then again using
	\eqref{E:POINTWISENORMCOMPARISON}
	(this time with $l=0$ and $m=3$)
	to estimate the term
	$\left\|
		\partial g
	\right\|_{L_g^{\infty}(\Sigma_t)}
	$,
	we deduce that the products under consideration are bounded in the norm
	$\| \cdot \|_{L_{Frame}^2(\Sigma_t)}$ as follows:
	\begin{align} \label{E:FIRSTSTEPSECONDFUNDCOMMUTEDJUNKTERM2ANNOYINGTOPORDERTERMS}
		& \lesssim
			t^{\Blowupexp + \Worstexp + \Room}
			\left\|
				n
			\right\|_{L^{\infty}(\Sigma_t}
			\left\|
				\partial g
			\right\|_{L_g^{\infty}(\Sigma_t)}
			\left\|
				\partial g
			\right\|_{\dot{H}_g^{N-1}(\Sigma_t)}
				\\
		& \lesssim
			t^{\Blowupexp - 2 \Worstexp + \Room}
			\left\|
				n
			\right\|_{L^{\infty}(\Sigma_t}
			\left\|
				g
			\right\|_{\dot{W}_{Frame}^{1,\infty}(\Sigma_t}
			\left\|
				\partial g
			\right\|_{\dot{H}_g^{N-1}(\Sigma_t)}.
			\notag
	\end{align}
From 
\eqref{E:WORSTEXPANDROOMLOTSOFUSEFULINEQUALITIES},
Defs.\,\ref{D:LOWNORMS} and~\ref{D:HIGHNORMS},
the estimate \eqref{E:UPTOFOURDERIVATIVESOFGLINFINITYSOBOLEV},
and the bootstrap assumptions,
we see that 
$
\mbox{\upshape RHS~\eqref{E:FIRSTSTEPSECONDFUNDCOMMUTEDJUNKTERM2ANNOYINGTOPORDERTERMS}} 
\lesssim
t^{\Room - 6 \Worstexp - \Blowupexp \updelta}
\MetricHighnorm(t)
$,
which, in view of \eqref{E:WORSTEXPANDROOMLOTSOFUSEFULINEQUALITIES},
is $\lesssim \mbox{RHS~\eqref{E:JUSTBELOWTOPORDERSECONDFUNDJUNKLINEERRORL2}}$ as desired.
It remains for us to consider the cases in which
$|\vec{I}_5| \leq N-2$
and
$|\vec{I}_6| \leq N-2$.
Using 
\eqref{E:BASICINTERPOLATION}
and
\eqref{E:FRAMENORML2PRODUCTBOUNDINERMSOFLINFINITYANDHMDOT},
we bound 
(using that $|\vec{I}| = N-1$)
the terms under consideration as follows:	
	\begin{align} \label{E:FIRSTSTEPJUSTBELOWTOPORDERSECONDFUNDJUNKLINEERRORL2TERM2}
	&
	\mathop{\sum_{\vec{I}_1 + \vec{I}_2 + \cdots + \vec{I}_6 = \vec{I}}}_{|\vec{I}_5|, |\vec{I}_6| \leq N-2}
		t^{\Blowupexp + 3 \Worstexp + \Room}
		\left\|
			(\partial_{\vec{I}_1} n)
			(\partial_{\vec{I}_2} g^{-1})
			(\partial_{\vec{I}_3} g^{-1})
			(\partial_{\vec{I}_4} g^{-1})
			(\partial \partial_{\vec{I}_5} g)
			\partial \partial_{\vec{I}_6} g
		\right\|_{L_{Frame}^2(\Sigma_t)}
			\\
		& \lesssim
		t^{\Blowupexp + 3 \Worstexp + \Room}
		\left\|
			n 
		\right\|_{W^{1,\infty}(\Sigma_t)}
		\left\|
			g^{-1}
		\right\|_{W_{Frame}^{1,\infty}(\Sigma_t)}^3
		\left\|
			g - \KasnerMetric
		\right\|_{W_{Frame}^{2,\infty}(\Sigma_t)}
		\left\|
			g
		\right\|_{\dot{H}_{Frame}^{N-1}(\Sigma_t)}
			\notag \\
		& \ \
		+
		t^{\Blowupexp + 3 \Worstexp + \Room}
		\left\|
			n 
		\right\|_{W^{1,\infty}(\Sigma_t)}
		\left\|
			g^{-1}
		\right\|_{W_{Frame}^{1,\infty}(\Sigma_t)}^2
		\left\|
			g - \KasnerMetric
		\right\|_{W_{Frame}^{2,\infty}(\Sigma_t)}^2
		\left\|
			g^{-1}
		\right\|_{\dot{H}_{Frame}^{N-1}(\Sigma_t)}
		\notag
			\\
		& \ \
		+
		t^{\Blowupexp + 3 \Worstexp + \Room}
		\left\|
			g^{-1}
		\right\|_{W_{Frame}^{1,\infty}(\Sigma_t)}^2
		\left\|
			g - \KasnerMetric
		\right\|_{W_{Frame}^{2,\infty}(\Sigma_t)}^2
		\left\|
			g^{-1}
		\right\|_{W_{Frame}^{1,\infty}(\Sigma_t)}
		\left\| 
			n 
		\right\|_{\dot{H}^{N-1}(\Sigma_t)}
		\notag
			\\
		& \ \
		+
		t^{\Blowupexp + 3 \Worstexp + \Room}
		\left\|
			n 
		\right\|_{W^{1,\infty}(\Sigma_t)}
		\left\|
			g^{-1}
		\right\|_{W_{Frame}^{1,\infty}(\Sigma_t)}^3
		\left\|
			g - \KasnerMetric
		\right\|_{W_{Frame}^{2,\infty}(\Sigma_t)}^2.
		\notag
	\end{align}
From 
\eqref{E:WORSTEXPANDROOMLOTSOFUSEFULINEQUALITIES},
Defs.\,\ref{D:LOWNORMS} and~\ref{D:HIGHNORMS},
the estimate \eqref{E:UPTOFOURDERIVATIVESOFGLINFINITYSOBOLEV},
and the bootstrap assumptions,
we see that 
$
\mbox{RHS~\eqref{E:FIRSTSTEPJUSTBELOWTOPORDERSECONDFUNDJUNKLINEERRORL2TERM2}}
\lesssim
t^{1 - 10 \Worstexp - 2 \Room - \Blowupexp \updelta}
\left\lbrace
	\MetricLownorm(t)
	+
	\MetricHighnorm(t)
\right\rbrace
$,
which, in view of \eqref{E:WORSTEXPANDROOMLOTSOFUSEFULINEQUALITIES},
is $\lesssim \mbox{RHS~\eqref{E:JUSTBELOWTOPORDERSECONDFUNDJUNKLINEERRORL2}}$ as desired.

  To bound the third sum on RHS~\eqref{E:SECONDFUNDCOMMUTEDJUNKTERM},
	we first use
	\eqref{E:BASICINTERPOLATION}
	and
	\eqref{E:FRAMENORML2PRODUCTBOUNDINERMSOFLINFINITYANDHMDOT}
	to bound 
	(using that $|\vec{I}| = N-1$)
	the terms under consideration as follows:
	\begin{align} \label{E:FIRSTSTEPJUSTBELOWTOPORDERSECONDFUNDJUNKLINEERRORL2TERM3}
	&
	\mathop{\sum_{\vec{I}_1 + \vec{I}_2 + \vec{I}_3 + \vec{I}_4 = \vec{I}}}_{|\vec{I}_4| \leq N-2}
		t^{\Blowupexp + 3 \Worstexp + \Room}
		\left\|
			(\partial_{\vec{I}_1} n)
			(\partial_{\vec{I}_2} g^{-1})
			(\partial_{\vec{I}_3} g^{-1})
			\partial^2 \partial_{\vec{I}_4} g
		\right\|_{L_{Frame}^2(\Sigma_t)}
			\\
		& \lesssim
		t^{\Blowupexp + 3 \Worstexp + \Room}
		\left\|
			n
		\right\|_{W^{1,\infty}(\Sigma_t)}
		\left\|
			g^{-1}
		\right\|_{W_{Frame}^{1,\infty}(\Sigma_t)}^2
		\left\|
			g
		\right\|_{\dot{H}_{Frame}^N(\Sigma_t)}
			\notag \\
		& \ \
		+
		t^{\Blowupexp + 3 \Worstexp + \Room}
		\left\|
			g^{-1}
		\right\|_{W_{Frame}^{1,\infty}(\Sigma_t)}^2
		\left\|
			g
		\right\|_{\dot{W}_{Frame}^{2,\infty}(\Sigma_t)}
		\left\|
			n
		\right\|_{\dot{H}^{N-1}(\Sigma_t)}
		\notag
			\\
		& \ \
		+	
		t^{\Blowupexp + 3 \Worstexp + \Room}
		\left\|
			n
		\right\|_{W^{1,\infty}(\Sigma_t)}
		\left\|
			g^{-1}
		\right\|_{W_{Frame}^{1,\infty}(\Sigma_t)}
		\left\|
			g
		\right\|_{\dot{W}_{Frame}^{2,\infty}(\Sigma_t)}
		\left\|
			g^{-1}
		\right\|_{\dot{H}_{Frame}^N(\Sigma_t)}
		\notag
			\\
		& \ \
		+	
		t^{\Blowupexp + 3 \Worstexp + \Room}
		\left\|
			n
		\right\|_{W^{1,\infty}(\Sigma_t)}
		\left\|
			g^{-1}
		\right\|_{W_{Frame}^{1,\infty}(\Sigma_t)}^2
		\left\|
			g
		\right\|_{\dot{W}_{Frame}^{2,\infty}(\Sigma_t)}.
		\notag
		\end{align}
From \eqref{E:WORSTEXPANDROOMLOTSOFUSEFULINEQUALITIES},
Defs.\,\ref{D:LOWNORMS} and~\ref{D:HIGHNORMS},
the estimates \eqref{E:KASNERMETRICESTIMATES}, 
\eqref{E:UPTOFOURDERIVATIVESOFGLINFINITYSOBOLEV},
and
\eqref{E:UPTOFOURDERIVATIVESOFGINVERSELINFINITYSOBOLEV},
and the bootstrap assumptions,
we see that 
$
\mbox{RHS~\eqref{E:FIRSTSTEPJUSTBELOWTOPORDERSECONDFUNDJUNKLINEERRORL2TERM2}}
\lesssim
t^{- 3 \Worstexp - \Blowupexp \updelta}
\left\lbrace
	\MetricLownorm(t)
	+
	\MetricHighnorm(t)
\right\rbrace
+
t^{1 + \Room - 4 \Worstexp  - \Blowupexp \updelta}
\left\lbrace
	\MetricLownorm(t)
	+
	\MetricHighnorm(t)
\right\rbrace
$,
which, in view of \eqref{E:WORSTEXPANDROOMLOTSOFUSEFULINEQUALITIES},
is $\lesssim \mbox{RHS~\eqref{E:JUSTBELOWTOPORDERSECONDFUNDJUNKLINEERRORL2}}$ as desired.

To bound the last sum on RHS~\eqref{E:SECONDFUNDCOMMUTEDJUNKTERM},
we first use
\eqref{E:FRAMENORML2PRODUCTBOUNDINERMSOFLINFINITYANDHMDOT}
to bound 
(using that $|\vec{I}| = N-1$)
the terms under consideration as follows:
\begin{align} \label{E:FIRSTSTEPJUSTBELOWTOPORDERSECONDFUNDJUNKLINEERRORL2TERM4}
	&
	\sum_{\vec{I}_1 + \vec{I}_2 + \vec{I}_3 + \vec{I}_4 = \vec{I}}
		t^{\Blowupexp + 3 \Worstexp + \Room}
		\left\|
			(\partial_{\vec{I}_1} g^{-1})
			(\partial_{\vec{I}_2} g^{-1})
			(\partial \partial_{\vec{I}_3} g)
			\partial \partial_{\vec{I}_4} n
		\right\|_{L_{Frame}^2(\Sigma_t)}
			\\
		& \lesssim
		t^{\Blowupexp + 3 \Worstexp + \Room}
		\left\|
			g^{-1}
		\right\|_{L_{Frame}^{\infty}(\Sigma_t)}^2
		\left\|
			g
		\right\|_{\dot{W}_{Frame}^{1,\infty}(\Sigma_t)}
		\left\|
			n
		\right\|_{\dot{H}^N(\Sigma_t)}
			\notag \\
		& \ \
		+
		t^{\Blowupexp + 3 \Worstexp + \Room}
		\left\|
			n
		\right\|_{\dot{W}^{1,\infty}(\Sigma_t)}
		\left\|
			g^{-1}
		\right\|_{L_{Frame}^{\infty}(\Sigma_t)}^2
		\left\|
			g
		\right\|_{\dot{H}_{Frame}^N(\Sigma_t)}
			\notag \\
		& \ \
		+
		t^{\Blowupexp + 3 \Worstexp + \Room}
		\left\|
			n
		\right\|_{\dot{W}^{1,\infty}(\Sigma_t)}
		\left\|
			g^{-1}
		\right\|_{L_{Frame}^{\infty}(\Sigma_t)}
		\left\|
			g
		\right\|_{\dot{W}_{Frame}^{1,\infty}(\Sigma_t)}
		\left\|
			g^{-1}
		\right\|_{\dot{H}_{Frame}^{N-1}(\Sigma_t)}.
		\notag
		\end{align}
From Defs.\,\ref{D:LOWNORMS} and~\ref{D:HIGHNORMS},
the estimates \eqref{E:KASNERMETRICESTIMATES},
\eqref{E:UPTOFOURDERIVATIVESOFGLINFINITYSOBOLEV},
and
\eqref{E:UPTOFOURDERIVATIVESOFLAPSELINFINITYSOBOLEV},
and the bootstrap assumptions,
we see that 
\begin{align} \label{E:NEARFINALBOUNDJUSTBELOWTOPORDERSECONDFUNDJUNKLINEERRORL2TERM4}
\mbox{RHS~\eqref{E:FIRSTSTEPJUSTBELOWTOPORDERSECONDFUNDJUNKLINEERRORL2TERM4}}
&
\lesssim
t^{\Room - 3 \Worstexp - \Blowupexp \updelta}
\left\lbrace
	\MetricLownorm(t)
	+
	\MetricHighnorm(t)
\right\rbrace
+
t^{2 - 13 \Worstexp - \Room - \Blowupexp \updelta}
\left\lbrace
	\MetricLownorm(t)
	+
	\MetricHighnorm(t)
\right\rbrace
	\\
& \ \
+ 
t^{3 - 16 \Worstexp - 3 \Room - \Blowupexp \updelta}
\left\lbrace
	\MetricLownorm(t)
	+
	\MetricHighnorm(t)
\right\rbrace
\notag
\end{align}
which, in view of \eqref{E:WORSTEXPANDROOMLOTSOFUSEFULINEQUALITIES},
is $\lesssim \mbox{RHS~\eqref{E:JUSTBELOWTOPORDERSECONDFUNDJUNKLINEERRORL2}}$ as desired.

\medskip
\noindent \underline{\textbf{Proof of \eqref{E:JUSTBELOWTOPORDERSECONDFUNDJUNKLINEERRORL2GNORM}}}:
We stress that for this estimate,
on RHS~\eqref{E:SECONDFUNDCOMMUTEDJUNKTERM}, 
we have $\Pos = 3 \Worstexp + \Room$ and 
$|\vec{I}| = N-1$.

We claim that we only have to bound 
(in the norm $\| \cdot \|_{L_g^2(\Sigma_t)}$)
the second sum on RHS~\eqref{E:SECONDFUNDCOMMUTEDJUNKTERM}
in the cases in which
$|\vec{I}_5| = N-1$
or
$|\vec{I}_6| = N-1$.
For by inspecting the proof of \eqref{E:JUSTBELOWTOPORDERSECONDFUNDJUNKLINEERRORL2} given above,
and using \eqref{E:WORSTEXPANDROOMLOTSOFUSEFULINEQUALITIES},
we see that all remaining products 
on RHS~\eqref{E:SECONDFUNDCOMMUTEDJUNKTERM}
are bounded in the norm
$\| \cdot \|_{L_{Frame}^2}$
by
$
\lesssim 
t^{2 \Worstexp + \Room - 1}
\left\lbrace
	\MetricLownorm(t)
	+
	\MetricHighnorm(t)
\right\rbrace
$.
Hence, using \eqref{E:POINTWISENORMCOMPARISON} with $l=m=1$
(since $\leftexp{(Junk;3 \Worstexp + \Room;\vec{I}}{\mathfrak{K}}$ is type $\binom{1}{1}$),
we find that these same products are bounded in the norm
$\| \cdot \|_{L_g^2}$
by
$
\lesssim 
t^{\Room - 1}
\left\lbrace
	\MetricLownorm(t)
	+
	\MetricHighnorm(t)
\right\rbrace
$,
which is $\lesssim \mbox{RHS~\eqref{E:JUSTBELOWTOPORDERSECONDFUNDJUNKLINEERRORL2GNORM}}$ as desired.

To handle the remaining cases in which 
$|\vec{I}_5| = N-1$
or
$|\vec{I}_6| = N-1$
in the second sum on RHS~\eqref{E:SECONDFUNDCOMMUTEDJUNKTERM},
we first use 
that $|g^{-1}|_g \lesssim 1$ 
and $g$-Cauchy--Schwarz
to deduce that the products under consideration are bounded in the norm
$\| \cdot \|_{L_g^2(\Sigma_t)}$ by
$
\lesssim
t^{\Blowupexp + 3 \Worstexp + \Room}
\left\|
	n
\right\|_{L^{\infty}(\Sigma_t}
\left\|
	\partial g
\right\|_{L_g^{\infty}(\Sigma_t)}
\left\|
	\partial g
\right\|_{\dot{H}_g^{N-1}(\Sigma_t)}
$.
Using \eqref{E:POINTWISENORMCOMPARISON} 
(with $l=0$ and $m=3$)
to estimate
$\left\|
	\partial g
\right\|_{L_g^{\infty}(\Sigma_t)}
$,
we deduce that the RHS of the previous expression is
$
\lesssim
t^{\Blowupexp + \Room}
\left\|
	n
\right\|_{L^{\infty}(\Sigma_t}
\left\|
	g
\right\|_{\dot{W}_{Frame}^{1,\infty}(\Sigma_t}
\left\|
	\partial g
\right\|_{\dot{H}_g^{N-1}(\Sigma_t)}
$.
From 
\eqref{E:WORSTEXPANDROOMLOTSOFUSEFULINEQUALITIES},
Defs.\,\ref{D:LOWNORMS} and~\ref{D:HIGHNORMS},
the estimate \eqref{E:UPTOFOURDERIVATIVESOFGLINFINITYSOBOLEV},
and the bootstrap assumptions,
we deduce that the RHS of the previous expression is
$
t^{- 4 \Worstexp - \Blowupexp \updelta}
\left\lbrace
	\MetricLownorm(t)
	+
	\MetricHighnorm(t)
\right\rbrace
$,
which, in view of \eqref{E:WORSTEXPANDROOMLOTSOFUSEFULINEQUALITIES},
is $\lesssim \mbox{RHS~\eqref{E:JUSTBELOWTOPORDERSECONDFUNDJUNKLINEERRORL2GNORM}}$ as desired.

	\medskip
	\noindent \underline{\textbf{Proof of \eqref{E:L2ESTIMATEFORTOPORDERMETRICANDLAPSEERRORTERMSINJUSTBELOWTOPORDERSTIMATESFORSECONDFUND}}}:
	First, using \eqref{E:POINTWISENORMCOMPARISON} with $l=m=1$
	(since $T$ is a type $\binom{1}{1}$ tensorfield),
	we deduce that
	$
	t^{\Blowupexp + 3 \Worstexp + \Room} 
	\| T \|_{L_{Frame}^2(\Sigma_t)}
	\lesssim
	t^{\Blowupexp + \Worstexp + \Room} 
	\| T \|_{L_g^2(\Sigma_t)} 
	$.
	Next, using that $|g^{-1}|_g \lesssim 1$,
	$g$-Cauchy--Schwarz,
	and the estimate $\| n \|_{L^{\infty}(\Sigma_t)} \lesssim 1$
	(which is a simple consequence of \eqref{E:WORSTEXPANDROOMLOTSOFUSEFULINEQUALITIES}, Defs~\ref{D:LOWNORMS}, and the bootstrap assumptions),
	we find that 
	$
	t^{\Blowupexp + \Worstexp + \Room} 
	\| T \|_{L_g^2(\Sigma_t)} 
	\lesssim
	t^{\Blowupexp + \Worstexp + \Room}
	\| \partial^2 n \|_{\dot{H}_g^{N-1}}
	+
	t^{\Blowupexp + \Worstexp + \Room}
	\| \partial^2 g \|_{\dot{H}_g^{N-1}}
	$.
	Using the definition of the norms
$\| \cdot \|_{\dot{H}_g^M}$
and
$\| \cdot \|_{L_{Frame}^{\infty}(\Sigma_t)}$,
we deduce that the RHS of the previous expression is
$
\lesssim
	t^{\Blowupexp + \Worstexp + \Room}
	\| g^{-1} \|_{L_{Frame}^{\infty}(\Sigma_t)}^{1/2}
	\| \partial n \|_{\dot{H}_g^N}
	+
	t^{\Blowupexp + \Worstexp + \Room}
	\| g^{-1} \|_{L_{Frame}^{\infty}(\Sigma_t)}^{1/2}
	\| \partial g \|_{\dot{H}_g^N}
	$.
	From Defs.\,\ref{D:LOWNORMS} and~\ref{D:HIGHNORMS},
	\eqref{E:KASNERMETRICESTIMATES},
	the elliptic estimate \eqref{E:LAPSETOPORDERELLIPTIC},
	and the bootstrap assumptions,
	we find that the RHS of the previous expression is
	$
	\lesssim 
	t^{\Room - 1} 
	\left\lbrace
		\MetricLownorm(t)
		+
		\MetricHighnorm(t)
	\right\rbrace$,
	which is $\lesssim \mbox{RHS~\eqref{E:L2ESTIMATEFORTOPORDERMETRICANDLAPSEERRORTERMSINJUSTBELOWTOPORDERSTIMATESFORSECONDFUND}}$
	as desired. 
		
	\medskip
	\noindent \underline{\textbf{Proof of \eqref{E:L2GNORMESTIMATEFORTOPORDERMETRICANDLAPSEERRORTERMSINJUSTBELOWTOPORDERSTIMATESFORSECONDFUND}}}:
	The arguments given in the proof of
	\eqref{E:L2ESTIMATEFORTOPORDERMETRICANDLAPSEERRORTERMSINJUSTBELOWTOPORDERSTIMATESFORSECONDFUND}
	yield that
	$
	t^{\Blowupexp + 3 \Worstexp + \Room} 
	\| T \|_{L_g^2(\Sigma_t)} 
	\lesssim
	t^{\Blowupexp + 3 \Worstexp + \Room} 
	\| g^{-1} \|_{L_{Frame}^{\infty}(\Sigma_t)}^{1/2}
	\| \partial n \|_{\dot{H}_g^N}
	+
	t^{\Blowupexp + 3 \Worstexp + \Room} 
	\| g^{-1} \|_{L_{Frame}^{\infty}(\Sigma_t)}^{1/2}
	\| \partial^2 g \|_{\dot{H}_g^N}
	\lesssim 
	t^{2 \Worstexp + \Room - 1} 
	\left\lbrace
		\MetricLownorm(t)
		+
		\MetricHighnorm(t)
	\right\rbrace
	$,
	which is a better bound than we need.

\end{proof}

\subsection{Proof of Prop.\,\ref{P:INTEGRALINEQUALITIESFORJUSTBELOWTOPDERIVATIVESOFMETRICANDSECONDFUND}}
\label{SS:INTEGRALINEQUALITIESFORJUSTBELOWTOPDERIVATIVESOFMETRICANDSECONDFUNDAMENTALFORM}
In this subsection, we prove Prop.\,\ref{P:INTEGRALINEQUALITIESFORJUSTBELOWTOPDERIVATIVESOFMETRICANDSECONDFUND}.
Throughout this proof, we will assume that $\Blowupexp \updelta$ is sufficiently small
(and in particular that $\Blowupexp \updelta < \Room$);
in view of the discussion in Subsect.\,\ref{SS:SOBOLEVEMBEDDING}, we see that
at fixed $\Blowupexp$, this can be achieved by choosing $N$ to be sufficiently large.

To prove \eqref{E:JUSTBELOWTOPORDERMETRICENERGYINTEGRALINEQUALITY},
we first use the fundamental theorem of calculus and the evolution equation \eqref{E:COMMUTEDPARTIALTGCMC}
with $\Pos = 2 \Worstexp + \Room$
to deduce that 
for any multi-index $\vec{I}$ with $|\vec{I}| = N$, we have
(where there is no summation over $i,j$ and we stress that $0 < t \leq 1$):
\begin{align} \label{E:FIRSTSTEPJUSTBELOWTOPORDERMETRICENERGYINTEGRALINEQUALITY}
	\left\|	
		t^{\Blowupexp + 2 \Worstexp + \Room} \partial_{\vec{I}} g_{ij}
	\right\|_{L^2(\Sigma_t)}^2
	& 
	= 
	\left\|	
		\partial_{\vec{I}} g_{ij}
	\right\|_{L^2(\Sigma_1)}^2
		\\
	& \ \
		-
		2	
		\int_{s=t}^1
			\int_{\Sigma_s}
				s^{-1}
				\left\lbrace
				(\Blowupexp + 2 \Worstexp + \Room)
				\delta_{\ j}^a
				- 2 n s \SecondFund_{\ j}^a
				\right\rbrace
				(s^{\Blowupexp + 2 \Worstexp + \Room} \partial_{\vec{I}} g_{ia}) 
				(s^{\Blowupexp + 2 \Worstexp + \Room} \partial_{\vec{I}} g_{ij}) 
				 dx
		\, ds
		\notag \\
	& \ \
		-
		2
		\int_{s=t}^1
			\int_{\Sigma_s}
				(s^{\Blowupexp + 2 \Worstexp + \Room} \partial_{\vec{I}} g_{ij}) 
				\leftexp{(2 \Worstexp + \Room;\vec{I})}{\mathfrak{G}}_{ij}
			\, dx
		\, ds.
	\notag
\end{align}
From Def.~\ref{D:HIGHNORMS},
Cauchy--Schwarz,
Young's inequality,
and the estimate \eqref{E:METRICJUSTBELOWTOPORDERERRORL2ESTIMATE},
we deduce that the
last integral on RHS~\eqref{E:FIRSTSTEPJUSTBELOWTOPORDERMETRICENERGYINTEGRALINEQUALITY}
can be bounded as follows:
\begin{align} \label{E:BOUNDFORLASTERRFIRSTSTEPJUSTBELOWTOPORDERMETRICENERGYINTEGRALINEQUALITY}
	2
		\int_{s=t}^1
			\int_{\Sigma_s}
				(s^{\Blowupexp + 2 \Worstexp} \partial_{\vec{I}} g_{ij}) 
				\leftexp{(\Pos;\vec{I})}{\mathfrak{G}}_{ij}
			\, dx
		\, ds
	& \leq
		 C
		\int_{s=t}^1
			s^{\Room - 1}
			\left\lbrace
				\MetricLownorm^2(s) + \MetricHighnorm^2(s)
			\right\rbrace
		\, ds.
\end{align}
Next, we use Def.~\ref{D:LOWNORMS},
the fact that
$s \KasnerSecondFund(s,x)$ is equal to the diagonal tensor $-\mbox{\upshape diag}(q_1,\cdots,q_{\mydim})$,
\eqref{E:WORSTEXPANDROOMLOTSOFUSEFULINEQUALITIES},
and the bootstrap assumptions to bound the (scalar) component $2 n s \SecondFund_{\ j}^a$
on RHS~\eqref{E:FIRSTSTEPJUSTBELOWTOPORDERMETRICENERGYINTEGRALINEQUALITY} as follows:
$\left|2 n s \SecondFund_{\ j}^a \right| \leq 2 \Worstexp \delta_{\ j}^a + C \varepsilon$,
where $\delta_{\ j}^a$ is the standard Kronecker delta.
It follows that the integral on the second line of 
RHS~\eqref{E:FIRSTSTEPJUSTBELOWTOPORDERMETRICENERGYINTEGRALINEQUALITY}
is bounded from above by
$
\leq
-
\left\lbrace
	2 \Blowupexp
	+
	\Worstexp
\right\rbrace
		\int_{s=t}^1
			\int_{\Sigma_s}
				s^{-1}
				(s^{\Blowupexp + 2 \Worstexp + \Room} \partial_{\vec{I}} g_{ij})^2
				 dx
		\, ds
+
C \varepsilon
\int_{s=t}^1
	s^{-1}
	\left\|	
		s^{\Blowupexp + 2 \Worstexp + \Room} g
	\right\|_{\dot{H}_{Frame}^N(\Sigma_t)}^2
\, ds
$.
Combining these estimates,
noting that
$
\left\|	
	\partial_{\vec{I}} g_{ij}
\right\|_{L^2(\Sigma_1)}^2
\leq
\MetricHighnorm^2(1)
$,
summing the resulting estimates over $1 \leq i,j \leq \mydim$
and over $\vec{I}$ with $|\vec{I}| = N$,
and taking $\varepsilon$ to be sufficiently small,
we arrive at \eqref{E:JUSTBELOWTOPORDERMETRICENERGYINTEGRALINEQUALITY}.

The estimate \eqref{E:JUSTBELOWTOPORDERINVERSEMETRICENERGYINTEGRALINEQUALITY}
can be proved using a similar argument based on
the evolution equation \eqref{E:COMMUTEDPARTIALTGINVERSECMC}
with $\Pos = 2 \Worstexp + \Room$ and $|\vec{I}| = N$
and the estimate \eqref{E:INVERSEMETRICJUSTBELOWTOPORDERERRORL2ESTIMATE},
we omit the details.

The estimate \eqref{E:TWOBELOWTOPORDERMETRICENERGYINTEGRALINEQUALITY}
can be proved using a similar argument based on
the evolution equation \eqref{E:COMMUTEDPARTIALTGCMC}
with $\Pos = 5 \Worstexp + 3 \Room - 1$ and $|\vec{I}| = N - 1$
and the estimate \eqref{E:METRICTWOBELOWTOPORDERERRORL2ESTIMATE},
and we omit the details.

The estimate \eqref{E:TWOBELOWTOPORDERINVERSEMETRICENERGYINTEGRALINEQUALITY}
can be proved using a similar argument based on
the evolution equation \eqref{E:COMMUTEDPARTIALTGINVERSECMC}
with $\Pos = 5 \Worstexp + 3 \Room - 1$ and $|\vec{I}| = N - 1$
and the estimate \eqref{E:INVERSEMETRICTWOBELOWTOPORDERERRORL2ESTIMATE},
and we omit the details.

The estimate \eqref{E:JUSTBELOWTOPORDERSECONDFUNDINTEGRALINEQUALITY}
can be proved using a similar argument based on
the evolution equation \eqref{E:COMMUTEDPARTIALTKCMC}
with $\Pos = 3 \Worstexp + \Room$ and $|\vec{I}| = N - 1$
and the estimates
\eqref{E:JUSTBELOWTOPORDERSECONDFUNDBORDERLINEERRORL2},
\eqref{E:JUSTBELOWTOPORDERSECONDFUNDJUNKLINEERRORL2},
and
\eqref{E:L2ESTIMATEFORTOPORDERMETRICANDLAPSEERRORTERMSINJUSTBELOWTOPORDERSTIMATESFORSECONDFUND},
and we omit the details.

To prove \eqref{E:GNORMJUSTBELOWTOPORDERMETRICENERGYINTEGRALINEQUALITY},
we let $\vec{I}$ be any
any multi-index with $|\vec{I}| = N$.
Using the evolution equation \eqref{E:COMMUTEDPARTIALTGCMC}
with $\Pos = \Worstexp + \Room$,
the definition of the norm $|\cdot|_g$,
and equation \eqref{E:PARTIALTGINVERSECMC}
(to substitute for the factors of $\partial_t g^{-1}$ that appear 
when $\partial_t$ falls on the factors of $g^{-1}$ that are inherent in the definition of $|\cdot|_g$),
we deduce that 
\begin{align} \label{E:FIRSTSTEPGNORMJUSTBELOWTOPORDERMETRICENERGYINTEGRALINEQUALITY}
\partial_t
\left\lbrace
\left|
	t^{\Blowupexp + \Worstexp + \Room} \partial_{\vec{I}} g
\right|_g^2
\right\rbrace
& =
		4 n
		g^{ac} \SecondFund^{bd}
		(t^{\Blowupexp + \Worstexp + \Room} \partial_{\vec{I}} g_{ab})
		(t^{\Blowupexp + \Worstexp + \Room} \partial_{\vec{I}} g_{cd})
	+
	2
	g^{ac} g^{bd}
	(t^{\Blowupexp + \Worstexp + \Room} \partial_{\vec{I}} g_{ab})
	\partial_t (t^{\Blowupexp + \Worstexp + \Room} \partial_{\vec{I}} g_{cd}).
\end{align}
Next, using equation \eqref{E:COMMUTEDPARTIALTGCMC} to substitute for
the factor 
$\partial_t (t^{\Blowupexp + 2 \Worstexp + \Room} \partial_{\vec{I}} g_{cd})$
on RHS~\eqref{E:FIRSTSTEPGNORMJUSTBELOWTOPORDERMETRICENERGYINTEGRALINEQUALITY},
we obtain
\begin{align} \label{E:SECONDSTEPGNORMJUSTBELOWTOPORDERMETRICENERGYINTEGRALINEQUALITY}
\partial_t
\left\lbrace
\left|
	t^{\Blowupexp + \Worstexp + \Room} \partial_{\vec{I}} g
\right|_g^2
\right\rbrace
& =
		\frac{2(\Blowupexp + \Worstexp + \Room)}{t}
		\left|
			t^{\Blowupexp + \Worstexp + \Room} \partial_{\vec{I}} g
		\right|_g^2
		+
		2
		g^{ac} g^{bd}
		(t^{\Blowupexp + \Worstexp + \Room} \partial_{\vec{I}} g_{ab})
		\leftexp{(\Worstexp + \Room;\vec{I})}{\mathfrak{G}}_{cd}.
\end{align}
Integrating \eqref{E:SECONDSTEPGNORMJUSTBELOWTOPORDERMETRICENERGYINTEGRALINEQUALITY}
over the spacetime slab $(t,1] \times \mathbb{T}^{\mydim}$,
using $g$-Cauchy--Schwarz, and appealing to Def.\,\ref{D:HIGHNORMS},
we obtain the following estimate (where we stress that $t < 1$):
\begin{align} \label{E:THIRDSTEPGNORMJUSTBELOWTOPORDERMETRICENERGYINTEGRALINEQUALITY}
	\left\|	
		t^{\Blowupexp + \Worstexp + \Room} \partial_{\vec{I}} g
	\right\|_{L_g^2(\Sigma_t)}^2
	& 
	\leq
	\left\|	
		\partial_{\vec{I}} g
	\right\|_{L_g^2(\Sigma_1)}^2
		\\
	& \ \
		-
		2	(\Blowupexp + \Worstexp + \Room)
		\int_{s=t}^1
				s^{-1}
					\left\|	
						s^{\Blowupexp + \Worstexp + \Room} \partial_{\vec{I}} g
					\right\|_{L_g^2(\Sigma_s)}^2
			\, ds
		\notag \\
	& \ \
		+
		2
		\int_{s=t}^1
			\int_{\Sigma_s}
				\MetricHighnorm(s)
				\left\|
					\leftexp{(\Worstexp + \Room;\vec{I})}{\mathfrak{G}}
				\right\|_{L_g^2(\Sigma_s)}
			\, dx
		\, ds.
	\notag
\end{align}
Using \eqref{E:METRICJUSTBELOWTOPORDERERRORGNORML2ESTIMATE}
to bound the integrand factor 
$
\left\|
	\leftexp{(\Worstexp + \Room;\vec{I})}{\mathfrak{G}}
\right\|_{L_g^2(\Sigma_s)}
$
on RHS~\eqref{E:THIRDSTEPGNORMJUSTBELOWTOPORDERMETRICENERGYINTEGRALINEQUALITY},
using Young's inequality, 
noting that
$
\left\|	
	\partial_{\vec{I}} g
\right\|_{L_g^2(\Sigma_1)}^2
\leq
\MetricHighnorm^2(1)
$,
and summing the resulting estimates over $\vec{I}$ with $|\vec{I}| = N$,
we arrive at the desired bound \eqref{E:GNORMJUSTBELOWTOPORDERMETRICENERGYINTEGRALINEQUALITY}.

The estimate \eqref{E:GNORMJUSTBELOWTOPORDERINVERSEMETRICENERGYINTEGRALINEQUALITY}
can be proved using a similar argument based on
equation \eqref{E:COMMUTEDPARTIALTGINVERSECMC} with 
$\Pos = \Worstexp + \Room$
and
$|\vec{I}| = N$,
equation \eqref{E:PARTIALTGCMC} 
(which one uses to substitute for the factors of $\partial_t g$ that appear 
when $\partial_t$ falls on the factors of $g$ that are inherent in the definition of $|\cdot|_g$),
and the estimate \eqref{E:INVERSEMETRICJUSTBELOWTOPORDERERRORGNORML2ESTIMATE};
we omit the details.

The estimate \eqref{E:GNORMJUSTBELOWTOPORDERWITHGRADIENTMETRICENERGYINTEGRALINEQUALITY}
can be proved using a similar argument based on
the evolution equation \eqref{E:COMMUTEDPARTIALTONEDERIVATIVEOFMETRICCMC} with $\Pos = 2 \Worstexp + \Room$,
and $|\vec{I}| = N-1$,
equation \eqref{E:PARTIALTGINVERSECMC} 
(to substitute for the factors of $\partial_t g^{-1}$ that appear 
when $\partial_t$ falls on the factors of $g^{-1}$ that are inherent in the definition of $|\cdot|_g$),
and the estimates 
\eqref{E:ONEBELOWTOPORDERDIFFERENTIATEDBORDERONCEDIFFERENTIATEDMETRICERRORTERML2ESTIMATE}-\eqref{E:ONEBELOWTOPORDERDIFFERENTIATEDJUNKONCEDIFFERENTIATEDMETRICERRORTERML2ESTIMATE}
and \eqref{E:L2GNORMESTIMATEFORTOPORDERSECONDFUNDTERMSINJUSTBELOWTOPORDERSTIMATESFORMETRIC}.
We omit the details, noting only that the factors of $\partial_t g^{-1}$ 
and the factor $- 2 t \SecondFund_{\ j}^a$
in the first braces on RHS~\eqref{E:COMMUTEDPARTIALTONEDERIVATIVEOFMETRICCMC}
lead to the terms
$
2 n \SecondFund^{ad} g^{be} g^{cf}
(t^{\Blowupexp + \Worstexp + \Room} \partial_a \partial_{\vec{I}} g_{bc})
(t^{\Blowupexp + \Worstexp + \Room} \partial_d \partial_{\vec{I}} g_{ef})
+
4 (n - 1) g^{ad} \SecondFund^{be} g^{cf}
(t^{\Blowupexp + \Worstexp + \Room} \partial_a \partial_{\vec{I}} g_{bc})
(t^{\Blowupexp + \Worstexp + \Room} \partial_d \partial_{\vec{I}} g_{ef})
$,
which we pointwise bound in magnitude as follows by using
$g$-Cauchy--Schwarz, 
the fact that $|g^{-1}|_g \leq C_*$,
\eqref{E:WORSTEXPANDROOMLOTSOFUSEFULINEQUALITIES},
\eqref{E:KASNERSECONDFUNDESTIMATES},
Def.\,\ref{D:LOWNORMS},
and the bootstrap assumptions:
\begin{align} \label{E:POINTWISEBOUNDSFORERRORTERMSWITHTIMEDERIVATIVESOFGINVERSE}
&
2 n \SecondFund^{ad} g^{be} g^{cf}
(t^{\Blowupexp + \Worstexp + \Room} \partial_a \partial_{\vec{I}} g_{bc})
(t^{\Blowupexp + \Worstexp + \Room} \partial_d \partial_{\vec{I}} g_{ef})
+
4 (n-1) g^{ad} \SecondFund^{be} g^{cf}
(t^{\Blowupexp + \Worstexp + \Room} \partial_a \partial_{\vec{I}} g_{bc})
(t^{\Blowupexp + \Worstexp + \Room} \partial_d \partial_{\vec{I}} g_{ef})
	\\
&
\leq
C_*
\left\lbrace
	\| n \|_{L^{\infty}(\Sigma_t)}
	+
	\| n - 1 \|_{L^{\infty}(\Sigma_t)}
\right\rbrace
\left\|
	\SecondFund
\right\|_{L_g^{\infty}(\Sigma_t)}
\left|
	t^{\Blowupexp + \Worstexp + \Room} \partial_a \partial_{\vec{I}} g
\right|_g^2
\leq
C_*
t^{-1}
\left|
	t^{\Blowupexp + \Worstexp + \Room} \partial_a \partial_{\vec{I}} g
\right|_g^2.
\notag
\end{align}
We further remark that these factors of $C_*$ lead to the $C_*$-dependent
products on RHS~\eqref{E:GNORMJUSTBELOWTOPORDERWITHGRADIENTMETRICENERGYINTEGRALINEQUALITY}.

The estimate \eqref{E:JUSTBELOWTOPORDERGNORMSECONDFUNDINTEGRALINEQUALITY}
can be proved using a similar argument based on
equation \eqref{E:COMMUTEDPARTIALTKCMC} with $\Pos = 3 \Worstexp + \Room$
and
$|\vec{I}| = N-1$
(where we use equations \eqref{E:PARTIALTGCMC}-\eqref{E:PARTIALTGINVERSECMC}
to substitute for the factors of $\partial_t g$ and $\partial_t g^{-1}$ that arise when 
$\partial_t$ falls on the factors of $g$ and $g^{-1}$ inherent in the definition of $|\cdot|_g$),
and the estimates \eqref{E:JUSTBELOWTOPORDERSECONDFUNDBORDERLINEERRORL2GNORM},
\eqref{E:JUSTBELOWTOPORDERSECONDFUNDJUNKLINEERRORL2GNORM},
and \eqref{E:L2GNORMESTIMATEFORTOPORDERMETRICANDLAPSEERRORTERMSINJUSTBELOWTOPORDERSTIMATESFORSECONDFUND};
we omit the details.

\hfill $\qed$

\section{The Main A Priori Estimates}
\label{S:APRIORIESTIMATES}
In this section, we use the estimates derived in
Sects.\,\ref{S:CONTROLOFLAPSEINTERMSOFMETRICANDSECONDFUND}-\ref{S:ESTIMATESFORJUSTBELOWTOPORDERDERIVATIVESOFMETRICANDSECONDFUNDAMENTALFORM}
to prove the main technical result of the article: 
Prop.\,\ref{P:MAINAPRIORIESTIMATES}, which provides
a priori estimates for the solution norms
from Defs.\,\ref{D:LOWNORMS} and~\ref{D:HIGHNORMS}.
The proposition in particular yields a strict improvement of the bootstrap assumptions.

\subsection{Integral inequality for the high norm}
\label{SS:INTEGRALINEQUALITYFORTHEHIGHMETRICNORM}
We start with the following lemma, in which we derive an integral inequality for
the high norm $\MetricHighnorm(t)$. The lemma is an analog of 
Prop.\,\ref{P:INTEGRALINEQUALITYFORLOWMETRICNORM},
in which we derived a similar but simpler inequality for
the low norm $\MetricLownorm(t)$.

\begin{lemma}[\textbf{Integral inequality for the high norm}]
	\label{L:INTEGRALINEQUALITYHIGHNORM}
	Assume that the bootstrap assumptions \eqref{E:BOOTSTRAPASSUMPTIONS} hold.
	There exists a universal constant $C_* > 0$ \underline{independent of $N$ and $\Blowupexp$}
	such that if $N$ is sufficiently large 
	and if $\varepsilon$ is sufficiently small, 
	then the following integral inequality holds for $t \in (\TBoot,1]$
	(where, as we described in Subsect.\,\ref{SSS:CONSTANTS}, constants ``$C$'' are allowed to depend on $N$ and other quantities):
	\begin{align} 
		\MetricHighnorm^2(t)
		&
		\leq
		C
		\MetricHighnorm^2(1)
		-
			\left\lbrace
					2 \Blowupexp 
					-
					C_*
			\right\rbrace
			\int_{s=t}^1
				s^{-1}
				\MetricHighnorm^2(s)
			\, ds
			\label{E:HIGHMETRICNORMINTEGRALINEQUALITY} \\
		& \ \
			+
			C
		\int_{s=t}^1
			s^{\Room - 1}
			\left\lbrace
				\MetricLownorm^2(s)
				+
				\MetricHighnorm^2(s)
			\right\rbrace
		\, ds.
			\notag 
	\end{align}

\end{lemma}	

\begin{proof}
		We sum the estimates \eqref{E:MAINTOPORDERMETRICENERGYESTIMATE}
		over all $\vec{I}$ with $|\vec{I}|=N$
		together with the estimates
		\eqref{E:JUSTBELOWTOPORDERMETRICENERGYINTEGRALINEQUALITY}-\eqref{E:JUSTBELOWTOPORDERGNORMSECONDFUNDINTEGRALINEQUALITY}.
		In view of definition \eqref{E:METRICHIGHNORM},
		we conclude the estimate
		\eqref{E:HIGHMETRICNORMINTEGRALINEQUALITY}.
\end{proof}	

\subsection{A priori estimates for the solution norms}
\label{SS:MAINAPRIORIESTIMATES}
In the next proposition, we provide the main result of Sect.\,\ref{S:APRIORIESTIMATES}.

\begin{proposition}[\textbf{A priori estimates for the solution norms}]
	\label{P:MAINAPRIORIESTIMATES}
	Recall that $\MetricLownorm(t)$, $\LapseLownorm(t)$, $\MetricHighnorm(t)$, and $\LapseHighnorm(t)$
	are the norms from Defs.\,\ref{D:LOWNORMS} and~\ref{D:HIGHNORMS},
	and assume that the bootstrap assumptions \eqref{E:BOOTSTRAPASSUMPTIONS} hold.
	Let $\mathring{\upepsilon}$ be the following norm of the difference between
	the Kasner initial data and the perturbed initial data:
	\begin{align} \label{E:APRIORIESTIMATESTATEMENTGEOMETRICDATASIZE}
		\mathring{\upepsilon} 
		& := 
		\left\|
			g - \KasnerMetric
		\right\|_{L_{Frame}^{\infty}(\Sigma_1)}
		+
		\left\|
			\SecondFund - \KasnerSecondFund
		\right\|_{L_{Frame}^{\infty}(\Sigma_1)}
		+
		\left\|
			g 
		\right\|_{\dot{H}_{Frame}^{N+1}(\Sigma_1)}
		+
		\left\|
			\SecondFund 
		\right\|_{\dot{H}_{Frame}^N(\Sigma_1)}.
	\end{align}
	If $\Blowupexp$ is sufficiently large
	and if $N$ is sufficiently large in a manner that depends on 
	$\Blowupexp$, 
	$\Worstexp$,
	$\Room$,
	and $\mydim$,
	then there exists a constant $C_{N,\Blowupexp,\Worstexp,\Room,\mydim} > 1$
	such that if $\varepsilon$ is sufficiently small in a manner that depends on
	$N$, $\Blowupexp$, $\Worstexp$, $\Room$, and $\mydim$, 
	then the following estimates hold for 
	$t \in (\TBoot,1]$:
	\begin{align} \label{E:MAINAPRIORIESTIMATE}
		\MetricLownorm(t)
		+
		\MetricHighnorm(t)
		+
		\LapseLownorm(t)
		+
		\LapseHighnorm(t)
		&
		\leq 
			C_{N,\Blowupexp,\Worstexp,\Room,\mydim}
			\mathring{\upepsilon}.
	\end{align}
	In particular, if $C_{N,\Blowupexp,\Worstexp,\Room,\mydim} \mathring{\upepsilon} < \varepsilon$,
	then \eqref{E:MAINAPRIORIESTIMATE}
	yields a strict improvement of the bootstrap assumptions. 
\end{proposition}

\begin{proof}
	We first square inequality \eqref{E:INTEGRALINEQUALITYFORLOWMETRICNORM}
	and use the Cauchy--Schwarz estimate
	\begin{align}
	&
	\left(
	\int_{s=t}^1
		s^{\Room - 1} 
		\left\lbrace
			\MetricLownorm(s) + \MetricHighnorm(s)
		\right\rbrace
	\, ds
	\right)^2
		\\
	&
	\leq C
	\int_{s=t}^1
		s^{\Room - 1} 
		\left\lbrace
			\MetricLownorm^2(s) + \MetricHighnorm^2(s)
		\right\rbrace
	\, ds
	\times
	\underbrace{
	\int_{s=t}^1
		s^{\Room - 1}
	\, ds
	}_{\leq C}
		\notag \\
& 
	\leq 
	C
	\int_{s=t}^1
		s^{\Room - 1} 
		\left\lbrace
			\MetricLownorm^2(s) + \MetricHighnorm^2(s)
		\right\rbrace
		\, ds
		\notag
	\end{align}
	to deduce
	\begin{align} \label{E:SQUAREDINTEGRALINEQUALITYFORLOWMETRICNORM}
	\MetricLownorm^2(t)
	&
	\leq
	\MetricLownorm^2(1)
	+
	C
	\int_{s=t}^1
		s^{\Room - 1} 
		\left\lbrace
			\MetricLownorm^2(s) + \MetricHighnorm^2(s)
		\right\rbrace
	\, ds.
\end{align}
	We now fix $\Blowupexp \geq \max \lbrace C_*/2,1 \rbrace$,
	where $C_* > 0$ is the universal constant (independent of $N$ and $\Blowupexp$)
	on RHS~\eqref{E:HIGHMETRICNORMINTEGRALINEQUALITY}, and the assumption $\Blowupexp \geq 1$
	was used in other parts of the paper (for example, the proof of \eqref{E:METRICTWOBELOWTOPORDERERRORL2ESTIMATE}).
	Given this choice of $\Blowupexp$, 
	we fix $N$ large enough such that whenever 
	$\varepsilon$ is sufficiently small, the estimates of
	Lemmas~\ref{P:INTEGRALINEQUALITYFORLOWMETRICNORM} and~\ref{L:INTEGRALINEQUALITYHIGHNORM} hold.
	Note that our choice of $\Blowupexp$ ensures that the factor
	$\left\lbrace
		2 \Blowupexp 
		-
		C_*
	\right\rbrace
	$
	on RHS~\eqref{E:HIGHMETRICNORMINTEGRALINEQUALITY} is non-positive;
	hence the first time integral on RHS~\eqref{E:HIGHMETRICNORMINTEGRALINEQUALITY}
	is non-positive and can be discarded.
	In particular, using \eqref{E:SQUAREDINTEGRALINEQUALITYFORLOWMETRICNORM} 
	and Lemma~\ref{L:INTEGRALINEQUALITYHIGHNORM},
	we see that for $t \in (\TBoot,1]$, the quantity
	$Q(t) := \MetricLownorm^2(t) + \MetricHighnorm^2(t)$
	verifies
	$Q(t) 
	\leq C Q(1) 
	+ 
	C 
	\int_{s=t}^1
		s^{\Room - 1}
		Q(s)
	\, ds
	$.
	Since the function $s^{\Room - 1}$ is integrable over the interval $s \in (0,1]$, 
	we conclude from Gronwall's inequality that $Q(t) \leq C Q(1)$ for $t \in (\TBoot,1]$.
	Moreover, from Defs.\,\ref{D:LOWNORMS} and~\ref{D:HIGHNORMS},
	definition \eqref{E:APRIORIESTIMATESTATEMENTGEOMETRICDATASIZE},
	standard Sobolev interpolation (i.e., \eqref{E:BASICINTERPOLATION}), 
	and Sobolev embedding,
	we deduce that if $N$ is sufficiently large,
	then $Q(1) \leq C \mathring{\upepsilon}^2$.
	From this bound and the bound $Q(t) \leq C Q(1)$, we deduce that
	$Q(t) \leq C \mathring{\upepsilon}^2$ for $t \in (\TBoot,1]$.
	From this bound,
	\eqref{E:LAPSELOWNORMELLIPTIC},
	\eqref{E:LAPSETOPORDERELLIPTIC},
	and
	\eqref{E:LAPSEJUSTBELOWTOPORDERELLIPTIC},
	we conclude, in view of Defs.\,\ref{D:LOWNORMS} and~\ref{D:HIGHNORMS}, 
	the desired bound \eqref{E:MAINAPRIORIESTIMATE}.
	
\end{proof}

\section{Estimates Tied to Curvature Blowup and the Length of Past-Directed Causal Geodesics}
\label{S:CURVATUREANDTIMELIKECURVELENGTHESTIMATES}
In this section, we derive the main estimates needed to show curvature blowup and geodesic incompleteness
for the solutions under study.

\subsection{Curvature estimates}
\label{SS:CURVATUREESTIMATES}
In the next lemma, we derive a pointwise estimate that shows in particular that
the Kretschmann scalar blows up as $t \downarrow 0$.

\begin{lemma}[\textbf{Pointwise estimate for the Kretschmann scalar}]
\label{L:KRETSCHMANNSCALARESTIMATE}
Under the hypotheses and conclusions of Prop.\,\ref{P:MAINAPRIORIESTIMATES},
perhaps enlarging $N$ if necessary,
the following pointwise estimate holds for $t \in (\TBoot,1]$:
\begin{align} \label{E:KRETSCHMANNSCALARESTIMATE}
		\Riemfour_{\alpha \beta \gamma \delta} \Riemfour^{\alpha \beta \gamma \delta}
	&
	=
	4 
	t^{-4}
	\left\lbrace
		\sum_{i=1}^{\mydim} (q_i^2 - q_i)^2 
		+ 
		\sum_{1 \leq i < j = \mydim} q_i^2 q_j^2
	\right\rbrace
	+
	t^{-4}
	\mathcal{O}(\mathring{\upepsilon}).
\end{align}

\end{lemma}

\begin{proof}
Throughout this proof, we will assume that $\Blowupexp \updelta$ is sufficiently small
(and in particular that $\Blowupexp \updelta < \Room$);
in view of the discussion in Subsect.\,\ref{SS:SOBOLEVEMBEDDING}, we see that
at fixed $\Blowupexp$, this can be achieved by choosing $N$ to be sufficiently large.
	
First, we observe the following identity, 
	which holds relative to CMC-transported spatial coordinates in view of
	\eqref{E:SPACETIMEMETRICDECOMPOSEDINTOLAPSEANDFIRSTFUNDAMENTALFORM}
	and the curvature properties
	$\Riemfour_{\alpha \beta \gamma \delta}
	=
	- \Riemfour_{\beta \alpha \gamma \delta}
	=
	- \Riemfour_{\alpha \beta \delta \gamma}
	= 
	\Riemfour_{\gamma \delta \alpha \beta}$:
	\begin{align} \label{E:KRETSCHMANNNORMDECOMP}
		t^4 \Riemfour_{\alpha \beta \gamma \delta} \Riemfour^{\alpha \beta \gamma \delta}
			& = t^4 \Riemfour_{ab}^{\ \ cd} \Riemfour_{cd}^{\ \ ab}
				+ 
				4 t^4 \Riemfour_{a0}^{\ \ c 0} \Riemfour_{c0}^{\ \ a0} 
				\\
		& \ \ - 4 n^{-2} t^4 |\Riemfour_{0 \cdot}|_g^2,
			\notag
\end{align}
where
$
\Riemfour_{0 \cdot}
$
is defined to be the type $\binom{0}{3}$
$\Sigma_t$-tangent tensorfield
with components
$\Riemfour_{0bcd}$
relative to the transported spatial coordinates.
Next, from standard calculations based in part on the Gauss and Codazzi equations
(see, for example, \cite{aS2010}*{Equation (4.14)} and \cite{aS2010}*{Equation (4.18)}), 
we find that relative to the CMC-transported spatial coordinates,
the components 
$\Riemfour_{\alpha \beta}^{\ \ \ \mu \nu}$
of the (type $\binom{2}{2}$) Riemann curvature tensor of $\gfour$ 
can be decomposed into principal terms and error terms as follows:	
	\begin{subequations}
	\begin{align}
		\Riemfour_{ab}^{\ \ cd} 
			& =
			\SecondFund_{\ a}^c \SecondFund_{\ b}^d
			- 
			\SecondFund_{\ a}^d \SecondFund_{\ b}^c
			+ 
			\pmb{\triangle}_{ab}^{\ \ cd},
			\label{E:RFOURALLSPATIALDECOMP} \\
		\Riemfour_{a0}^{\ \ c 0} & 
		= t^{-1} \SecondFund_{\ a}^c
			+ \SecondFund_{\ e}^c \SecondFund_{\ a}^e 
			+ \pmb{\triangle}_{a0}^{\ \ c 0}, 
			\label{E:RFOURTWO0DECOMP} \\
		n^{-1} \Riemfour_{0b}^{\ \ cd} & = \pmb{\triangle}_{0b}^{\ \ cd},
		\label{E:RFOURONE0DECOMP}
	\end{align}
	\end{subequations}
	where the error terms are defined by
	\begin{subequations}
	\begin{align}
		\pmb{\triangle}_{ab}^{\ \ cd} & := \Riemann_{ab}^{\ \ cd},
		 \label{E:RFOURFOURERROR} \\
		\pmb{\triangle}_{a0}^{\ \ c 0} 
			& := - t^{-1} n^{-1} \partial_t (t \SecondFund_{\ a}^c)
			+ 
			t^{-1} (n^{-1} - 1)\SecondFund_{\ a}^c 
			\label{E:RFOURTWOERROR} \\
		& \ \ - n^{-1} g^{ec} \partial_a \partial_e n
			+ n^{-1} g^{ec} \Gamma_{a \ e}^{\ f} \partial_f n, 
			 \notag \\
		\pmb{\triangle}_{0b}^{\ \ cd} & := 
					g^{ce} \partial_e (\SecondFund_{\ b}^d)
				- g^{de} \partial_e (\SecondFund_{\ b}^c) 
				\label{E:RFOURTHREEERROR} \\
		& \ \ 
			+ g^{ce} \Gamma_{e \ f}^{\ d} \SecondFund_{\ b}^f
			- g^{ce} \Gamma_{e \ b}^{\ f} \SecondFund_{\ f}^d
			- g^{de} \Gamma_{e \ f}^{\ c} \SecondFund_{\ b}^f
			+ g^{de} \Gamma_{e \ b}^{\ f} \SecondFund_{\ f}^c. 
			\notag
	\end{align}
\end{subequations}
In \eqref{E:RFOURFOURERROR}, $\Riemann_{ab}^{\ \ cd}$ denotes a component of the (type $\binom{2}{2}$) 
Riemann curvature tensor of of $g$.

We now claim that the following estimates hold,
where $\KasnerSecondFund$ is the Kasner second fundamental form:
	\begin{align}
		\left\| t^2 \Riemfour_{ab}^{\ \ cd} 
			- (t^2 \KasnerSecondFund_{\ a}^c \KasnerSecondFund_{\ b}^d
			- t^2 \KasnerSecondFund_{\ a}^d \KasnerSecondFund_{\ b}^c) 
		\right\|_{L^{\infty}(\Sigma_t)} 
			& \lesssim \mathring{\upepsilon}, 
			\label{E:4RIEMANNSPATIALBLOWUP} \\
		\left\| 
			t^2 \Riemfour_{a0}^{\ \ c 0} 
			- 
			(t \KasnerSecondFund_{\ a}^c
			+ 
			t^2 \KasnerSecondFund_{\ e}^c \KasnerSecondFund_{\ a}^e) 
		\right\|_{L^{\infty}(\Sigma_t)}
			& \lesssim \mathring{\upepsilon}, 
			\label{E:4RIEMANN0SPATIAL0SPATIALBLOWUP} \\
		\left\| 
			n^{-1} t^2 \Riemfour_{0 \cdot} 
		\right\|_{L_g^{\infty}(\Sigma_t)}
		& \lesssim 
		\mathring{\upepsilon},
			\label{E:4RIEMANN0SPATIALSPATIALSPATIALBLOWUP}
	\end{align}
	where we stress that \eqref{E:4RIEMANNSPATIALBLOWUP}-\eqref{E:4RIEMANN0SPATIAL0SPATIALBLOWUP}
	are estimates for \emph{components} of tensorfields and
	\eqref{E:4RIEMANN0SPATIALSPATIALSPATIALBLOWUP} is an estimate for the \emph{norm} $|\cdot|_g$ of 
	the tensorfield
	$
		\Riemfour_{0 \cdot}
	$.
	Let us momentarily accept 
	\eqref{E:4RIEMANNSPATIALBLOWUP}-\eqref{E:4RIEMANN0SPATIALSPATIALSPATIALBLOWUP}.
	Then from \eqref{E:KRETSCHMANNNORMDECOMP},
	\eqref{E:4RIEMANNSPATIALBLOWUP}-\eqref{E:4RIEMANN0SPATIALSPATIALSPATIALBLOWUP},
	Def.\,\ref{D:LOWNORMS},
	and the estimate \eqref{E:MAINAPRIORIESTIMATE}
	(which in particular implies the component bound
	$t \SecondFund_{\ b}^a = t \KasnerSecondFund_{\ b}^a + \mathcal{O}(\mathring{\upepsilon})$),
	we deduce that
	\begin{align} \label{E:ALMOSTKRETSCHMANNSCALARESTIMATE}
	t^4 \Riemfour_{\alpha \beta \gamma \delta} \Riemfour^{\alpha \beta \gamma \delta}
	&
	=
	2 
	t^4 (\KasnerSecondFund_{\ a}^c \KasnerSecondFund_{\ c}^a)^2
	+
	4 t^2 \KasnerSecondFund_{\ a}^c \KasnerSecondFund_{\ c}^a
	+
	8 t^3 \KasnerSecondFund_{\ a}^c \KasnerSecondFund_{\ d}^a \KasnerSecondFund_{\ c}^d
	+
	2 t^4 \KasnerSecondFund_{\ e}^c \KasnerSecondFund_{\ a}^e \KasnerSecondFund_{\ d}^a \KasnerSecondFund_{\ c}^d
	+
	\mathcal{O}(\mathring{\upepsilon}).
\end{align}
Next, using the fact that in CMC-transported spatial coordinates,
$t \KasnerSecondFund$ is equal to the diagonal tensor $-\mbox{\upshape diag}(q_1,\cdots,q_{\mydim})$,
we compute that
\begin{align}
	t^2 \KasnerSecondFund_{\ a}^c \KasnerSecondFund_{\ c}^a
	& 
	= \sum_{i=1}^{\mydim} q_i^2,
		 \label{E:KASNEREXPONENTSUMID1} \\
	t^3 \KasnerSecondFund_{\ a}^c \KasnerSecondFund_{\ d}^a \KasnerSecondFund_{\ c}^d
	& = - \sum_{i=1}^{\mydim} q_i^3,
		\label{E:KASNEREXPONENTSUMID2} \\
	t^4 \KasnerSecondFund_{\ e}^c \KasnerSecondFund_{\ a}^e \KasnerSecondFund_{\ d}^a \KasnerSecondFund_{\ c}^d
	& = \sum_{i=1}^{\mydim} q_i^4.
	\label{E:KASNEREXPONENTSUMID3}
\end{align}
From \eqref{E:ALMOSTKRETSCHMANNSCALARESTIMATE} 
and \eqref{E:KASNEREXPONENTSUMID1}-\eqref{E:KASNEREXPONENTSUMID3},
we arrive at the desired bound \eqref{E:KRETSCHMANNSCALARESTIMATE}.

It remains for us to prove \eqref{E:4RIEMANNSPATIALBLOWUP}-\eqref{E:4RIEMANN0SPATIALSPATIALSPATIALBLOWUP}.
To prove \eqref{E:4RIEMANNSPATIALBLOWUP}, we first note the schematic identity
$\Riemann_{ab}^{\ \ cd}
\mycong g^{-2} \partial^2 g + g^{-3} (\partial g)^2
$,
which follows from \eqref{E:RIEMANNCURVATUREEXACT}.
Hence, bounding each factor on the RHS of the schematic identity
in the norm
$\| \cdot \|_{L_{Frame}^{\infty}(\Sigma_t)}$
with the help of the estimates 
\eqref{E:KASNERMETRICESTIMATES},
\eqref{E:UPTOFOURDERIVATIVESOFGLINFINITYSOBOLEV}-\eqref{E:UPTOFOURDERIVATIVESOFGINVERSELINFINITYSOBOLEV},
and \eqref{E:MAINAPRIORIESTIMATE},
we deduce the estimate
$\| \Riemann_{ab}^{\ \ cd} \|_{L^{\infty}(\Sigma_t)}
\lesssim 
\mathring{\upepsilon}
t^{- 10 \Worstexp - \Blowupexp \updelta}
$.
From this estimate,
\eqref{E:RFOURALLSPATIALDECOMP},
\eqref{E:RFOURFOURERROR},
the aforementioned component estimate
$t \SecondFund_{\ b}^a = t \KasnerSecondFund_{\ b}^a + \mathcal{O}(\mathring{\upepsilon})$,
and \eqref{E:WORSTEXPANDROOMLOTSOFUSEFULINEQUALITIES},
we conclude \eqref{E:4RIEMANNSPATIALBLOWUP}.
The estimate \eqref{E:4RIEMANN0SPATIAL0SPATIALBLOWUP} can be proved by combining similar arguments
based on equation \eqref{E:RFOURTWO0DECOMP} 
with the additional bounds 
\eqref{E:UPTOFOURDERIVATIVESOFLAPSELINFINITYSOBOLEV}
and
\eqref{E:SECONDFUNDORDER0TIMEDERIVATIVEESTIMATE},
which are needed to help control the terms on RHS~\eqref{E:RFOURTWOERROR}; 
we omit the straightforward details.
To prove \eqref{E:4RIEMANN0SPATIALSPATIALSPATIALBLOWUP}, 
we first use the fact that $|g^{-1}|_g \lesssim 1$ and the $g$-Cauchy--Schwarz
inequality to deduce that the norm $|\cdot|_g$ of RHS~\eqref{E:RFOURTHREEERROR}
is bounded by 
$\lesssim |\partial \SecondFund|_g + |\partial g|_g |\SecondFund|_g$.
Next, using \eqref{E:POINTWISENORMCOMPARISON}, we bound the RHS of the previous expression
as follows:
\begin{align} \label{E:SECONDBOUNDFORCURVATURETENSOR0SPACEERRORTERM}
|\partial \SecondFund|_g + |\partial g|_g |\SecondFund|_g
&
\lesssim 
t^{-3 \Worstexp}
\| \SecondFund \|_{W_{Frame}^{1,\infty}(\Sigma_t)}
+
t^{-3 \Worstexp}
\| g \|_{\dot{W}_{Frame}^{1,\infty}(\Sigma_t)}^{1/2}
\| \SecondFund \|_{L_g^{\infty}(\Sigma_t)}.
\end{align}
From Def.~\ref{D:LOWNORMS} and the estimates 
\eqref{E:KASNERSECONDFUNDESTIMATES},
\eqref{E:UPTOFOURDERIVATIVESOFGLINFINITYSOBOLEV},
and
\eqref{E:MAINAPRIORIESTIMATE},
we deduce,
in view of \eqref{E:WORSTEXPANDROOMLOTSOFUSEFULINEQUALITIES},
that 
$\mbox{RHS~\eqref{E:SECONDBOUNDFORCURVATURETENSOR0SPACEERRORTERM}} 
\lesssim 
\mathring{\upepsilon} t^{-1 - 4 \Worstexp - \Blowupexp \updelta}
\lesssim
\mathring{\upepsilon} t^{-2}$.
Considering also equation \eqref{E:RFOURONE0DECOMP},
we see that we have proved the desired bound \eqref{E:4RIEMANN0SPATIALSPATIALSPATIALBLOWUP}.
\end{proof}

\subsection{Estimates for the length of past-directed causal geodesic segments}
\label{SS:LENGTHOFCAUSALGEODESICS}
In this subsection, we show that for the solutions under consideration,
the length of any past-directed causal geodesic segment is uniformly bounded by a constant.

\begin{lemma}[\textbf{Estimates for the length of past-directed causal geodesic segments}]
\label{L:LENGTHOFCAUSALGEODESICS}
Under the hypotheses and conclusions of Prop.\,\ref{P:MAINAPRIORIESTIMATES},
perhaps enlarging $N$ and shrinking $\mathring{\upepsilon}$ if necessary,
the following holds:
any past-directed causal geodesic $\pmb{\zeta}$ that emanates from $\Sigma_1$ 
and is contained in the region $(\TBoot,1] \times \mathbb{T}^{\mydim}$
has an affine length that is bounded from above by
\begin{align} \label{E:AFFINTEBLOWUPTIME}
	\leq \mathscr{A}(\TBoot) \leq \mathscr{A}(0) \leq \frac{\left|\mathscr{A}^{'}(1)\right|}{1 - \Worstexp},
\end{align}	
where $\mathscr{A}(t)$ is the affine parameter along $\pmb{\zeta}$ viewed as a function of $t$ along $\pmb{\zeta}$ 
(normalized by $\mathscr{A}(1) = 0$).
\end{lemma}

\begin{proof}
Throughout this proof, we will assume that $\Blowupexp \updelta$ is sufficiently small
(and in particular that $\Blowupexp \updelta < \Room$);
in view of the discussion in Subsect.\,\ref{SS:SOBOLEVEMBEDDING}, we see that
at fixed $\Blowupexp$, this can be achieved by choosing $N$ to be sufficiently large.

Let $\pmb{\zeta}(\mathscr{A})$ be a past-directed affinely parametrized geodesic verifying the hypotheses of the lemma.
Note that the component $\pmb{\zeta}^0$ can be identified with the CMC time coordinate. 
In the rest of the proof, we view the affine parameter $\mathscr{A}$ as a function of 
$t = \pmb{\zeta}^0$ along $\pmb{\zeta}$. 
We normalize $\mathscr{A}(t)$ by setting $\mathscr{A}(1) = 0$. 
We also define $\dot{\pmb{\zeta}}^{\mu} := \frac{d}{d \mathscr{A}} \pmb{\zeta}^{\mu}$ and $\mathscr{A}' = \frac{d}{dt} \mathscr{A}$. 
By the chain rule, we have 
\begin{align} 
	\mathscr{A}^{'} 
	& = \frac{1}{\dot{\pmb{\zeta}}^0}, 
	&&
	\ddot{\pmb{\zeta}}^0 
	= \dot{\pmb{\zeta}}^0 \frac{d}{dt} \dot{\pmb{\zeta}}^0 = - (\mathscr{A}^{'})^{-3} \mathscr{A}^{''}.
		\label{E:GEODESICDOTANDDDOTCHAINRULE}
\end{align}

For use below, we note that since $\pmb{\zeta}$ is a causal curve,
we have (by the definition of a causal curve) that $\gfour(\dot{\pmb{\zeta}},\dot{\pmb{\zeta}}) \leq 0$.
Considering also the expression \eqref{E:SPACETIMEMETRICDECOMPOSEDINTOLAPSEANDFIRSTFUNDAMENTALFORM} for $\gfour$, 
we deduce that relative to the CMC-transported coordinates, causal curves satisfy
\begin{align} \label{E:GEODESICSPATIALVELOCITYBOUNDEDBY0VELOCITY}
	g_{ab}\dot{\zeta}^a\dot{\zeta}^b \leq n^2 (\dot{\pmb{\zeta}}^0)^2. 
\end{align}

Next, we note that (relative to CMC-transported spatial coordinates),
the $0$ component of the geodesic equation is
$\ddot{\pmb{\zeta}}^0 
+
\Chfour_{\alpha \ \beta}^{\ 0}|_{\pmb{\zeta}}
\dot{\pmb{\zeta}}^{\alpha}
\dot{\pmb{\zeta}}^{\beta}
=
0
$,
where 
$\Chfour_{\alpha \ \beta}^{\ 0}
$
are Christoffel symbols of $\gfour$
(see \eqref{E:FOURCHRISTOFFEL}).
Using 
\eqref{E:SPACETIMEMETRICDECOMPOSEDINTOLAPSEANDFIRSTFUNDAMENTALFORM} and \eqref{E:PARTIALTGCMC},
we compute that this geodesic equation component 
can be written in the following more explicit form:
\begin{align} \label{E:GEO0}
		\ddot{\pmb{\zeta}}^0 
			+ 
			(\partial_t \ln n)|_{\pmb{\zeta}} (\dot{\pmb{\zeta}}^0)^2 
			+ 
			2 (\partial_a \ln n)|_{\pmb{\zeta}} \dot{\pmb{\zeta}}^a \dot{\pmb{\zeta}}^0 
			- 
			(n^{-1} g_{ac} \SecondFund_{\ b}^c)|_{\pmb{\zeta}} \dot{\pmb{\zeta}}^a \dot{\pmb{\zeta}}^b 
			& = 0. 
\end{align}
Multiplying \eqref{E:GEO0} by $- (\mathscr{A}^{'})^{-3}$
and using the $g$-Cauchy--Schwarz inequality,
\eqref{E:GEODESICDOTANDDDOTCHAINRULE},
and
\eqref{E:GEODESICSPATIALVELOCITYBOUNDEDBY0VELOCITY}, 
we deduce
\begin{align} \label{E:ADOUBLEPRIMEFIRSTEQUATION}
	\left|
		\mathscr{A}^{''} 
	\right| & \leq 
		\left|
			n^{-1}
			\SecondFund_{\ b}^a \dot{\pmb{\zeta}}_a \dot{\pmb{\zeta}}^b 
			(\mathscr{A}')^3
		\right|
		+
		\big(n^{-1} |\partial_t n| + 2 |\partial n|_g \big) |\mathscr{A}^{'}|.
\end{align}
From Def.\,\ref{D:LOWNORMS} and the estimate \eqref{E:MAINAPRIORIESTIMATE},
we see that in CMC-transported spatial coordinates, 
$t \SecondFund$ is equal to the diagonal tensor $\mbox{\upshape diag} - (q_1,\cdots,q_{\mydim})$
plus an $\mathcal{O}(\mathring{\upepsilon})$ correction.
Hence,
the eigenvalues of $t \SecondFund$ are all bounded in magnitude by $q_{(Max)} + \mathcal{O}(\mathring{\upepsilon})$,
where $q_{(Max)} := \max_{i=1,\cdots,\mydim} |q_i|$.
Also using \eqref{E:GEODESICDOTANDDDOTCHAINRULE}-\eqref{E:GEODESICSPATIALVELOCITYBOUNDEDBY0VELOCITY},
we see that
\begin{align} \label{E:BOUNDFORMAINTERMINLENGTHEST}
\left|
	n^{-1}
	\SecondFund_{\ b}^a \dot{\pmb{\zeta}}_a \dot{\pmb{\zeta}}^b 
	(\mathscr{A}')^3
\right|
&
\leq
\left\lbrace
	q_{(Max)} 
	+ 
	\mathcal{O}(\mathring{\upepsilon})
\right\rbrace
t^{-1} 
n^{-1}
(\mathscr{A}')^3
\left|
	\dot{\pmb{\zeta}}	
\right|_g^2
	\\
&
\leq
\left\lbrace
	q_{(Max)} 
	+ 
	\mathcal{O}(\mathring{\upepsilon})
\right\rbrace
t^{-1} 
|\mathscr{A}^{'}|
+
\left\lbrace
	q_{(Max)} 
	+ 
	\mathcal{O}(\mathring{\upepsilon})
\right\rbrace
t^{-1} 
|n-1|
|\mathscr{A}^{'}|.
\notag
\end{align}
Moreover, from Def.\,\ref{D:LOWNORMS},
\eqref{E:WORSTEXPANDROOMLOTSOFUSEFULINEQUALITIES},
and the estimate \eqref{E:MAINAPRIORIESTIMATE},
we see that the last product on RHS~\eqref{E:BOUNDFORMAINTERMINLENGTHEST}
obeys the bound
\begin{align}
\left\lbrace
	q_{(Max)} 
	+ 
	\mathcal{O}(\mathring{\upepsilon})
\right\rbrace
t^{-1} 
|n-1|
|\mathscr{A}^{'}|
\leq C \mathring{\upepsilon} t^{-1} |\mathscr{A}^{'}|.
\end{align}
Furthermore, 
also using 
\eqref{E:KASNERMETRICESTIMATES},
\eqref{E:POINTWISENORMCOMPARISON},
\eqref{E:UPTOFOURDERIVATIVESOFLAPSELINFINITYSOBOLEV},
and
\eqref{E:LAPSETIMEDERIVATIVEESTIMATE}, 
we see that the last term on RHS~\eqref{E:ADOUBLEPRIMEFIRSTEQUATION} is bounded as follows:
\begin{align} \label{E:ADOUBLEPRIMELASTTERMESTIMATE}
\big(n^{-1} |\partial_t n| + 2 |\partial n|_g \big) |\mathscr{A}^{'}|
&
\leq C \mathring{\upepsilon} t^{-1} |\mathscr{A}^{'}|.
\end{align}
Combining \eqref{E:ADOUBLEPRIMEFIRSTEQUATION}-\eqref{E:ADOUBLEPRIMELASTTERMESTIMATE}
and taking into account \eqref{E:WORSTEXPANDROOMLOTSOFUSEFULINEQUALITIES}, 
we deduce that if $\mathring{\upepsilon}$ is sufficiently small,
then the following bound holds:
\begin{align} \label{E:DDTAFFINEGRONWALLREADY}
	\left|\mathscr{A}^{''} \right| & \leq t^{-1} \Worstexp |\mathscr{A}^{'}|,
	& t & \in (\TBoot,1]. 
\end{align}
Applying Gronwall's inequality to \eqref{E:DDTAFFINEGRONWALLREADY}, we deduce that
\begin{align} \label{E:DDTAFFINEESTIMATED}
	|\mathscr{A}^{'}(t)| & \leq |\mathscr{A}^{'}(1)| t^{- \Worstexp}, & t & \in (\TBoot,1].
\end{align}
Integrating \eqref{E:DDTAFFINEESTIMATED} from time $t$ to time $1$ and using the assumption $\mathscr{A}(1)=0$, 
we find that
\begin{align} \label{E:AFFINEESTIMATE}
	\mathscr{A}(t) 
	& \leq 
	\frac{|\mathscr{A}^{'}(1)|}{1 - \Worstexp}
	(1 - t^{1 - \Worstexp}), & t  & \in (\TBoot,1],
\end{align}
from which the desired estimate \eqref{E:AFFINTEBLOWUPTIME} follows.

\end{proof}

\section{The Main Stable Blowup Theorem}
\label{S:MAINTHM}
We now state and prove our main stable blowup theorem.
As we noted in Remark~\ref{R:ADDITONALINFORMATION}, it
is possible to derive substantial additional information
about the solution, going beyond that provided by the theorem.

\begin{theorem}[\textbf{The main stable blowup theorem}]
\label{T:MAINTHM}
	Let $\widetilde{\gfour} = - dt \otimes dt + \KasnerMetric_{ab} dx^a \otimes dx^b$
	be an Einstein-vacuum Kasner solution on $(0,\infty) \times \mathbb{T}^{\mydim}$ 
	i.e., $\KasnerMetric = \mbox{\upshape diag}(t^{2 q_1},t^{2 q_2},\cdots,t^{2 q_{\mydim}})$,
	where $\sum_{i=1}^{\mydim} q_i = \sum_{i=1}^{\mydim} q_i^2 = 1$, 
	and assume that
	\begin{align} \label{E:MAINTHMKASNEREXPONENTSMALLNESS}
		\max_{i=1,\cdots,\mydim} |q_i| < \frac{1}{6}.
	\end{align}
	Recall that in Subsect.\,\ref{SS:EXISTENCEOFKASNER}, 
	we showed that such Kasner solutions
	exist when $\mydim \geq 38$.
	Let $\KasnerSecondFund = - t^{-1} \mbox{\upshape diag}(q_1,q_2,\cdots,q_{\mydim})$
	denote the corresponding Kasner (mixed) second fundamental form.
	Let $(\Sigma_1 = \mathbb{T}^{\mydim},\mathring{g},\mathring{\SecondFund})$
	be initial data for the Einstein-vacuum equations
	verifying 
	the constraints \eqref{E:HAMILTONIAN}-\eqref{E:MOMENTUM}
	and the CMC condition $\mathring{\SecondFund}_{\ a}^a = - 1$
	(see, however, Remark~\ref{R:NONEEDFORCMC}),
	and let
	\begin{align} \label{E:GEOMETRICDATASIZE}
		\mathring{\upepsilon} 
		& := 
		\left\|
			\mathring{g} - \KasnerMetric
		\right\|_{L_{Frame}^{\infty}(\Sigma_1)}
		+
		\left\|
			\mathring{\SecondFund} - \KasnerSecondFund
		\right\|_{L_{Frame}^{\infty}(\Sigma_1)}
		+
		\left\|
			\mathring{g} 
		\right\|_{\dot{H}_{Frame}^{N+1}(\Sigma_1)}
		+
		\left\|
			\mathring{\SecondFund}
		\right\|_{\dot{H}_{Frame}^N(\Sigma_1)}.
	\end{align}
	Assume that
	\begin{itemize}
		\item $\Blowupexp > 0$ is sufficiently large.
		\item $N > 0$ is sufficiently large, where the required largeness depends on
		$\Blowupexp$, $\Worstexp$, $\Room$, and $\mydim$.
			Here we recall that $\Worstexp > 0$ and $\Room > 0$ are the constants fixed in Subsect.\,\ref{SS:PARAMETERSINNORMS}.
		\item $\mathring{\upepsilon}$ is sufficiently small,
			where the required smallness depends on $N$, $\Blowupexp$, $\Worstexp$, $\Room$, and $\mydim$.
	\end{itemize}
	Then the following conclusions hold.
	
	\medskip
	
	\noindent \underline{\textbf{Existence and norm estimates on $(0,1] \times \mathbb{T}^{\mydim}$}}.
	The initial data launch a solution $(g,\SecondFund,n)$ 
	to the Einstein-vacuum equations in 
	CMC-transported spatial coordinates 
	(that is, the equations of Prop.\,\ref{P:EINSTEINVACUUMEQUATIONSINCMCTRANSPORTEDSPATIALCOORDINATES},
	where $\gfour = - n^2 dt \otimes dt + g_{ab} dx^a \otimes dx^b$ is the spacetime metric)
	that exists classically for $(t,x) \in (0,1] \times \mathbb{T}^{\mydim}$. Moreover, there exists a constant 
	$C > 1$ (depending on $N$, $\Blowupexp$, $\Worstexp$, $\Room$, and $\mydim$)
	such that the $(N,\Blowupexp,\Worstexp,\Room)$-dependent norms
	from Defs.\,\ref{D:LOWNORMS} and~\ref{D:HIGHNORMS} 
	verify the following
	estimate for $t \in (0,1]$:
	\begin{align} \label{E:MAINTHMAPRIORIESTIMATE}
	\MetricLownorm(t)
		+
	\MetricHighnorm(t)
		+
	\LapseLownorm(t)
		+
	\LapseHighnorm(t)
	& \leq C \mathring{\upepsilon}.
\end{align}

\medskip

\noindent \underline{\textbf{Description of the maximal globally hyperbolic development and curvature blowup}}.
The spacetime Kretschmann scalar verifies the following estimate
for $t \in (0,1]$:
\begin{align} \label{E:MATINTHMKRETSCHMANNSCALARESTIMATE}
		t^4 \Riemfour_{\alpha \beta \gamma \delta} \Riemfour^{\alpha \beta \gamma \delta}
	&
	=
	4 
	\left\lbrace
		\sum_{i=1}^{\mydim} (q_i^2 - q_i)^2 
		+ 
		\sum_{1 \leq i < j = \mydim} q_i^2 q_j^2
	\right\rbrace
	+
	\mathcal{O}(\mathring{\upepsilon}).
\end{align}
In particular, for $\mathring{\upepsilon}$ sufficiently small, 
$\Riemfour_{\alpha \beta \gamma \delta} \Riemfour^{\alpha \beta \gamma \delta}$
blows up like $\lbrace C + \mathcal{O}(\mathring{\upepsilon}) \rbrace t^{-4}$ as $t \downarrow 0$ 
(where 
$C = 4 
	\left\lbrace
		\sum_{i=1}^{\mydim} (q_i^2 - q_i)^2 
		+ 
		\sum_{1 \leq i < j = \mydim} q_i^2 q_j^2
	\right\rbrace$).
Consequently, the maximal (classical) globally hyperbolic development of the data
is $((0,1] \times \mathbb{T}^{\mydim},\gfour)$, and $\gfour$ cannot be continued
as a $C^2$ Lorentzian metric to the past of the singular hypersurface $\Sigma_0$.
That is, the past of $\Sigma_1$ in the maximal (classical) globally hyperbolic development of the data
is foliated by the family of spacelike hypersurfaces $\Sigma_t$, 
along which the CMC condition $\SecondFund_{\ a}^a = - t^{-1}$ holds.

\medskip
\noindent \underline{\textbf{Bounded length of past-directed causal geodesics}}.
Let $\pmb{\zeta}$ be any past-directed causal geodesic 
that emanates from $\Sigma_1$, 
and let $\mathscr{A} = \mathscr{A}(t)$ be an affine parameter along $\pmb{\zeta}$,
where $\mathscr{A}$ is viewed as a function of $t$ along $\pmb{\zeta}$ 
that is normalized by $\mathscr{A}(1) = 0$.
Then $\pmb{\zeta}$ crashes into the singular hypersurface $\Sigma_0$ in finite affine parameter time
	\begin{align} \label{E:MAINTHMAFFINTEBLOWUPTIME}
	\mathscr{A}(0) \leq \frac{\left|\mathscr{A}^{'}(1)\right|}{1 - \Worstexp},
\end{align}	
where 
$
\displaystyle
\mathscr{A}^{'}(t) := \frac{d}{dt} \mathscr{A}(t)
$.
\end{theorem}

\begin{proof}
	We first fix $\Blowupexp$ and $N$ to be large enough so that
	all of the estimates 
	proved (under bootstrap assumptions)
	in the previous sections hold true.
	By standard local well-posedness, 
	if $\mathring{\upepsilon}$ is sufficiently small
	and $C'$ is sufficiently large,
	then there exists a maximal time $T_{(Max)} \in [0,1)$
	such that the solution $(g,\SecondFund,n)$
	exists classically for $(t,x) \in (T_{(Max)},1] \times \mathbb{T}^{\mydim}$
	and such that the bootstrap assumptions \eqref{E:BOOTSTRAPASSUMPTIONS}
	hold with 
	$\TBoot := T_{(Max)}$
	and
	$\varepsilon := C' \mathring{\upepsilon}$.
	By enlarging $C'$ if necessary, we can assume that
	$C' \geq 2 C_{N,\Blowupexp,\Worstexp,\Room,\mydim}$, where
	$C_{N,\Blowupexp,\Worstexp,\Room,\mydim} > 1$ is the constant from inequality \eqref{E:MAINAPRIORIESTIMATE}.
	Readers can consult \cite{lAvM2003} for the main ideas behind the proof of local well-posedness, 
	in a similar but distinct gauge for Einstein's equations.
	Moreover, in view of Defs.\,\ref{D:LOWNORMS} and~\ref{D:HIGHNORMS},
	it is a standard result (again, see \cite{lAvM2003} for the main ideas)
	that if $\varepsilon$ is sufficiently small,
	then either $T_{(Max)} = 0$
	or the bootstrap assumptions are saturated on
	the time interval $(T_{(Max)},1]$, that is,
	$
	\sup_{t \in (T_{(Max)},1]}
	\left\lbrace
		\MetricLownorm(t)
		+
		\MetricHighnorm(t)
		+
		\LapseLownorm(t)
		+
		\LapseHighnorm(t)
	\right\rbrace
	=
	C' \mathring{\upepsilon}
	$.
	The latter possibility is ruled out by inequality \eqref{E:MAINAPRIORIESTIMATE}.
	Thus, $T_{(Max)} = 0$.
	In particular, the solution exists classically for
	$(t,x) \in (0,1] \times \mathbb{T}^{\mydim}$,
	and the estimate \eqref{E:MAINTHMAPRIORIESTIMATE} holds
	for $t \in (0,1]$.
	
	The remaining aspects of the theorem follow from
	Lemmas~\ref{L:KRETSCHMANNSCALARESTIMATE} and~\ref{L:LENGTHOFCAUSALGEODESICS}.

\end{proof}

\bibliographystyle{amsalpha}
\bibliography{JBib}

\end{document}